\documentclass{amsart}
\usepackage{amsmath,amscd,xypic,amssymb,combelow,color,enumitem,graphicx,float}

\usepackage[normalem]{ulem}
\usepackage[dvipsnames]{xcolor}
\makeatletter
\def\squiggly{\bgroup \markoverwith{\textcolor{red}{\lower3.5\p@\hbox{\sixly \char58}}}\ULon}
\makeatother

\xyoption{all}
\CompileMatrices

\makeatletter
\@addtoreset{equation}{section}
\makeatother

\newtheorem{theorem}[subsection]{Theorem}

\newtheorem{proposition}[subsection]{Proposition}
\newtheorem{lemma}[subsection]{Lemma}
\newtheorem{corollary}[subsection]{Corollary}

\newtheorem{definition}[subsection]{Definition}

\newtheorem{claim}[subsection]{Claim}

\newtheorem{remark}[subsection]{Remark}

\def\loccit{\emph{loc. cit. }}

\def\fg{{\mathfrak{g}}}

\def\fsl{{\mathfrak{sl}}}
\def\fgl{{\mathfrak{gl}}}

\def\hgl{{\widehat{\fgl}}}

\def\BC{{\mathbb{C}}}

\def\BN{{\mathbb{N}}}

\def\BP{{\mathbb{P}}}

\def\BQ{{\mathbb{Q}}}
\def\BZ{{\mathbb{Z}}}

\def\woo{\widehat{\otimes}}

\def\CA{{\mathcal{A}}}
\def\CB{{\mathcal{B}}}
\def\DD{{\mathcal{D}}}
\def\CE{{\mathcal{E}}}

\def\CS{{\mathcal{S}}}
\def\CT{{\mathcal{T}}}

\newcommand\CV{{\mathcal{V}}}

\def\tCA{\widetilde{\CA}}
\def\tCB{\widetilde{\CB}}

\def\tR{\widetilde{R}}
\def\oR{\overline{R}}

\def\ph{\varphi}
\def\coop{{\textrm{coop}}}

\def\e{\varepsilon}

\def\vs{\varsigma}

\def\and{\textrm{ }\&\textrm{ }}

\def\sym{\textrm{Sym}}

\def\esym{\emph{Sym}}

\def\eop{\emph{op}}
\def\op{\text{op}}

\def\Tr{\textrm{Tr}}

\def\sym{\textrm{Sym}}

\def\nn{{{\BN}}^n}
\def\zz{{{{\mathbb{Z}}}^n}}

\def\su{{U_q(\dot{\fsl}_n)}}

\def\uui{{U_q(\dot{\fgl}_1)}}
\def\uuip{{U_q^+(\dot{\fgl}_1)}}

\def\uuig{{U_q^\geq(\dot{\fgl}_1)}}
\def\uuil{{U_q^\leq(\dot{\fgl}_1)}}

\def\uu{{U_q(\dot{\fgl}_n)}}
\def\uupm{{U_q^\pm(\dot{\fgl}_n)}}

\def\uug{{U_q^\geq(\dot{\fgl}_n)}}
\def\uul{{U_q^\leq(\dot{\fgl}_n)}}

\def\UU{{U_{q,\oq}(\ddot{\fgl}_n)}}
\def\UUp{{U_{q,\oq}^+(\ddot{\fgl}_n)}}
\def\UUm{{U_{q,\oq}^-(\ddot{\fgl}_n)}}
\def\UUpm{{U_{q,\oq}^\pm(\ddot{\fgl}_n)}}
\def\UUg{{U_{q,\oq}^\geq(\ddot{\fgl}_n)}}
\def\UUl{{U_{q,\oq}^\leq(\ddot{\fgl}_n)}}

\def\UUup{{U_{q,\oq}^\uparrow(\ddot{\fgl}_n)}}
\def\UUdown{{U_{q,\oq}^\downarrow(\ddot{\fgl}_n)}}
\def\UUleft{{U_{q,\oq}^\leftarrow(\ddot{\fgl}_n)}}
\def\UUright{{U_{q,\oq}^\rightarrow(\ddot{\fgl}_n)}}

\def\tUUup{{\widetilde{U}_{q,\oq}^\uparrow(\ddot{\fgl}_n)}}
\def\tUUdown{{\widetilde{U}_{q,\oq}^\downarrow(\ddot{\fgl}_n)}}
\def\tUUleft{{\widetilde{U}_{q,\oq}^\leftarrow(\ddot{\fgl}_n)}}
\def\tUUright{{\widetilde{U}_{q,\oq}^\rightarrow(\ddot{\fgl}_n)}}

\def\bd{{\mathbf{d}}}
\def\be{{\mathbf{e}}}

\def\bs{{\boldsymbol{\vs}}}

\def\bde{{\boldsymbol{\delta}}}

\def\oq{{\overline{q}}}

\def\bari{\bar{i}}
\def\barj{\bar{j}}

\def\bE{\bar{E}}
\def\bF{\bar{F}}

\def\End{\text{End}}
\def\Res{\text{Res}}
\def\Aut{\text{Aut}}

\def\eRes{\emph{Res}}
\def\eEnd{\emph{End}}

\def\ebig{\emph{big}}

\def\tCA{\widetilde{\CA}}

\def\zzz{\frac {\BZ^2}{(n,n)\BZ}}

\def\slope{\text{slope}}

\def\bF{\bar{F}}

\def\bQ{\bar{Q}}

\def\barb{\bar{b}}
\def\barc{\bar{c}}
\def\barf{\bar{f}}
\def\slz{SL_2(\BZ)}

\def\hdeg{\text{hdeg }}
\def\vdeg{\text{vdeg }}

\def\lhs{\text{LHS}}
\def\rhs{\text{RHS}}
\def\oo{\overline}

\def\mtd{\text{m.i.d. }}

\def\fff{\BQ(q,\oq^{\frac 1n})}

\begin{document}

\title[A tale of two shuffle algebras]{\large{\textbf{A tale of two shuffle algebras}}}
\author[Andrei Negu\cb t]{Andrei Negu\cb t}
\address{MIT, Department of Mathematics, Cambridge, MA, USA}
\address{Simion Stoilow Institute of Mathematics, Bucharest, Romania}
\email{andrei.negut@gmail.com}

\maketitle

%\renewcommand{\thefootnote}{\fnsymbol{footnote}} 
%\footnotetext{\emph{2010 Mathematics Subject Classification: } 14D20, 14J60}     
%\renewcommand{\thefootnote}{\arabic{footnote}} 

\begin{abstract} As a quantum affinization, the quantum toroidal algebra $\UU$ is defined in terms of its ``left" and ``right" halves, which both admit shuffle algebra presentations (\cite{E, FO}). In the present paper, we take an orthogonal viewpoint, and give shuffle algebra presentations for the ``top" and ``bottom" halves instead, starting from the evaluation representation $\uu \curvearrowright \BC^n(z)$ and its usual $R$--matrix $R(z) \in \End(\BC^n \otimes \BC^n)(z)$ (see \cite{FRT}). An upshot of this construction is a new topological coproduct on $\UU$ which extends the Drinfeld-Jimbo coproduct on the horizontal subalgebra $\uu \subset \UU$.
	
\end{abstract}

\section{Introduction}

\subsection{} The affine quantum group $\su = U_q(\widehat{\fsl}_n)$ (hats will be replaced by points in the present paper) has the following two presentations: \\

\begin{itemize}

\item as the affinization of $U_q(\fsl_n)$ \\
	
\item as the Drinfeld-Jimbo quantum group whose Dynkin diagram is an $n$-cycle \\

\end{itemize} 

\noindent However, the two presentations above yield different bialgebra structures on $\su$, which is evidenced by the fact that the coproduct in the first bullet is only topological (i.e. $\Delta$ is an infinite sum, which only makes sense in a certain completion). Moreover, the two bullets above yield different triangular decompositions of $\su$ into positive, Cartan, and negative halves:
\begin{equation}  
\label{eqn:quantum 1}
\su \cong U_q^\leftarrow(\dot{\fsl}_n) \otimes \left( \text{Cartan subalgebra} \right) \otimes U_q^\rightarrow(\dot{\fsl}_n)  
\end{equation}
\begin{equation}  
\label{eqn:quantum 2}
\su \ \cong \ U_q^\uparrow(\dot{\fsl}_n) \otimes \left( \text{Cartan subalgebra} \right) \otimes U_q^\downarrow(\dot{\fsl}_n)
\end{equation} 
The two decompositions above are quite different: the positive subalgebra $U_q^\rightarrow(\dot{\fsl}_n)$ of \eqref{eqn:quantum 1} is generated by Drinfeld's elements $e_{i,k}$ over all $1 \leq i < n$ and $k \in \BZ$, while the positive subalgebra $U_q^\uparrow(\dot{\fsl}_n)$ of \eqref{eqn:quantum 2} is generated by the Drinfeld-Jimbo elements $\{e_i\}_{i \in \BZ/n\BZ}$. The connection between these two presentations was given in \cite{B}. \\

\subsection{} 
\label{sub:left right} 

The main purpose of the present paper is to extend the description above to the quantum toroidal algebra $\UU$, which is defined as in the first bullet above:
$$
\UU := \text{affinization of } \uu 
$$
This construction naturally comes with a triangular decomposition (see Subsection \ref{sub:uu} for an overview of the quantum toroidal algebra, as well as of our conventions):
\begin{equation}
\label{eqn:decomp 0}
\UU \cong \tUUleft \otimes \tUUright
\end{equation}
Our $\tUUleft$ and $\tUUright$ are the ``Borel" subalgebras of the quantum toroidal algebra, and they explicitly arise as tensor products:
\begin{align} 
&\tUUright \cong \UUright \otimes \uuig^{n} \label{eqn:nord} \\
&\tUUleft \cong \UUleft \otimes \uuil^{n} \label{eqn:sud}
\end{align}
There is a well-known topological coproduct of $\UU$, which preserves the subalgebras \eqref{eqn:nord} and \eqref{eqn:sud}, and extends the (almost) cocommutative coproduct on the ``vertical" subalgebra:
\begin{equation}
\label{eqn:ver ver}
\uuig^{n} \otimes \uuil^{n} = \uui^{n} \subset \UU
\end{equation}
The main goal of this paper is to define another decomposition into subalgebras:
\begin{equation}
\label{eqn:decomp 1}
\UU \cong \tUUup \otimes \tUUdown 
\end{equation}
(see Corollary \ref{cor:explicit}). We will explicitly construct the tensor factors of \eqref{eqn:decomp 1} as:
\begin{align} 
&\tUUup \cong \UUup \otimes \uug \label{eqn:half plus} \\
&\tUUdown \cong \UUdown \otimes \uul \label{eqn:half minus}
\end{align}
where the ``horizontal" subalgebra:
\begin{equation}
\label{eqn:hor hor}
\uug \otimes \uul = \uu \subset \UU
\end{equation}
will be the quantum group in the RTT presentation (\cite{FRT}). Moreover, we endow $\UU$ with a new topological coproduct which preserves the subalgebras \eqref{eqn:half plus}, \eqref{eqn:half minus}, and extends the usual (Drinfeld-Jimbo) coproduct on $\uu \subset \UU$. \\

\subsection{} To represent the aforementioned decompositions pictorially, we will recall that the quantum toroidal algebra is graded by $\zz \times \BZ$, where $\zz$ is the root lattice of $\su$ and $\BZ$ is the affinization direction. Then the following picture indicates the various subalgebras of $\UU$, by displaying which degrees they live in:

\begin{picture}(100,230)(-110,-75)
\label{pic:par}

\put(0,0){\circle*{2}}\put(20,0){\circle*{2}}\put(40,0){\circle*{2}}\put(60,0){\circle*{2}}\put(80,0){\circle*{2}}\put(100,0){\circle*{2}}\put(120,0){\circle*{2}}\put(0,20){\circle*{2}}\put(20,20){\circle*{2}}\put(40,20){\circle*{2}}\put(60,20){\circle*{2}}\put(80,20){\circle*{2}}\put(100,20){\circle*{2}}\put(120,20){\circle*{2}}\put(0,40){\circle*{2}}\put(20,40){\circle*{2}}\put(40,40){\circle*{2}}\put(60,40){\circle*{2}}\put(80,40){\circle*{2}}\put(100,40){\circle*{2}}\put(120,40){\circle*{2}}\put(0,60){\circle*{2}}\put(20,60){\circle*{2}}\put(40,60){\circle*{2}}\put(60,60){\circle*{2}}\put(80,60){\circle*{2}}\put(100,60){\circle*{2}}\put(120,60){\circle*{2}}\put(0,80){\circle*{2}}\put(20,80){\circle*{2}}\put(40,80){\circle*{2}}\put(60,80){\circle*{2}}\put(80,80){\circle*{2}}\put(100,80){\circle*{2}}\put(120,80){\circle*{2}}\put(0,100){\circle*{2}}\put(20,100){\circle*{2}}\put(40,100){\circle*{2}}\put(60,100){\circle*{2}}\put(80,100){\circle*{2}}\put(100,100){\circle*{2}}\put(120,100){\circle*{2}}\put(0,120){\circle*{2}}\put(20,120){\circle*{2}}\put(40,120){\circle*{2}}\put(60,120){\circle*{2}}\put(80,120){\circle*{2}}\put(100,120){\circle*{2}}\put(120,120){\circle*{2}}

\put(60,-10){\vector(0,1){140}}
\put(-10,60){\vector(1,0){140}}

\put(-25,10){\scalebox{4}{$\{$}}
\put(-63,17){$U^\downarrow_{q,\oq}(\ddot{\fgl}_n)$}
\put(135,56){$\uug$}
\put(-25,90){\scalebox{4}{$\{$}}
\put(-63,97){$U^\uparrow_{q,\oq}(\ddot{\fgl}_n)$}
\put(-48,57){$\uul$}
\put(0,-15){\scalebox{4}{\rotatebox{270}{$\}$}}}
\put(-15,-43){$\UUleft$}
\put(80,-15){\scalebox{4}{\rotatebox{270}{$\}$}}}
\put(81,-43){$\UUright$}
\put(63,120){$\BZ$}
\put(120,63){$\zz$}
\put(42,-18){$\uuil^{n}$}
\put(42,135){$\uuig^{n}$}

\put(-80,-65){F{\scriptsize IGURE} 0. The grading of $\UU$ and its various subalgebras}

\end{picture}

\noindent In the particular case $n=1$, the quantum toroidal algebra is isomorphic to the well-known Ding-Iohara-Miki algebra (\cite{DI, M}) a.k.a. the elliptic Hall (\cite{BS, S}) algebra, on which the universal cover of $\slz$ acts by automorphisms. With respect to this action, the decomposition \eqref{eqn:decomp 1} is obtained from the decomposition \eqref{eqn:decomp 0} by applying the automorphism corresponding to rotation by 90 degrees. However, in the general $n$ case, the algebras featuring in the two decompositions are not isomorphic to each other, which is sensible given the fact that the grading axes $\zz$ and $\BZ$ are quite different. \\

\subsection{} 
\label{sub:describe} 

To describe $\UUright$ of \eqref{eqn:nord}, let us consider the vector space:
\begin{equation}
\label{eqn:old shuf intro}
\CS^+ \subset \bigoplus_{(d_1,...,d_n) \in \nn} \fff (z_{11},...,z_{1d_1},...,z_{n1},...,z_{nd_n})^{\sym}
\end{equation}
of rational functions which satisfy the wheel conditions (as in \cite{FHHSY,FO}): namely that such rational functions have at most simple poles at $z_{ia} q^2 - z_{i+1,b}$ (for all $i,a,b$) and that the residue at such a pole is divisible by $z_{i a'} - z_{i+1,b}$ and $z_{ia} - z_{i+1,b'}$ for all $a' \neq a$ and $b' \neq b$. The vector space \eqref{eqn:old shuf intro} is called a shuffle algebra, akin to the classical construction of Feigin and Odesskii concerning certain elliptic algebras (\cite{FO}). Explicitly, the product on \eqref{eqn:old shuf intro} is constructed using the rational function \eqref{eqn:def zeta}, see Definition \ref{def:shuf classic}. An algebra homomorphism was constructed in \cite{E}:
$$
\UUright \rightarrow \CS^+
$$
and it was shown to be an isomorphism in \cite{Tor}. Similarly, $\UUleft \cong \CS^- := (\CS^+)^\op$. \\

\noindent To describe the subalgebras $U^\uparrow_{q,\bar{q}}(\ddot{\fgl}_n)$ and $U^\downarrow_{q,\bar{q}}(\ddot{\fgl}_n)$ of \eqref{eqn:half plus}--\eqref{eqn:half minus}, we will introduce a new kind of shuffle algebra (let $V$ be an $n$--dimensional vector space):
\begin{equation}
\label{eqn:new shuf intro}
\CA^+ \subset \bigoplus_{k=0}^\infty \End_{\fff}(\underbrace{V \otimes ... \otimes V}_{k \text{ factors}})(z_1,...,z_k)
\end{equation}
and the algebra structure on the RHS is constructed using the $R$--matrix \eqref{eqn:r}, see Propositions \ref{prop:shuf aff 1} and \ref{prop:shuf aff 2}. By definition, the subspace \eqref{eqn:new shuf intro} precisely consists of $\End(V^{\otimes k})$--valued rational functions which have at most simple poles at $z_a \oq^2 - z_b$ (for all $a,b$) and whose residue at such a pole satisfies the conditions outlined in Definition \ref{def:shuf aff}. The subalgebra $\CA^{-}$ is defined similarly, but with $\oq^{-1} q^{-n}$ instead of $\oq$. \\

\begin{theorem}
\label{thm:main}
	
There exist injective algebra homomorphisms:
$$
\CA^+, \CA^{-,\eop} \hookrightarrow \UU
$$
Denoting the images of these maps by $U^\uparrow_{q,\oq}(\ddot{\fgl}_n)$ and $U^\downarrow_{q,\oq}(\ddot{\fgl}_n)$ yields the decompositions featuring in \eqref{eqn:decomp 1}, \eqref{eqn:half plus}, \eqref{eqn:half minus}. Moreover, there exist topological coproducts on the subalgebras:
\begin{align}
\tCA^+ = \CA^+ \otimes \uug \ \hookrightarrow \ \UU \label{eqn:half 1} \\
\tCA^{-,\eop} = (\CA^- \otimes \uul)^\eop  \ \hookrightarrow \ \UU \label{eqn:half 2}
\end{align}
which extend the Drinfeld-Jimbo coproduct on the horizontal subalgebra \eqref{eqn:hor hor}, and realize $\UU$ as the Drinfeld double of its subalgebras \eqref{eqn:half 1} and \eqref{eqn:half 2}. \\ 
	
\end{theorem}

\noindent We emphasize the fact that $\UUup$ is not the same as the ``vertical subalgebra" that was studied in \cite{FJMM} and numerous other works. The latter construction has to do with $\su$ presented as the affinization of $U_q(\fsl_n)$ and thus implicitly breaks the symmetry among the vertices of the cyclic quiver. Meanwhile, our construction takes the ``horizontal subalgebra" $\uu$ and its evaluation representation $V = \BC^n(z)$ as an input, and outputs half of the quantum toroidal algebra. \\

\noindent More generally, starting from a quantum group $U_q(\fg)$ and a representation $V$ endowed with a unitary $R$--matrix, one may ask if the double shuffle algebra:
\begin{equation}
\label{eqn:mo}
\mathcal{D} \left( \text{an appropriate subalgebra of } \bigoplus_{k=0}^\infty \End(V^{\otimes k}) \right)
\end{equation}
(defined as in Section \ref{sec:old}) is related to the quantum group $U_q(\dot{\fg})$. Theorem \ref{thm:main} deals with the case $\fg = \dot{\fgl}_n$ and $V = \BC^n(z)$. If something along these lines is true in affine types other than $A$, we venture to speculate that the algebra \eqref{eqn:mo} might be related to the extended Yangians of \cite{W}, appropriately $q$--deformed and doubled. \\

\subsection{} The structure of the present paper is the following: \\

\begin{itemize} [leftmargin=*]
	
\item In Section \ref{sec:old}, we construct a \textbf{shuffle algebra} $\CA^+$ starting from a vector space $V$ and a unitary $R$--matrix $\in \End(V^{\otimes 2})$ (see also \cite{Mu}). By adding certain elements, we construct the \textbf{extended} shuffle algebra $\tCA^+$, which admits a coproduct. From two such extended shuffle algebras, we construct their Drinfeld \textbf{double} $\CA$. \\
	
\item In Section \ref{sec:quantum}, we recall the quantum group $\UU$ and its PBW presentation from \cite{PBW}. This will allow us to construct the decomposition \eqref{eqn:decomp 1} as algebras. \\
		
\item In Section \ref{sec:new}, we construct a version of the \textbf{shuffle algebra} of Section \ref{sec:old} that corresponds to the $R$--matrix with spectral parameter \eqref{eqn:r}, thus yielding \eqref{eqn:new shuf intro}. \\
			
\item In Section \ref{sec:extended}, we construct the \textbf{extended} version of the shuffle algebra of Section \ref{sec:new}, endow it with a topological coproduct, and construct a PBW basis of it. \\

\item In Section \ref{sec:double}, we construct a bialgebra pairing between two copies of the extended shuffle algebras of Section \ref{sec:extended}. The corresponding Drinfeld \textbf{double} will precisely match $\UU$, thus completing the proof of Theorem \ref{thm:main}. \\
	
\end{itemize}

\noindent I would like to thank Pavel Etingof, Sachin Gautam, Victor Kac, Andrei Okounkov, and Alexander Tsymbaliuk for many valuable conversations, and all their help along the years. I gratefully acknowledge the NSF grants DMS--1600375,  DMS--1760264 and DMS--1845034, as well as support from the Alfred P. Sloan Foundation. \\

\subsection{} 
\label{sub:notation}

Given a finite-dimensional vector space $V$, we will often write elements $X \in \End(V^{\otimes k})$ as $X_{1...k}$ in order to point out the set of indices of $X$. If $V = \BC^n$, then:
\begin{equation}
\label{eqn:basis}
X = \sum_{i_1,...,i_k,j_1,...,j_k = 1}^n \text{coefficient} \cdot E_{i_1j_1} \otimes ... \otimes E_{i_kj_k}
\end{equation}
for certain coefficients, where $E_{ij} \in \End(V)$ denotes the matrix with entry 1 on row $i$ and column $j$, and 0 everywhere else. For any permutation $\sigma \in S(k)$, we write:
\begin{equation}
\label{eqn:conjugation}
\sigma X \sigma^{-1} = X_{\sigma(1)...\sigma(k)}
\end{equation}
where $\sigma \curvearrowright V^{\otimes k}$ by permuting the factors (therefore, the effect of conjugating \eqref{eqn:basis} by $\sigma$ is to replace the indices $i_1,...,j_k$ by $i_{\sigma(1)},...,j_{\sigma(k)}$). Moreover, we will write:
\begin{equation}
\label{eqn:pseudo sweedler}
X_{1...k} = X_{1...i} \otimes X_{i+1...k} \in \End(V^{\otimes i}) \otimes \End(V^{\otimes k-i}) \cong \End(V^{\otimes k})
\end{equation}
if we wish to set apart the first $i$ tensor factors from the last $k-i$ tensor factors of $X$. There is an implicit summation in the right-hand side of \eqref{eqn:pseudo sweedler} which we will not write down, much alike Sweedler notation. For any $a \in \BN$, we will write:
$$
E_{ij}^{(a)} = 1 \otimes ... \otimes \underbrace{E_{ij}}_{a\text{--th position}} \otimes \ ... \otimes 1 \in \End(V^{\otimes k})
$$
(the number $k \geq a$ will always be clear from context). More generally, for any $X \in \End(V^{\otimes k})$ and any collection of distinct natural numbers $a_1,...,a_k$, write:
$$
X_{a_1...a_k} \in \End(V^{\otimes N})
$$
(the number $N \geq a_1,...,a_k$ will always be clear from context) for the image of $X$ under the map $\End(V^{\otimes k}) \rightarrow \End(V^{\otimes N})$ that sends the $i$--th factor of the domain to the $a_i$--th factor of the codomain, and maps to the unit in all factors $\neq \{a_1,...,a_k\}$. \\

\section{Shuffle algebras and $R$--matrices}
\label{sec:old}

\subsection{} 
\label{sub:old}

The main goal of the present Section is to study shuffle algebras associated to the data contained in the four bullets below: \\

\begin{itemize}

\item a vector space $V$, assumed finite-dimensional for simplicity \\

\item an element ($R$--matrix) $R \in \Aut(V^{\otimes 2})$ satisfying the Yang-Baxter equation:
\begin{equation}
\label{eqn:ybe}
R_{12} R_{13} R_{23} = R_{23} R_{13} R_{12}
\end{equation}

\item an element $\tR \in \Aut(V^{\otimes 2})$ satisfying:
\begin{equation}
\label{eqn:quasi-ybe 1}
\tR_{21} \tR_{31} R_{23} = R_{23} \tR_{31} \tR_{21}
\end{equation}
\begin{equation}
\label{eqn:quasi-ybe 2}
R_{12} \tR_{31} \tR_{32} = \tR_{32} \tR_{31} R_{12}
\end{equation}

\item a scalar $f$ so that:
\begin{equation}
\label{eqn:unitary}
R_{12} R_{21} = f \cdot \text{Id}_{V \otimes V} = R_{21} R_{12}
\end{equation}

\end{itemize}

\noindent The present Section will be concerned with developing a certain algebraic framework in the generality above, while Section \ref{sec:new} will deal with the particular case of:
\begin{equation}
\label{eqn:r intro}
R(x) = \text{RHS of \eqref{eqn:explicit r}} \in \End(\BC^n \otimes \BC^n)(x)
\end{equation}
and $\tR(x) = R_{21}\left( x^{-1} \oq^{-2} \right)$, for a parameter $\oq$. Many Propositions in the current Section have counterparts in Section \ref{sec:new}, and we will only prove such statements once. \\

\subsection{}
\label{sub:braids}

We will represent the tensor product $V^{\otimes k}$ as $k$ labeled dots on a vertical line, and certain elements of $\End(V^{\otimes k})$ will be represented as braids between two such collections of $k$ labeled dots situated on parallel vertical lines (the labels will not change along strands, so they will be represented pictorially as colors). Specifically, the crossings below represent either the automorphisms $R$ or $\tR$, with indices given by the labels of the strands (which are inherited from the labels of their endpoints):
\begin{figure}[h]
\centering
\includegraphics[scale=0.55]{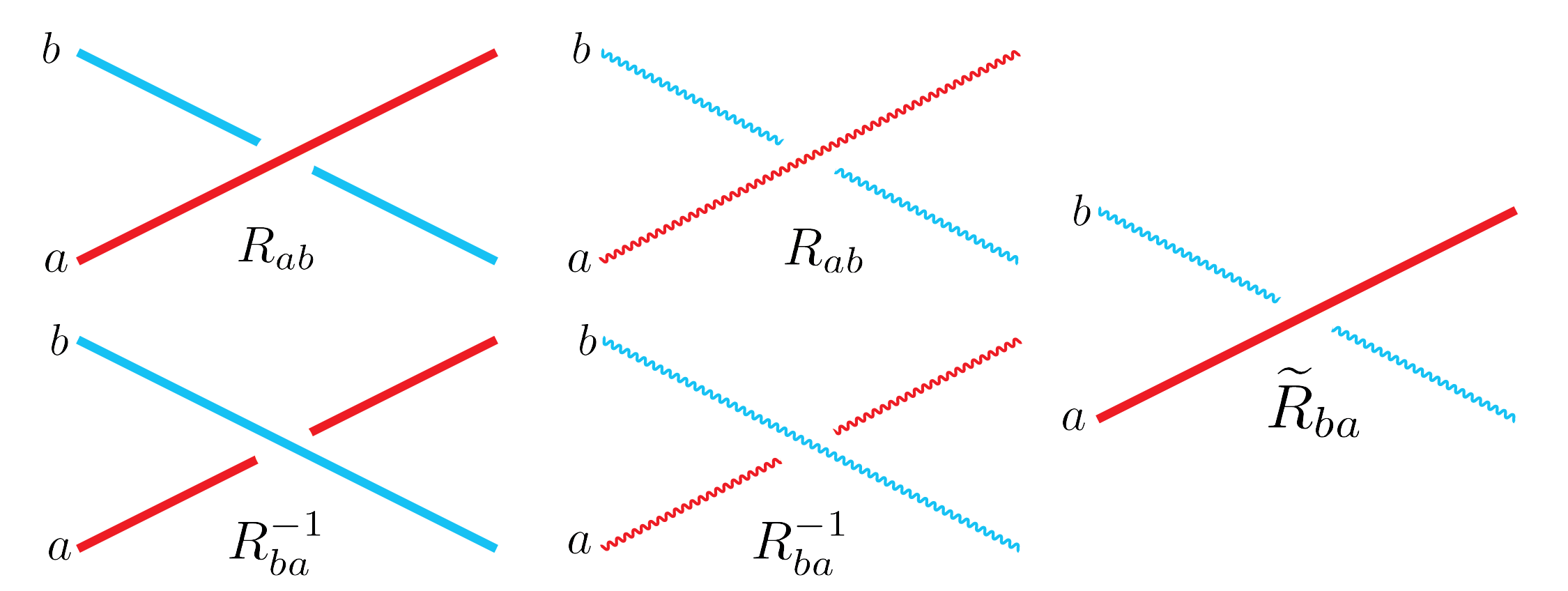} 
\caption{Various crossings}
\end{figure}

\noindent The strands are represented either as straight or squiggly lines, because we wish to indicate whether the picture in question refers to either $R$ or $\tR$. Compositions are always read left-to-right, for example the following equivalence of isotopy classes of braids underlies the Yang-Baxter relation \eqref{eqn:ybe}:
\begin{figure}[h]
\centering
\includegraphics[scale=0.5]{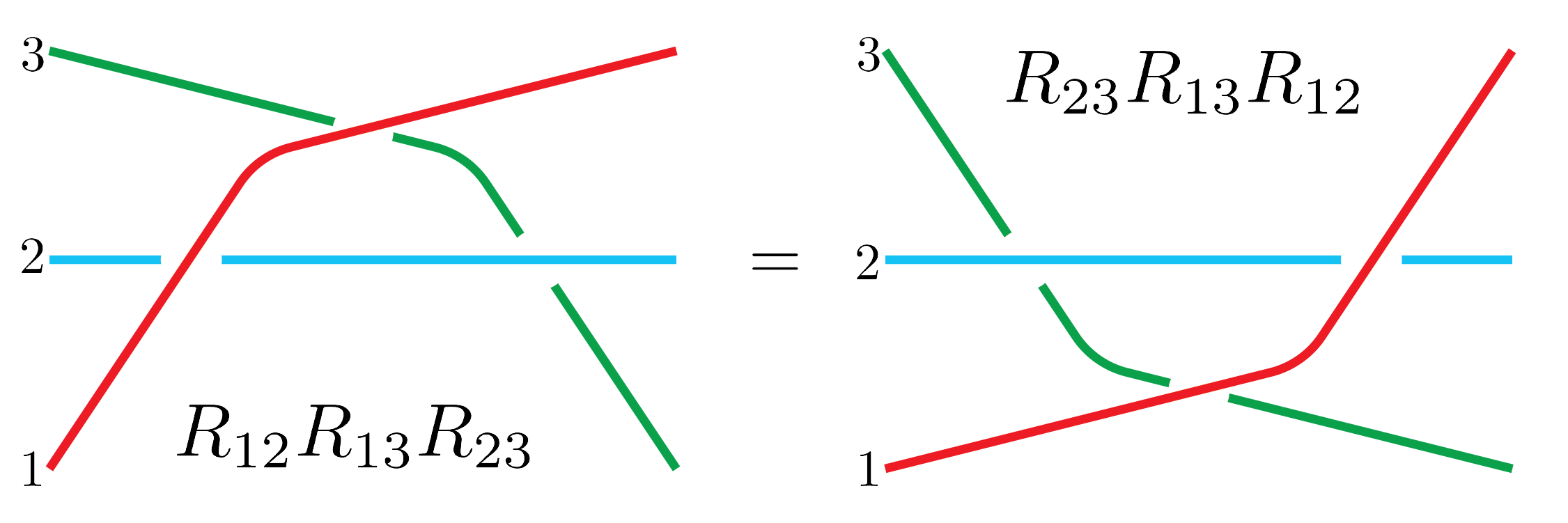} 
\caption{Reidemeister III move - version 1}
\end{figure}

\noindent The following equivalences underlie equations \eqref{eqn:quasi-ybe 1} and \eqref{eqn:quasi-ybe 2}, respectively:
\begin{figure}[H]
\centering
\includegraphics[scale=0.5]{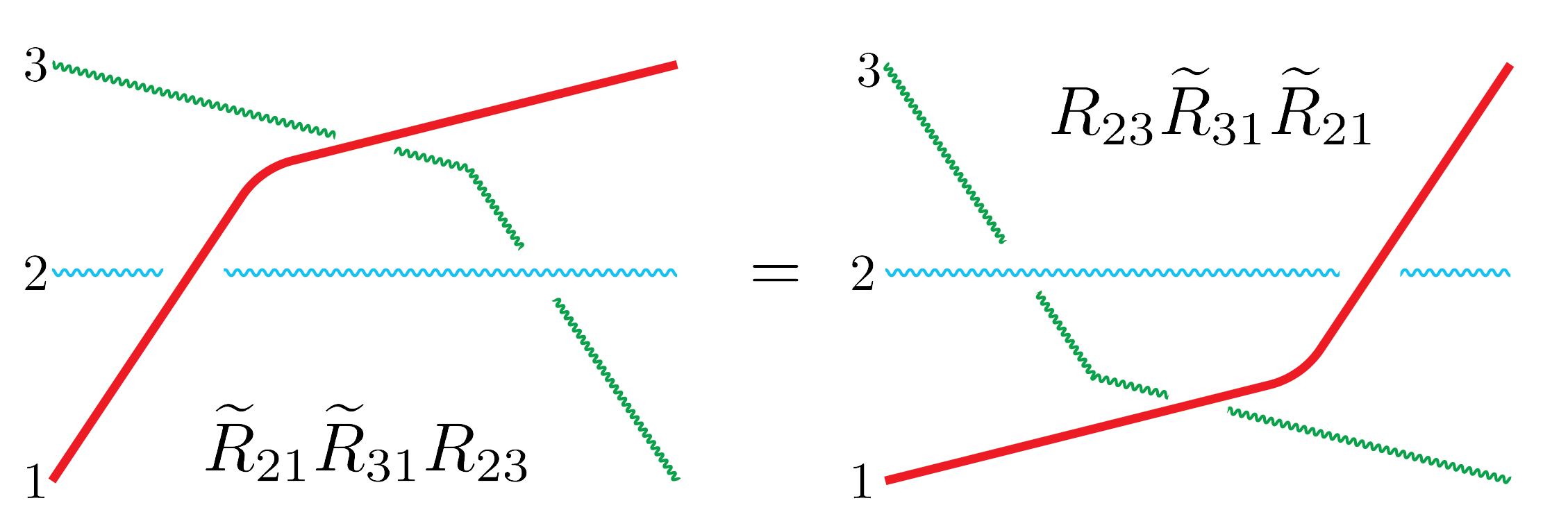}
\caption{Reidemeister III move - version 2}
\end{figure}
\begin{figure}[H]
\centering
\includegraphics[scale=0.5]{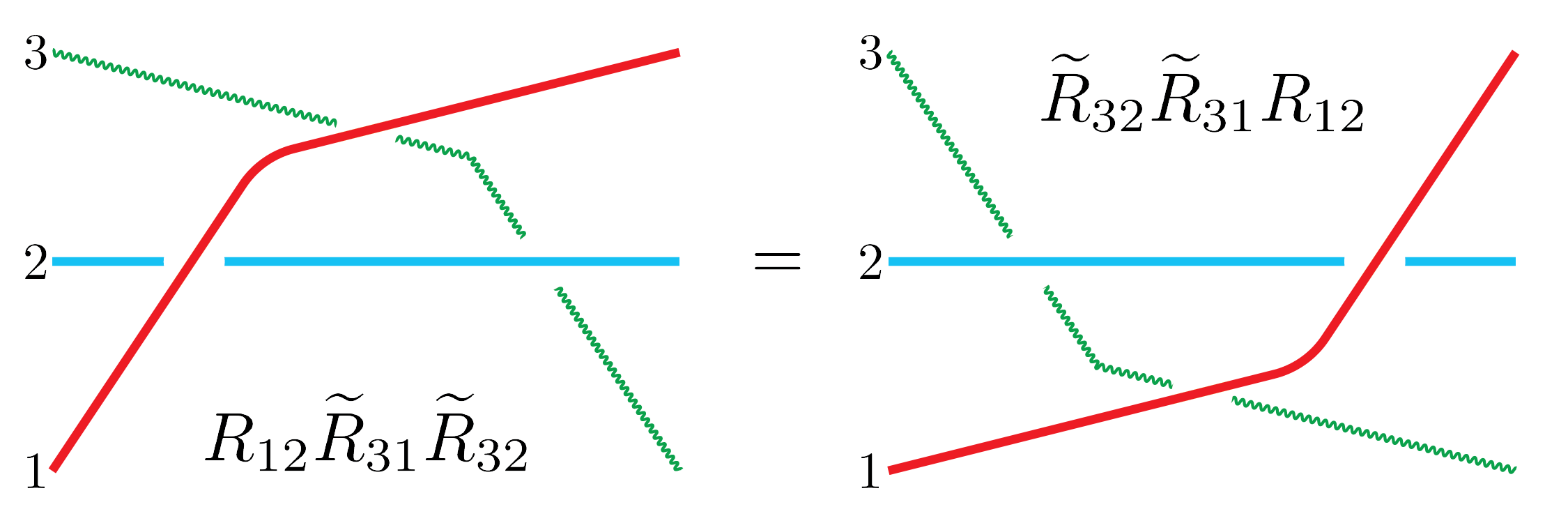}
\caption{Reidemeister III move - version 3}
\end{figure}

\noindent We will equivalate braids connected by the Reidemeister III type moves above. \\

\subsection{} 
\label{sub:classic shuf}

The following construction (when $\tR_{12} = R_{21}$) bears similarities with that of Section 5.2 of \cite{Mu}, itself a dual version of the construction of \cite{FRT}. We will then construct an extended shuffle algebra which admits a coproduct and bialgebra pairing, and then define the corresponding Drinfeld double. \\

\begin{proposition}
\label{prop:shuf 1}

For $V, R, \tR$ as in Subsection \ref{sub:old}, the assignment:
\begin{equation}
\label{eqn:shuf prod}
A_{1...k} * B_{1...l} = \sum^{a_1<...<a_k, \ b_1<...<b_l}_{\{1,...,k+l\} = \{a_1,...,a_k\} \sqcup \{b_1,...,b_l\}} 
\end{equation}
$$
\underbrace{\Big[ R_{a_kb_1} ... R_{a_1b_l} \Big]}_{\text{only if }a_i < b_j}
A_{a_1...a_k} \Big[ \tR_{a_1b_l} ... \tR_{a_kb_1} \Big] B_{b_1...b_l}  \underbrace{\Big[ R_{a_kb_1} ... R_{a_1b_l} \Big]}_{\text{only if }a_i > b_j}
$$
\footnote{The meaning of the indexing sets in the three products of $R$'s or $\tR$'s is that the factors in:
$$
\tR_{a_1b_l}... \tR_{a_kb_1} 
$$
are taken in any order such that $(a_i,b_j)$ is to the left of $(a_{i'},b_j)$ if $i<i'$ and to the right of $(a_i,b_{j'})$ if $j<j'$. The text ``only if $a_i < b_j$" or ``only if $a_i > b_j$" under such a product means that only those pairs of indices $(a_i,b_j)$ satisfying the respective inequalities occur in the product.} yields an associative algebra structure on the vector space:
\begin{equation}
\label{eqn:general shuffle}
\bigoplus_{k=0}^\infty \emph{End}(V^{\otimes k})
\end{equation}
with unit $1 \in \eEnd(V^{\otimes 0})$. We will call \eqref{eqn:shuf prod} the ``shuffle product". \\

\end{proposition}

\noindent We note that the second line of \eqref{eqn:shuf prod} can be represented by the braid in Figure 5, depicted here for $k=2$, $l=2$ and $(a_1,a_2) = (1,3)$, $(b_1,b_2) = (2,4)$. \\

\begin{figure}[ht]              
\centering
\includegraphics[scale=0.45]{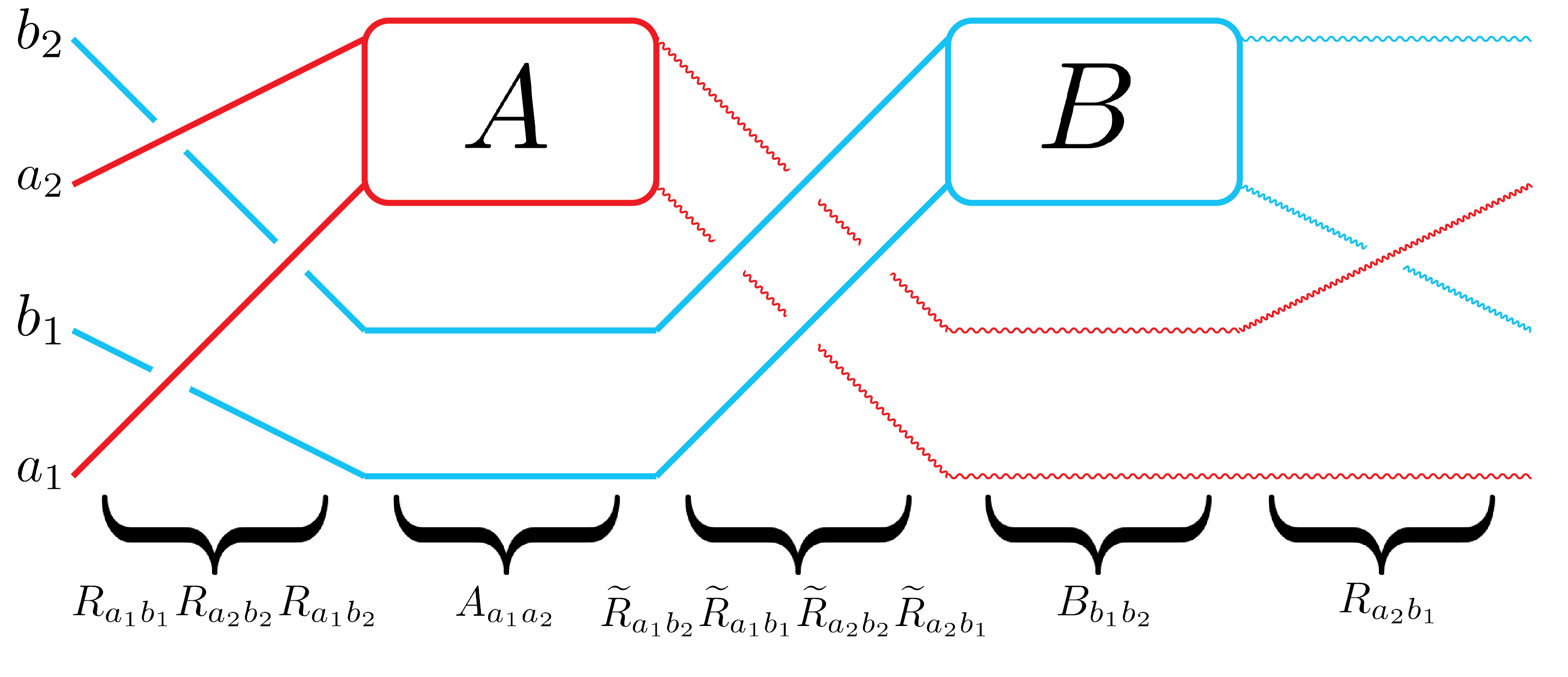}
\caption{$A*B$ as a braid}
\end{figure}

\begin{proof} The associativity of multiplication \eqref{eqn:shuf prod} is pictorially summarized by the equivalence of the braids in Figures 6 and 7. Indeed, in Figure 6, one can pull the straight red strands to the left of the blue-green crossings, the squiggly red strands below the blue-green crossings, and the straight green strands above the red-blue crossings. This procedure is simply a succession of the Reidemeister III moves of Figures 2,3 and 4, which in the end produces the braid in Figure 7. 

\end{proof}

\begin{figure}[ht]              
\centering
\includegraphics[scale=0.25]{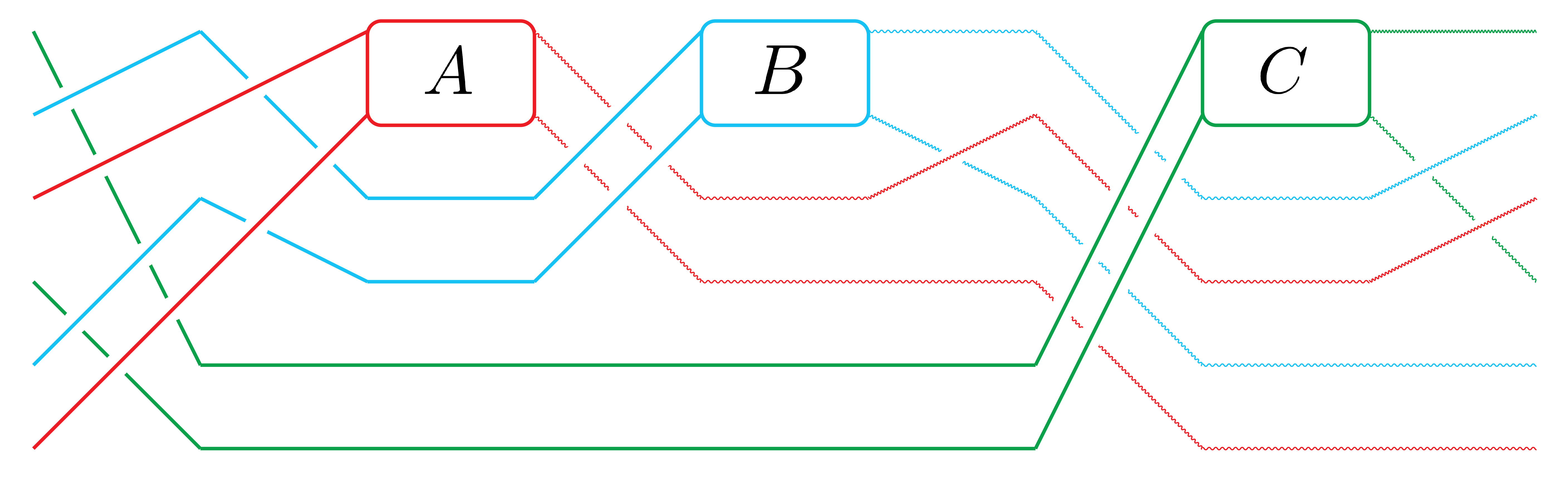}
\caption{$(A*B)*C$}
\end{figure}
\begin{figure}[ht]              
\centering
\includegraphics[scale=0.25]{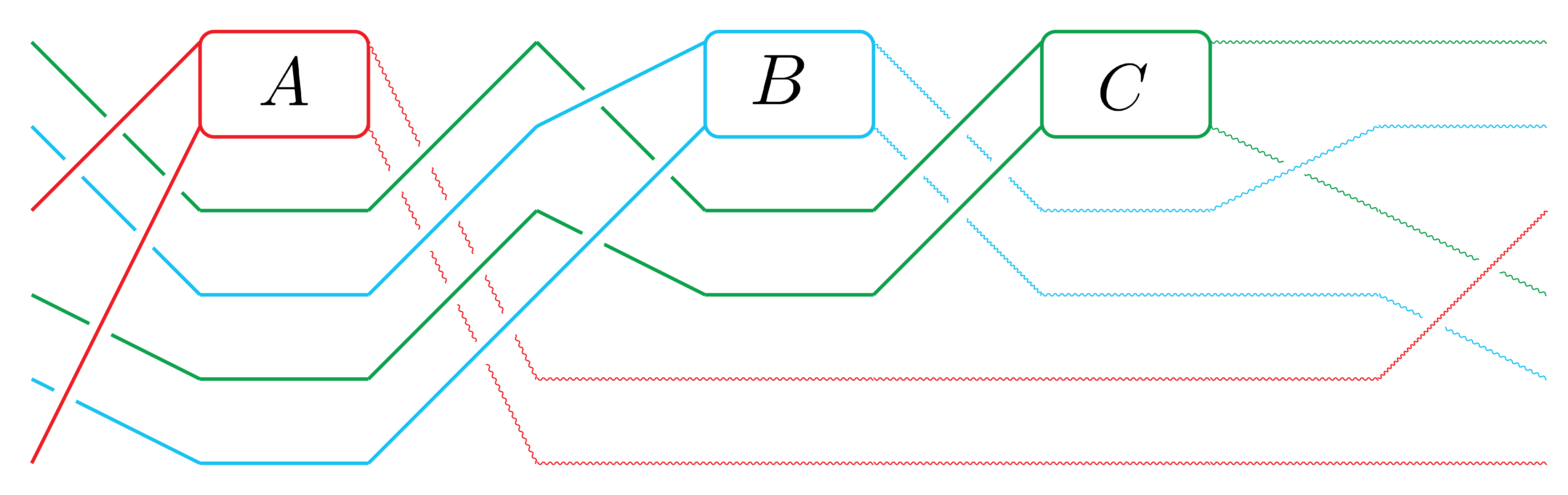}
\caption{$A*(B*C)$}
\end{figure}

\subsection{} A tensor $X \in \End(V^{\otimes k})$ is called symmetric if for all permutations $\sigma \in S(k)$:
\begin{equation}
\label{eqn:symmetry}
X = R_\sigma \cdot (\sigma X \sigma^{-1}) \cdot R_\sigma^{-1}
\end{equation}
where $\sigma X \sigma^{-1}$ is defined in \eqref{eqn:conjugation}, and $R_\sigma \in \End(V^{\otimes k})$ is any braid lift of $\sigma$ (i.e. any braid connecting the $i$--th endpoint on the right with the $\sigma(i)$--th endpoint on the left). Choosing one braid lift of $\sigma$ over another is just the ambiguity of choosing $R_{ab}$ over $R_{ba}^{-1}$ for any crossing between the strands labeled $a$ and $b$. Since \eqref{eqn:unitary} says that these two endomorphisms differ by a scalar, the ambiguity does not affect \eqref{eqn:symmetry}. \\

\noindent Pictorially, the right-hand side of \eqref{eqn:symmetry} is represented by the following braid:

\begin{figure}[ht]              
\centering
\includegraphics[scale=0.35]{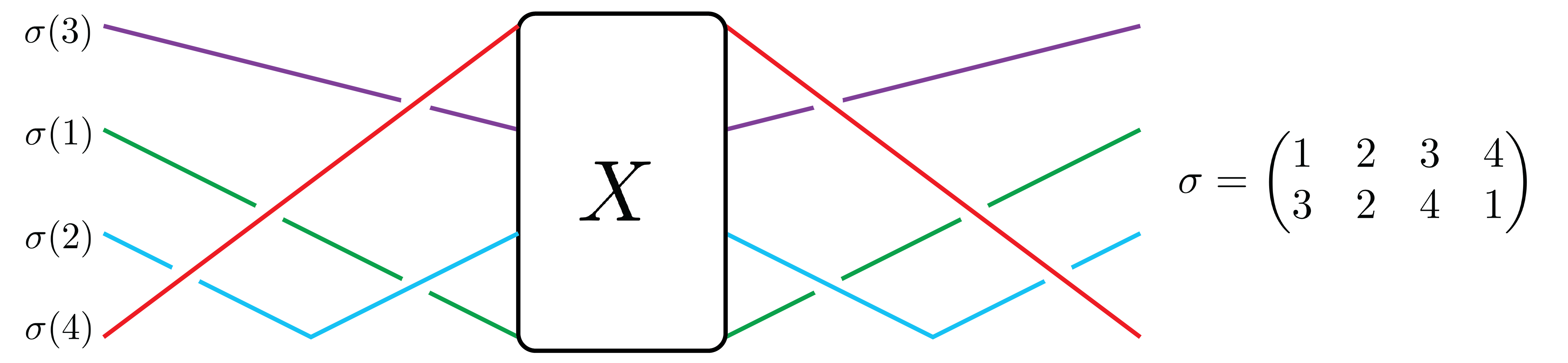}
\caption{A braid representation of $R_\sigma \cdot (\sigma X \sigma^{-1}) \cdot R_\sigma^{-1}$}
\end{figure}

\noindent The symmetrization of any tensor $X \in \End(V^{\otimes k})$ is defined as:
\begin{equation}
\label{eqn:symmetrization}
\sym \ X = \sum_{\sigma \in S(k)} R_\sigma \cdot (\sigma X \sigma^{-1}) \cdot R_\sigma^{-1}
\end{equation}
It is easy to see that the symmetrization of an abitrary tensor is symmetric. \\

\begin{proposition}
\label{prop:shuf 2}

The shuffle product of Proposition \ref{prop:shuf 1} preserves the vector space:
\begin{equation}
\label{eqn:general shuf}
\CA^+ \subset \bigoplus_{k=0}^\infty \emph{End}(V^{\otimes k})
\end{equation}
consisting of symmetric tensors. We will call $\CA^+$ the shuffle algebra. \\

\end{proposition}

\begin{proof} Let $A_{1...k} \in \End(V^{\otimes k})$ and $B_{1...l} \in \End(V^{\otimes l})$ be any symmetric tensors. A permutation $\mu \in S(k+l)$ is called a $(k,l)$--shuffle if:
\begin{equation}
\label{eqn:kl shuffle}
\begin{array}{l@{}cc}
a_1 := \mu(1) < ... < a_k := \mu(k) \\ \\
b_1 := \mu(k+1) < ... < b_l := \mu(k+l) 
\end{array}
\end{equation}
Because of the diagram (depicted for $k=2$, $l=2$, $(a_1,a_2) = (1,3)$, $(b_1,b_2) = (2,4)$):
\begin{figure}[H]              
\centering
\includegraphics[scale=0.3]{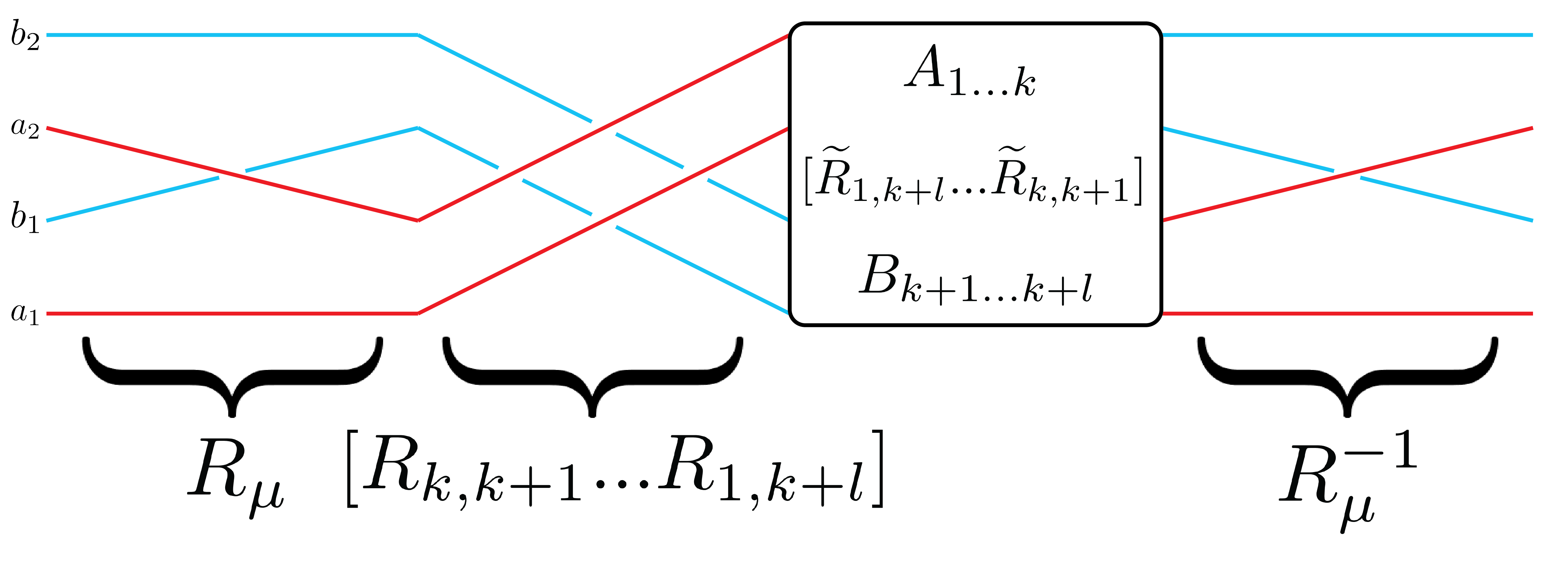}
\caption{$R_\mu \cdot (\mu \Phi\mu^{-1}) \cdot R_\mu^{-1}$}
\end{figure}

\noindent it is easy to see that the definition \eqref{eqn:shuf prod} can be restated as:
\begin{equation}
\label{eqn:shuf symm 1}
A * B = \sum^{\mu \text{ goes over}}_{(k,l)\text{--shuffles}} R_\mu \cdot (\mu \Phi \mu^{-1}) \cdot R_\mu^{-1}
\end{equation}
where:
\begin{equation}
\label{eqn:phi}
\Phi = \Big[R_{k,k+1} ... R_{1,k+l} \Big] A_{1...k} \Big[\tR_{1,k+l}... \tR_{k,k+1} \Big] B_{k+1...k+l}
\end{equation}
For any $\tau \in S(k) \times S(l)$ standardly embedded in $S(k+l)$, we have:
$$
R_\tau \cdot (\tau\Phi \tau^{-1}) \cdot R_\tau^{-1} = \Phi 
$$
which can be seen from the fact that the braid below:
\begin{figure}[ht]              
\centering
\includegraphics[scale=0.3]{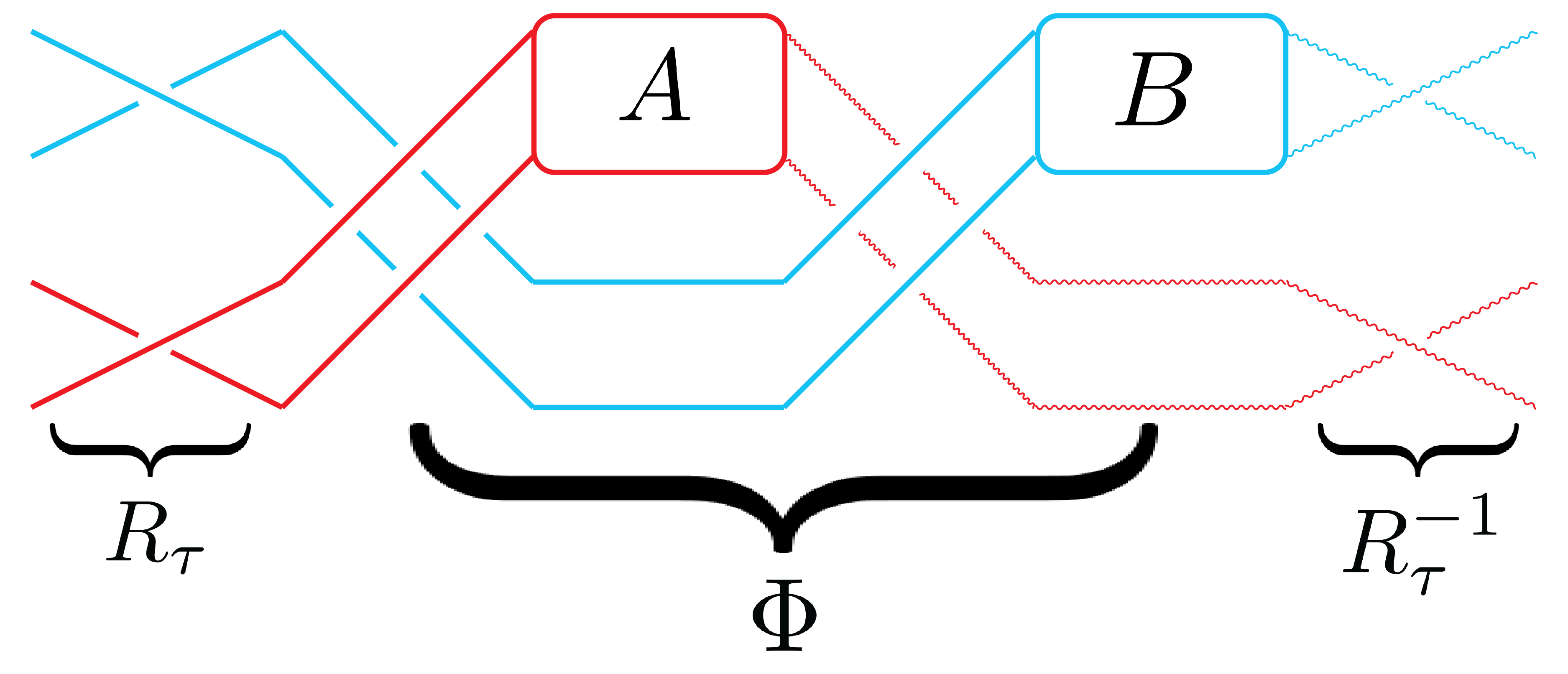}
\caption{$R_\tau \cdot (\tau \Phi \tau^{-1}) \cdot R_\tau^{-1}$}
\end{figure}

\noindent is equivalent to $\Phi$ (since we can cancel the braids representing $R_\tau$ and $R_\tau^{-1}$ by pulling them through the symmetric tensors $A$ and $B$). Then \eqref{eqn:shuf symm 1} implies:
$$
A * B = \frac 1{k! l!} \sum^{\mu \text{ goes over}}_{(k,l)\text{--shuffles}} \sum_{\tau \in S(k) \times S(l)} R_\mu \cdot \mu \left( R_\tau \cdot \tau \Phi \tau^{-1} \cdot R_\tau^{-1} \right) \mu^{-1} \cdot R_\mu^{-1}
$$
Since any $\sigma \in S(k+l)$ can be written uniquely as $\mu \circ \tau$, where $\mu$ is a $(k,l)$--shuffle and $\tau \in S(k) \times S(l) \subset S(k+l)$, the formula above yields:
\begin{equation}
\label{eqn:shuf symm 2}
A * B = \frac 1{k! l!} \sum_{\sigma \in S(k+l)} R_\sigma \cdot (\sigma \Phi \sigma^{-1}) \cdot R_\sigma^{-1} = \frac 1{k!l!} \cdot \sym \ \Phi
\end{equation}
(we have used the fact that $R_{\mu \tau} = R_\mu \cdot \mu R_\tau \mu^{-1}$, times a product of scalars \eqref{eqn:unitary}). Since the symmetrization of any tensor is symmetric, this concludes the proof. 		

\end{proof} 

\noindent By analogy with formula \eqref{eqn:shuf symm 2}, one has the following:
\begin{equation}
\label{eqn:shuf symm 3}
A * B = \frac 1{k!l!} \cdot \sym \ \Psi
\end{equation}
where $\Psi = A_{l+1...l+k} \Big[\tR_{l+1,l}... \tR_{l+k,1} \Big] B_{1...l} \Big[R_{l+k,1} ... R_{l+1,l} \Big]$. The proof is analogous. \\

\subsection{} 
\label{sub:extended}

\noindent Let us fix a basis $v_1,...,v_n$ of $V$ and write $E_{ij}$ for the matrix unit with a single 1 at the intersection of row $i$ and column $j$, and 0 elsewhere. \\

\begin{definition}
\label{def:extended}

Consider the extended shuffle algebra:
$$
\tCA^+ = \Big \langle \CA^+, s_{ij}, t_{ij} \Big \rangle_{1 \leq i,j \leq n} \Big/ \text{relations \eqref{eqn:ext rel 1}--\eqref{eqn:ext rel 5}}
$$
In order to concisely state the relations, we encode the new generators $s_{ij}$, $t_{ij}$ into:
\begin{align*} 
&S = \sum_{i,j = 1}^n s_{ij} \otimes E_{ij} \in \tCA^+ \otimes \emph{End}(V) \\
&T = \sum_{i,j = 1}^n t_{ij} \otimes E_{ij} \in \tCA^+ \otimes \emph{End}(V) 
\end{align*}
and the required relations take the form:
\begin{equation}
\label{eqn:ext rel 1}
R S_1 S_2 = S_2 S_1 R \quad \in \quad \tCA^+ \otimes \emph{End}(V^{\otimes 2})
\end{equation}
\begin{equation}
\label{eqn:ext rel 2}
T_1 T_2 R = R T_2 T_1 \quad \in \quad \tCA^+ \otimes \emph{End}(V^{\otimes 2})
\end{equation}
\begin{equation}
\label{eqn:ext rel 3}
T_1 \tR S_2 = S_2 \tR T_1 \quad \in \quad \tCA^+ \otimes \emph{End}(V^{\otimes 2})
\end{equation}
as well as:
\begin{equation}
\label{eqn:ext rel 4}
X_{1...k} \cdot S_0 = S_0 \cdot \frac {R_{k0}...R_{10}}{f_{k0}...f_{10}} X_{1...k} \tR_{10}...\tR_{k0}
\end{equation}
\begin{equation}
\label{eqn:ext rel 5}
T_0 \cdot X_{1...k} = \tR_{0k}... \tR_{01} X_{1...k} \frac {R_{01}...R_{0k}}{f_{01}...f_{0k}} \cdot T_0
\end{equation}
$\forall X_{1...k} \in \eEnd(V^{\otimes k}) \subset \CA^+$. The latter two formulas should be interpreted as identities in $\tCA^+ \otimes \eEnd(V)$, where the latter copy of $V$ is the one represented by index 0. \footnote{We write $f_{ij} = f_{ji}$ for the scalar $f \in \text{Aut}(V \otimes V)$, interpreted as an element of $\text{Aut}(V^{\otimes k})$ via the inclusion of the $i$--th and $j$--th tensor factors. This notation will come in handy in Section \ref{sec:extended}, when $f_{ij}$ will be a scalar-valued rational function in the variable $z_i/z_j$.} \\

\end{definition}

\begin{proposition} 
\label{prop:extended}

The following assignments make $\tCA^+$ into a bialgebra:
\begin{equation}
\label{eqn:cop shuf 1}
\Delta(S) = (1 \otimes S)(S \otimes 1), \qquad \e(S) = \emph{Id}
\end{equation}
\begin{equation}
\label{eqn:cop shuf 2}
\Delta(T) = (T \otimes 1)(1 \otimes T), \qquad \e(T) = \emph{Id}
\end{equation}
while for all $X = X_{1...k} \in \eEnd(V^{\otimes k}) \subset \CA^+ \subset \tCA^+$ for $k \geq 1$, we set $\e(X) = 0$ and:
\begin{equation}
\label{eqn:cop shuf 3}
\Delta(X) = \sum_{i=0}^k (S_{k}...S_{i+1} \otimes 1) \left( \frac {X_{1...i} \otimes X_{i+1...k}}{\prod_{1 \leq u \leq i < v \leq k} f_{uv}} \right) (T_{i+1} ... T_{k} \otimes 1)
\end{equation}
where the notation $X_{1...k} = X_{1...i} \otimes X_{i+1...k}$ is explained in \eqref{eqn:pseudo sweedler}. \\

\end{proposition}

\begin{proof} The facts that the counit extends to an algebra homomorphism, and that it interacts appropriately with the coproduct, are easy to see. In order to show that the coproduct extends to an algebra homomorphism: 
$$
\tCA^+ \rightarrow \tCA^+ \otimes \tCA^+
$$
we must show that \eqref{eqn:cop shuf 1}, \eqref{eqn:cop shuf 2}, \eqref{eqn:cop shuf 3} respect relations \eqref{eqn:ext rel 1}--\eqref{eqn:ext rel 5}, as well as formula \eqref{eqn:shuf prod} for the shuffle product. To this end, note that:
$$
\Delta(R S_1 S_2) \stackrel{\eqref{eqn:cop shuf 1}}= R (1 \otimes S_1)(S_1 \otimes 1)(1 \otimes S_2)(S_2 \otimes 1) =
$$ 
$$
= R (1 \otimes S_1S_2)(S_1S_2 \otimes 1) \stackrel{\eqref{eqn:ext rel 1}}= (1 \otimes S_2 S_1)R(S_1 S_2 \otimes 1) \stackrel{\eqref{eqn:ext rel 1}}= (1 \otimes S_2 S_1)(S_2 S_1 \otimes 1)R = 
$$
$$
= (1 \otimes S_2)(S_2 \otimes 1)(1 \otimes S_1)(S_1 \otimes 1)R \stackrel{\eqref{eqn:cop shuf 1}}= \Delta(S_2 S_1 R)
$$
An analogous argument shows that $\Delta(T_1 T_2 R) = \Delta(R T_2 T_1)$. As for \eqref{eqn:ext rel 3}:
$$
\Delta (T_1 \tR S_2) = (T_1 \otimes 1)(1 \otimes T_1) \tR (1 \otimes S_2)(S_2 \otimes 1) \stackrel{\eqref{eqn:ext rel 3}}= (T_1 \otimes S_2) \tR (S_2 \otimes T_1) = 
$$
$$
= (1 \otimes S_2)(T_1 \otimes 1) \tR (S_2 \otimes 1)(1 \otimes T_1) \stackrel{\eqref{eqn:ext rel 3}}= (1 \otimes S_2)(S_2 \otimes 1) \tR (T_1 \otimes 1)(1 \otimes T_1) = \Delta(S_2 \tR T_1)
$$
Let us now apply the coproduct to the left-hand side of \eqref{eqn:ext rel 4}:
$$
\Delta(X_{1...k} S_0) = \sum_{i=0}^k (S_{k}...S_{i+1} \otimes 1) (X_{1...i} \otimes X_{i+1...k} ) (T_{i+1} ... T_{k} \otimes 1)(1 \otimes S_0)(S_0 \otimes 1) \stackrel{\eqref{eqn:ext rel 4}}=	
$$
$$
\sum_{i=0}^k (1 \otimes S_0) \frac {R_{k0} ... R_{i+1,0}}{f_{k0}...f_{i+1,0}} (S_{k}...S_{i+1} \otimes 1) (X_{1...i} \otimes X_{i+1...k} ) (T_{i+1} ... T_{k} \otimes 1) \tR_{i+1,0}... \tR_{k0} (S_0 \otimes 1) \stackrel{\eqref{eqn:ext rel 3}}=
$$
$$
\sum_{i=0}^k (1 \otimes S_0) \frac {R_{k0} ... R_{i+1,0}}{f_{k0}...f_{i+1,0}} (S_{k}...S_{i+1} \otimes 1) (X_{1...i} \otimes X_{i+1...k} )  (S_0 \otimes 1)  \tR_{i+1,0}... \tR_{k0} (T_{i+1} ... T_{k} \otimes 1)  \stackrel{\eqref{eqn:ext rel 4}}=
$$
$$
\sum_{i=0}^k (1 \otimes S_0) \frac {R_{k0} ... R_{i+1,0}}{f_{k0}...f_{i+1,0}} (S_{k}...S_{i+1} S_0 \otimes 1) \frac {R_{i0}...R_{10}}{f_{i0}...f_{10}} (X_{1...i} \otimes X_{i+1...k} ) \tR_{10}... \tR_{k0} (T_{i+1} ... T_{k} \otimes 1) \stackrel{\eqref{eqn:ext rel 1}}=
$$
$$
\sum_{i=0}^k (1 \otimes S_0) (S_0 \otimes 1) (S_{k}...S_{i+1} \otimes 1) \frac {R_{k0} ... R_{10}}{f_{k0}...f_{10}} (X_{1...i} \otimes X_{i+1...k} ) \tR_{10}... \tR_{k0} (T_{i+1} ... T_{k} \otimes 1)
$$
The last line above equals $\Delta(\text{RHS of \eqref{eqn:ext rel 4}})$, so we are done. The computation showing that $\Delta$ respects relation \eqref{eqn:ext rel 5} is analogous, and therefore left to the interested reader. As for the shuffle product itself, we must show that $\Delta(A_{1...k} * B_{1...l}) = \Delta(A_{1...k}) * \Delta(B_{1...l})$. Applying formulas \eqref{eqn:shuf prod} and  \eqref{eqn:cop shuf 3} implies:
$$
\Delta(A_{1...k}) * \Delta(B_{1...l}) = \sum_{d=0}^k \sum_{e=0}^l (S_{k}...S_{d+1} \otimes 1) \left( \frac {A_{1...d} \otimes A_{d+1,...k}}{\prod_{1 \leq u \leq d < v \leq k} f_{uv}} \right) (T_{d+1} ... T_{k} \otimes 1) * 
$$
$$
(S_{l}...S_{e+1} \otimes 1) \left( \frac {B_{1...e} \otimes B_{e+1...l}}{\prod_{1 \leq u \leq e < v \leq l} f_{uv}} \right) (T_{e+1} ... T_{l} \otimes 1) = \sum_{d=0}^k \sum_{e=0}^l \mathop{\sum^{a_1<...<a_k, \ b_1<...<b_l}_{\{1,...,d+e\} = \{a_1,...,a_d\} \sqcup \{b_1,...,b_e\}}}_{\{d+e+1,...,k+l\} = \{a_{d+1},...,a_k\} \sqcup \{b_{e+1},...,b_l\}}
$$
$$
\left[\underbrace{R_{a_db_1}... R_{a_1b_e}}_{\text{if }a_i < b_j} \right]  \left[\underbrace{R_{a_kb_{e+1}}...R_{a_{d+1}b_l}}_{\text{if }a_i < b_j} \right] (S_{a_{k}}...S_{a_{d+1}} \otimes 1) \left( \frac {A_{a_1...a_d} \otimes A_{a_{d+1}...a_k}}{\prod_{1 \leq u \leq d < v \leq k} f_{a_ua_v}} \right)
$$
$$
(T_{a_{d+1}} ... T_{a_{k}} \otimes 1) \left[\tR_{a_1b_e} ... \tR_{a_db_1} \right] \left[ \tR_{a_{d+1} b_l} ... \tR_{a_k b_{e+1}} \right] (S_{b_{l}}...S_{b_{e+1}} \otimes 1) 	
$$
$$
\left( \frac {B_{b_1...b_e} \otimes B_{b_{e+1}...b_l}}{\prod_{1 \leq u \leq e < v \leq l} f_{b_ub_v}} \right) (T_{b_{e+1}} ... T_{b_{l}} \otimes 1) \left[\underbrace{R_{a_db_1}... R_{a_1b_e}}_{\text{if }a_i > b_j} \right]  \left[\underbrace{R_{a_kb_{e+1}}...R_{a_{d+1}b_l}}_{\text{if }a_i > b_j} \right]
$$
In the next-to-last row of the expression above, we may apply \eqref{eqn:ext rel 3} in order to move the $T$'s to the right and the $S$'s to the left. Afterwards, we apply \eqref{eqn:ext rel 4} and \eqref{eqn:ext rel 5} to move the $S$'s to the left of $A_{a_1...a_d}$ and the $T$'s to the right of $B_{b_1...b_e}$:
$$	
\Delta(A_{1...k}) * \Delta(B_{1...l}) =  \sum_{d=0}^k \sum_{e=0}^l \mathop{\sum^{a_1<...<a_k, \ b_1<...<b_l}_{\{1,...,d+e\} = \{a_1,...,a_d\} \sqcup \{b_1,...,b_e\}}}_{\{d+e+1,...,k+l\} = \{a_{d+1},...,a_k\} \sqcup \{b_{e+1},...,b_l\}}
$$
$$
\underbrace{R_{a_db_1}... R_{a_1b_e}}_{\text{if }a_i < b_j} \underbrace{R_{a_kb_{e+1}}...R_{a_{d+1}b_l}}_{\text{if }a_i < b_j} (S_{a_{k}}...S_{a_{d+1}} S_{b_{l}}...S_{b_{e+1}} \otimes 1)  \frac {R_{a_db_{e+1}} ... R_{a_1b_l}}{\prod^{1 \leq u \leq d}_{e < v \leq l} f_{a_ub_v}} \frac {A_{a_1...a_d} \otimes A_{a_{d+1},...a_k}}{\prod_{1 \leq u \leq d < v \leq k} f_{a_ua_v}} 
$$
$$
\left[\tR_{a_1b_l} ... \tR_{a_db_{e+1}} \right] \left[\tR_{a_1b_e} ... \tR_{a_db_1} \right] \left[ \tR_{a_{d+1} b_l} ... \tR_{a_k b_{e+1}} \right] \left[ \tR_{a_{d+1} b_e} ... \tR_{a_k b_1} \right]
$$
$$
\frac {B_{b_1...b_e} \otimes B_{b_{e+1}...b_l}}{\prod_{1 \leq u \leq e < v \leq l} f_{b_ub_v}} \frac {R_{a_k b_1} ... R_{a_{d+1} b_e}}{\prod^{d < u \leq k}_{1 \leq v \leq e} f_{a_ub_v}}  (T_{a_{d+1}} ... T_{a_{k}} T_{b_{e+1}} ... T_{b_{l}} \otimes 1) \underbrace{R_{a_db_1}... R_{a_1b_e}}_{\text{if }a_i > b_j} \underbrace{R_{a_kb_{e+1}}...R_{a_{d+1}b_l}}_{\text{if }a_i > b_j} 
$$
Finally, we may use \eqref{eqn:ext rel 1} and \eqref{eqn:ext rel 2} to move the outermost products of $R_{a_ib_j}$ past the $S$'s and the $T$'s, at the cost of re-ordering the latter:
$$
\Delta(A_{1...k}) * \Delta(B_{1...l}) =  \sum_{d=0}^k \sum_{e=0}^l \mathop{\sum^{a_1<...<a_k, \ b_1<...<b_l}_{\{1,...,d+e\} = \{a_1,...,a_d\} \sqcup \{b_1,...,b_e\}}}_{\{d+e+1,...,k+l\} = \{a_{d+1},...,a_k\} \sqcup \{b_{e+1},...,b_l\}} \frac 1{\prod_{1 \leq u \leq d+e < v \leq k+l} f_{uv}}
$$
$$
\left( \prod_{x \in \{a_{d+1},...,a_k, b_{e+1},...,b_l\}}^{\text{decreasing }x} S_x \otimes 1 \right) \left[\underbrace{R_{a_kb_{e+1}}...R_{a_{d+1}b_l}}_{\text{if }a_i < b_j} \right] \left[\underbrace{R_{a_db_1}... R_{a_1b_e}}_{\text{if }a_i < b_j} \right] \left[ R_{a_db_{e+1}} ... R_{a_1b_l} \right] 
$$
$$
(A_{a_1...a_d} \otimes A_{a_{d+1}...a_k}) \left[\tR_{a_1b_l} ...  \tR_{a_k b_1} \right] (B_{b_1...b_e} \otimes B_{b_{e+1}...b_l}) $$
$$
\left[ R_{a_k b_1} ... R_{a_{d+1} b_e} \right] \left[\underbrace{R_{a_kb_{e+1}}...R_{a_{d+1}b_l}}_{\text{if }a_i > b_j} \right] \left[\underbrace{R_{a_db_1}... R_{a_1b_e}}_{\text{if }a_i > b_j} \right] \left( \prod_{x \in \{a_{d+1},...,a_k, b_{e+1},...,b_l\}}^{\text{increasing }x} T_x \otimes 1 \right)
$$
The right-hand side is simply $\Delta$ applied to the RHS of \eqref{eqn:shuf prod}, as we needed to prove. 
	
\end{proof}

\subsection{}
\label{sub:pairing}

Let us consider two copies of the extended shuffle algebra, denoted $\tCA^+$, $\tCA^-$, defined as in the previous Subsections with respect to the same $R$, but with:
\begin{align}
&\tR^+ = \tR, \\ &\tR^- = \left( ( \tR^{\dagger_1} )^{-1} \right)^{\dagger_1}_{21} \label{eqn:tr plus minus}
\end{align}
where $\End(V^{\otimes k}) \stackrel{\dagger_{s}}\rightarrow \End(V^{\otimes k})$ denotes the transposition of the $s$--tensor factor: 
\begin{equation}
\label{eqn:partial transposition}
\left( E_{i_1j_1} \otimes ... \otimes E_{i_sj_s} \otimes ... \otimes E_{i_kj_k} \right)^{\dagger_s} = E_{i_1j_1} \otimes ... \otimes E_{j_si_s} \otimes ... \otimes E_{i_kj_k}
\end{equation}
It is an elementary exercise to show that if properties \eqref{eqn:quasi-ybe 1}--\eqref{eqn:quasi-ybe 2} hold for $\tR^+ = \tR$, then they also hold for $\tR^-$ given by formula \eqref{eqn:tr plus minus}. We will now define a pairing:
\begin{equation}
\label{eqn:shuf pair}
\tCA^+ \otimes \tCA^- \stackrel{\langle \cdot , \cdot \rangle}\longrightarrow \text{ground field}
\end{equation}
which respects the bialgebra structure in the following sense:
\begin{equation}
\label{eqn:bialg 1}
\langle ab, c \rangle = \langle b \otimes a, \Delta(c) \rangle \qquad \forall a,b \in \tCA^+, \ c \in \tCA^-
\end{equation}
\begin{equation}
\label{eqn:bialg 2}
\langle a , b c \rangle = \langle \Delta(a), b \otimes c \rangle \qquad \forall a \in \tCA^+, \ b,c \in \tCA^-
\end{equation}
We will henceforth write $X^\pm$ for the copy of an arbitrary $X \in \End(V^{\otimes k})$ in $\tCA^\pm$. The analogous notation will apply to $S^\pm, T^\pm \in \tCA^\pm \otimes \End(V)$. \\

\begin{proposition} 
\label{prop:bialg pair}

The assignments:
\begin{equation}
\label{eqn:bialg pair 1}
\langle S_2^+, S_1^- \rangle = \tR^+, \qquad \langle T_2^+, T_1^- \rangle = \tR^-
\end{equation}
\begin{equation}
\label{eqn:bialg pair 2}
\langle S_2^+, T_1^- \rangle = \frac Rf, \qquad \langle T_2^+, S_1^- \rangle = R
\end{equation}
and for all $X,Y \in \eEnd(V^{\otimes k})$:
\begin{equation}
\label{eqn:bialg pair 4}
\langle X^+, Y^- \rangle = \frac 1{k!} \emph{Tr}_{V^{\otimes k}}\left(XY \prod_{1\leq i < j \leq k} \frac 1{f_{ij}} \right)
\end{equation}
(the pairings between $X^+$, $Y^-$ on one side and $S^+, T^+, S^-, T^-$ on the other side are defined to be 0) generate a bialgebra pairing \eqref{eqn:shuf pair} satisfying \eqref{eqn:bialg 1}--\eqref{eqn:bialg 2}. \\

\end{proposition}

\begin{proof} The data provided are sufficient to completely define the pairing, in virtue of \eqref{eqn:bialg 1} and \eqref{eqn:bialg 2}. The thing that we need to check is that the defining relations of the extended shuffle algebras, namely \eqref{eqn:ext rel 1}--\eqref{eqn:ext rel 5} and the definition of the shuffle product in \eqref{eqn:shuf prod}, are preserved by the pairing. For \eqref{eqn:ext rel 1}, we have:
$$
\langle R_{12} S^+_1 S^+_2 , S_3^- \rangle \stackrel{\eqref{eqn:bialg 1}}= R_{12} \langle S_1^+ \otimes S_2^+, S_3^- \otimes S_3^- \rangle = R_{12} \tR_{31} \tR_{32} \stackrel{\eqref{eqn:quasi-ybe 2}}= 
$$
$$
= \tR_{32} \tR_{31} R_{12} = \langle S_2^+ \otimes S_1^+ , S_3^- \otimes S_3^- \rangle R_{12} \stackrel{\eqref{eqn:bialg 1}}= \langle S_2^+ S_1^+ R_{12}, S_3^- \rangle
$$
and:
$$
\langle R_{12} S_1^+ S_2^+ , T_3^- \rangle \stackrel{\eqref{eqn:bialg 1}}= R_{12} \langle S_2^+ \otimes S_1^+, T_3^- \otimes T_3^- \rangle = \frac {R_{12} R_{32} R_{31}}{f_{32}f_{31}} \stackrel{\eqref{eqn:ybe}}= 
$$
$$
= \frac {R_{31} R_{32} R_{12}}{f_{32}f_{31}} = \langle S_1^+ \otimes S_2^+, T_3^- \otimes T_3^- \rangle R_{12} \stackrel{\eqref{eqn:bialg 1}}= \langle S_2^+ S_1^+ R_{12}, T_3^- \rangle
$$
We leave the analogous formulas when \eqref{eqn:ext rel 1} is replaced by \eqref{eqn:ext rel 2}, or when the roles of the two arguments of the pairing are switched, as exercises to the reader. As for \eqref{eqn:ext rel 3}, we have:
$$
\langle T_1^+ \tR_{12} S_2^+, S_3^- \rangle \stackrel{\eqref{eqn:bialg 1}}= \langle T_1^+, S_3^- \rangle \tR_{12}  \langle S_2^+, S_3^- \rangle = R_{31} \tR_{12} \tR_{32} \stackrel{\eqref{eqn:quasi-ybe 1}}=
$$
$$
= \tR_{32} \tR_{12} R_{31} = \langle S_2^+, S_3^- \rangle \tR_{12} \langle T_1^+, S_3^- \rangle \stackrel{\eqref{eqn:bialg 1}}= \langle S_2^+ \tR_{12} T_1^+, S_3^- \rangle 
$$
The analogous formulas when $S_3^-$ is replaced by $T_3^-$, or when the roles of the arguments of the pairing are switched, are left as exercises to the interested reader. To prove that \eqref{eqn:ext rel 4} pairs correctly with elements of $\CA^-$, note that \eqref{eqn:cop shuf 3} implies:
$$
\Delta(Y^-_{1...k}) = Y^-_{1...k} \otimes 1 + (S^-_k...S_1^- \otimes 1) (1 \otimes Y^-_{1...k})  (T^-_1...T^-_k \otimes 1) + \text{other terms}
$$
where ``other terms" stands for summands in which $Y_{1...k}^-$ has a non-zero number of indices on either side of the $\otimes$ sign. Then we claim that:
$$
\langle X^+_{1...k} S^+_0, Y^-_{1...k} \rangle \stackrel{\eqref{eqn:bialg 1}}= \Big \langle S^+_0 \otimes X^+_{1...k}, (S^-_k...S^-_1 \otimes 1) (1 \otimes Y^-_{1...k})  (T^-_1...T^-_k \otimes 1) \Big \rangle =
$$
\begin{equation}
\label{eqn:temu}
= \frac 1{k!}  \Tr_{V^{\otimes k}} \left[ \frac {R_{k0} ... R_{10}}{f_{k0}...f_{10}} X_{1...k} \tR_{10}...\tR_{k0}  Y_{1...k} \prod_{1\leq i < j \leq n} \frac 1{f_{ij}} \right]
\end{equation}
where $\Tr_{V^{\otimes k}}$ denotes trace with respect to the indices $1,...,k$ only (therefore the expression above is valued in $\End(V)$, corresponding to the index $0$). The equality between the two rows of \eqref{eqn:temu} is proved as follows: because both sides are bilinear in the tensors $X_{1...k}$ and $Y_{1...k}$, it suffices to prove that they are equal for:
$$
X = E_{i_1j_1} \otimes ... \otimes E_{i_kj_k}, \qquad Y = E_{i_1'j_1'} \otimes ... \otimes E_{i_k'j_k'}
$$
for arbitrary $i_a,j_a,i_a',j_a' \in \{1,...,n\}$. In this case, the equality between the two rows of \eqref{eqn:temu} is a straightforward exercise, which is performed by expanding $S_a^\pm$ and $T_a^\pm$ in terms of the elementary matrices $E_{ij}$, and using \eqref{eqn:bialg 2}, \eqref{eqn:bialg pair 1}, \eqref{eqn:bialg pair 2}. Similarly, because $\e(S^+_0) = \text{Id}$, one sees that:
$$
\Big \langle S_0^+ \frac {R_{k0} ... R_{10}}{f_{k0}...f_{10}} X^+_{1...k} \tR_{10}...\tR_{k0}, Y^-_{1...k} \Big \rangle = \text{RHS of \eqref{eqn:temu}}
$$
We conclude that relation \eqref{eqn:ext rel 4} is preserved by the pairing. The proof that \eqref{eqn:ext rel 5} is preserved by the pairing is analogous, and left as an exercise to the reader. \\
	
\noindent Before proving that the pairing \eqref{eqn:bialg pair 4} intertwines the shuffle product with the coproduct, let us show that the trace pairing is symmetric, in the sense that:
\begin{equation}
\label{eqn:symmetric pairing 1} 
\frac 1{k!} \Tr \Big( [ \sym \ A ] \cdot Y \Big) = \Tr(A \cdot Y)
\end{equation}
\begin{equation}
\label{eqn:symmetric pairing 2} 
\frac 1{k!} \Tr \Big(X \cdot [ \sym \ B ] \Big) = \Tr(X \cdot B)
\end{equation}
for all symmetric tensors $X,Y \in \End(V^{\otimes k})$ and all tensors $A,B \in \End(V^{\otimes k})$. Indeed, \eqref{eqn:symmetric pairing 1} follows from the fact that $\forall \sigma \in S(k)$, we have:
$$
\Tr \Big( R_\sigma (\sigma A \sigma^{-1}) R_{\sigma}^{-1} Y \Big) \stackrel{\eqref{eqn:symmetrization}}= \Tr \Big( R_\sigma (\sigma A \sigma^{-1}) (\sigma Y \sigma^{-1}) R_\sigma^{-1} \Big) = \Tr(A Y)
$$
where the latter equality is the conjugation invariance of trace. Property \eqref{eqn:symmetric pairing 2} is proved likewise. As a consequence of \eqref{eqn:shuf symm 3} and \eqref{eqn:symmetric pairing 1}, proving formula \eqref{eqn:bialg 1} for $a = A^+$, $b = B^+$ and $c = Y^-$ boils down to the following equality:
$$
\frac 1{k! l!}\Tr \left( A_{l+1...k+l} [\tR^+_{l+1,l}... \tR^+_{l+k,1}] B_{1...l} [R_{l+k,1} ... R_{l+1,l}] Y_{1...k+l} \prod_{1 \leq i < j \leq k+l} \frac 1{f_{ij}} \right) =
$$
\begin{equation} 
\label{eqn:big want 1}
= \left \langle B^+_{1...l} \otimes A^+_{l+1...l+k}, (S^-_{k+l} ... S^-_{l+1} \otimes 1) \frac {Y^-_{1...l} \otimes Y^-_{l+1...k+l}}{\prod_{1 \leq u \leq l < v \leq k+l} f_{uv}} (T^-_{l+1} ... T^-_{k+l} \otimes 1) \right \rangle
\end{equation}
which we will now prove. We have:
$$
\Delta(B^+_{1...l}) = B^+_{1...l} \otimes 1 + (S_l^+...S^+_1 \otimes 1) (1 \otimes B^+_{1...l}) (T^+_1... T^+_l \otimes 1) + \text{other terms}
$$
where the phrase ``other terms" denotes summands whose second tensor factor has a non-zero number of indices on either side of the $\otimes$ sign. Because of this, formula \eqref{eqn:bialg 2} when one of $b$ and $c$ is either $S^-$ or $T^-$ (which we have already checked yields a consistent bialgebra pairing) implies that:
\begin{align*}
&\langle B^+_{1...l}, S^-_{l+k} ... S^-_{l+1} Y^-_{1...l} T^-_{l+1} ... T^-_{k+l} \rangle = \\ &= \Big \langle (\Delta \otimes \text{Id}) \circ \Delta(B_{1...l}^+), S^-_{l+k} ... S^-_{l+1} \otimes Y_{1...l}^- \otimes T^-_{l+1} ... T^-_{k+l} \Big \rangle = \\ &= \Tr\Big( [\tR^+_{l+1,l}...\tR^+_{k+l,1}] B_{1...l} [R_{k+l,1}...R_{l+1,l}] Y_{1...l} \Big)
\end{align*}
Therefore, the RHS of \eqref{eqn:big want 1} is precisely equal to the LHS of \eqref{eqn:big want 1}, as required. Similarly, proving \eqref{eqn:bialg 2} for $a = X^+$, $b = A^-$, $c = B^-$ boils down to the equality:
$$
\frac 1{k! l!}\Tr_{V^{\otimes k+l}} \left( [R_{k,k+1} ... R_{1,k+l}] A_{1...k} [\tR^-_{1,k+l}... \tR^-_{k,k+1}] B_{k+1...k+l} X_{1...k+l} \prod_{1 \leq i < j \leq k+l} \frac 1{f_{ij}} \right)
$$
\begin{equation} 
\label{eqn:big want 2}
= \left \langle (S^+_{k+l} ... S^+_{k+1} \otimes 1) \frac {X^+_{1...k} \otimes X^+_{k+1...k+l}}{\prod_{1 \leq u \leq k < v \leq k+l} f_{uv}} (T^+_{k+1} ... T^+_{k+l} \otimes 1), A^-_{1...k} \otimes B^-_{k+1...k+l} \right \rangle
\end{equation}
The equality \eqref{eqn:big want 2} is proved by analogy with \eqref{eqn:big want 1}, so we skip the details. 

\end{proof}

\subsection{}
\label{sub:double}

Proposition \ref{prop:bialg pair} allows us to define the Drinfeld double:
\begin{equation}
\label{eqn:double}
\CA = \tCA^+ \otimes \tCA^{-, \op, \coop}
\end{equation}
such that $\tCA^+ \cong \tCA^+ \otimes 1$ and $\tCA^{-, \op,\coop} \cong 1 \otimes \tCA^{-, \op, \coop}$ are sub-bialgebras of $\CA$ (``op" and ``coop" denote the opposite algebra and coalgebra structure, respectively), and the commutation of elements coming from the two factors is governed by:
\begin{equation}
\label{eqn:dd}
\langle a_1,b_1 \rangle a_2b_2 = b_1 a_1 \langle a_2, b_2 \rangle
\end{equation}
for all $a \in \tCA^+$ and $b \in \tCA^{-, \op,\coop}$. Let us now spell out formula \eqref{eqn:dd} when $a$ and $b$ are among the generators of the double algebra \eqref{eqn:double}. The quadratic relations \eqref{eqn:ext rel 1}--\eqref{eqn:ext rel 5} hold as stated between the $+$ generators, and hold with the opposite multiplication between the $-$ generators. As for the relations that involve one of the $+$ generators and one of the $-$ generators, we have: \\

\begin{proposition}
\label{prop:double}

We have the following formulas in $\CA$:
\begin{equation}
\label{eqn:drinfeld 0}
S_2^+ \tR S_1^- = S_1^- \tR S_2^+, \quad \qquad R S_1^+ T_2^- = T_2^- S_1^+ R 
\end{equation}
\begin{equation}
\label{eqn:drinfeld 0 bis}
T_1^+ \tR T_2^- = T_2^- \tR T_1^+, \quad \qquad R T_2^+ S_1^- = S_1^- T_2^+ R 
\end{equation}
as well as:
\begin{equation}
\label{eqn:drinfeld 1} 
S_0^\mp \cdot_{\pm} X_{1...k}^\pm = \tR_{0k}^\pm ... \tR_{01}^\pm X_{1...k}^\pm R_{01}...R_{0k} \cdot_\pm S_0^\mp
\end{equation}
\begin{equation}
\label{eqn:drinfeld 2} 
X_{1...k}^\pm \cdot_\pm T_0^\mp = T_0^\mp \cdot_\pm R_{k0} ... R_{10} X_{1...k}^\pm \tR^\pm_{10} ... \tR^\pm_{k0}
\end{equation}
where $\cdot_+ = \cdot$ and $\cdot_- = \cdot^{\eop}$ (the opposite multiplication in $\CA$). Finally, we have:
\begin{equation}
\label{eqn:drinfeld 5}
[E_{ij}^+, E_{i'j'}^-] = s_{j'i}^+ t_{ji'}^+ - t_{j'i}^- s_{ji'}^- 
\end{equation}
$\forall \ i,j,i',j' \in \{1,...,n\}$, where $E_{ij}^\pm$ are elements in the $k=1$ summand of \eqref{eqn:general shuffle}. \\

\end{proposition}

\begin{proof} Let us now prove the first formula in \eqref{eqn:drinfeld 0} and leave the second one and \eqref{eqn:drinfeld 0 bis} as exercises for the reader. Since:
\begin{align} 
&\Delta(S^+) = (1 \otimes S^+)(S^+ \otimes 1) \quad \text{in } \tCA^+ \label{eqn:cop pos} \\
&\Delta(S^-) = (S^- \otimes 1)(1 \otimes S^-) \quad \text{in } \tCA^{-, \op,\coop} \label{eqn:cop neg}
\end{align} 
formula \eqref{eqn:dd} for $a = S_2^+$ and $b = S_1^-$ implies:
$$
S^+_2 \langle S_2^+, S_1^- \rangle S_1^- = S_1^- \langle S_2^+, S_1^- \rangle S_2^+
$$
Using \eqref{eqn:bialg pair 1} to evaluate the pairing implies precisely the first formula in \eqref{eqn:drinfeld 0}. \\
	
\noindent Let us now prove \eqref{eqn:drinfeld 1} and leave the analogous formula \eqref{eqn:drinfeld 2} as an exercise. We will do so in the case $\pm = +$, as $\pm = -$ just involves the opposite of all relations.
$$
\Delta(X_{1...k}^+) = X_{1...k}^+ \otimes 1 + (S^+_k...S_1^+ \otimes 1)(1 \otimes X_{1...k}^+)(T^+_1...T_k^+ \otimes 1) + \text{other terms}
$$
where ``other terms" stands for terms which have a non-zero number of indices on either side of the $\otimes$ sign, so they pair trivially with $S^-$. Meanwhile, $\Delta(S^-)$ is given by \eqref{eqn:cop neg}. Applying \eqref{eqn:dd} for $a = X_{1...k}^+$ and $b_0 = S^-$ yields:
$$
\langle S_k^+ ... S_1^+ , S_0^- \rangle X_{1...k}^+ \langle T_1^+ ... T_k^+, S_0^- \rangle S_0^- = S_0^- X_{1...k}^+
$$
Formulas \eqref{eqn:bialg 1}, \eqref{eqn:bialg pair 1} and \eqref{eqn:bialg pair 2} transform the formula above precisely into \eqref{eqn:drinfeld 1}. \\
	
\noindent As for \eqref{eqn:drinfeld 5}, consider relation \eqref{eqn:cop shuf 3} for $k=1$ and $X = E^+_{ij}$:
\begin{equation}
\label{eqn:delta up basic}
\Delta(E_{ij}^+) = E_{ij}^+ \otimes 1 + \sum_{x,y = 1}^n s^+_{xi} t^+_{jy} \otimes E_{xy}^+
\end{equation}
as well as the $(\op,\coop)$ version of the above equality that holds in $\tCA^{-,\op,\coop}$:
\begin{equation}
\label{eqn:delta down basic}
\Delta(E_{i'j'}^-) = \sum_{x',y' = 1}^n E_{x'y'}^- \otimes t^-_{j'y'} s^-_{x'i'}  + 1 \otimes E_{i'j'}^-
\end{equation}
Then \eqref{eqn:drinfeld 5} follows from \eqref{eqn:dd} for $a = E_{ij}^+$ and $b = E_{i'j'}^-$.
	
\end{proof}

\section{Quantum toroidal $\fgl_n$}
\label{sec:quantum}

\noindent In the present Section, we will define quantum affine and toroidal $\fgl_n$, and recall (in Subsections \ref{sub:uu}--\ref{sub:formulation}) a generators-and-relations presentation   for the latter, which will play a big role in the subsequent Sections. As for quantum affine $\fgl_n$, we recall its Drinfeld-Jimbo presentation in Subsections \ref{sub:special}--\ref{sub:coproducts}, and its RTT presentation in Subsections \ref{sub:bill}--\ref{sub:e infinity}. We also note that Subsections \ref{sub:two realizations}--\ref{sub:bars} serve as a bridge between the Drinfeld-Jimbo and RTT presentations; the reader who is familiar with our notation and conventions may skip this bridge, but otherwise it will serve as useful motivation for the Subsections that follow. \\

\subsection{} Fix $n>1$, and let us start with a few notational remarks. The symbol $\delta$ will refer to several different notions throughout the present paper. Specifically, the main usage of $\delta$ is the following ``mod $n$" variant of the Kronecker symbol:
\begin{equation}
\label{eqn:delta 1}
\delta_j^i = \begin{cases} 1 &\text{if } i \equiv j \text{ mod }n \\
0 &\text{otherwise} \end{cases}
\end{equation}
We will write $\delta_{j \text{ mod }g}^i$ if we need the notion above for congruences modulo another number $g$ instead of $n$. Moreover, if $i,j,i',j' \in \BZ$, we will write:
\begin{equation}
\label{eqn:delta 2}
\delta_{(i',j')}^{(i,j)} = \begin{cases} 1 &\text{if } (i,j) \equiv (i',j') \text{ mod }(n,n) \\
0 &\text{otherwise} \end{cases}
\end{equation}
The ordinary Kronecker symbol will be denoted by $\tilde{\delta}_j^i$, and it takes the value 1 if and only if $i = j$ as integers, and 0 otherwise. Finally, we will use the notation:
$$
\delta(z) = \sum_{k\in \BZ} z^k
$$
for the formal $\delta$ series. Since the various uses of $\delta$ will arise in different contexts, we hope this abuse of notation will not cause confusion. \\

\subsection{} 
\label{sub:special}

Let $\dot{\fsl}_n$ be the Kac-Moody Lie algebra of type $\widehat{A}_{n-1}$. The corresponding Drinfeld-Jimbo quantum group \footnote{We note that the algebra defined below is slightly larger than the usual quantum group, since the Cartan part of the latter is generated by the ratios $\frac {\psi_i}{\psi_j}$, instead of $\psi^{\pm 1}_1,...,\psi_n^{\pm 1}$ themselves} is defined to be the associative algebra:
\begin{equation}
\label{eqn:special}
\su = \BQ(q) \Big \langle x_i^\pm, \psi_s^{\pm 1}, c^{\pm 1} \Big \rangle^{i \in \BZ/n\BZ}_{s \in \{1,...,n\}}
\end{equation}
modulo the fact that $c$ is central, as well as the following relations for all $i,j \in \BZ/n\BZ$ and  $s,s' \in \{1,...,n\}$ (by setting $\psi_{s+n} = c\psi_s$, we extend the notation $\psi_s$ to all $s \in \BZ$):
\begin{equation}
\label{eqn:special 0}
\psi_s \psi_{s'} = \psi_{s'} \psi_s
\end{equation}
\begin{equation}
\label{eqn:special 1}
\psi_s x^\pm_i  = q^{\pm (\delta_s^{i+1} - \delta^i_s)} x^\pm_i \psi_s 
\end{equation}
\begin{equation}
\label{eqn:special 2}
\begin{cases}
\ [x_i^\pm, x_j^\pm] = 0 &\text{if } j \notin \{i-1,i+1\} \\ \\
\ (x_i^\pm)^2 x_j^\pm - (q+q^{-1}) x_i^\pm x_j^\pm x_i^\pm + x_j^\pm (x_i^\pm)^2 = 0  &\text{if } j \in \{i-1,i+1\} 
\end{cases} 
\end{equation}
\begin{equation}
\label{eqn:special 3}
[x_i^+, x_j^-] =  \frac {\delta_i^j}{q-q^{-1}} \left(\frac {\psi_{i+1}}{\psi_i} - \frac {\psi_i}{\psi_{i+1}} \right)
\end{equation}
(note that relation \eqref{eqn:special 2} must be amended in the case $n=2$, but we will not need the modification in question). We also consider the $q$--deformed Heisenberg algebra:
\begin{equation}
\label{eqn:heisenberg}
\uui = \BQ(q) \Big \langle p_{\pm k}, c \Big \rangle_{k \in \BN}
\end{equation}
where $c$ is central, and the $p_{\pm k}$ satisfy the commutation relation:
\begin{equation}
\label{eqn:heis 0}
[p_k, p_l] = k \tilde{\delta}_{k+l}^0 \cdot \frac {c^k - c^{-k}}{q^k-q^{-k}} 
\end{equation}
Then we will consider the algebra:
\begin{equation}
\label{eqn:quantum as tensor}
\uu = \su \otimes \uui \Big/ (c \otimes 1 - 1 \otimes c)
\end{equation}
which serves as an affine $q$--version of the Lie algebra isomorphism $\fgl_n \cong \fsl_n \oplus \fgl_1$. \\

\subsection{}
\label{sub:coproducts}

We can make $\uu$ into a bialgebra by using the counit $\e(c) = 1$, $\e(\psi_s) = 1$, $\e(x_i^\pm) = 0$, $\e(p_k) = 0$ and the coproduct given by $\Delta(c) = c \otimes c$ and:
\begin{equation}
\label{eqn:cop special 1}
\Delta(\psi_s) = \psi_s \otimes \psi_s  
\end{equation}
\begin{align}
&\Delta(x_i^+) = \frac {\psi_i}{\psi_{i+1}} \otimes x_i^+ + x_i^+ \otimes 1 \label{eqn:cop special 2} \\
&\Delta(x_i^-) = 1 \otimes x_i^- + x_i^- \otimes \frac {\psi_{i+1}}{\psi_i} \label{eqn:cop special 3} 
\end{align} 
\begin{equation}
\label{eqn:cop heis 1}
\Delta(p_k) = \frac 1{c^k} \otimes p_k + p_k \otimes 1  
\end{equation}
\begin{equation}
\label{eqn:cop heis 2}
\Delta(p_{-k}) = 1 \otimes p_{-k} + p_{-k} \otimes c^k
\end{equation}
Moreover, the sub-bialgebras:
\begin{align}
&\uug = \BQ(q) \Big \langle x_i^+, p_k, \psi_s^{\pm 1}, c^{\pm 1} \Big \rangle^{i \in \BZ/n\BZ, k \in \BN}_{s \in \{1,...,n\}} \ \ \subset \uu \label{eqn:positive} \\
&\uul = \BQ(q) \Big \langle x_i^-, p_{-k}, \psi_s^{\pm 1}, c^{\pm 1} \Big \rangle^{i \in \BZ/n\BZ, k \in \BN}_{s \in \{1,...,n\}} \subset \uu \label{eqn:negative}
\end{align}
are endowed with a bialgebra pairing:
\begin{equation}
\label{eqn:double 0}
\uug \otimes \uul^{\op,\coop} \stackrel{\langle \cdot, \cdot \rangle}\longrightarrow \BQ(q)
\end{equation}
generated by properties \eqref{eqn:bialg 1}, \eqref{eqn:bialg 2} and:
$$
\langle \psi_s, \psi_{s'} \rangle = q^{-\delta_{s'}^s}, \qquad \langle x_i^+, x_j^-\rangle = \frac {\delta_j^i}{q^{-1}-q}, \qquad \langle p_k, p_{-l} \rangle = \frac {k\tilde{\delta}_k^l}{q^{-k}-q^k}
$$
and all other parings between generators are 0. It is well-known that $\uu$ is the Drinfeld double corresponding to the data \eqref{eqn:double 0} (modulo the identification of the symbols $\psi_s,c$ in the two factors of \eqref{eqn:double 0}). The algebra $\uu$ is $\zz$--graded:
$$
\deg c = 0, \quad \deg \psi_s = 0, \quad \deg x_i^\pm = \pm \bs^i, \quad \deg p_{\pm k} = \pm k\bde 
$$
where $\bs^i = \underbrace{(0,...,0,1,0,...,0)}_{\text{1 on }i\text{--th position}}$ and $\bde = (1,...,1)$. \\

\begin{remark}
\label{rem:simple imaginary}
	
The elements $x_i^\pm$ of $\uu$ are called simple (root) generators, while the elements $p_{\pm k}$ are called imaginary (root) generators. Up to constant multiples, these are all the primitive elements of $\uu$ (see Definition \ref{def:primitive}). \\
	
\end{remark}

\subsection{} 
\label{sub:two realizations}

We will now give a different incarnation of the bialgebra \eqref{eqn:quantum as tensor}, which is obtained by combining the results of \cite{DF} and \cite{D}. We will use the notation of \cite{Tor}. \\

\begin{definition}
\label{def:two realizations}
	
Consider the algebra:
\begin{equation}
\label{eqn:two realizations}
\CE= \BQ(q) \Big \langle f_{\pm [i;j)}, \psi_s^{\pm 1}, c^{\pm 1} \Big \rangle^{1\leq s \leq n}_{(i <j) \in \zzz} \Big/ \text{relations \eqref{eqn:psi f}--\eqref{eqn:f f}}
\end{equation}
where $c$ is central, and the quadratic relations \eqref{eqn:psi f}--\eqref{eqn:f f} will be explained later. \\
	
\end{definition}

\noindent The algebra $\CE$ is a bialgebra with respect to the counit $\e(\psi_s) = \e(c) = 1, \forall s$ and $\e(f_{\pm [i;j)}) = 0, \forall i<j$, as well as the coproduct $\Delta(c) = c \otimes c$ and:
\begin{align}
&\Delta(\psi_s) = \psi_s \otimes \psi_s \label{eqn:cop 0} \\
&\Delta(f_{[i;j)}) = \sum_{s = i}^j f_{[s;j)} \frac {\psi_i}{\psi_s}  \otimes f_{[i;s)} \label{eqn:cop 1} \\
&\Delta(f_{-[i;j)}) = \sum_{s = i}^j f_{-[i;s)}  \otimes f_{-[s;j)} \frac {\psi_s}{\psi_i}  \label{eqn:cop 2}
\end{align}
where the notation $\psi_s$ is extended to all $s \in \BZ$ by $\psi_{s+n} = c\psi_s$. The sub-bialgebras:
\begin{align} 
&\CE^\geq = \BQ(q) \Big \langle f_{+[i;j)}, \psi_s^{\pm 1}, c^{\pm 1} \Big \rangle^{1\leq s \leq n}_{(i < j) \in \zzz} \subset \CE \label{eqn:e geq} \\ 
&\CE^\leq = \BQ(q) \Big \langle f_{-[i;j)}, \psi_s^{\pm 1}, c^{\pm 1} \Big \rangle^{1\leq s \leq n}_{(i < j) \in \zzz} \subset \CE \label{eqn:e leq}
\end{align}
are endowed with a bialgebra pairing:
\begin{equation}
\label{eqn:double 00}
\CE^\geq \otimes \CE^{\leq,\op,\coop} \stackrel{\langle \cdot, \cdot \rangle}\longrightarrow \BQ(q)
\end{equation}
generated by properties \eqref{eqn:bialg 1}, \eqref{eqn:bialg 2} and:
\begin{equation}
\label{eqn:pairing}
\langle \psi_s, \psi_{s'} \rangle = q^{-\delta_{s'}^s}, \qquad \langle f_{[i;j)}, f_{-[i';j')} \rangle = \delta_{(i',j')}^{(i,j)} (1-q^{-2})
\end{equation}
where the right-most delta symbol is defined in \eqref{eqn:delta 2}. With respect to the pairing \eqref{eqn:double 00}, the algebra $\CE$ is isomorphic to the Drinfeld double $\CE^{\geq} \otimes \CE^{\leq}$ modulo the identification of $\psi_s, c$ in the two tensor factors. The algebra $\CE$ is $\zz$--graded:
$$
\deg c = 0, \quad \deg \psi_s = 0, \quad \deg f_{\pm [i;j)} = \pm [i;j)
$$ 
where $[i;j) = \bs^i + ... + \bs^{j-1} \in \nn$ (we write $\bs^k = \bs^{k \text{ mod }n}$). \\

\subsection{} 
\label{sub:sub}

The subalgebras:
$$
\CE \supset \CE^\pm = \BQ(q) \Big \langle f_{\pm [i;j)} \Big \rangle_{(i < j) \in \zzz} 
$$
are graded by $\pm \nn$, and we will write $\CE_{\pm \bd}$ for their graded pieces, for all $\bd \in \nn$. The dimensions of these graded pieces are given by the following formula:
\begin{equation}
\label{eqn:bound affine}
\dim \CE_{\pm \bd} = \# \text{ partitions } \Big \{ \bd = [i_1;j_1) + ... + [i_u;j_u) \Big \}^{u \in \BN}_{(i_a<j_a) \in \zzz}
\end{equation}
In fact, ordered products of the various $f_{\pm [i;j)}$ give rise to a PBW basis of $\CE^\pm$ (\cite{Tor}). \\

\begin{definition}
\label{def:primitive}

An element $x \in \CE_{\pm \bd}$ is called primitive if:
\begin{equation}
\label{eqn:primitive def}
\Delta(x) \in \langle \psi_s^{\pm 1} \rangle_{s \in \BZ} \otimes x + x \otimes \langle \psi_s^{\pm 1} \rangle_{s \in \BZ}
\end{equation}
It is easy to see that any sum of primitive elements is primitive. With this in mind, we will write $\CE_{\pm \bd}^{\emph{prim}} \subset \CE_{\pm \bd}$ for the vector subspace of primitive elements. \\

\end{definition}

\noindent As shown in \cite[Lemma 2.3]{PBW}, we have:
$$
\dim \CE_{\pm \bd}^{\text{prim}} = \begin{cases} 1 &\text{if } \bd \text{ is either } [i;i+1) \text{ or } k\bde \\ 0 &\text{otherwise} \end{cases} 
$$
for various $i \in \BZ/n\BZ$ and $k \in \BN$. Therefore, up to scalar multiples, there is a unique choice of primitive elements:
\begin{equation}
\label{eqn:simple imaginary}
x_i^\pm \in \CE_{\pm [i;i+1)}, \qquad p_{\pm k} \in \CE_{\pm k\bde}
\end{equation}
which will be called simple and imaginary (respectively) primitive generators of $\CE$. Comparing this with Remark \ref{rem:simple imaginary} yields the following: \\

\begin{theorem} (\cite{Tor}) \label{thm:affine} Any choice of simple and imaginary primitive elements \eqref{eqn:simple imaginary} of $\CE$ gives rise to an isomorphism of $\zz$--graded bialgebras $\uu \cong \CE$. \\
	
\end{theorem}

\subsection{} As is clear from the Theorem above, understanding the bialgebra structure of $\CE$ boils down to controlling the up-to-scalar ambiguity in choosing the primitive elements. To do this, we consider a formal parameter $\oq$ and define linear functionals:
\begin{equation}
\label{eqn:alpha 0}
\alpha_{\pm [i;j)} : \CE_{\pm [i;j)} \longrightarrow \fff
\end{equation}
for all $(i < j) \in \zzz$, satisfying the following properties:
\begin{equation}
\label{eqn:pseudo} 
\alpha_{\pm [i;j)} (r \cdot r') = \begin{cases} \alpha_{\pm [s;j)}(r) \alpha_{\pm [i;s)}(r') &\text{if }\exists  s \text{ s.t. } r \in \CE_{\pm [s;j)}, r' \in \CE_{\pm [i;s)} \\ 0 &\text{otherwise} \end{cases} 
\end{equation}
and (let $\oq_+ = \oq$ and $\oq_- = (q^n \oq)^{-1}$):
\begin{equation} 
\label{eqn:alpha affine}
\alpha_{\pm [i;j)}(f_{\pm [i';j')}) = \delta_{(i',j')}^{(i,j)} (1-q^2) \oq_\pm^{\frac {j-i}n}
\end{equation}
We henceforth fix the elements \eqref{eqn:simple imaginary} by making the choice of \cite{PBW}, namely:
\begin{equation}
\label{eqn:norm 1}
\alpha_{\pm [i;i+1)}(x_i^\pm) = \pm 1
\end{equation}
\begin{equation}
\label{eqn:norm 2}
\alpha_{\pm [s;s + nk)}(p_{\pm k}) = \pm 1
\end{equation}
$\forall k > 0$, $i,s \in \BZ/n\BZ$. The fact that the right-hand side of \eqref{eqn:norm 2} does not depend on $s$ is a consequence of the $\BZ/n\BZ$--invariance of the elements $p_{\pm k}$, see \cite{Tor}. \\

\subsection{} 
\label{sub:bars}

It is easy to note that the bialgebra $\uu \cong \CE$ possesses an antipode: \footnote{The antipode is a bialgebra anti-automorphism $A : \CE \rightarrow \CE$ obeying certain compatibility conditions with the product, coproduct and pairing. We choose to write $A(x)$ instead of the more common $S(x)$ so as to not confuse the antipode with the series $S(x)$ of Subsection \ref{sub:bill}}
$$
A(\psi_s) = \psi_s^{-1}, \quad A(x_i^+) = - \frac {\psi_{i+1}}{\psi_i} x_i^+, \quad A(x_i^-) = - x_i^- \frac {\psi_i}{\psi_{i+1}}, \quad A(p_{\pm k}) = - c^{\pm k} p_{\pm k}
$$
In terms of the generators $f_{\pm [i;j)}$, we may write:
\begin{equation}
\label{eqn:bars}
A^{\pm 1}(f_{\pm[i;j)}) = \frac {\psi^{\pm 1}_j}{\psi^{\pm 1}_i} \barf_{\pm[i;j)} \oq_\mp^{\frac {2(i-j)}n} 
\end{equation}
where the elements $\barf_{\pm [i;j)}$ are inductively defined in terms of $f_{\pm[i;j)}$ by the formulas:
\begin{equation}
\label{eqn:antipode}
\sum_{s = i}^j \barf_{\pm [s;j)} f_{\pm [i;s)} \oq_\mp^{\frac {2(s-i)}n} = \sum_{s = i}^j f_{\pm [s;j)} \barf_{\pm [i;s)} \oq_\mp^{\frac {2(j-s)}n} = 0
\end{equation}
Alternatively, the elements $\barf_{\pm [i;j)}$ are completely determined by their coproduct:
\begin{align}
&\Delta(\barf_{[i;j)}) = \sum_{s = i}^j  \frac {\psi_s}{\psi_j} \barf_{[i;s)} \otimes \barf_{[s;j)} \label{eqn:cop 11} \\
&\Delta(\barf_{-[i;j)}) = \sum_{s = i}^j \barf_{-[s;j)} \otimes \frac {\psi_j}{\psi_s} \barf_{-[i;s)} \label{eqn:cop 22}
\end{align} 
and by their values under the linear maps \eqref{eqn:alpha 0}:
\begin{equation}
\label{eqn:alpha affine bar}
\alpha_{\pm [i;j)}(\barf_{\pm [i';j')}) = \delta_{(i',j')}^{(i,j)} (1-q^{-2}) \oq_\pm^{\frac {i-j}n}
\end{equation}
Since $\CE \cong \uu$, we will use the notation $f_{\pm [i;j)}$ (respectively $\barf_{\pm [i;j)}$) for the elements of either algebra. These elements will be called root generators of either algebra $\CE \cong \uu$, because $[i;j)$ are positive roots of the affine $A_n$ root system. \\

\subsection{} 
\label{sub:uu}

Affinizations of quantum groups are defined by replacing each generator $x_i^\pm$ as in Subsection \ref{sub:special} by an infinite family of generators $\{x_{i,k}^\pm\}_{k \in \BZ}$. To define affinizations explicitly, let us consider variables $z$ as being colored by an integer $i$, denoted by ``$\text{col }z$". Then we may define the following color-dependent rational function:
\begin{equation}
\label{eqn:def zeta}
\zeta \left( \frac zw \right) = \begin{cases} \displaystyle \frac {zq \oq^{2k} - wq^{-1}}{z \oq^{2k} - w} & \text{if } i - j = nk \\ \displaystyle 1 & \text{if } i - j \notin \{-1,0\} \text{ mod } n \\ \displaystyle \frac {z \oq^{2k} - w}{zq \oq^{2k} - wq^{-1}} & \text{if } i - j = nk - 1 \end{cases}
\end{equation}
for any variables $z,w$ of colors $i$ and $j$ respectively. \\

\begin{definition}
\label{def:toroidal} 
	
The quantum toroidal algebra is:
$$
\UU = \fff \Big \langle x_{i,k}^\pm, \ph_{i,k'}^\pm, \psi_s^{\pm 1}, c^{\pm 1}, \barc^{\pm 1} \Big \rangle_{i \in \BZ/n\BZ, s \in \BZ}^{k \in \BZ, k' >0 } \Big/ \text{relations \eqref{eqn:rel aff 0}--\eqref{eqn:rel aff 6}}
$$
\end{definition} 

\text{}

\noindent In order to spell out the defining relations in the quantum toroidal algebra, let us consider the series $x_i^\pm(z) = \sum_{k \in \BZ} \frac {x_{i,k}^\pm}{z^k}$ and $\ph_s^\pm(w) = \frac {\psi_{s+1}^{\pm 1}}{\psi_s^{\pm 1}}  + \sum_{k=1}^\infty \frac {\ph_{s,k}^\pm}{w^{\pm k}}$, and set:
\begin{equation}
\label{eqn:rel aff 0}
c,\barc \text{ central}, \qquad \psi_{s+n} = \psi_s c, \ \forall s \in \BZ
\end{equation} 
\begin{equation}
\label{eqn:rel aff 1}
\psi_s \text{ commutes with } \psi\text{'s and } \ph\text{'s, and satisfies \eqref{eqn:special 1} with } x_i^\pm \leadsto x_i^\pm(z)
\end{equation}
\begin{equation}
\label{eqn:rel aff 2}
\ph_i^\pm(z) \ph_j^{\pm'}(w) \frac {\zeta \left( \frac {w \barc^{\mp 1}}z \right)}{\zeta \left( \frac {w \barc^{\mp' 1}}z \right)} = \ph_j^{\pm'}(w) \ph_i^{\pm}(z) \frac {\zeta \left( \frac {z \barc^{\mp' 1}}w \right)}{\zeta \left( \frac {z \barc^{\mp 1}}w \right)}
\end{equation}
\begin{equation}
\label{eqn:rel aff 3}
x_i^\pm(z) \ph_j^{\pm'} (w) \zeta \left(\frac {w}{z\barc^{\delta_{\pm'}^\mp}} \right)^{\pm 1} = \ph_j^{\pm'} (w) x_i^\pm(z) \zeta \left(\frac {z\barc^{\delta_{\pm'}^\mp}}{w} \right)^{\pm 1}
\end{equation}
\begin{equation}
\label{eqn:rel aff 4}
x^\pm_i(z) x^\pm_j(w) \zeta \left(\frac wz \right)^{\pm 1} = x^\pm_j(w) x_i^\pm(z) \zeta \left(\frac zw \right)^{\pm 1}
\end{equation} 
\begin{equation}
\label{eqn:rel aff 5}
[x_i^+(z), x_j^-(w)] = \frac {\delta_i^j}{q-q^{-1}} \left[ \delta \left( \frac {z}{w\barc} \right) \ph_i^+(z) - \delta \left( \frac {w}{z\barc} \right) \ph_i^-(w) \right]
\end{equation}
and, for $j \in \{i-1,i+1\}$:
$$
x_i^\pm(z_1) x_i^\pm(z_2) x_j^\pm(w) - (q+q^{-1}) x_i^\pm(z_1) x_j^\pm(w) x_i^\pm(z_2) + x_j^\pm(w) x_i^\pm(z_1) x_i^\pm(z_2) +
$$
\begin{equation}
\label{eqn:rel aff 6}
+ \text{ same expression with } z_1 \text{ and } z_2 \text{ switched} = 0
\end{equation}
for all choices of $\pm, \pm', i, j$, where the variables $z$ and $w$ have color $i$ and $j$, respectively, for the purpose of defining the rational function $\zeta$ (when $n=2$, one needs to modify \eqref{eqn:rel aff 6} similarly to how one needs to modify \eqref{eqn:special 2}). Note that we extend the index $i$ to arbitrary integers, by applying the convention:
$$
x^\pm_{i+n,k} = x^\pm_{i,k} \oq^{-2k}, \qquad \ph^\pm_{i+n,k} = \ph^\pm_{i,k} \oq^{\mp 2k}
$$
We consider the subalgebras $\UUpm \subset \UU$ generated by $\{x^\pm_{i,k}\}^{i \in \BZ/n\BZ}_{k \in \BZ}$. \\

\begin{remark}

In the notation of Subsection \ref{sub:left right}, we have $\UUp = \UUright$ and $\UUm = \UUleft$, but we will henceforth use the $\pm$ notation. \\

\end{remark}

\subsection{} 
\label{sub:old shuf}

Let us now recall the classic shuffle algebra realization of $\UUpm$. Consider variables $z_{ia}$ of color $i$, for various $i \in \{1,...,n\}$ and $a \in \BN$. We call a function $R(z_{11}...,z_{1d_1},...,z_{n1}...,z_{nd_n})$ color-symmetric if it is symmetric in $z_{i1},...,z_{id_i}$ for all $i \in \{1,...,n\}$ separately. Depending on the context, the symbol ``Sym" will refer to either color-symmetric functions in variables $z_{ia}$, or to the symmetrization operator:
$$
\sym \ F(...,z_{i1},...,z_{id_i},...) = \sum_{(\sigma^1,...,\sigma^n) \in S(d_1) \times ... \times S(d_n)} F(...,z_{i,\sigma^i(1)},...,z_{i,\sigma^i(d_i)},...)
$$
Let $\bd! = d_1!...d_n!$. The following construction arose in the context of quantum groups in \cite{E}, by analogy to the work of Feigin-Odesskii on certain elliptic algebras. \\

\begin{definition}
\label{def:shuf classic}
	
Consider the vector space:
\begin{equation}
\label{eqn:shuf classic}
\CV = \bigoplus_{\bd = (d_1,...,d_n) \in \nn} \fff(...,z_{i1},...,z_{id_i},...)^{\emph{Sym}}
\end{equation}
and endow it with an associative algebra structure, by setting $R*R'$ equal to:
$$
\esym \left[ \frac {R(...,z_{i1},...,z_{id_i},...)}{\bd!} \frac {R'(...,z_{i,d_i+1},...,z_{i,d_i+d'_i},...)}{\bd'!} \prod_{i,i' = 1}^n \prod^{1 \leq a \leq d_i}_{d_{i'} < a' \leq d_{i'}+d'_{i'}} \zeta \left( \frac {z_{ia}}{z_{i'a'}} \right) \right] 
$$
Let $\CS^+$ the subalgebra of $\CV$ generated by $\{ z_{i1}^k \}^{1\leq i \leq n}_{k \in \BZ}$ and we let $\CS^- = \left( \CS^+ \right)^\emph{op}$. \\
	
\end{definition}

\noindent The algebras $\CS^\pm$ are graded by $\pm \nn \times \BZ$, where $\deg R(...,z_{i1},...,z_{id_i},...) = (\pm \bd, k)$ if $\bd = (d_1,..,d_n)$, while $k$ denotes the homogeneous degree of $R$. We will write $\CS_{\pm \bd}$ for the graded pieces of $\CS^\pm$ with respect to the $\pm \nn$ direction only. \\

\begin{theorem} (\cite{Tor}) The subalgebras $\CS^\pm$ coincide with the $\fff$--vector subspaces of $\CV$ consisting of rational functions $R(...,z_{ia},...)$ that satisfy:
	$$
	R(...,z_{ia},...) = \frac {r(...,z_{ia},...)}{\prod_{i=1}^n \prod_{a,a'} (z_{ia} q - z_{i+1,a'}q^{-1})} 
	$$
	where $r$ is any Laurent polynomial which vanishes at the specializations:
	$$
	(z_{i1}, z_{i2}, z_{i-1,1}) \mapsto (w,wq^2,w), \qquad  (z_{i1}, z_{i2}, z_{i+1,1}) \mapsto (w,wq^{-2},w)
	$$
	for any $i \in \{1,...,n\}$. This vanishing property is the natural analogue of the wheel conditions studied in \cite{FHHSY, FO}. By convention, we set $z_{n+1,a} = z_{1a} \oq^{-2}$, $z_{0a} = z_{na} \oq^2$. \\
	
\end{theorem}

\noindent It is a well-known fact (see \cite{E}) that $\CS^\pm \cong \UUpm$. In order to obtain the entire quantum toroidal algebra and not just its halves, define the double shuffle algebra:
\begin{equation}
\label{eqn:whole}
\CS = \CS^+ \otimes \CS^0 \otimes \CS^- \Big / \text{relations modeled after \eqref{eqn:rel aff 3}, \eqref{eqn:rel aff 5}}
\end{equation}
where: 
$$
\CS^0 = \displaystyle \frac {\fff \Big \langle \ph_{i,k}^\pm, \psi_s^{\pm 1}, c^{\pm 1}, \barc^{\pm 1} \Big \rangle^{k \in \BN, i \in \BZ/n\BZ}_{s \in \BZ}}{\text{relations \eqref{eqn:rel aff 0}--\eqref{eqn:rel aff 2}}}
$$
Therefore, there is a isomorphism:
\begin{equation}
\label{eqn:bialg iso}
\CS \cong \UU
\end{equation}
given by sending $\left( z_{i1}^k \right)^\pm \leadsto x_{i,k}^\pm, \ \ph_{i,k}^\pm \leadsto \ph_{i,k}^\pm, \ \psi_s \leadsto \psi_s, c \leadsto c, \barc \leadsto \barc$. \\

\subsection{} Let us consider the following halves of $\CS$:
\begin{align}
&\CS^\geq = \CS^+ \otimes \frac {\fff \Big \langle \ph_{i,k}^+, \psi^{\pm 1}_s, c^{\pm 1}, \barc^{\pm 1} \Big \rangle}{\text{relations \eqref{eqn:rel aff 0}--\eqref{eqn:rel aff 2}}} \Big/ \text{relations modeled after \eqref{eqn:rel aff 3}} \label{eqn:geq} \\
&\CS^\leq = \CS^- \otimes \frac {\fff \Big \langle \ph_{i,k}^-, \psi^{\pm 1}_s, c^{\pm 1}, \barc^{\pm 1} \Big \rangle}{\text{relations \eqref{eqn:rel aff 0}--\eqref{eqn:rel aff 2}}} \Big/ \text{relations modeled after \eqref{eqn:rel aff 3}} \label{eqn:leq} 
\end{align}
For $\be, \bd \in \nn$, we will write $\be \leq \bd$ if $e_i \leq d_i$ for all $i \in \{1,\dots,n\}$. With this in mind, the algebras \eqref{eqn:geq} and \eqref{eqn:leq} are endowed with topological coproducts:
\begin{align}
&\Delta(R^+) = \sum_{\be \in \nn}^{\be \leq \bd}  \frac {\left[ \prod_{1 \leq i \leq n}^{a>e_i} \ph^+_i(z_{ia}\barc_1) \otimes 1 \right] R^+(z_{i, a \leq e_i} \otimes z_{i, a>e_i} \barc_1)}{\prod_{1 \leq i' \leq n}^{1 \leq i \leq n} \prod^{a \leq e_{i}}_{a' > e_{i'}} \zeta(z_{i'a'}\barc_1/z_{ia})} \label{eqn:coproduct 3} \\
&\Delta(R^-) = \sum_{\be \in \nn}^{\be \leq \bd}  \frac {R^-(z_{i, a \leq e_i}\barc_2 \otimes z_{i, a>e_i}) \left[ \prod_{1 \leq i \leq n}^{a \leq e_i} 1 \otimes \ph^-_i(z_{ia}\barc_2) \right]}{\prod_{1 \leq i' \leq n}^{1 \leq i \leq n} \prod^{a \leq e_{i}}_{a' > e_{i'}} \zeta(z_{ia} \barc_2/z_{i'a'})} \label{eqn:coproduct 4}
\end{align}

\text{}

\begin{remark}
\label{rem:coproduct}
	
To think of \eqref{eqn:coproduct 3} as a tensor, we expand the right-hand side in non-negative powers of $z_{ia} / z_{i'a'}$ for $a\leq e_i$, $e_{i'} < a'$, thus obtaining an infinite sum of monomials. In each of these monomials, we put the symbols $\ph^+_{i,d}$ to the very left of the expression, then all powers of $z_{ia}$ with $a\leq e_i$, then the $\otimes$ sign, and finally all powers of $z_{ia}$ with $a>e_i$. The powers of the central element $\barc_1 = \barc \otimes 1$ are placed in the first tensor factor. The resulting expression will be a power series, and therefore lies in a completion of $\CS^\geq \otimes \CS^\geq$. The same argument applies to \eqref{eqn:coproduct 4}. \\	
	
\end{remark}

\noindent There exists a pairing between the halves $\CS^\geq$ and $\CS^\leq$, given by:
\begin{equation}
\label{eqn:pairshuf1}
\left \langle R^+,R^- \right \rangle = \frac {(1-q^{-2})^{|\bd|}}{\bd!} \oint \frac {R^+(...,z_{ia},...)R^-(...,z_{ia},...)}{\prod_{i,j=1}^{n} \prod^{(i,a) \neq (j,b)}_{a\leq d_i, b \leq d_j} \zeta(z_{ia}/z_{jb})} \prod^{1 \leq i \leq n}_{1 \leq a \leq d_i} \frac {dz_{ia}}{2\pi i z_{ia}} \qquad 
\end{equation}
for any $R^+ \in \CS_{\bd}$ and $R^- \in \CS_{-\bd}$ (we refer the reader to \cite{Tor} or \cite{PBW} for details). \\

\subsection{}
\label{sub:uuu}

In \cite{Tor, PBW}, we constructed a PBW basis of $\CS^\pm \cong \UUpm$. More precisely we construct particular elements of $\CS^\pm$ called ``PBW generators", indexed by a totally ordered set, and claim that a linear basis of $\CS^\pm$ is given by ordered products of the PBW generators. In the case of the algebra $\CE^\pm \cong \uupm$, we have already seen in Subsection \ref{sub:sub} that the PBW generators of $\CE^\pm$ are indexed by:
$$
(i<j) \in \zzz 
$$ 
It should come as no surprise that the PBW generators of $\CS^\pm$ are indexed by:
$$
\Big( (i<j), k \Big) \in \zzz \times \BZ
$$
If we write $\mu = \frac {j-i}k$, we will find it more useful to index the PBW generators by:
$$
\Big( (i<j), \mu \Big) \in \zzz \times (\BQ \sqcup \infty)
$$
such that $\frac {j-i}{\mu} \in \BZ$. For any choice of $i<j$ and $\mu$ as above, we define:
\begin{align} 
&A^\mu_{\pm [i;j)} = \sym \left[ \frac {\prod_{a=i}^{j-1} (z_a \oq^{\frac {2a}n})^{\lceil \frac {a-i+1}{\mu} \rceil  - \lceil \frac {a-i}{\mu} \rceil}}{\left(1 - \frac {z_i q^2}{z_{i+1}}\right) ... \left(1 - \frac {z_{j-2} q^2}{z_{j-1}}\right)} \prod_{i\leq a < b < j} \zeta \left( \frac {z_b}{z_a} \right)  \right] \in \CS^\pm \label{eqn:shuf 1} \\
&B^\mu_{\pm [i;j)} =\sym \left[ \frac {\prod_{a=i}^{j-1} (z_a \oq^{\frac {2a}n})^{\lfloor \frac {a-i+1}{\mu} \rfloor  - \lfloor \frac {a-i}{\mu} \rfloor}}{\left(1 - \frac {z_{i+1}}{z_i}\right) ... \left(1 - \frac {z_{j-1}}{z_{j-2}}\right)} \prod_{i\leq a < b < j} \zeta \left( \frac {z_a}{z_b} \right)  \right] \in \CS^\pm \label{eqn:shuf 2}
\end{align}
In order to think of the RHS of \eqref{eqn:shuf 1} and \eqref{eqn:shuf 2} as shuffle elements, we relabel the variables $z_i,...,z_{j-1}$ according to the following rule $\forall a \in \{i,...,j-1\}$:
\begin{equation}
\label{eqn:identify}
z_{a} \text{ should be replaced with } z_{\bar{a}\bullet_a} \oq^{-2\left \lfloor \frac {a-1}n \right \rfloor}
\end{equation}
where $\bar{a}$ is the residue class of $a$ in the set $\{1,...,n\}$ (and the precise value of $\bullet_a \in \BN$ is immaterial due to the ``Sym" in \eqref{eqn:shuf 1}--\eqref{eqn:shuf 2}). If $\frac {j-i}{\mu} \notin \BZ$, the LHS of \eqref{eqn:shuf 1} and \eqref{eqn:shuf 2} are defined to be 0. We will occasionally write:
$$
A_{\pm [i;j)}^{(\pm k)} = A_{\pm [i;j)}^\mu \qquad \text{and} \qquad  B_{\pm [i;j)}^{(\pm k)} = B_{\pm[i;j)}^\mu
$$
where $k = \frac {j-i}{\mu}$, in order to indicate the fact that $\deg A_{\pm [i;j)}^{(\pm k)} , B_{\pm [i;j)}^{(\pm k)}= \pm ([i;j), k)$. \\

\subsection{} 
\label{sub:slope stuff}

We will now define an algebra $\DD$ that is isomorphic to $\CS \cong \UU$, much like the algebra $\CE$ was isomorphic to $\uu$ (see Theorem \ref{thm:affine}). The first step is to define an infinite family of the latter algebras. Specifically, for any coprime $(a,b) \in \BZ \times \BN$, define:
\begin{equation}
\label{eqn:e mu}
\CE_{\frac ab} = U_q(\dot{\fgl}_{\frac ng})^{\otimes g}
\end{equation}
where $g = \gcd(n,a)$. The root generators of \eqref{eqn:e mu} are parametrized by:
$$
\text{triples } (u,v,r) \quad \text{where} \quad (u , v) \in \frac {\BZ^2}{\left( \frac ng, \frac ng \right) \BZ} \quad \text{and} \quad  r \in \{1,...,g\}
$$
However, we choose to replace a triple $(u,v,r)$ as above by:
\begin{equation}
\label{eqn:eis}
(i,j) = (r + a u, r + a v) \in \zzz
\end{equation}
and therefore we will use the following notation for the root generators of \eqref{eqn:e mu}:
\begin{align}
&f_{[i;j)}^{\frac ab} = 1^{\otimes r-1} \otimes f_{[u;v)} \otimes 1^{\otimes g-r} \label{eqn:root generators} \\
&\barf_{[i;j)}^{\frac ab} = 1^{\otimes r-1} \otimes \barf_{[u;v)} \otimes 1^{\otimes g-r} \label{eqn:root generators 2}
\end{align}
for any indices $i,j$, which are connected to $u,v,r$ by \eqref{eqn:eis}. Note that we allow $i>j$ in formula \eqref{eqn:eis}, via the convention:
$$
[i;j) = - [j;i) \quad \text{if } i > j
$$
Moreover, formula \eqref{eqn:eis} implies that $k := (j-i) \frac ba$ is an integer, so we will write:
\begin{equation}
\label{eqn:f mu}
f_{[i;j)}^{(k)} = f_{[i;j)}^{\frac ab} \quad \text{and} \quad \barf_{[i;j)}^{(k)} = \barf_{[i;j)}^{\frac ab}
\end{equation}
We make the convention that $f_{[i;j)}^{(k)} = \barf_{[i;j)}^{(k)} = 0$ if $k = (j-i)  \frac ba \notin \BZ$. \\

\subsection{} For any $(i,j) \in \zzz$ and $k \in \BN$, let us write $\mu = \frac {j-i}k$. The assignment:
$$
\deg f_{[i;j)}^{(k)} = \deg \barf_{[i;j)}^{(k)} = ([i;j),k) 
$$
makes $\CE_\mu$ into a $\zz \times \BZ$ graded algebra. For all $\bd \in \zz$, we will write $\CE_{\mu|\bd}$ for its degree $\bd \times \BZ$ graded piece. Consider two invertible central elements $c,\barc$, and recall the Cartan elements $\psi_1,...,\psi_n$ of $\CE_\infty = \CE$. As $\mu$ ranges over $\BQ \sqcup \infty$, we will identify the central elements of the algebras $\CE_\mu$ according to the rule:
$$
\Big(\text{central element of }\CE_{\frac ab}\Big) = c^{\frac ag} \barc^{\frac {bn}g} 
$$
and identify the Cartan elements of the algebras $\CE_\mu$ according to the rule:
$$
\Big(\psi_s \text{ on } r\text{--th factor of } U_{q}(\dot{\fgl}_{\frac ng})^{\otimes g} = \CE_{\frac ab} \Big) = \psi_{r+sa} \barc^{bs}
$$
where $g = \gcd(n,a)$. Hence we have the following relation:
\begin{equation}
\label{eqn:rel 1}
\psi_s X = q^{- \langle \bs^s, \bd \rangle} X \psi_s \qquad \forall X \in \CE_{\mu|\bd}, \ \forall s \in \{1,...,n\}
\end{equation} 
where $\langle \cdot, \cdot \rangle$ is the bilinear form on $\BZ^n$ determined by $\langle \bs^i, \bs^j \rangle = \delta_{i}^{j} - \delta_{i}^{j+1}$. \\

\noindent We henceforth extend the scalars in $\CE_\mu$ from $\BQ(q)$ to $\BQ(q,\oq^{\frac 1n})$. The following is an obvious consequence of the structure defined in Subsections \ref{sub:two realizations}--\ref{sub:bars}. \\

\begin{proposition} 
\label{prop:endowed}

For any $\mu$, the algebra $\CE_\mu$ has a coproduct $\Delta_\mu$, for which: \footnote{The formulas below differ from \eqref{eqn:cop 1}--\eqref{eqn:cop 2} and \eqref{eqn:cop 11}--\eqref{eqn:cop 22} by powers of $\barc$. This is achieved by rescaling the $f$ and $\barf$ generators by certain powers of $\barc$, which we will tacitly do.}
\begin{align*}
&\Delta_\mu(f^\mu_{[i;j)}) = \sum_{s \in \{i,...,j\}} f^\mu_{[s;j)} \frac {\psi_i}{\psi_s} \barc^{\frac {i-s}{\mu}} \otimes f^\mu_{[i;s)} \\
&\Delta_\mu(\barf^\mu_{[i;j)}) = \sum_{s \in \{i,...,j\}}  \frac {\psi_s}{\psi_j} \barf^\mu_{[i;s)} \barc^{\frac {s-j}{\mu}} \otimes \barf^\mu_{[s;j)} \\
&\Delta_\mu(f^\mu_{-[i;j)}) = \sum_{s \in \{i,...,j\}} f^\mu_{-[i;s)} \otimes f^\mu_{-[s;j)} \frac {\psi_s}{\psi_i} \barc^{\frac {s-i}{\mu}} \\
&\Delta_\mu(\barf^\mu_{-[i;j)}) = \sum_{s \in \{i,...,j\}} \barf^\mu_{-[s;j)} \otimes \frac {\psi_j}{\psi_s} \barf^\mu_{-[i;s)} \barc^{\frac {j-s}{\mu}}
\end{align*} 
\footnote{The notation $s \in \{i,...,j\}$ means ``$s$ runs between $i$ and $j$", for either $i<j$ or $i>j$.} $\forall (i,j) \in \zzz$ such that $\frac {j-i}{\mu} \in \BN$. For all such $i,j$, we have linear maps:
$$
\alpha_{\pm [i;j)} : \CE_{\mu|\pm [i;j)} \rightarrow \fff
$$ 
that satisfy property \eqref{eqn:pseudo} and:
\begin{align}
&\alpha_{\pm [i;j)}(f^{(\pm k)}_{\pm [i';j')}) = \delta_{(i',j')}^{(i,j)} (1-q^2) \oq_{\pm}^{\frac {\gcd(j-i,k)}n} \label{eqn:f alpha}\\
&\alpha_{\pm [i;j)}(\barf_{\pm [i';j')}^{(k)}) = \delta_{(i',j')}^{(i,j)} (1-q^{-2}) \oq_{\pm}^{-\frac {\gcd(j-i,k)}n} \label{eqn:barf alpha}
\end{align} 

\end{proposition} 

\text{}

\subsection{} 
\label{sub:pbw stuff}

Let us consider the subalgebras $\CE^\pm_\mu \subset \CE_\mu$ generated by those elements \eqref{eqn:f mu} where the sign of $k$ is equal to $\pm$. As in \eqref{eqn:simple imaginary}, we obtain the following elements:
\begin{align}
&\{p_{\pm [i;i+a)}^\mu\}_{i \in \BZ/n\BZ} \subset \CE_\mu^\pm \text{ are simple generators, if } n \nmid a \label{eqn:simple gen} \\
&\{p_{\pm l\bde,r}^\mu\}^{r \in \{1,...,g\}}_{l \in \BN \frac ag} \subset \CE_\mu^\pm \text{ are imaginary generators} \label{eqn:imaginary gen}
\end{align}
$\forall \mu = \frac ab \in \BQ \sqcup \infty$. These elements are all primitive for the coproduct $\Delta_\mu$ and satisfy:
\begin{align}
&\alpha_{\pm [u;v)} \left(p_{\pm [i;i+a)}^\mu\right) = \pm \delta_{(u,v)}^{(i,i+a)} \label{eqn:simple alpha} \\
&\alpha_{\pm [s;s+ln)} \left(p_{\pm l\bde,r}^\mu\right) = \pm \delta_{s \text{ mod }g}^r \label{eqn:imaginary alpha}
\end{align} 
for any $u,v,s$. We may use the notation:
\begin{equation}
\label{eqn:primitive elements}
p_{\pm [i;i+a)}^{(\pm b)} = p_{\pm [i;i+a)}^\mu \qquad \text{and} \qquad p_{\pm l\bde,r}^{\left(\pm \frac {ln}{\mu} \right)} = p_{\pm l\bde,r}^\mu
\end{equation}
to emphasize the fact that $\deg p_{\bd}^{(k)} = (\bd,k) \in \zz \times \BZ$. Let us write:
\begin{equation}
\label{eqn:residue class}
\bari = i - n \left \lfloor \frac {i-1}n \right \rfloor 
\end{equation}
for all $i \in \BZ$, and recall that $\delta_i^j = 1$ if $i \equiv j$ modulo $n$, and 0 otherwise. \\

\begin{definition}
\label{def:explicit d} 

Consider the algebras:
\begin{equation}
\label{eqn:explicit d}
\DD^\pm = \bigotimes_{\mu \in \BQ}^\rightarrow \CE^\pm_\mu \Big/ \text{\eqref{eqn:rel 2 pbw}--\eqref{eqn:rel 3 pbw}}
\end{equation}
whose generators, by the discussion above, are denoted by:
$$
\left\{ p_{\pm [i;j)}^{(\pm k)}, p_{\pm l\bde, r}^{(\pm k')}\right\}^{k,k'>0,  l \in \BZ \backslash 0}_{(i,j)\in \zzz, r \in \BZ/g\BZ}
$$
Whenever $d:=  \det \begin{pmatrix} k & k' \\ j-i & nl \end{pmatrix}$ satisfies $|d| = \gcd(k',nl)$, we impose the relation:
\begin{equation}
\label{eqn:rel 2 pbw}
\Big[ p_{\pm [i;j)}^{(\pm k)}, p_{\pm l\bde, r}^{(\pm k')} \Big] = \pm p_{\pm [i;j+ln)}^{(\pm k \pm k')} \left(\delta_{i \text{ mod }g}^r \oq_\pm^{\frac dn} - \delta_{j \text{ mod }g}^r \oq_\pm^{-\frac dn} \right)
\end{equation}
and whenever $\det \begin{pmatrix} k & k' \\ j-i & j'-i' \end{pmatrix} = \gcd(k+k',j+j'-i-i')$, we impose: 
\begin{equation}
\label{eqn:rel 3 pbw}
p_{\pm [i;j)}^{(\pm k)} p_{\pm [i';j')}^{(\pm k')} q^{\delta_{j'}^{i} - \delta_{i'}^{i}} - p_{\pm [i';j')}^{(\pm k')} p_{\pm [i;j)}^{(\pm k)} q^{\delta_{j'}^{j} - \delta_{i'}^{j}} =
\end{equation}
$$
= \sum_{[t;s) = [i';j')} f^\mu_{\pm [t,j)} \barf_{\pm [i;s)}^\mu \left(\delta_{j'}^{s}  \frac {q^{-\delta_{j'}^{i'}}}{q^{-1}-q} - \delta_{j'}^{i'} \frac {\oq_\mp^{\frac {2(\overline{k'(s-i')})}n}}{\oq_\mp^2 - 1} \right)
$$
where $\mu = \frac {j+j'-i-i'}{k+k'}$. \\ 

\end{definition}

\subsection{} 
\label{sub:explicit} 

Let us write $\DD^0 = \CE_\infty$, and use the notation $p^{(0)}_{\pm [i;i+1)}$ and $p_{\pm k\bde, 1}^{(0)}$ for its simple and imaginary generators, respectively, as defined in Subsection \ref{sub:pbw stuff}. \\

\begin{definition}
\label{def:explicit double d} 

Let us define the double of the algebras \eqref{eqn:explicit d} as:
\begin{equation}
\label{eqn:explicit double d}
\DD = \DD^+ \otimes \DD^0 \otimes \DD^- \Big/\text{relations \eqref{eqn:f zero imaginary}--\eqref{eqn:e plus minus}}
\end{equation}
where:
\begin{equation}
\label{eqn:f zero imaginary} 
\Big[ p_{\pm [i;j)}^{(\pm 1)}, p_{\pm l\bde, 1}^{(0)} \Big] = \pm p_{\pm [i;j+ln)}^{(\pm 1)} \left(\oq_\pm^{l} - \oq_\pm^{-l} \right)
\end{equation}
\begin{equation}
\label{eqn:e zero imaginary} 
\left[ p^{(\pm 1)}_{\pm [i;j)}, p_{\mp l\bde, 1}^{(0)} \right] = \pm p^{(\pm 1)}_{\pm [i;j-nl)} c^{\pm l} \left( \oq_\mp^{l} -  \oq_\mp^{-l} \right)
\end{equation}
\begin{multline}
p_{\pm [i;j)}^{(\pm 1)} p_{\pm [s;s+1)}^{(0)} q^{\delta_{s+1}^{i} - \delta_{s}^{i}} - p_{\pm [s;s+1)}^{(0)} p_{\pm [i;j)}^{(\pm 1)} q^{\delta_{s+1}^{j} - \delta_{s}^{j}} = \\
= \pm \left( \delta_{s+1}^{i} \cdot \oq_\mp^{- \frac 1n} p^{(\pm 1)}_{\pm [i-1;j)} - \delta_{s}^{j} \cdot \oq_\mp^{\frac 1n} p^{(\pm 1)}_{\pm [i;j+1)} \right) \label{eqn:f zero simple} 
\end{multline}
\begin{equation}
\label{eqn:e zero simple} 
\left[ p^{(\pm 1)}_{\pm [i;j)}, p^{(0)}_{\mp [s;s+1)} \right] = \pm \left(\delta_{s}^{i} \cdot \oq_\mp^{\frac 1n} p^{(\pm 1)}_{\pm [i+1;j)}  \frac {\psi_{i+1}^{\pm 1}}{\psi_{i}^{\pm 1}} -  \delta_{s+1}^{j} \cdot \oq_\mp^{-\frac 1n} \frac {\psi_j^{\pm 1}}{\psi_{j-1}^{\pm 1}} p^{(\pm 1)}_{\pm [i;j-1)} \right)
\end{equation}
and:
\begin{equation}
\label{eqn:e plus minus} 
\left[ p_{[i;j)}^{(1)}, p_{[i';j')}^{(-1)} \right] = \frac 1{q^{-1}-q}
\end{equation}
$$
\left( \sum_{\left \lceil \frac {i-j'}n \right \rceil \leq k \leq \left \lfloor \frac {j-i'}n \right \rfloor} f^{(0)}_{[i'+nk;j)} \frac {\psi_{j'}}{\psi_{i'}\barc} \barf^{(0)}_{[i;j'+nk)} - \sum_{\left \lceil \frac {j'-i}n \right \rceil \leq k \leq \left \lfloor \frac {i'-j}n \right \rfloor} f^{(0)}_{-[j+nk;i')} \frac {\psi_{j}\barc}{\psi_{i}} \barf_{-[j';i+nk)}^{(0)} \right)
$$

\end{definition} 

\noindent Relations \eqref{eqn:f zero imaginary}--\eqref{eqn:e plus minus} are sufficient to describe all the commutation relations between the three tensor factors of \eqref{eqn:explicit double d}, because the algebras $\DD^\pm$ are generated by:
\begin{equation}
\label{eqn:generators subalgebra}
\Big \{p_{\pm [i;j)}^{(\pm 1)} \Big \}_{(i,j) \in \zzz}
\end{equation}
(we will prove this in Proposition \ref{prop:gen}). Therefore, relations \eqref{eqn:f zero imaginary}--\eqref{eqn:e plus minus} allow us to ``reorder" any product of elements from the subalgebras $\DD^+$, $\DD^0$, $\DD^-$, i.e. to write said product as a sum of products of elements from $\DD^+$, $\DD^0$, $\DD^-$, \underline{in this order}. \\

\begin{theorem}
\label{thm:pbw} (\cite{PBW}) There is an isomorphism $\DD \cong \CS$. \\
\end{theorem}

\begin{proof} \emph{(sketch, see \cite[Theorem 5.7, Proposition 5.8]{PBW} for details \footnote{Note that the slope $\mu$ in \loccit is actually $\frac 1{\mu}$ in our notation})} The subalgebra:
$$
\CS \supset \CT_\mu = \begin{cases} \Big \langle A_{[i;j)}^\mu, B_{-[i;j)}^\mu  \Big \rangle^{i<j}_{\frac {j-i}{\mu} \in \BZ} &\text{if } \mu > 0 \text{ or } \mu = \infty \\ \Big \langle A_{-[i;j)}^\mu, B_{[i;j)}^\mu \Big \rangle^{i<j}_{\frac {j-i}{\mu} \in \BZ} &\text{if } \mu < 0 \end{cases}
$$
is isomorphic to $\CE_\mu$ of \eqref{eqn:e mu} for all $\mu \in \BQ \sqcup \infty$, by sending: 
\begin{align}
&f_{[i;j)}^\mu \leadsto \frac 1{\psi_j}  A_{[i;j)}^\mu \psi_i \barc^{\frac {i-j}{\mu}} \cdot \left(-\oq_-^{\frac 2n} \right)^{i-j}  \label{eqn:mu pos 1} \\
&f^\mu_{-[i;j)} \leadsto \frac 1{\psi_i} B_{-[i;j)}^\mu \psi_j \barc^{\frac {i-j}{\mu}} \cdot \left(-\oq_+^{\frac 2n} \right)^{i-j} \label{eqn:mu pos 2}
\end{align}
if $\mu > 0$ or $\mu = \infty$, and:
\begin{align}
&f_{[i;j)}^\mu \leadsto B_{[i;j)}^\mu \cdot q^{j-i} \label{eqn:mu neg 1} \\
&f^\mu_{-[i;j)} \leadsto  A_{-[i;j)}^\mu \cdot q^{j-i} \label{eqn:mu neg 2}
\end{align}
if $\mu < 0$. Similarly, $\CT_0 := \CS^0$ is isomorphic to a tensor product of $n$ Heisenberg algebras. As in Subsection \ref{sub:pbw stuff}, this allows us to construct the images of the simple and imaginary generators:
$$
p_{\pm [i;j)}^{(\pm k)} \leadsto X_{\pm [i;j)}^{(\pm k)} \in \CS, \qquad p_{\pm l\bde, r}^{(\pm k')} \leadsto X_{\pm l\bde, r}^{(\pm k')} \in \CS
$$
In \cite{PBW}, we showed that the simple and imaginary generators $X_{...}^{...} \in \CS$ defined above satisfy relations \eqref{eqn:rel 2 pbw}--\eqref{eqn:rel 3 pbw} and \eqref{eqn:f zero imaginary} --\eqref{eqn:e plus minus} with $p$'s replaced by $X$'s, hence we obtain an algebra homomorphism:
\begin{equation}
\label{eqn:isomorphism}
\Phi : \DD \rightarrow \CS
\end{equation}
This map is an isomorphism because ordered products of the elements $f_{[i;j)}^\mu$ (respectively their images under the assignments \eqref{eqn:mu pos 1}--\eqref{eqn:mu neg 2}) in increasing order of $\mu$ were shown in \cite{PBW} (respectively \cite{Tor}) to form a linear basis of $\DD$ (respectively $\CS$).

\end{proof} 

\begin{corollary}
\label{cor:explicit} 

If we combine \eqref{eqn:bialg iso} with \eqref{eqn:isomorphism}, we obtain an isomorphism:
$$
\Psi : \DD \cong \UU 
$$
Let us define:
$$
\Psi(\DD^+) = \UUup , \quad \Psi(\DD^-) = \UUdown
$$
and consider the usual triangular decomposition (see \eqref{eqn:positive} and \eqref{eqn:negative}):
$$
\Psi(\DD^0) = \uu = \uug \otimes \uul 
$$
Then the algebras 
\begin{align*}
&\tUUup \cong \UUup \otimes \uug \\
& \tUUdown \cong \UUdown \otimes \uul
\end{align*}
yield the decomposition $\UU = \tUUup \otimes \tUUdown$ of \eqref{eqn:decomp 1}. \\

\end{corollary}

\subsection{} 
\label{sub:formulation} 

The purpose of the remainder of the present paper is to make:
$$
\DD^+ \otimes \uug \quad \text{and} \quad \DD^- \otimes \uul 
$$
into bialgebras, in such a way that $\DD$ becomes their Drinfeld double. This will be done by recasting $\DD^\pm$ as a new type of shuffle algebra, as described in Subsection \ref{sub:describe}. With this in mind, we will prove the following more explicit version of Theorem \ref{thm:main}. \\

\noindent \textbf{Theorem 1.5 (explicit):}  \emph{If $\CA^+$ and $\CA^-$ are the shuffle algebras that will be introduced in Definitions \ref{def:shuf aff} and \ref{def:minus algebra}, respectively, then we have algebra isomorphisms:}
$$
\DD^+ \cong \CA^+, \qquad \DD^- \cong \CA^{-,\op}
$$
\emph{Moreover, the extended algebras $\tCA^+ = \CA^+ \otimes \UUg$ and $\tCA^- = \CA^- \otimes \UUl$ defined in \eqref{eqn:extended aff pm} have topological bialgebra structures and a bialgebra pairing between them, such that the Drinfeld double:}
$$
\CA = \tCA^+ \otimes \tCA^{-,\op,\coop}
$$
\emph{is isomorphic to $\DD$ (and hence also with $\CS$ and $\UU$) as an algebra.} \\

\subsection{} 
\label{sub:bill}

In the remainder of this Section, we will study the algebra $\CE \cong \uu$ in more detail, and fill the gaps left in the discussion above. Let us consider the following matrix-valued rational function called an $R$--matrix with parameter $x$:
\begin{equation}
\label{eqn:explicit r}
R(x) = \sum_{1\leq i,j \leq n} E_{ii} \otimes E_{jj} \left(\frac {q - xq^{-1}}{1-x} \right)^{\delta_i^j} + (q - q^{-1}) \sum_{1 \leq i \neq j \leq n} E_{ij} \otimes E_{ji} \frac {x^{\delta_{i<j}}}{1-x}
\end{equation}
where $E_{ij} \in \End(\BC^n)$ denotes the $n \times n$ matrix with a single $1$ at the intersection of row $i$ and column $j$, and zeroes everywhere else. We will now give an alternate version of Definition \ref{def:two realizations}, and afterwards show how to match notations: \\

\begin{definition}
\label{def:two realizations 2}
	
(\cite{DF}) Consider the algebra:
\begin{equation}
\label{eqn:two realizations 2}
\CE:= \BQ(q) \Big \langle s_{[i;j)}, t_{[i;j)}, c^{\pm 1} \Big \rangle^{1 \leq i \leq n}_{i \leq j \in \BZ} \Big/ \text{relations \eqref{eqn:rtt 0}--\eqref{eqn:rtt 3}}
\end{equation}
where:
\begin{equation}
\label{eqn:rtt 0}
c \text{ is central, and } s_{[i;i)} t_{[i;i)} = 1
\end{equation}
\begin{align}
R\left( \frac xy \right) S_1(x) S_2(y) &=  S_2(y)   S_1(x) R\left( \frac xy \right) \label{eqn:rtt 1} \\
R\left( \frac xy \right) T_1(x) T_2(y) &=  T_2(y)   T_1(x) R\left( \frac xy \right) \label{eqn:rtt 2} \\
R\left( \frac {xc}{y} \right) S_1(x) T_2(y) &= T_2(y) S_1(x) R\left( \frac {x}{yc} \right) \label{eqn:rtt 3} 
\end{align}
\footnote{Relations \eqref{eqn:rtt 1} and \eqref{eqn:rtt 2} can be made explicit by expanding in either positive or negative powers of $x/y$, but \eqref{eqn:rtt 3} must be expanded in negative powers of $x/y$.} where $Z_1 = Z \otimes \emph{Id}$ and $Z_2 = \emph{Id} \otimes Z$ for any symbol $Z$, and:
\begin{align}
&S(x) = \sum_{1 \leq i,j \leq n, d \geq 0}^{\text{if } d = 0 \text{ then } i\leq j} s_{[i;j+nd)} \cdot E_{ij}x^{-d} \label{eqn:series s} \\
&T(x) = \sum_{1 \leq i,j \leq n, d \geq 0}^{\text{if } d = 0 \text{ then } i\leq j} t_{[i;j+nd)} \cdot E_{ji}x^{d} \label{eqn:series t}
\end{align}
	
\end{definition}

\noindent The series $S(x)$, $T(y)$ are the transposes of the series denoted $T^-(x)$, $T^+(y)$ in \cite{Tor}, which explains the discrepancy between our conventions and those of \loccit \\

\subsection{} 
\label{sub:hmmm}

 We will write $\psi_k = s_{[k;k)}^{-1} = t_{[k;k)} \in \CE$ for all $1 \leq k \leq n$, and set:
\begin{equation}
\label{eqn:no bars}
f_{[i;j)} = s_{[i;j)} \psi_i, \qquad f_{-[i;j)} = t_{[i;j)}  \psi_i^{-1}
\end{equation}
$\forall 1 \leq i \leq n$ and $i \leq j \in \BZ$. We will extend our notation to all integers by setting:
$$
f_{\pm [i+n;j+n)} = f_{\pm [i;j)}, \qquad \psi_{k+n} = c \psi_k 
$$
$\forall i \leq j, \forall k$. It is elementary to see that relations \eqref{eqn:rtt 1}--\eqref{eqn:rtt 3} can be rewritten as:
\begin{equation}
\label{eqn:f f}
\sum_{\pm [i;j) \pm' [i';j') = \bd} \text{coefficient} \cdot f_{\pm [i;j)} f_{\pm' [i';j')} = 0
\end{equation}
for all $\pm, \pm' \in \{+,-\}$ and $\bd \in \pm \nn \pm' \nn$. We will not need to spell out the coefficients in \eqref{eqn:f f} explicitly, but they can easily be obtained by expanding \eqref{eqn:rtt 1}--\eqref{eqn:rtt 3} as power series in $x/y$ and equating matrix coefficients of every $E_{ij} \otimes E_{i'j'}$. In particular, an important special case of relation \eqref{eqn:f f} reads:
\begin{equation}
\label{eqn:psi f}
\psi_k f_{\pm [i;j)} = q^{\pm \delta_{k}^{j} \mp \delta_{k}^{i}} f_{\pm [i;j)} \psi_k
\end{equation}
The bialgebra (and Drinfeld double) structure on $\CE = \CE^\geq \otimes \CE^\leq$ from \eqref{eqn:cop 0}--\eqref{eqn:cop 2} can be presented in terms of the matrix-valued power series \eqref{eqn:series s}--\eqref{eqn:series t} as: 
\begin{align}
&\Delta(S(x)) = (1 \otimes S(x c_1)) \cdot (S(x) \otimes 1) \label{eqn:cop rtt 1} \\ 
&\Delta(T(x)) = (1 \otimes T(x)) \cdot (T(xc_2) \otimes 1) \label{eqn:cop rtt 2} 
\end{align}
where $\cdot$ denotes matrix multiplication (i.e. the formula $E_{ij} \cdot E_{i'j'} = \delta_j^{i'} E_{ij'}$), and $c_1 = c \otimes 1$, $c_2 = 1 \otimes c$. It is straightforward to check that these coproducts respect relations \eqref{eqn:rtt 1}--\eqref{eqn:rtt 3}, i.e. extend to well-defined coproducts on the algebra $\CE$. \\

\subsection{} The pairing \eqref{eqn:pairing} takes the form:
\begin{equation}
\label{eqn:pairing affine}
\Big \langle S_1(x), T_2(y) \Big \rangle = R \left(\frac xy \right)^{-1} \quad \Leftrightarrow \quad \Big \langle S_1(x), T_2(y)^{-1} \Big \rangle = R \left(\frac xy \right)
\end{equation}
with $R$ expanded in non-negative powers of $\frac yx$. It is elementary to show (see \cite{Tor} for details) that \eqref{eqn:pairing affine} generates a bialgebra pairing, i.e. it intertwines the product with the coproduct on $\CE$, and that $\CE$ is the Drinfeld double of its halves with respect to this pairing. Moreover, we claim that the linear maps \eqref{eqn:alpha 0} which were used to normalize primitive elements of $\CE$ can be recovered from the assignment:
\begin{equation}
\label{eqn:def alpha}
\CE^\pm \stackrel{\alpha^\pm}\longrightarrow \End(\BC^n)[[x^{\mp 1}]], \qquad \begin{cases} \alpha^+(r) = \langle r, T(x)^{-1} \rangle &\text{if } r \in \CE^+ \\ \alpha^-(r) = \langle S(x)^{-1},r \rangle &\text{if } r \in \CE^- \end{cases}
\end{equation}
Indeed, we define:
\begin{equation}
\label{eqn:alpha plus}
\alpha_{\pm [i;j)} : \CE_{\pm [i;j)} \longrightarrow \fff
\end{equation}
as the coefficients of the maps $\alpha^\pm$, appropriately renormalized as follows:
\begin{align}
&\alpha^+(r) = \sum_{(i \leq j) \in \zzz}  \alpha_{[i;j)}(r) \cdot E_{ji} x^{\left \lfloor \frac {j-1}n \right \rfloor - \left \lfloor \frac {i-1}n \right \rfloor} \oq_+^{\frac {i-j}n} \label{eqn:alpha renorm plus} \\
&\alpha^-(r) = \sum_{(i \leq j) \in \zzz}  \alpha_{-[i;j)}(r) \cdot E_{ij} x^{\left \lfloor \frac {i-1}n \right \rfloor - \left \lfloor \frac {j-1}n \right \rfloor} \oq_-^{\frac {i-j}n} \label{eqn:alpha renorm minus} 
\end{align}
Then it is elementary to observe that the bialgebra property \eqref{eqn:bialg 1}--\eqref{eqn:bialg 2} of the pairing, together with definition \eqref{eqn:def alpha}, imply the multiplicativity property \eqref{eqn:pseudo}. \\

\subsection{}
\label{sub:e infinity}

We will henceforth write $S^+(x) = S(x)$ and $T^-(x) = T(x)$, so we have:
\begin{align*} 
&S^+(x) = \sum_{1 \leq i,j \leq n, d \geq 0}^{\text{if } d = 0 \text{ then } i\leq j} \underbrace{f_{[i;j+nd)} \psi_i^{-1}}_{s^+_{[i;j+nd)}} \cdot \ E_{ij} x^{-d} \\
&T^-(x) = \sum_{1 \leq i,j \leq n, d \geq 0}^{\text{if } d = 0 \text{ then } i\leq j} \underbrace{f_{-[i;j+nd)} \psi_i}_{t^-_{[i;j+nd)}} \cdot \ E_{ji} x^{d}
\end{align*}
as elements of $\CE\otimes \End(\BC^n)[[x^{\pm 1}]]$. We define series $S^-(x)$ and $T^+(x)$ by:
\begin{equation}
\label{eqn:new series s}
S^-(x) T^-(x \oq^2) = 1 
\end{equation}
\begin{equation}
 \label{eqn:new series t}
D^{-1} S^+ (x q^{2n} \oq^2)^\dagger D T^+(x)^\dagger = 1
\end{equation}
where $D = \text{diag}(q^2,...,q^{2n})$ and $\dagger$ denotes the transpose operation $E_{ij}^\dagger = E_{ji}$. The coefficients of these series will be denoted by:
\begin{align*} 
&T^+(x) = \sum_{1 \leq i,j \leq n, d \geq 0}^{\text{if } d = 0 \text{ then } i\leq j} \underbrace{\psi_j \barf_{[i;j+nd)} \oq^{\frac {2(j-i)}n}}_{t^+_{[i;j+nd)}} \cdot \ E_{ij} x^{-d} \\
&S^-(x) = \sum_{1 \leq i,j \leq n, d \geq 0}^{\text{if } d = 0 \text{ then } i\leq j} \underbrace{\psi_j^{-1} \barf_{-[i;j+nd)} \oq^{\frac {2(i-j)}n}}_{s^-_{[i;j+nd)}} \cdot \ E_{ji} x^{d}
\end{align*}
It is straightforward to show that formulas \eqref{eqn:new series s} and \eqref{eqn:new series t} imply \eqref{eqn:antipode}. The following identities are easy to prove, as consequences of \eqref{eqn:rtt 1}, \eqref{eqn:rtt 2}, \eqref{eqn:rtt 3}:
\begin{equation}
\label{eqn:remarkable 1}
T^+_1(x) R_{21} \left( \frac y{x\oq^2} \right) S^+_2(y) = S^+_2(y) R_{21} \left( \frac y{x\oq^2} \right) T^+_1(x)
\end{equation}
\begin{equation}
\label{eqn:remarkable 2}
S^-_1(x) R_{21} \left( \frac y{x\oq^2} \right) T^-_2(y) = T^-_2(y) R_{21} \left( \frac y{x\oq^2} \right) S^-_1(x)
\end{equation}
\begin{equation}
\label{eqn:remarkable 3}
R \left( \frac {xc}{y} \right) S_1^+(x) T_2^-(y) = T_2^-(y) S_1^+(x) R \left( \frac {x}{yc} \right) 
\end{equation}
\begin{equation}
\label{eqn:remarkable 4}
R \left( \frac {xc}{y} \right) T_2^+(y) S_1^-(x) = S_1^-(x) T_2^+(y) R \left( \frac x{yc} \right) 
\end{equation}
\begin{equation}
\label{eqn:remarkable 5}
S_1^+(x) R \left( \frac {x}{yc\oq^2} \right) S_2^-(y) = S_2^-(y) R \left( \frac {xc}{y\oq^2} \right)S_1^+(x) 
\end{equation}
\begin{equation}
\label{eqn:remarkable 6}
T_2^+(y) R \left( \frac {x}{y c \oq^2} \right) T_1^-(x) = T_1^-(x) R \left( \frac {x c}{y\oq^2} \right) T_2^+(y)
\end{equation}

\subsection{} In the present Subsection, we will show how to rewrite relations \eqref{eqn:f zero imaginary}--\eqref{eqn:e zero simple} between the generators of $\DD^+$, $\DD^0$, $\DD^-$ in terms of the series $S^\pm(x)$ and $T^\pm(x)$. \\ 

\begin{proposition}
\label{prop:connection plus}
	
Under the substitution: 
\begin{equation}
\label{eqn:correspondence plus}
p_{[i,j+nd)}^{(1)} \leadsto E_{ji} z^d \cdot \frac {\oq^{\frac {2i-1}n}}{1-q^2}
\end{equation}
for all $1\leq i,j\leq n$ and $d \in \BZ$, the following relations hold in $\DD$:
\begin{align}
&X^+_1(z) \cdot S_2^+(w) = S_2^+(w) \cdot \frac {R_{12} \left(\frac zw \right)}{f \left( \frac zw \right)}  X^+_1(z) R_{21} \left(\frac w{z\oq^2} \right) \label{eqn:formula 1} \\
&T_2^+(w) \cdot X^+_1(z) = R_{12} \left(\frac z{w\oq^2} \right) X^+_1(z) \frac {R_{21} \left(\frac wz \right)}{f \left( \frac wz \right)} \cdot T_2^+(w)  \label{eqn:formula 2} \\
&S_2^-(w) \cdot X^+_1(z) = R_{12} \left(\frac {z}{w c \oq^2} \right) X^+_1(z) R_{21} \left(\frac {w c}{z} \right) \cdot S_2^-(w) \label{eqn:formula 3} \\
&X^+_1(z) \cdot T_2^-(w) = T_2^-(w) \cdot R_{12} \left(\frac {z}{w c} \right) X^+_1(z) R_{21} \left(\frac {w c}{z\oq^2} \right) \label{eqn:formula 4}
\end{align}
for any $X^+(z) \in \eEnd(\BC^n)[z^{\pm 1}]$, where $f(x) = \frac {(1-xq^2)(1-xq^{-2})}{(1-x)^2}$. \\

\end{proposition} 

\noindent We emphasize the fact that $X^+(z)$ does not denote a generating series, but is instead a placeholder for $E_{ji}z^d$ for all $i,j \in \{1,\dots,n\}$ and $d \in \BZ$. Thus, to understand the meaning of relations \eqref{eqn:formula 1}--\eqref{eqn:formula 4}, let us spell out the first of these. Letting $S^+(w) = \sum_{u,v}^{k \geq 0} s^+_{[u;v+nk)} \frac {E_{uv}}{w^{k}}$ and $X(z) = \frac {E_{ij}}{z^d}$, formula \eqref{eqn:formula 1} reads:
\begin{equation}
\label{eqn:explicit}
\sum_{1 \leq u,v \leq n}^{k \geq 0}  \frac {E_{ij}}{z^d} \otimes \frac {E_{uv}}{w^k} \cdot s_{[u;v+nk)}^+= \sum_{1 \leq u,v,\bullet,*,x,y \leq n}^{k,a,b \geq 0} s^+_{[u;\bullet+n(k-a-b))}
\end{equation}
$$
\left(1 \otimes \frac {E_{u \bullet}}{w^{k-a-b}} \right) \left(r_{xi,a}^{\bullet *} \cdot E_{xi} \otimes E_{\bullet *} \frac {z^a}{w^a} \right) \left( \frac {E_{ij}}{z^d} \otimes 1 \right) \left(\oo{r}_{jy,b}^{* v} \cdot E_{jy} \otimes E_{*v} \frac {z^b}{w^b} \right)
$$
where $r$ and $r'$ are the coefficients of the power series expansions:
\begin{align*} 
&R_{12} \left(\frac zw \right) \cdot f \left( \frac zw \right)^{-1} = \sum_{i,j,u,v}^{k \geq 0} r_{ij,k}^{uv} \cdot E_{ij} \otimes E_{uv} \frac {z^k}{w^k} \\
&R_{21} \left(\frac w{z\oq^2} \right) = \sum_{i,j,u,v}^{k \geq 0} \oo{r}_{ij,k}^{uv} \cdot E_{ij} \otimes E_{uv} \frac {z^k}{w^k}
\end{align*}
Equating the coefficients of $... \otimes \frac {E_{uv}}{w^k}$ in the two sides of \eqref{eqn:explicit} yields the identity:
$$
\frac {E_{ij}}{z^d} \cdot s_{[u;v+nk)}^+ = \sum^{a,b \geq 0}_{1 \leq \bullet, *, x, y \leq n} r_{xi,a}^{\bullet *} \oo{r}_{jy,b}^{* v} s^+_{[u;\bullet+n(k-a-b))} \cdot \frac {E_{xy}}{z^{d-a-b}}
$$
which is a relation in the algebra $\DD$, once we perform the substitution \eqref{eqn:correspondence plus}. \\

\begin{proof} We will only prove relation \eqref{eqn:formula 1}, and leave the analogous formulas \eqref{eqn:formula 2}--\eqref{eqn:formula 4} as exercises to the reader. Let us rewrite \eqref{eqn:formula 1} as:
\begin{equation}
\label{eqn:milhous}
S_2^+(w)^{-1} X^+_1(z) S_2^+(w) = R_{21}\left( \frac wz \right)^{-1}  X^+_1(z) R_{21} \left(\frac w{z\oq^2} \right) 
\end{equation}
(note that we are using \eqref{eqn:unitary r}). If we let $A:\CE \rightarrow \CE$ denote the antipode, then \eqref{eqn:cop rtt 1} implies $A^{-1}(S(w)) = S(w/c)^{-1}$. With this in mind, relation \eqref{eqn:milhous} reads:
\begin{equation}
\label{eqn:richard}
X^+_1(z) \ \spadesuit \ S_2^+(w) = X^+_1(z) \ \clubsuit \ S_2^+(w)
\end{equation}
where for any $e \in \uug$, we write:
\begin{align} 
&X^+(z) \ \spadesuit \ e = A^{-1}(e_2) X^+(z) e_1 \label{eqn:spade} \\ 
&X^+(z) \ \clubsuit \ e = \Big \langle e_2, T(z) \Big \rangle  X^+(z) \Big \langle e_1, T(z\oq^2)^{-1} \Big \rangle \label{eqn:club}
\end{align} 
Indeed, when $e = S(w)$, the right-hand sides of \eqref{eqn:spade} and \eqref{eqn:club} match the LHS and RHS of \eqref{eqn:milhous}, respectively, due to \eqref{eqn:pairing affine}. It is easy to see that the operations $\spadesuit$ and $\clubsuit$ are additive in $e$. Moreover, as a consequence of the properties of the antipode and bialgebra pairing (respectively), we have the following identities:
\begin{align*} 
&X^+(z) \ \spadesuit \ (ee') = (X^+(z) \ \spadesuit \ e) \ \spadesuit \ e' \\
&X^+(z) \ \clubsuit \ (ee') = (X^+(z) \ \clubsuit \ e) \ \clubsuit \ e'
\end{align*}
Since $\CE^+$ is generated by the $p_{l\bde, 1}^{(0)}$'s and $p_{[s;s+1)}^{(0)}$'s, formula \eqref{eqn:richard} is equivalent to: 
\begin{equation}
\label{eqn:club spade}
X^+(z) \ \spadesuit \ e = X^+(z) \ \clubsuit \ e \qquad \forall e \in \left\{p_{l\bde, 1}^{(0)}, p_{[s;s+1)}^{(0)} \right\}^{l \in \BN}_{s\in \BZ/n\BZ}
\end{equation}
When $e = p_{l\bde, 1}^{(0)}$ (whose coproduct is $\Delta(e) = e \otimes 1 + c^{-l} \otimes e$), relation \eqref{eqn:club spade} reads:
$$
X^+(z) p_{l\bde, 1}^{(0)} - p_{l\bde, 1}^{(0)} X^+(z) = X^+(z) \Big \langle p_{l\bde, 1}^{(0)}, T(z\oq^2)^{-1} \Big \rangle  + \Big \langle p_{l\bde, 1}^{(0)}, T(z) \Big \rangle X^+(z) 
$$
Formulas \eqref{eqn:imaginary alpha} and \eqref{eqn:def alpha}--\eqref{eqn:alpha renorm plus} imply that the pairings in the right-hand side of the formula above are equal to $z^l\oq^l$ and $-z^l \oq^{-l}$, respectively, hence we obtain:
$$
\left[ X^+(z), p_{l\bde, 1}^{(0)} \right] = z^l X^+(z) (\oq^l - \oq^{-l})
$$
If we plug $X^+(z) = E_{ji} z^d \oq^{\frac {2i}n}$ into the equation above and use the correspondence \eqref{eqn:correspondence plus}, the formula above reduces to \eqref{eqn:f zero imaginary} for $\pm = +$. Similarly, if we plug $e = p_{[s;s+1)}^{(0)}$ into \eqref{eqn:club spade}, then the resulting formula reduces to \eqref{eqn:f zero simple} for $\pm = +$.
	
\end{proof}

\noindent By analogy with Proposition \ref{prop:connection plus}, we have the following result (whose proof is quite close to the one above, hence left as an exercise to the reader): \\

\begin{proposition}
\label{prop:connection minus}
	
Under the substitution: 
\begin{equation}
\label{eqn:correspondence minus}
p_{[i,j+nd)}^{(-1)} \leadsto E_{ji} z^d \cdot \frac {\oq^{\frac {1-2j}n}}{q-q^{-1}}
\end{equation}
the following relations hold in $\DD$ (recall that $\cdot_\eop$ denotes the opposite product):
\begin{align}
&X^-_1(z) \cdot_{\eop} S_2^-(w) = S_2^-(w) \cdot_{\eop} \frac {R_{12} \left(\frac zw \right)}{f \left(\frac zw \right)} X^-_1(z) D_2 R_{21} \left(\frac {w \oq^2 q^{2n}}{z} \right) D_2^{-1} \label{eqn:formula 5} \\
&T_2^-(w) \cdot_{\eop} X^-_1(z) = D_1 R_{12} \left(\frac {z \oq^2 q^{2n}}{w} \right) D_1^{-1} X^-_1(z) \frac {R_{21} \left(\frac wz \right)}{f \left(\frac wz \right)} \cdot_{\eop} T_2^-(w)  \label{eqn:formula 6} \\
&S_2^+(w) \cdot_{\eop} X^-_1(z) = D_1 R_{12} \left(\frac {z \oq^2 q^{2n}}{wc} \right) D_1^{-1} X^-_1(z) R_{21} \left(\frac {wc}z \right) \cdot_{\eop} S_2^+(w)  \label{eqn:formula 7} \\
&X^-_1(z) \cdot_{\eop} T_2^+(w) = T_2^+(w) \cdot_{\eop} R_{12} \left(\frac {z}{w c} \right) X^-_1(z)  D_2 R_{21} \left(\frac {w c \oq^2 q^{2n}}{z} \right) D_2^{-1}  \label{eqn:formula 8}
\end{align}
for any $X^-(z) \in \eEnd(\BC^n)[z^{\pm 1}]$, where $D = \emph{diag}(q^2,...,q^{2n})$. \\ 
	
\end{proposition}

\noindent Finally, let us perform both substitutions \eqref{eqn:correspondence plus} and \eqref{eqn:correspondence minus} simultaneously:
\begin{align*} 
&p_{[i,j+nd)}^{(+1)} \leadsto \left(E_{ji} z^d\right)^+ \cdot \frac {\oq^{\frac {2i-1}n}}{1-q^2} \\
&p_{[i,j+nd)}^{(-1)} \leadsto \left(E_{ji} z^d\right)^- \cdot \frac {\oq^{\frac {1-2j}n}}{q-q^{-1}}
\end{align*} 
(for any $X \in \text{End}(\BC^n)[z^{\pm 1}]$, we use the notation $X^+$ and $X^-$ to differentiate among the matrices representing $p_{[i,j+nd)}^{(+  1)}$ and $p_{[i,j+nd)}^{(- 1)}$, respectively). Then we have:
\begin{equation}
\label{eqn:plus minus}
\left[ \left( \frac {E_{ij}}{z^d} \right)^+, \left( \frac {E_{i'j'}}{z^{d'}} \right)^- \right] =
\end{equation}
$$
=  (q^2-1) \sum_{k \in \BZ} \left( s^+_{[j';i+nk)} t^+_{[j;i'+n(-d-d'-k))} c^{-d'} \barc^{-1} -  t^-_{[i;j'+nk)} s^-_{[i';j+n(d+d'-k))} c^{-d} \barc \right)
$$
as an immediate consequence of formula \eqref{eqn:e plus minus} (we set $s^\pm_{[i;j)} = t^\pm_{[i;j)} = 0$ if $i>j$). \\

\subsection{}

We will now prove a useful Lemma about the structure of the algebra $\CE^+$ of Subsection \ref{sub:sub}. Let us write $\lhs_\bd$ for the quantity in the left-hand side of \eqref{eqn:f f}, when the signs are $\pm = \pm' = +$. Then we have:
$$
\CE^+ = \BQ(q) \Big \langle f_{[i;j)} \Big \rangle_{(i < j) \in \zzz} \Big / \left( \lhs_\bd \right)_{\bd \in \nn}
$$
In Section \ref{sec:extended}, we will find ourselves in the situation of having an algebra $\CB^+$ and wanting to construct an algebra isomorphism $\Upsilon : \CE^+ \cong \CB^+$. Of course, the straightforward way to do this is to construct elements:
$$
f_{[i;j)}' \in \CB^+, 
$$
declare that $\Upsilon(f_{[i;j)}) = f'_{[i;j)}$, and directly check that $\Upsilon(\lhs_\bd) = 0$, $\forall \bd \in \nn$. However, such a check is not handy in our situation, and we will instead rely on some additional structure. More precisely, we will prove the following. \\

\begin{lemma}
\label{lem:any}
	
Assume $\CB^+$ is a $\nn$--graded $\BQ(q)$--algebra, such that:
$$
\tCB^+ = \frac {\Big \langle \CB^+, \psi_s^{\pm 1}, c^{\pm 1}\Big \rangle_{s \in \BZ}}{\Big ( \psi_s x - q^{- \langle \bs^s, \deg x \rangle} x \psi_s, \ \psi_{s+n} - c \psi_s, \ c \text{ central} \Big)_{\forall x \in \CB^+, s \in \BZ}}
$$
\footnote{In the formula above, $\langle \cdot, \cdot \rangle$ is the bilinear form on $\BZ^n$ given by $\langle \bs^i, \bs^j \rangle = \delta_{i}^{j} - \delta_{i}^{j+1}$} is a bialgebra. Assume there exist elements $0 \neq f'_{[i;j)} \in \CB_{[i;j)}$ and linear maps: 
$$
\alpha'_{[i;j)} : \CB_{[i;j)} \rightarrow \BQ(q) \qquad \forall i \in \{1,...,n\} \text{ and } j>i
$$
such that the analogues of \eqref{eqn:cop 1}, \eqref{eqn:pseudo}, \eqref{eqn:alpha affine} hold. If:
\begin{equation}
\label{eqn:primitive prop}
\Big\{ x \text{ primitive and } \alpha'_{[i;j)}(x) = 0, \ \forall i,j \Big\} \quad \Rightarrow \quad x = 0
\end{equation}
then the map $\CE^+ \stackrel{\Upsilon}\rightarrow \CB^+$, $\Upsilon(f_{[i;j)}) = f'_{[i;j)}$ is an injective algebra homomorphism. \\
	
\end{lemma}

\begin{proof} Let us consider the left-hand side of \eqref{eqn:f f} with $f'_{[i;j)}$ instead of the $f_{[i;j)}$:
$$
\lhs_{\bd}' = \sum_{[i;j) + [i';j') = \bd} \text{coefficient} \cdot f_{[i;j)}' f_{[i';j')}' \in \CB_{\bd}
$$	
To show that $\Upsilon$ is an algebra homomorphism, we would need to show that $\lhs_\bd' = 0$, which we will prove by induction on $\bd \in \nn$. The base case $\bd = 0$ is trivial, so we will only prove the induction step. We have:
$$
\Delta(\lhs'_\bd) \in \langle \psi_s^{\pm 1} \rangle_{s \in \BZ} \otimes \lhs'_\bd + ... + \lhs'_\bd \otimes \langle \psi_s^{\pm 1} \rangle_{s \in \BZ}
$$
where the middle terms denoted by the ellipsis are equal to $\Upsilon \otimes \Upsilon$ applied to the middle terms of $\Delta(\lhs_\bd)$. Since the latter are 0 (by using $\lhs_\bd = 0$ in $\CE^+$ and the induction hypothesis), we conclude that $\lhs'_\bd$ is primitive. Moreover, the analogues of  \eqref{eqn:pseudo} and \eqref{eqn:alpha affine} imply that:
$$
\alpha'_{[i;j)}(\lhs'_\bd) = \alpha_{[i;j)}(\lhs_\bd) = 0 \qquad \forall \bd, \ i < j
$$
Therefore, the assumption \eqref{eqn:primitive prop} implies that $\lhs'_\bd = 0$ for all $\bd$, thus establishing the fact that $\Upsilon$ is a well-defined algebra homomorphism. To show that $\Upsilon$ is injective, assume that its kernel is non-empty. Since $\Upsilon$ preserves degrees, we may choose $0 \neq x \in \CE^+$ of minimal degree $\bd \in \nn$ such that $\Upsilon(x) = 0$. Since $\Upsilon$ preserves the coproduct and is injective in degrees $<\bd$ (by the minimality of $\bd$), we conclude that $x$ is primitive. However, since $\Upsilon$ intertwines the linear maps $\alpha_{[i;j)}$ with $\alpha'_{[i;j)}$, we conclude that $x$ is also annihilated by the linear maps $\alpha_{[i;j)}$. However, \cite[Lemma 3.11]{PBW} tells us that under these circumstances, we must have $x = 0$.
	
\end{proof}

\noindent Since an injective linear map of finite-dimensional vector spaces $\Phi : V_1 \hookrightarrow V_2$ is an isomorphism if $\dim V_2 \leq \dim V_1$ (as well as the similar statement in the graded case, if $V_1$ and $V_2$ have finite-dimensional graded pieces which are preserved by $\Phi$) we obtain the following: \\

\begin{corollary}
\label{cor:any} 
	
If the assumption of Lemma \ref{lem:any} holds, and moreover, if $\dim \CB_\bd \leq$ the RHS of \eqref{eqn:bound affine} for all $\bd \in \nn$, then $\Upsilon : \CE^+ \cong \CB^+$ is an isomorphism. \\
	
\end{corollary}

\begin{proposition} 
	\label{prop:gen}
	
The $\BQ(q,\oq^{\frac 1n})$--algebra $\DD^\pm$ is generated by the elements:
\begin{equation}
\label{eqn:degree one piece}
\Big \{ p_{\pm [i;j)}^{(\pm 1)} \Big \}_{(i,j) \in \zzz}
\end{equation}
	
\end{proposition} 

\begin{proof} Without loss of generality, let $\pm = +$. We will prove that $p_{[i;j)}^{(k)}$ (resp. $p_{l\bde,r}^{(k')}$) lies in the subalgebra generated by the elements \eqref{eqn:degree one piece} for all choices of indices $i,j,l,r$, by induction on $k$ (respectively $k'$). To this end, let us choose a lattice triangle of minimal area with the vector $(j-i,k)$ (respectively $(nl,k')$) as an edge:
	
	\begin{picture}(100,125)(-30,10)
	
	\put(0,40){\circle*{2}}\put(20,40){\circle*{2}}\put(40,40){\circle*{2}}\put(60,40){\circle*{2}}\put(80,40){\circle*{2}}\put(0,60){\circle*{2}}\put(20,60){\circle*{2}}\put(40,60){\circle*{2}}\put(60,60){\circle*{2}}\put(80,60){\circle*{2}}\put(0,80){\circle*{2}}\put(20,80){\circle*{2}}\put(40,80){\circle*{2}}\put(60,80){\circle*{2}}\put(80,80){\circle*{2}}\put(0,100){\circle*{2}}\put(20,100){\circle*{2}}\put(40,100){\circle*{2}}\put(60,100){\circle*{2}}\put(80,100){\circle*{2}}\put(0,120){\circle*{2}}\put(20,120){\circle*{2}}\put(40,120){\circle*{2}}\put(60,120){\circle*{2}}\put(80,120){\circle*{2}}
	
	\put(200,60){\line(1,1){20}}
	\put(200,60){\line(2,1){80}}
	\put(220,80){\line(3,1){60}}
	
	\put(180,57){\scriptsize{$(0,0)$}}
	\put(266,103){\scriptsize{$(nl,k')$}}
	\put(202,85){\scriptsize{$(a,b)$}}
	
	\put(115,75){respectively}
	
	\put(200,40){\circle*{2}}\put(220,40){\circle*{2}}\put(240,40){\circle*{2}}\put(260,40){\circle*{2}}\put(280,40){\circle*{2}}\put(200,60){\circle*{2}}\put(220,60){\circle*{2}}\put(240,60){\circle*{2}}\put(260,60){\circle*{2}}\put(280,60){\circle*{2}}\put(200,80){\circle*{2}}\put(220,80){\circle*{2}}\put(240,80){\circle*{2}}\put(260,80){\circle*{2}}\put(280,80){\circle*{2}}\put(200,100){\circle*{2}}\put(220,100){\circle*{2}}\put(240,100){\circle*{2}}\put(260,100){\circle*{2}}\put(280,100){\circle*{2}}\put(200,120){\circle*{2}}\put(220,120){\circle*{2}}\put(240,120){\circle*{2}}\put(260,120){\circle*{2}}\put(280,120){\circle*{2}}\put(200,120){\circle*{2}}\put(220,120){\circle*{2}}\put(240,120){\circle*{2}}\put(260,120){\circle*{2}}\put(280,120){\circle*{2}}
	
	\put(0,60){\line(1,1){20}}
	\put(0,60){\line(3,2){60}}
	\put(20,80){\line(2,1){40}}
	
	\put(-20,57){\scriptsize{$(0,0)$}}
	\put(46,105){\scriptsize{$(j-i,k)$}}
	\put(2,85){\scriptsize{$(a,b)$}}
	
	\end{picture}
	
\noindent In the case of the picture on the left, namely that of the element $p_{[i;j)}^{(k)}$, the minimality of the area of the triangle implies that:
\begin{equation}
\label{eqn:determinant}
\det \begin{pmatrix} b & k-b \\ a & j-i-a \end{pmatrix} = 1
\end{equation}
Recall that $i \not \equiv j$ mod $n$. If $a \equiv j-i$ modulo $n$, then relation \eqref{eqn:rel 2 pbw} gives us:
$$
\left [p_{[i;i+a)}^{(b)}, p_{\frac {j-i-a}n\bde,i}^{(k-b)} \right] = p_{[i;j)}^{(k)} \left(\oq^{\frac 1n} - \delta_a^0 \oq^{-\frac 1n} \right)
$$
while if $a \not \equiv j-i$ modulo $n$, then relation \eqref{eqn:rel 3 pbw} gives us:
$$
p_{[i;i+a)}^{(b)} p_{[i+a;j)}^{(k-b)} q^{- \delta_a^0}   - p_{[i+a;j)}^{(k-b)} p_{[i;i+a)}^{(b)} q^{-1}  = \frac {\barf^{(k)}_{[i;j)}}{q^{-1} - q} \stackrel{\eqref{eqn:barf alpha}, \eqref{eqn:simple alpha}}= - q^{-1} \oq^{-\frac 1n} p^{(k)}_{[i;j)}
$$
In either of the two formulas above, the induction hypothesis implies that the left-hand side lies in the algebra generated by the elements \eqref{eqn:degree one piece}. Therefore, so does the right-hand side, and the induction step is complete. \\
	
\noindent The case of the picture on the right, namely that of the element $p_{l\bde, r}^{(k')}$, is proved analogously. We will therefore only sketch the main idea, and leave the details to the interested reader. As shown in \cite[Lemma 4.2]{PBW}, relation \eqref{eqn:rel 3 pbw} implies for all $u$:
$$
p_{[u;u+a)}^{(b)} \tilde{p}_{[u+a;u+nl)}^{(k'-b)} q^{\delta_{a}^0 - 1} - \tilde{p}_{[u+a;u+nl)}^{(k'-b)} p_{[u;u+a)}^{(b)} q^{1 - \delta_a^0} = 
\sum_{l_1+l_2=l} \frac {f^\mu_{[u+a+nl_1;u+a)} \barf^\mu_{[u;u+nl_2)}}{q^{-1}-q}
$$
where $\mu = \frac {k'}{nl}$, and $\tilde{p}_{[v;v+x)}^{(y)} = \sum_{v' \in \BZ/n\BZ} p_{[v';v'+x)}^{(y)} \cdot \left[q^{-\delta_x^0} \delta_{v'}^v + \delta_x^0(q-q^{-1}) \frac {\oq^{\frac 2n \cdot \overline{y(v-v')}}}{\oq^2-1} \right]$. By the induction hypothesis, the left-hand side of the expression above lies in the subalgebra generated by the elements \eqref{eqn:degree one piece}, hence so does the right-hand side, which we henceforth denote $\rhs$. Clearly, $\rhs \in \CB_\mu^+$, and it can therefore be expressed as a sum of products of the simple and imaginary generators:
$$
\rhs = \sum_{r} \alpha_{u}^r \cdot p^\mu_{l\bde, r} + \text{sum of products of more than one generator of } \CB_\mu^+
$$
By the induction hypothesis, all products of more than one simple or imaginary generator lie in the subalgebra generated by the elements \eqref{eqn:degree one piece}. We conclude that $\sum_{r} \alpha_{u}^r \cdot p^\mu_{l\bde, r}$ also lies in the same algebra, for all $u \in \{1,...,n\}$. Since the matrix $\alpha^r_u$ has full rank (as shown in \cite{PBW}), we conclude that $p^\mu_{l\bde, r}$	 also lies in the subalgebra generated by the elements \eqref{eqn:degree one piece}, precisely what we needed to show. 

\end{proof} 

\section{The shuffle algebra with spectral parameters}
\label{sec:new}

\subsection{}

We will now generalize the construction of Section \ref{sec:old} by allowing the coefficients of matrices in $\End(V^{\otimes k})$ to be rational functions. We will recycle all the notations from Section \ref{sec:old}, so fix a basis of a $n$--dimensional vector space $V$, and let:
\begin{equation}
\label{eqn:r}
R(x) \in \End_{\BQ(q)}(V \otimes V)(x) 
\end{equation}
be given by formula \eqref{eqn:explicit r}. For a parameter $\oq$, we define:
\begin{equation}
\label{eqn:tr}
\tR (x) = R_{21} \left( \displaystyle \frac 1{x\oq^2} \right) \in \End_{\BQ(q,\oq^{\frac 1n})}(V \otimes V)(x) 
\end{equation}
$$
\tR(x) = \sum_{1\leq i,j \leq n} E_{ii} \otimes E_{jj} \left(\frac {q^{-1} - x q \oq^2}{1-x \oq^2} \right)^{\delta_i^j} - (q - q^{-1}) \sum_{1 \leq i \neq j \leq n} E_{ij} \otimes E_{ji} \frac {(x \oq^2)^{\delta_{i<j}}}{1-x \oq^2} 
$$
It is well-known that $R(x)$ satisfies the Yang-Baxter equation with parameter:
\begin{equation}
\label{eqn:ybe aff}
R_{12} \left( \frac {z_1}{z_2} \right) R_{13} \left( \frac {z_1}{z_3} \right) R_{23} \left( \frac {z_2}{z_3} \right) = R_{23} \left( \frac {z_2}{z_3} \right)  R_{13} \left( \frac {z_1}{z_3} \right) R_{12} \left( \frac {z_1}{z_2} \right) 
\end{equation}
and it is easy to show that $\tR(z)$ satisfies the following analogue of \eqref{eqn:quasi-ybe 1}--\eqref{eqn:quasi-ybe 2}:
\begin{equation}
\label{eqn:quasi-ybe 1 aff}
\tR_{21} \left( \frac {z_2}{z_1} \right) \tR_{31} \left( \frac {z_3}{z_1} \right) R_{23}  \left( \frac {z_2}{z_3} \right) = R_{23}  \left( \frac {z_2}{z_3} \right) \tR_{31} \left( \frac {z_3}{z_1} \right) \tR_{21} \left( \frac {z_2}{z_1} \right) 
\end{equation}
\begin{equation}
\label{eqn:quasi-ybe 2 aff}
R_{12} \left( \frac {z_1}{z_2} \right) \tR_{31} \left( \frac {z_3}{z_1} \right) \tR_{32}  \left( \frac {z_3}{z_2} \right) = \tR_{32}  \left( \frac {z_3}{z_2} \right) \tR_{31} \left( \frac {z_3}{z_1} \right) R_{12} \left( \frac {z_1}{z_2} \right)
\end{equation}
Finally, we note that the $R$--matrix \eqref{eqn:explicit r} is (almost) unitary, in the sense that:
\begin{equation}
\label{eqn:unitary r}
R_{12}(x) R_{21} \left( \frac 1x \right) = f(x) \cdot \text{Id}_{V \otimes V}
\end{equation}
where:
\begin{equation}
\label{eqn:def f}
f(x) = \frac {(1-xq^2)(1-xq^{-2})}{(1-x)^2}
\end{equation}
It is easy to see that ``half" of the rational function $f(x)$ could have been absorbed in the definition of $R(x)$, thus making $R(x)$ to be unitary on the nose, but we will prefer our current conventions. \\

\subsection{} 
\label{sub:braids aff} 

We will represent elements of $\End(V^{\otimes k})(z_1,...,z_k)$ as braids on $k$ strands. The only difference between the present setup and that of Section \ref{sec:old} is that each strand carries not only a label $i \in  \{1,...,k\}$ but also a variable $z_i$. With this in mind, we make the convention that the endomorphism corresponding to a positive crossing of strands labeled $a$ and $b$, endowed with variables $z_a$ and $z_b$ respectively, is:
$$
R_{ab} \left( \frac {z_a}{z_b} \right)
$$
Because of \eqref{eqn:tr}, we can represent both $R$ and $\tR$ as crossings of braids of the same kind (i.e. we do not need the dichotomy of straight strands versus squiggly strands, of Subsection \ref{sub:braids}) if we remember to change the variable on one of our strands. We will always write the variable next to every braid. For example, the braids:
\begin{figure}[ht]    
\centering
\includegraphics[scale=0.2]{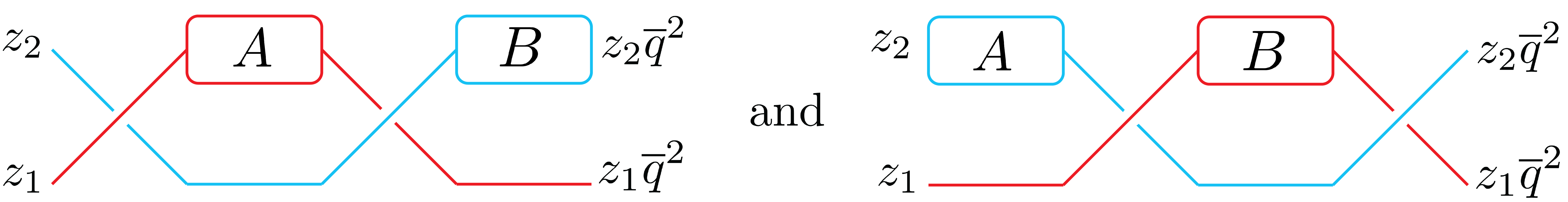}
\caption{Braids decorated with variables}
\end{figure}

\noindent represent the following compositions $\in \End(V^{\otimes 2}) (z_1,z_2)$: 
$$
R_{12} \left(\frac {z_1}{z_2} \right) A_1(z_1) \underbrace{\tR_{12} \left(\frac {z_1}{z_2} \right)}_{R_{21} \left(\frac {z_2}{z_1\oq^2} \right)} B_2(z_2) \quad \text{and} \quad A_2(z_2) \underbrace{\tR_{21} \left(\frac {z_2}{z_1} \right)}_{R_{12} \left(\frac {z_1}{z_2\oq^2} \right)} B_1(z_1) R_{21} \left(\frac {z_2}{z_1} \right) 
$$
respectively. The variable does not change along a strand, except at a box. \\

\subsection{} We make an unusual convention on residues of rational functions, by stipulating that $\frac {\alpha}{\alpha - x}$ has residue $1$ at $x=\alpha$ (instead of the more usual $-\alpha$). With this in mind, note that $R(x)$ has a pole at $x = 1$, with residue $(q - q^{-1}) \cdot (12)$. Thus:
\begin{equation}
\label{eqn:blob}
\underset{x=\oq^{-2}}{\text{Res}} \ \tR(x) = (q^{-1} - q) \cdot (12) \in \End_{\BQ(q,\oq)}(V \otimes V)
\end{equation}
where $(12)$ denotes the permutation operator of the two factors. Pictorially, the endomorphism \eqref{eqn:blob} will be represented by two black dots indicating a color change (recall that the color encodes the index $\in \{1,...,k\}$ of a strand) between two strands:

\begin{figure}[ht]    
\centering
\includegraphics[scale=0.32]{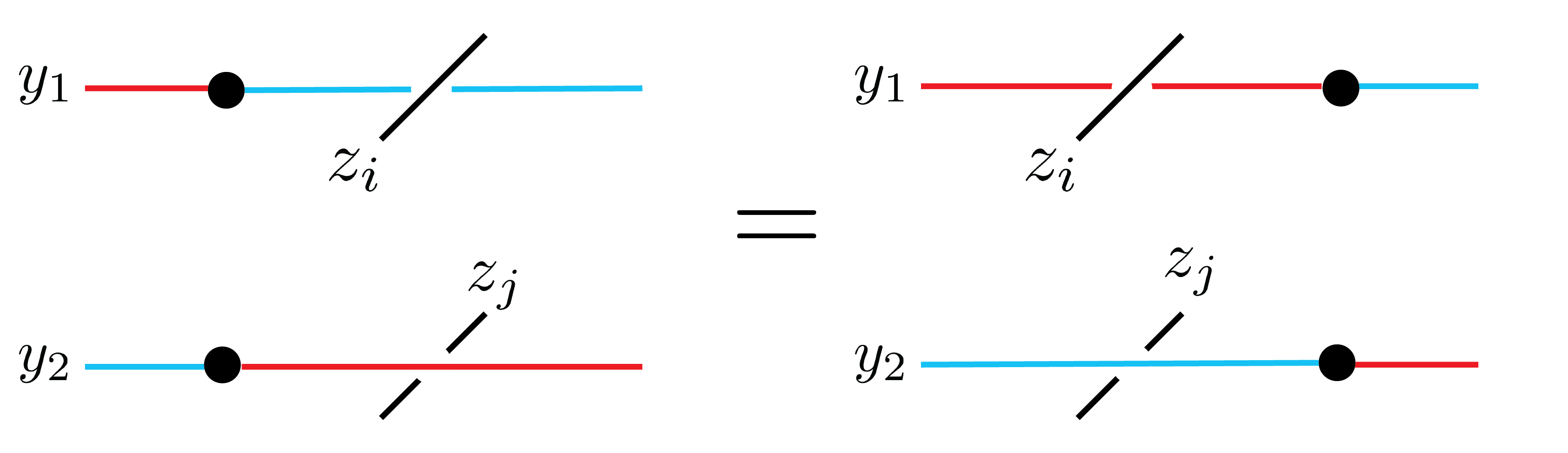}
\caption{Black dots can slide past arbitrary strands}
\end{figure}

\noindent The equality of braids depicted in Figure 12 means that one can move the black dots as far left or as far right as we wish, no matter how many other strands we pass over or under. Explicitly, the equality depicted above reads:
$$
(q^{-1} - q)(12) \cdot R_{i2} \left(\frac {z_i}{y_1} \right) R_{1j}^{-1} \left(\frac {y_2}{z_j} \right) =  R_{i1} \left(\frac {z_i}{y_1} \right) R_{2j}^{-1} \left(\frac {y_2}{z_j} \right) \cdot (q^{-1} - q)(12)
$$
which is a true identity. Finally, we note that due to formula \eqref{eqn:unitary r}, we can always change a crossing in a braid, at the cost of multiplying with the function \eqref{eqn:def f}:

\begin{figure}[ht]    
\centering
\includegraphics[scale=0.35]{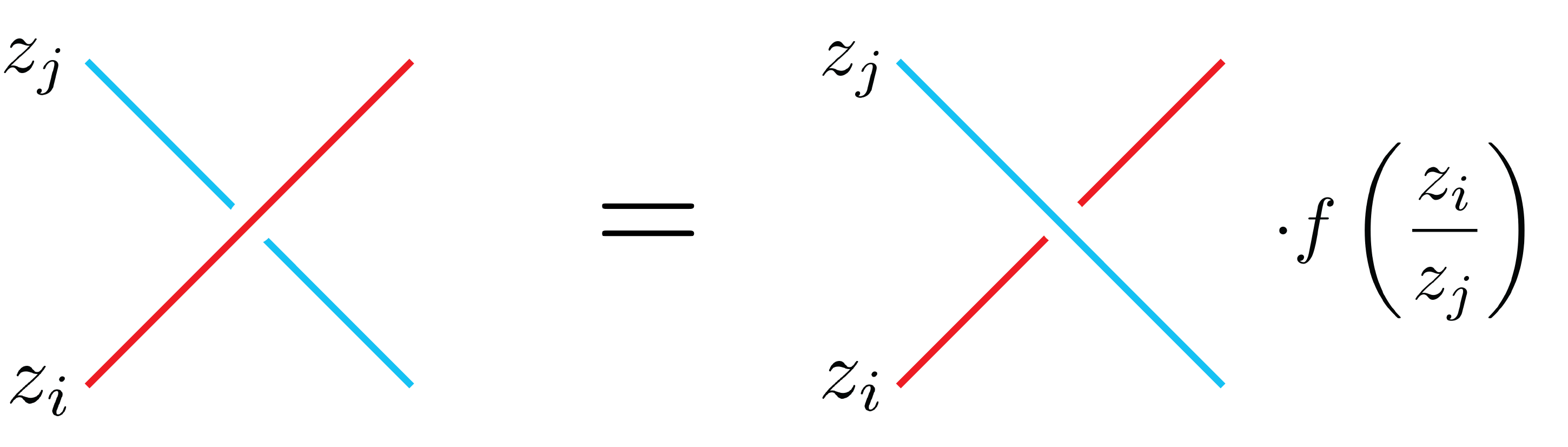}
\caption{Changing a crossing}
\end{figure}

\subsection{} The following is proved just like Proposition \ref{prop:shuf 1}, so we leave it as an exercise. \\

\begin{proposition}
\label{prop:shuf aff 1}
	
Let $A = A_{1...k}(z_1,...,z_k)$, $B = B_{1...l}(z_1,...,z_l)$. The assignment:
\begin{equation}
\label{eqn:shuf prod aff}
A * B = \sum^{a_1<...<a_k, \ b_1<...<b_l}_{\{1,...,k+l\} = \{a_1,...,a_k\} \sqcup \{b_1,...,b_l\}} \left[ \prod_{i=k}^{1} \prod_{j=1}^l \underbrace{R_{a_ib_j} \left( \frac {z_{a_i}}{z_{b_j}} \right)}_{\text{only if }a_i < b_j} \right]
\end{equation}
$$
A_{a_1...a_k}(z_{a_1},...,z_{a_k}) \left[ \prod_{i=1}^k \prod_{j=l}^{1} \tR_{a_ib_j}  \left( \frac {z_{a_i}}{z_{b_j}} \right) \right] B_{b_1...b_l}(z_{b_1},...,z_{b_l})  \left[  \prod_{i=k}^{1} \prod_{j=1}^l \underbrace{R_{a_ib_j} \left( \frac {z_{a_i}}{z_{b_j}} \right)}_{\text{only if }a_i > b_j} \right] 
$$
yields an associative algebra structure on the vector space:
$$
\bigoplus_{k=0}^\infty \eEnd_{\fff}(V^{\otimes k})(z_1,...,z_k)
$$
with unit $1 \in \eEnd_{\fff}(V^{\otimes 0})$. We call \eqref{eqn:shuf prod aff} the ``shuffle product". \\
	
\end{proposition}

\begin{proposition}
\label{prop:shuf aff 2}
	
The shuffle product above preserves the vector space:
$$
\CA^+_{\emph{big}} \subset \bigoplus_{k=0}^\infty \eEnd_{\fff}(V^{\otimes k})(z_1,...,z_k)
$$
consisting of tensors $X = X_{1...k}(z_1,...,z_k)$ which simultaneously satisfy: \\

\begin{itemize}[leftmargin=*]

\item $X = \displaystyle \frac {x(z_1,...,z_k)}{\prod_{1\leq i \neq j \leq k} (z_i - z_j \oq^2)}$ for some $x \in \eEnd_{\fff}(V^{\otimes k})[z_1^{\pm 1},...,z_k^{\pm 1}]$. \\ \\

\item $X$ is symmetric, in the sense that: 
\begin{equation}
\label{eqn:symm aff}
X = R_\sigma \cdot (\sigma X \sigma^{-1}) \cdot R_{\sigma}^{-1}
\end{equation}
$\forall \sigma \in S(k)$, where $R_\sigma = R_\sigma(z_1,...,z_k)$ is any braid lift of the permutation $\sigma$, and:
$$
\sigma X \sigma^{-1} = X_{\sigma(1)...\sigma(k)}(z_{\sigma(1)},...,z_{\sigma(k)})
$$

\end{itemize}

\end{proposition}

\begin{proof} The fact that the shuffle product \eqref{eqn:shuf prod aff} preserves the vector space of symmetric tensors is proved word-for-word as in Proposition \ref{prop:shuf 2}. Therefore, it remains to prove that if $A$ and $B$ have no poles other than simple poles at $z_i = z_j \oq^2$, then $A * B$ has the same property. Besides simple poles at $z_i = z_j \oq^2$ (which we allow), the $R$-matrices in the right-hand side of \eqref{eqn:shuf prod aff} only produce simple poles at $z_i = z_j$, so it remains to show that the residues at the latter poles vanish. \\
	
\noindent We will show that any symmetric tensor $X = X_{1...k}(z_1,...,z_k)$ with at most a simple pole at $z_1 = z_k$ is actually regular there. Without loss of generality, we will prove the vanishing of the residue at $z_1 = z_k$. Since only the indices/variables 1 and $k$ will play an important role in the following, we will use ellipses ... for the indices/variables $2,...,k-1$. Let us consider \eqref{eqn:symm aff} in the particular case $\sigma = (1k)$:
\begin{equation}
\label{eqn:sam}
-Y_{12...k-1,k}(z_1,z_2,...,z_{k-1},z_k) \cdot R_\sigma = R_\sigma \cdot Y_{k2...k-1,1}(z_k,z_2,...,z_{k-1},z_1)
\end{equation}
where $Y(z_1,...,z_k) = X(z_1,...,z_k) \cdot (z_1-z_k)$ is regular at $z_1 = z_k$. We may choose:
$$
R_\sigma = R_{12} \left(\frac {z_1}{z_2} \right) ... R_{1,k-1} \left(\frac {z_1}{z_{k-1}} \right)R_{1k} \left(\frac {z_1}{z_k} \right) R_{k,k-1}^{-1} \left(\frac {z_k}{z_{k-1}} \right) ...  R_{k2}^{-1} \left(\frac {z_k}{z_2} \right)
$$
Since the residue of $R_{1k} \left( \frac {z_1}{z_k} \right)$ at $z_1 = z_k$ is $(q-q^{-1})\cdot (1k)$, the residue of \eqref{eqn:sam} is:
$$
-Y_{1...k}(x,...,x) \cdot R_{12}\left(\frac x{z_2}\right) ...  R_{1,k-1}\left(\frac x{z_{k-1}}\right) \cdot (1k) \cdot R_{k,k-1}^{-1} \left( \frac x{z_{k-1}} \right) ... R_{k2}^{-1} \left( \frac x{z_2} \right) = 
$$
$$
= R_{12}\left(\frac x{z_2}\right) ...  R_{1,k-1}\left(\frac x{z_{k-1}}\right) \cdot (1k) \cdot R_{k,k-1}^{-1} \left( \frac x{z_{k-1}} \right) ... R_{k2}^{-1} \left( \frac x{z_2} \right) \cdot Y_{k...1}(x,...,x)
$$
(above, we wrote $x = z_1 = z_k$ and canceled an overall scalar factor of $q-q^{-1}$). We may move the permutation $(1k)$ to the very right of the equations above, obtaining:
$$
-Y_{1...k}(x,...,x) \cdot R_{12}\left(\frac x{z_2}\right) ...  R_{1,k-1}\left(\frac x{z_{k-1}}\right) R_{1,k-1}^{-1} \left( \frac x{z_{k-1}} \right) ... R_{12}^{-1} \left( \frac x{z_2} \right) \cdot (1k) = 
$$
$$
= R_{12}\left(\frac x{z_2}\right) ...  R_{1,k-1}\left(\frac x{z_{k-1}}\right) R_{1,k-1}^{-1} \left( \frac x{z_{k-1}} \right) ... R_{12}^{-1} \left( \frac x{z_2} \right) \cdot Y_{1...k}(x,...,x) \cdot (1k)
$$
After canceling all the $R$ factors and the permutation operators $(1k)$, we are left with $Y_{1...k}(x,...,x) = 0$, which implies that $X$ was regular at $z_1 = z_k$ to begin with. 
	
\end{proof}

\subsection{} For a rational function $X(z_1,...,z_k)$ with at most simple poles, we let:
\begin{equation}
\label{eqn:def residue}
\underset{\{z_1 = y, z_2 = y \oq^2,..., z_i = y \oq^{2(i-1)}\}}{\Res} X
\end{equation}
be the rational function in $y,z_{i+1},...,z_k$ obtained by successively taking the residue at $z_2 = z_1 \oq^2$, then at $z_3 = z_1 \oq^4$,..., then at $z_i = z_1 \oq^{2(i-1)}$ and finally relabeling the variable $z_1 \leadsto y$. More generally, for any collection of natural numbers:
$$
1 = c_1 < c_2 < ... < c_u < c_{u+1} = k + 1
$$
we will write:
$$
\underset{\{z_{c_s} = y_s, z_{c_s+1} = y_s \oq^2, ..., z_{c_{s+1}-1} = y_s \oq^{2(c_{s+1}-c_s - 1)} \}_{\forall s \in \{1,...,u\}}}{\Res} X
$$
for the rational function in $y_1,..,y_u$ obtained by applying the iterated residue \eqref{eqn:def residue} construction for the collections of variables indexed by $\{c_1,c_1+1,...,c_2-1\}$, ..., $\{c_u,c_u+1,...,c_{u+1}-1\}$. Moreover, we write:
$$
\prod_{i=1}^k x_i = \prod_{1 \leq i \leq k} x_i = x_1 x_2 ... x_k \qquad \qquad \prod_{i=k}^1 x_i = \prod_{k \geq i \geq 1} x_i = x_k x_{k-1} ... x_1 
$$
for any collection of potentially non-commuting symbols $x_1,...,x_k$. \\ 

\begin{definition}
\label{def:shuf aff}
	
Let $\CA^+ \subset \CA^+_\ebig$ be the subspace of tensors $X = X_{1...k}(z_1,...,z_k)$ such that for any composition $k = \lambda_1+...+\lambda_u$ we have (let $\lambda_s = c_{s+1}-c_s$, $\forall s$):
\begin{equation}
\label{eqn:wheel}
\underset{\{z_{c_s} = y_s, z_{c_s+1} = y_s \oq^2, ..., z_{c_{s+1}-1} = y_s \oq^{2(\lambda_s - 1)} \}_{\forall s \in \{1,...,u\}}}{\eRes} X  = 
\end{equation}
$$
= (q^{-1} - q)^{k-u} \prod^{\text{unordered pairs } (s,d) \neq (t,e)}_{\text{with } 1 \leq s,t \leq u, 1 \leq d < \lambda_s, 1 \leq e < \lambda_t} f \left( \frac {y_s \oq^{2d}}{y_t \oq^{2e}} \right) 
$$
$$
\left[ \prod_{u \geq s \geq 1} \prod_{s \leq t \leq u} \prod_{1 \leq e < \lambda_t}^{(t,e) \neq (s,0)} R_{c_s,c_t+e} \left( \frac {y_s}{y_t \oq^{2e}}\right) \right] \cdot X^{(\lambda_1,...,\lambda_u)}_{c_1...c_u}(y_{1},...,y_{u}) \cdot
$$
$$
\left[ \prod_{1 \leq s \leq u} \prod_{u \geq t > s} \prod_{\lambda_t > e \geq 1} R_{c_t + e, c_s} \left( \frac {y_t \oq^{2e}}{y_s \oq^{2\lambda_s}} \right) \right] \left[ \prod_{s=1}^u \begin{pmatrix}
c_s  & ... & c_{s+1} - 2 & c_{s+1}-1  \\
c_s + 1 & ... & c_{s+1} - 1 & c_s
\end{pmatrix} \right]
$$
for some tensor $X^{(\lambda_1,...,\lambda_u)} \in \eEnd(V^{\otimes u})(y_1,...,y_u)$. \\

\end{definition}

\noindent Pictorially, the RHS of \eqref{eqn:wheel} may be represented as follows:

\begin{figure}[ht]    
\centering
\includegraphics[scale=0.3]{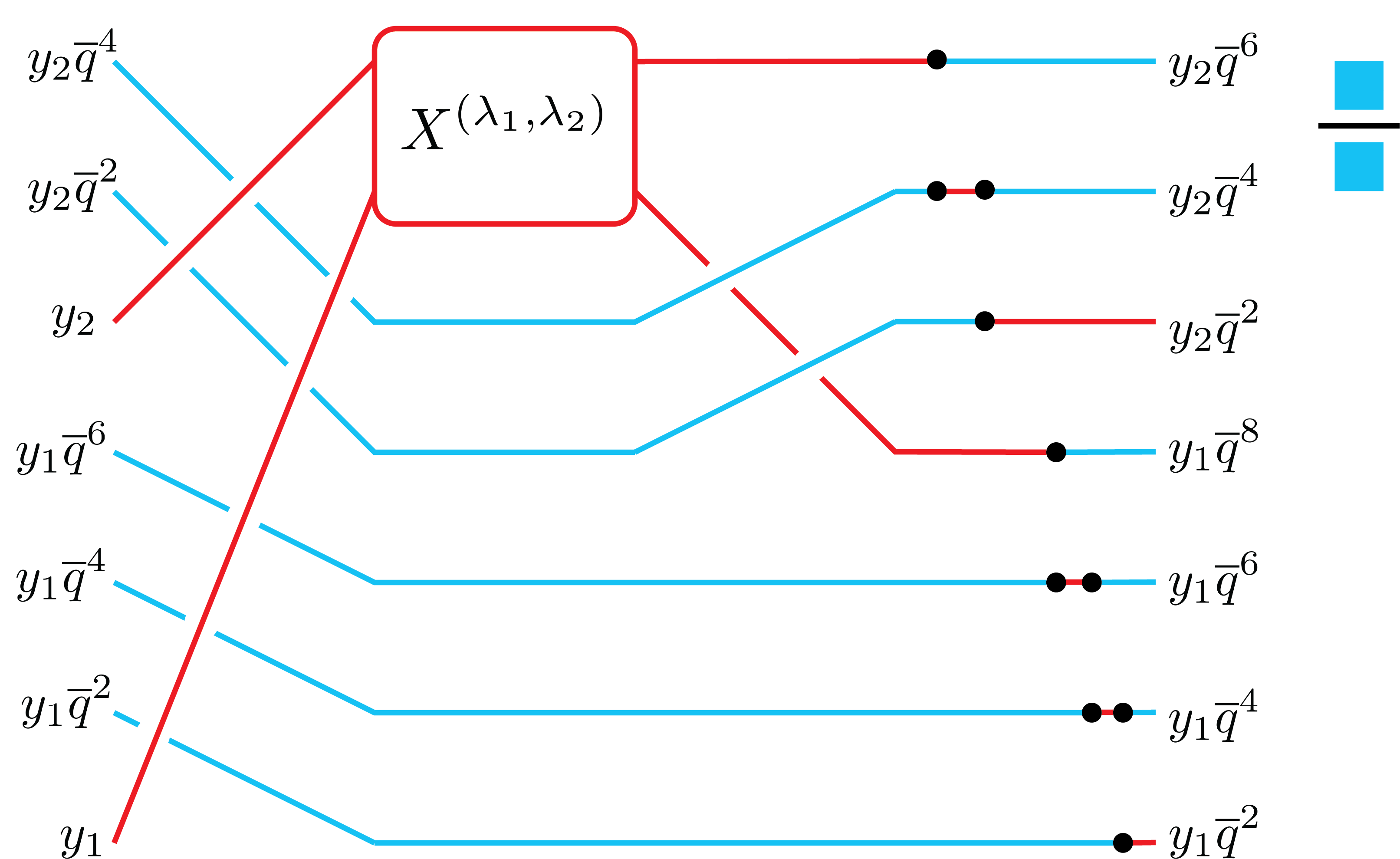}
\caption{The RHS of \eqref{eqn:wheel} for $u=2$, $\lambda_1 = 4$, $\lambda_2 = 3$}
\end{figure}

\noindent Note the symbol ``blue over blue" to the right of Figure 14. Given two colors $\gamma_1$ and $\gamma_2$, placing $\gamma_1$ over $\gamma_2$ is a prescription that indicates that the braid in question be multiplied by the product of $f(y/y')$, where $y$ (respectively $y'$) goes over all variables on strands whose left endpoint has color $\gamma_1$ (respectively $\gamma_2$), and the leftmost endpoint with variable $y$ is above the leftmost endpoint with variable $y'$. This is meant to account for the product of $f$'s on the second line of \eqref{eqn:wheel}. \\

\begin{remark} When $n=1$, we may identify $1 \times 1$ matrices with scalars, and so elements of the shuffle algebra are merely symmetric rational functions in $z_1,\dots,z_k$ with at most simple poles at $z_i = z_j \oq^2$. The shuffle product \eqref{eqn:shuf prod aff} takes the form
$$
A * B = \sum^{a_1<...<a_k, \ b_1<...<b_l}_{\{1,...,k+l\} = \{a_1,...,a_k\} \sqcup \{b_1,...,b_l\}}  A(z_{a_1},...,z_{a_k})  B(z_{b_1},...,z_{b_l}) \prod_{i=1}^k \prod_{j=1}^{l} \zeta \left( \frac {z_{a_i}}{z_{b_j}} \right)
$$
where
$$
\zeta(x) = \frac {(xq\oq^2-q^{-1})(xq^{-1}-q)}{(x\oq^2-1)(x-1)}
$$
Meanwhile, the wheel conditions \eqref{eqn:wheel} for the composition $1+...+1+2$ impose the following restriction on a rational function $X(z_1,\dots,z_k)$:
$$
\underset{z_k = z_{k-1} \oq^2}{\eRes} X  = \prod_{i=1}^{k-2} \Big[ (qz_{k-1} - z_i\oq^2q^{-1})(qz_i - z_{k-1}q^{-1}) \Big] \cdot \Big(\text{function of }z_1,...,z_{k-1} \Big)
$$
where the function in the round brackets can only have simple poles at $z_i-z_j\oq^{\pm 2}$ and $z_i - z_{k-1}\oq^{-2,0,2,4}$ for all $i < j \in \{1,\dots,k-1\}$ (see Proposition \ref{prop:symmetric}). Therefore, the residue above vanishes when we set $z_i = z_{k-1} q^{-2}$ or $z_i = z_{k-1} q^{2} \oq^{2}$, which in our notations and normalization is precisely the wheel condition of \cite{FHHSY}. \\

\end{remark}

\begin{proposition}
\label{prop:wheel preserved}
	
The vector subspace $\CA^+$ of Definition \ref{def:shuf aff} is preserved by the shuffle product (and will henceforth be called the ``shuffle algebra"). \\
	
\end{proposition}

\begin{proof} Assume that two matrix-valued rational functions $A$ and $B$ in $k$ and $l$ variables, respectively, satisfy the wheel condition \eqref{eqn:wheel}. To prove that their shuffle product $A*B$ also satisfies the wheel condition, we must take the iterated residue of the right-hand side of \eqref{eqn:shuf prod aff} at:
$$
z_{c_s} = y_s, z_{c_s+1} = y_s \oq^2 ,..., z_{c_{s+1}-1} = y_s \oq^{2(\lambda_s-1)}
$$
for any composition $k+l = \lambda_1+...+\lambda_u$. We will show that, at such a specialization, each summand in the RHS of \eqref{eqn:shuf prod aff} has the form predicated in the RHS of \eqref{eqn:wheel}, so we henceforth fix a shuffle $a_1<...<a_k$, $b_1<...<b_l$. Because $\tR(z)$ has a simple pole at $z = \oq^{-2}$, the only way such a shuffle can have a non-zero residue is if:
\begin{align*} 
&\{a_1,...,a_k\} = \bigsqcup_{s=1}^u \{c_s,c_s+1,...,r_s-2,r_s-1\} \\
&\{b_1,...,b_l\} = \bigsqcup_{s=1}^u \{r_s,r_s+1,...,c_{s+1}-2, c_{s+1}-1\}
\end{align*}
for some choice of $r_s \in \{c_s,...,c_{s+1}-1\}$ for all $s \in \{1,...,u\}$. We will indicate this choice by using the following colors for the strands of our braids:
\begin{align*}
&{\color{red} \text{red }} \text{for } c_s, \qquad \ {\color{cyan} \text{blue }} \text{for } c_s+1,...,r_s-1 \\
&{\color{Plum} \text{purple }} \text{for } r_s, \quad {\color{ForestGreen} \text{green }} \text{for } r_s+1,...,c_{s+1}-1
\end{align*}
With this in mind, the summand of \eqref{eqn:shuf prod aff} corresponding to our chosen shuffle is represented by the following braid (to keep the pictures reasonable, we will only depict the case $u=2$, but the modifications that lead to the general case are straightforward; although we only depict a single blue and green strand in each of the $u$ groups, the reader may obtain the general case by replacing each of them with any number of parallel blue and green strands, respectively):

\begin{figure}[ht]    
\centering
\includegraphics[scale=0.16]{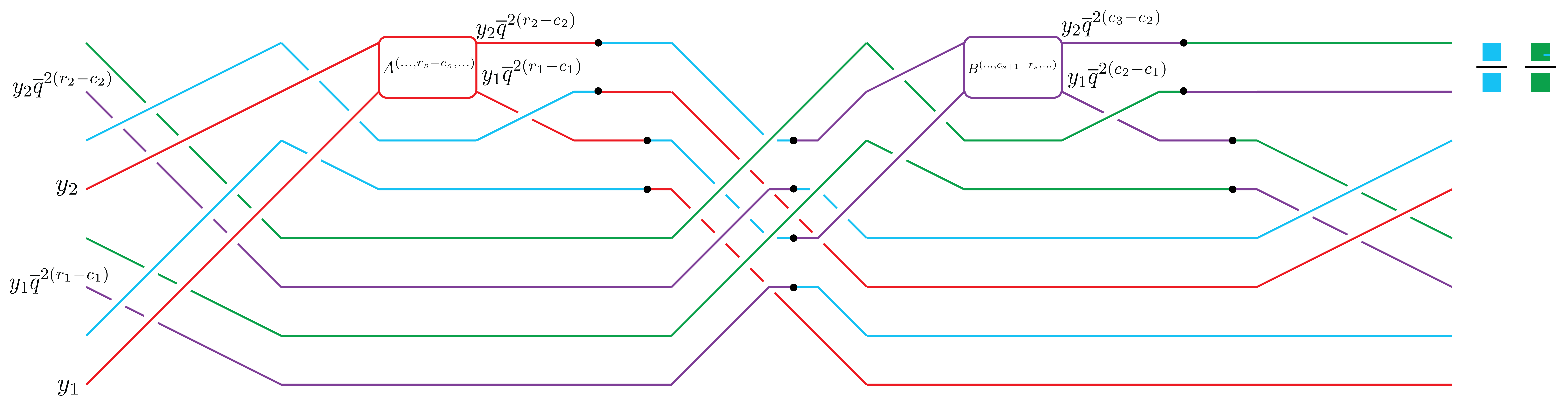}
\caption{}
\end{figure}

\noindent The black dots in the middle of the braid appear because the variables on the braids in question are set equal to each other in the iterated residue. By sliding the black dots as far to the right as possible (which is allowed, due to Figure 12), we obtain:
	
\begin{figure}[ht]    
\centering
\includegraphics[scale=0.16]{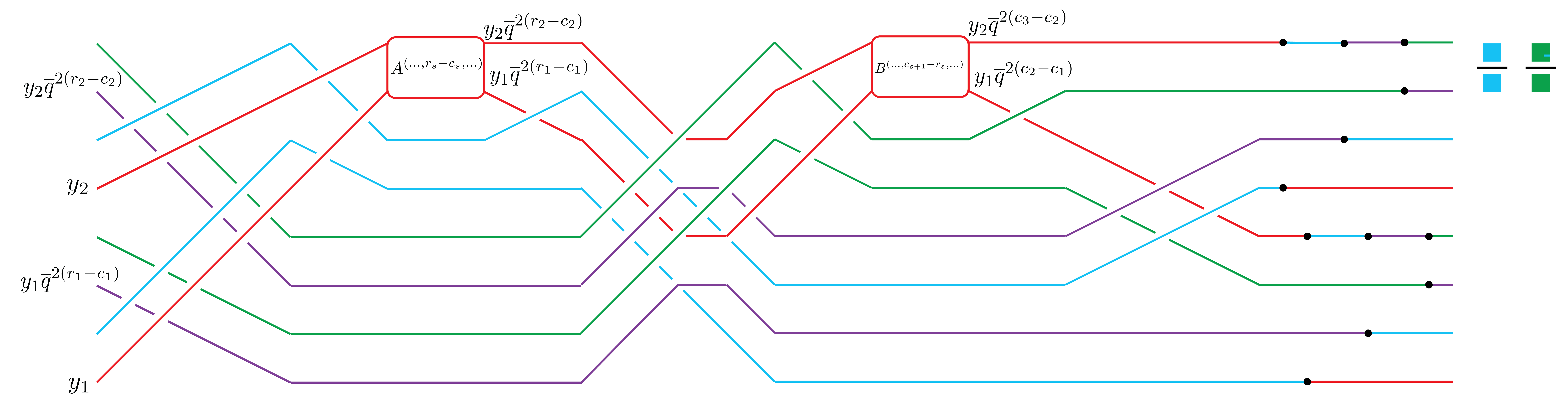}
\caption{}
\end{figure}

\noindent One readily notices that certain pairs of braids are twisted twice around each other, and these twists can be canceled up to a factor of $f(y/y')$ (due to the identity in Figure 13), where $y$ and $y'$ are the variables on the braids in question. Keeping in mind that the variables on the red strands are modified to the right of the red boxes, this yields the braid in Figure 17. Note that the black dots on the right side of Figure 17 yield the same permutation as the black dots on the right side of the braid in Figure 14, due to the following identity:
$$
\begin{pmatrix} 1 & ... & k-1 & k \\ 2 & ... & k & 1 \end{pmatrix} = (12)(23)...(k-1,k) = (1k)(1,k-1)...(12)
$$
in the symmetric group $S(k)$. Therefore, the braid in Figure 17 is precisely of the form predicated in the right-hand side of \eqref{eqn:wheel}, which concludes our proof. 

\begin{figure}
\centering
\includegraphics[scale=0.16]{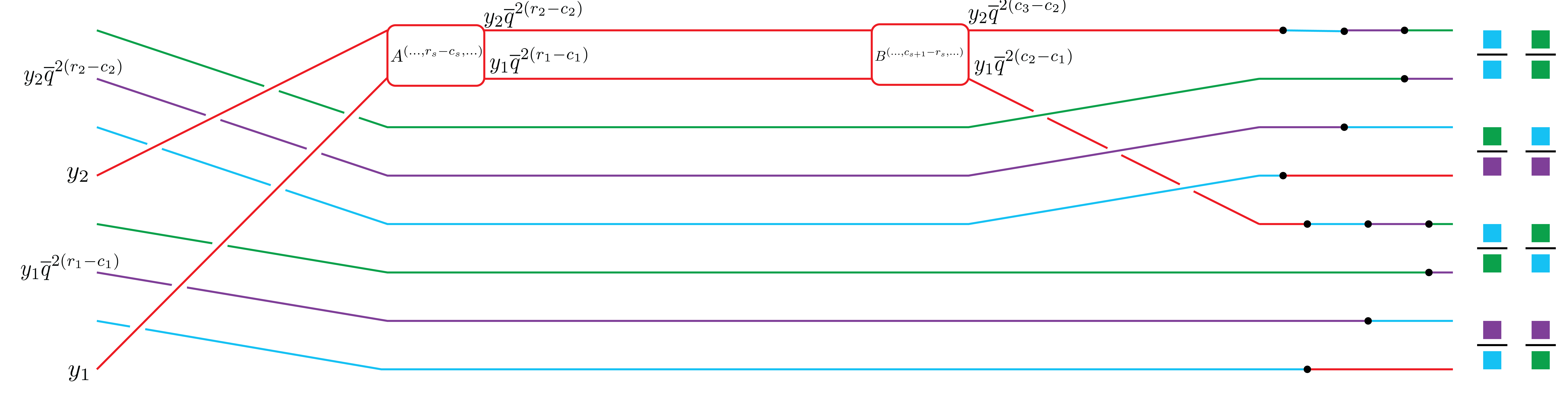}
\caption{}
\end{figure}

\end{proof} 

\begin{proposition}
\label{prop:symmetric}
	
For any $X \in \CA^+$ and any composition $k = \lambda_1 + ... + \lambda_u$, the tensor $Y = X^{(\lambda_1,...,\lambda_u)}$ that appears in \eqref{eqn:wheel} has at most simple poles at: 
\begin{equation}
\label{eqn:poles gen}
y_s \oq^{2d} - y_t \oq^{-2} \qquad \text{and} \qquad y_s \oq^{2d} - y_t \oq^{2\lambda_t}
\end{equation}
for all $1 \leq s < t \leq u$ and any $0 \leq d < \lambda_s$. Moreover, if $\lambda_s = \lambda_t$ then:
\begin{equation}
\label{eqn:symm y}
Y_{...,s,...t,...}(...,y_s,...,y_t,...) = R_{(st)} \cdot Y_{...,t,...s,...}(...,y_t,...,y_s,...) \cdot R_{(st)}^{-1}
\end{equation}
for any braid lift $R_{(st)} = R_{(st)}(y_1,...,y_u)$ of the transposition $(st)$ \footnote{Although we will not need it, formula \eqref{eqn:symm y} admits the following generalization for any $X \in \CA^+$ and any composition $\lambda = (\lambda_1,...,\lambda_u)$ of $k$:
$$
X^{\lambda}_{...,s,...t,...}(...,y_s,...,y_t,...) = R_{(st)} \cdot X^{\lambda'}_{...,t,...s,...}(...,y_t,...,y_s,...) \cdot R_{(st)}^{-1}
$$
where $\lambda'$ denotes the composition obtained from $\lambda$ by switching the $s$-th and the $t$-th entries.}. \\
	
\end{proposition}

\begin{proof} Let us first prove the statement about the poles of $Y$. In the course of this proof, all poles will be counted with multiplicities, in the sense that whenever we refer to a ``set of poles", the reader should assume this means ``multiset of poles". Because of the first bullet of Proposition \ref{prop:shuf aff 2}, which determines the allowable poles of $X \in \CA^+$, the left-hand side of \eqref{eqn:wheel} has a simple pole at:
\begin{equation}
\label{eqn:poles 1}
y_s \oq^{2d} - y_t \oq^{2e \pm 2} \quad \forall \ 1 \leq s < t \leq u, \ 0 \leq d < \lambda_s, \ 0 \leq e < \lambda_t 	
\end{equation}
On the other hand, the right-hand side of \eqref{eqn:wheel} has a double pole at:
\begin{equation}
\label{eqn:poles 2}
y_s \oq^{2d} - y_t \oq^{2e} \qquad \forall \ 1 \leq s < t \leq u, \ 1 \leq d < \lambda_s, \ 1 \leq e < \lambda_t 	
\end{equation}
because of the $f$ factors, and a simple pole at:
\begin{equation}
\label{eqn:poles 3}
\begin{cases} y_s - y_t \oq^{2e}, \text{ and} \\ y_s \oq^{2\lambda_s}  - y_t \oq^{2e} \end{cases} \qquad \forall \ 1 \leq s < t \leq u, \ 1 \leq e < \lambda_t
\end{equation}
because of the simple pole of $R(z)$ at $z=1$. Eliminating the multiset of poles in \eqref{eqn:poles 2} and \eqref{eqn:poles 3} from the multiset of poles in \eqref{eqn:poles 1} yields the allowable poles of $Y(y_1,...,y_u)$, and it is elementary to see that they are precisely of the form \eqref{eqn:poles gen}. \\

\noindent As for \eqref{eqn:symm y}, we will prove it pictorially. To keep the pictures legible, we will only show the case $u=2$ (hence $(s,t) = (1,2)$), but the reader may easily generalize the argument. Because of property \eqref{eqn:symm aff}, the tensor $X(y_s,y_s\oq^2,...,y_t,y_t\oq^2,...)$ (which is represented by a braid akin to Figure 14) is also equal to the following braid:
	
\begin{figure}[ht]    
\label{fig:symm 1}
\centering
\includegraphics[scale=0.3]{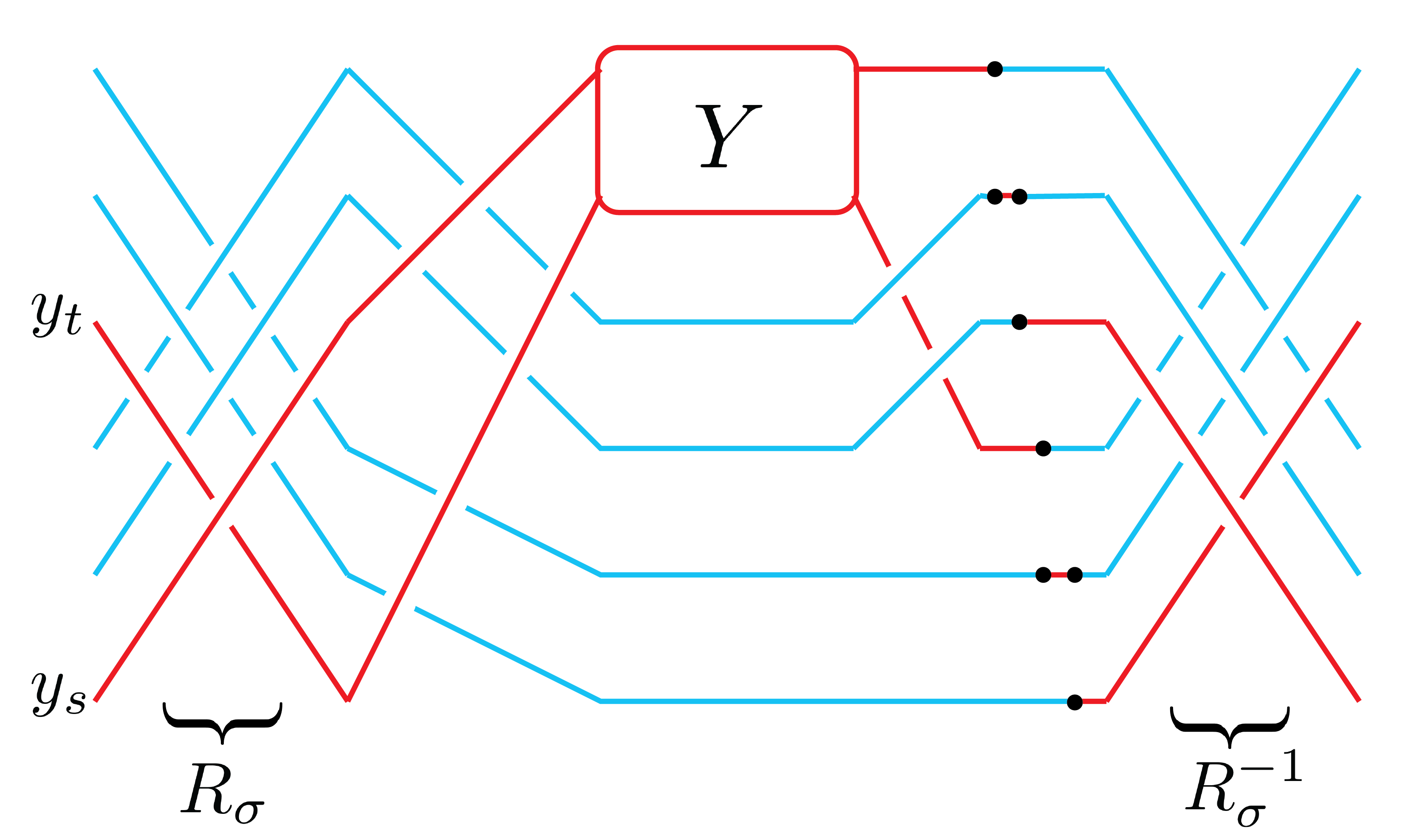}
\caption{}
\end{figure}
	
\noindent (we ignore the scalar-valed rational functions $f$ in the diagrams above, as they commute with all the braids involved). The braid called $R_\sigma$ interchanges the two collections of $\lambda_s = \lambda_t$ braids corresponding to the variables $y_s \oq^{2*}$ and $y_t \oq^{2*}$. Although we could choose the crossings of $R_\sigma$ arbitrarily, the choice we make above is that the two red strands cross above all other ones, then the two blue strands next to the red strands cross above all remaining ones, then the two blue strands next to the previous blue strands cross etc. In virtue of Figure 12, we may move the black dots to the very right of the picture above, obtaining the braid below:

\begin{figure}[ht]    
	\label{fig:symm 2}
\centering
\includegraphics[scale=0.3]{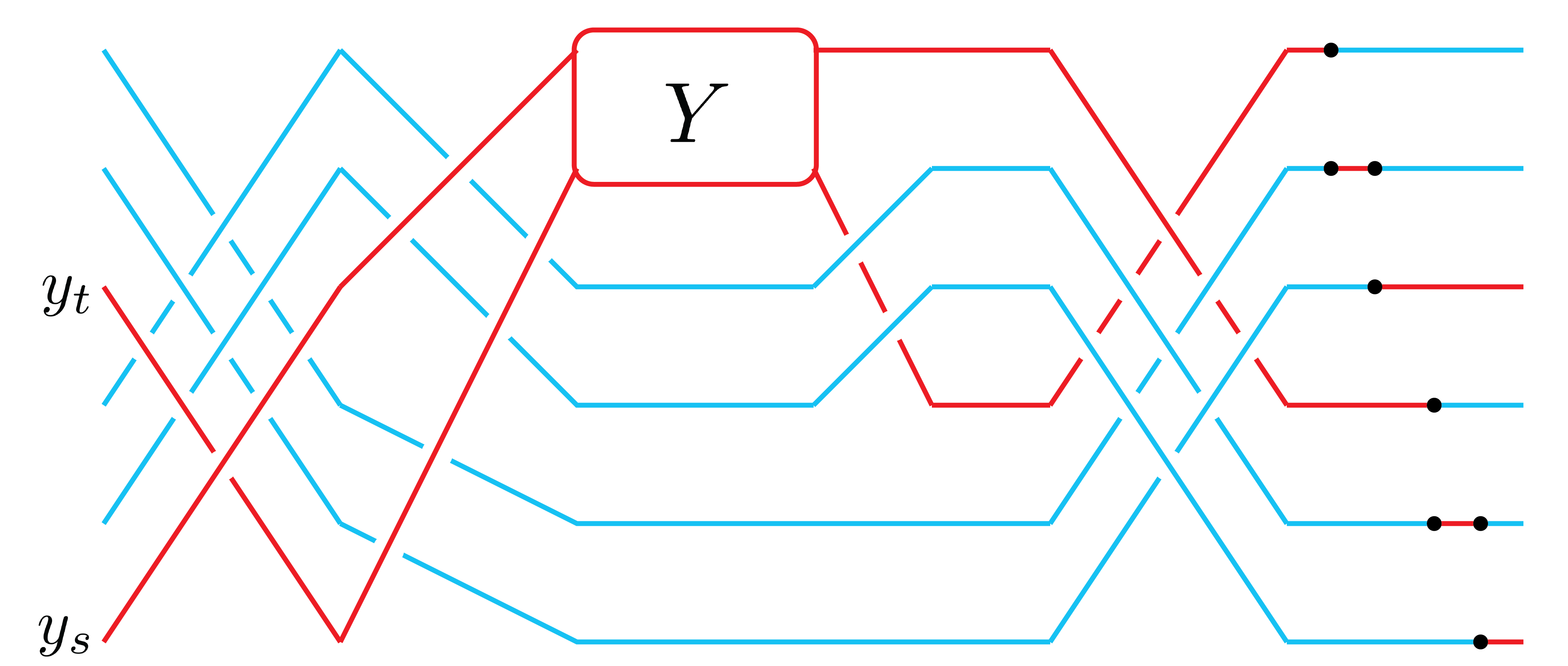}
\caption{}
\end{figure}

\noindent Then we pull the red strands as far up as possible, and notice that the blue strands are all unlinked, thus yielding the braid in Figure 20 below. The red strands in Figure 20 correspond to the endomorphism: 
$$
R_{(st)} \cdot (st)Y (st) \cdot R_{(st)}^{-1} \in \End(V^{\otimes 2})(y_s,y_t)
$$
which we may equate with $Y$ due to the braid equivalences described above. 

\begin{figure}  
	\label{fig:symm 3}
\centering
\includegraphics[scale=0.3]{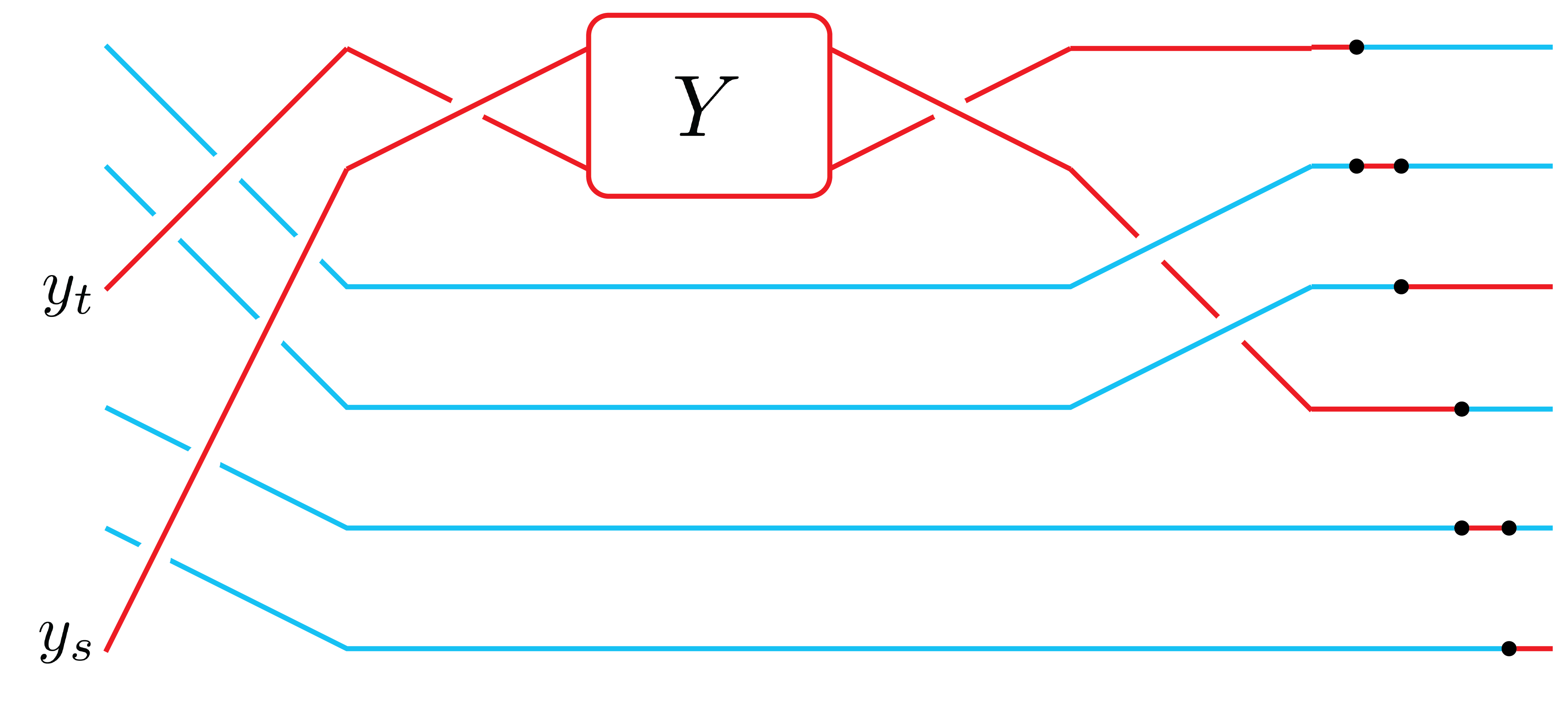}
\caption{}
\end{figure}
	
\end{proof}

\subsection{} 

The shuffle algebra $\CA^+$ has a ``vertical" and a ``horizontal" grading:
\begin{align}
& \BN \ni \vdeg f(z_1,...,z_k) E_{i_1j_1} \otimes ... \otimes E_{i_kj_k} = k \label{eqn:vertical} \\
& \zz \ni \hdeg f(z_1,...,z_k) E_{i_1j_1} \otimes ... \otimes E_{i_kj_k} = (\text{hom deg }f) \bde + \sum_{a=1}^k \deg E_{i_aj_a} \label{eqn:horizontal}
\end{align}
where $\bde = (1,...,1)$ and the grading on $\End(V)$ is defined by:
\begin{equation}
\label{eqn:degree} 
\deg E_{ij} = - [i;j) 
\end{equation}
We will find it convenient to extend the notation $E_{ij}$ to all $i,j \in \BZ$, according to:
\begin{equation}
\label{eqn:notation}
E_{ij} = E_{\bari \barj} z^{\left \lfloor \frac {i-1}n \right \rfloor - \left \lfloor \frac {j-1}n \right \rfloor} \in \End(V) [z^{\pm 1}]
\end{equation}
where $\bari$ denotes the residue class of $i$ in the set $\{1,...,n\}$.\footnote{More generally, we will extend the notation above to a $k$--fold tensor, by the rule:
\begin{equation}
\label{eqn:tensor convention}
E_{i_1j_1} \otimes ... \otimes E_{i_kj_k} = E_{\bari_1 \barj_1} \otimes ... \otimes E_{\bari_k \barj_k} z_1^{\left \lfloor \frac {i_1-1}n \right \rfloor - \left \lfloor \frac {j_1-1}n \right \rfloor} ... z_k^{\left \lfloor \frac {i_k-1}n \right \rfloor - \left \lfloor \frac {j_k-1}n \right \rfloor} 
\end{equation}
as elements of $\End(V^{\otimes k}) [z_1^{\pm 1},..., z_k^{\pm 1}]$} Then formulas \eqref{eqn:horizontal} and \eqref{eqn:degree} hold for arbitrary integer indices $i_1,j_1,...,i_k,j_k$ and $i,j$. We will denote the graded pieces of the shuffle algebra by:
$$
\CA^+ = \bigoplus_{k=0}^\infty \CA_k, \qquad \CA_k = \bigoplus_{\bd \in \zz} \CA_{\bd,k}
$$
and refer to $(\bd,k)$ as the degree of homogeneous elements. Finally, we write:
$$
|\bd| = d_1 + ... + d_n
$$
for any $\bd = (d_1,...,d_n) \in \BZ$ and refer to the number:
\begin{equation}
\label{eqn:slope}
\mu = \frac {|(\text{hom deg }f) \bde + \sum_{a=1}^k \deg E_{i_aj_a}|}k \in \BQ
\end{equation}
as the slope of the matrix-valued rational function $f(z_1,...,z_k) E_{i_1j_1} \otimes ... \otimes E_{i_kj_k}$. We will consider the partial ordering on $\BZ^n$ given by:
\begin{equation}
\label{eqn:partial ordering}
(d_1,...,d_n) \leq (d_1',...,d_n') \text{ if } \begin{cases} d_n < d_n' \text{ or} \\ d_n = d_n' \text{ and } \sum_{i=1}^{n-1} d_i \leq \sum_{i=1}^{n-1} d_i' \end{cases}
\end{equation}
$$$$
	
\section{The extended shuffle algebra with spectral parameter}
\label{sec:extended}

\subsection{} We will now replicate the construction of Subsection \ref{sub:extended} in the situation of the $R$--matrix with spectral parameter \eqref{eqn:explicit r}. \\ 

\begin{definition}
\label{def:extended aff}
	
Consider the extended shuffle algebra:
$$
\tCA^+ = \Big \langle \CA^+, s_{[i;j)} \Big \rangle^{i \leq j \in \BZ}_{1\leq i \leq n} \Big/ \text{relations \eqref{eqn:ext rel 1 aff} and \eqref{eqn:ext rel 4 aff}}
$$
In order to concisely state the relations, it makes sense to package the new generators $s_{[i;j)}$ into the following matrix-valued generating function:
$$
S(x) = \sum_{1 \leq i, j \leq n, \ d \geq 0}^{d = 0 \text{ only if } i \leq j} s_{[i;j+nd)} \otimes \frac {E_{ij}}{x^d} \in \tCA^+ \otimes \emph{End}(V) [[x^{-1}]]
$$
We impose the following analogues of relations \eqref{eqn:ext rel 1} and \eqref{eqn:ext rel 4}:
\begin{equation}
\label{eqn:ext rel 1 aff}
R\left(\frac xy \right) S_1(x) S_2(y) = S_2(y) S_1(x) R \left(\frac xy \right)
\end{equation}
\begin{equation}
\label{eqn:ext rel 4 aff}
X \cdot S_0(y) = S_0(y) \cdot \frac {R_{k0} \left(\frac {z_k}y \right) ... R_{10} \left(\frac {z_1}y \right)}{f \left(\frac {z_k}y \right) ... f \left(\frac {z_1}y \right)} X \tR_{10} \left( \frac {z_1}y \right) ... \tR_{k0} \left( \frac {z_k}y \right) 
\end{equation}
for any $X = X_{1...k}(z_1,...,z_k) \in \CA^+ \subset \tCA^+$. \\
	
\end{definition}

\noindent Note that the grading on $\CA^+$ extends to one on $\tCA^+$, by setting: 
$$
\deg s_{[i;j)} = ([i;j),0) \in \zz \times \BN
$$
The series $S(x)$ defined above is supposed to match the homonymous series in \eqref{eqn:series s}. \\

\subsection{} We may also consider the elements $t_{[i;j)} \in \tCA^+$ defined by \eqref{eqn:new series t}, where:
\begin{equation}
\label{eqn:s and t}
T(x) = \sum_{1 \leq i, j \leq n, \ d \geq 0}^{d = 0 \text{ only if } i \leq j} t_{[i;j+nd)} \otimes \frac {E_{ij}}{x^d}
\end{equation}
Then it is a straightforward computation (which we leave as an exercise to the interested reader) to see that \eqref{eqn:ext rel 1 aff}, \eqref{eqn:ext rel 4 aff} imply analogues of \eqref{eqn:ext rel 2}, \eqref{eqn:ext rel 3}, \eqref{eqn:ext rel 5}:
\begin{equation}
\label{eqn:ext rel 2 aff}
T_1(x) T_2(y) R\left(\frac xy \right) = R \left( \frac xy \right) T_2(y) T_1(x) 
\end{equation}
\begin{equation}
\label{eqn:ext rel 3 aff}
T_1(x) \tR\left(\frac xy \right) S_2(y)  =  S_2(y) \tR \left( \frac xy \right) T_1(x) 
\end{equation}
\begin{equation}
\label{eqn:ext rel 5 aff}
T_0(y) \cdot X = \tR_{0k} \left( \frac y{z_k} \right) ... \tR_{01} \left( \frac y{z_1} \right) X \frac {R_{01} \left( \frac y{z_1} \right) ... R_{0k} \left( \frac y{z_k} \right)}{f \left( \frac y{z_1} \right) ... f \left( \frac y{z_k} \right)} \cdot T_0(y)
\end{equation} 
Therefore, we conclude that the series $S(x)$ and $T(x)$ satisfy the same relations as in Definition \ref{def:extended} (modified in order to account for the variables $z_i$), even though the $s$'s and the $t$'s are not independent of each other anymore. We will write:
$$
s_{[i;i)}^{-1} = \psi_i = t_{[i;i)}
$$
for all $i \in \{1,\dots,n\}$ and note that formulas \eqref{eqn:ext rel 1 aff} imply that:
\begin{equation}
\label{eqn:psi commute}
\psi_i \psi_j = \psi_j \psi_i
\end{equation}
\begin{equation}
\label{eqn:psi x}
\psi_i X = q^{- \langle \bs^i, \hdeg X \rangle} X \psi_i
\end{equation}
$\forall X \in \tCA^+$, where $\langle \cdot, \cdot \rangle$ is the bilinear form on $\BZ^n$ given by $\langle \bs^i, \bs^j \rangle = \delta_{i}^{j} - \delta_{i}^{j+1}$. \\

\subsection{} Consider the following topological coproduct on the algebra $\tCA^+$, which is the natural analogue of the coproduct studied in Proposition \ref{prop:extended}:
\begin{equation}
\label{eqn:cop shuf aff 1}
\Delta(S(x)) = (1 \otimes S(x)) \cdot (S(x) \otimes 1) 
\end{equation}
(hence $\Delta(T(x)) = \left(T(x) \otimes 1 \right) \cdot (1 \otimes T(x))$ by \eqref{eqn:new series t}) and:
\begin{equation}
\label{eqn:cop shuf aff 3}
\Delta(X_{1...k}) = \sum_{i=0}^k (S_{k}(z_k)...S_{i+1}(z_{i+1}) \otimes 1) \cdot
\end{equation}
$$
\cdot \left[ \frac {X_{1...i} \left( z_1,...,z_i \right) \otimes X_{i+1...k} \left(z_{i+1}, ..., z_k \right)}{\prod_{1\leq u \leq i < v \leq k} f \left( \frac {z_u}{z_v} \right)} \right] \cdot (T_{i+1}(z_{i+1}) ... T_{k}(z_k) \otimes 1) 
$$
for all $X_{1...k}(z_1,...,z_k) \in \CA^+ \subset \tCA^+$. The fact that $\Delta$ defined above gives rise to a coasociative coalgebra structure which respects the algebra structure is proved by analogy with Proposition \ref{prop:extended}, and we leave the details to the interested reader. \\ 

\begin{remark} 
\label{rem:completed} 

Because $S(x)$ is a power series in $x$, the coproduct defined above takes values in a completion of $\tCA^+ \otimes \tCA^+$. Specifically, to make sense of the second line of \eqref{eqn:cop shuf aff 3}, we must expand the rational function:
$$
\frac {X_{1...k} \left( z_1,...,z_k \right)}{\prod_{1\leq u \leq i < v \leq k} f \left( \frac {z_u}{z_v} \right)}
$$
for $z_1,...,z_i \ll z_{i+1},...,z_k$, then collect all tensors of the form $X_{1...i}(z_1,...,z_i)$ to the left of $\otimes$, and all tensors of the form $X_{i+1,...,k}(z_{i+1},...,z_k)$ to the right of $\otimes$. See \cite[Equation (4.4)]{Shuf} for an analogous situation, which corresponds to our case $n=1$. \\
	
\end{remark} 

\subsection{} 
\label{sub:hinge}

Given $X \in \CA^+$, we will represent its degree $\deg X = (\bd,k)$ on a 2 dimensional lattice, via the projection $(\bd,k) \twoheadrightarrow (|\bd|,k)$, and hence assign to $X$ the lattice point $(|\bd|,k)$. Similarly, we will assign to the tensor $X_1 \otimes X_2$ the two-segment path:

\begin{picture}(100,130)(-110,-40)
\label{pic:1}

\put(0,0){\circle*{2}}\put(20,0){\circle*{2}}\put(40,0){\circle*{2}}\put(60,0){\circle*{2}}\put(80,0){\circle*{2}}\put(100,0){\circle*{2}}\put(120,0){\circle*{2}}\put(0,20){\circle*{2}}\put(20,20){\circle*{2}}\put(40,20){\circle*{2}}\put(60,20){\circle*{2}}\put(80,20){\circle*{2}}\put(100,20){\circle*{2}}\put(120,20){\circle*{2}}\put(0,40){\circle*{2}}\put(20,40){\circle*{2}}\put(40,40){\circle*{2}}\put(60,40){\circle*{2}}\put(80,40){\circle*{2}}\put(100,40){\circle*{2}}\put(120,40){\circle*{2}}\put(0,60){\circle*{2}}\put(20,60){\circle*{2}}\put(40,60){\circle*{2}}\put(60,60){\circle*{2}}\put(80,60){\circle*{2}}\put(100,60){\circle*{2}}\put(120,60){\circle*{2}}

\put(0,0){\vector(3,1){120}}
\put(120,40){\vector(-4,1){80}}

\put(-20,-3){\scriptsize{$(0,0)$}}
\put(5,68){\scriptsize{$(|\bd_1+\bd_2|,k_1+k_2)$}}
\put(123,38){\scriptsize{$(|\bd_2|,k_2)$}}

\put(-10,-25){\mbox{F{\scriptsize IGURE} 21. The hinge of a tensor}}

\end{picture}

\noindent The intersection of the arrows, namely $(|\bd_2|,k_2)$, is called the \textbf{hinge} of $X_1 \otimes X_2$. Recall the definition of slope in \eqref{eqn:slope}. \\

\begin{definition}
\label{def:leading order}
	
Let $\mu \in \BQ$. We let $\CA_{\leq \mu}^+ \subset \CA^+$ be the set of those $X$ such that:
\begin{equation}
\label{eqn:leading order}
\Delta(X) = \Delta_{\mu}(X) + (\text{anything}) \otimes (\slope < \mu)
\end{equation}
where $\Delta_\mu(X)$ consists only of summands $X_1 \otimes X_2$ with $\text{slope } X_2 = \mu$ (we also choose to include the summand $X \otimes 1$ in $\Delta_\mu(X)$). In terms of the pictorial definitions of hinges in Figure 21, $X \in \CA^+_{\leq \mu}$ if and only if all summands in $\Delta(X)$ have hinge at slope $|\bd|/k \leq \mu$ as measured from the origin. \\ 	
	
\end{definition}

\noindent We will say that an element of $\CA^+_{\leq \mu}$ has \textbf{slope} $\leq \mu$. Since having slope $\leq \mu$ is an additive property, $\CA^+_{\leq \mu}$ is a vector space. Let us define its graded pieces:
\begin{equation}
\label{eqn:pieces a}
\CA_{\leq \mu|k} = \CA^+_{\leq \mu} \cap \CA_k, \qquad \CA_{\leq \mu|\bd,k} = \CA^+_{\leq \mu} \cap \CA_{\bd,k}
\end{equation}
and note that $\CA_{\leq \mu|\bd,k} \neq 0$ only if $|\bd| \leq k \mu$. \\

\begin{proposition} 
\label{prop:multi} 

For any $\mu \in \BQ$, the subspace $\CA^+_{\leq \mu}$ is a subalgebra of $\CA^+$, and:
$$
\Delta_\mu(X * Y) = \Delta_\mu(X) * \Delta_\mu(Y)
$$
for all $X,Y \in \CA^+_{\leq \mu}$. \\
	
\end{proposition}

\begin{proof} Note that degree is multiplicative, i.e. (assume the LHS is non-zero):
$$
\deg \Big( f(z_1,...,z_k)E_{i_1j_1} \otimes ... \otimes E_{i_kj_k} \Big) \Big( f'(z_1,...,z_k)E_{i'_1j'_1} \otimes ... \otimes E_{i'_kj'_k} \Big) =
$$
$$
= \deg f(z_1,...,z_k)E_{i_1j_1} \otimes ... \otimes E_{i_kj_k} + \deg f'(z_1,...,z_k)E_{i'_1j'_1} \otimes ... \otimes E_{i'_kj'_k} 
$$
Therefore, if $\Delta(X) = X_1 \otimes X_2$ and $\Delta(Y) = Y_1 \otimes Y_2$ with $\slope \ X_2, \slope \ Y_2 \leq \mu$, then $\slope \ X_2 Y_2 \leq \mu$. Since $\Delta(XY) = X_1 Y_1 \otimes X_2 Y_2$, this implies the conclusion. 
	
\end{proof}

\subsection{}

Our reason for introducing slopes and the algebras $\CA^+_{\leq \mu}$ is that $\{\CA_{\leq \mu| \bd,k}\}_{\mu \in \BQ}$ yield a filtration of $\CA_{\bd,k}$ by finite-dimensional vector spaces. \\

\begin{lemma}
\label{lem:magic}
	
The dimension of $\CA_{\leq \mu|\bd,k}$ as a vector space over $\fff$ is at most the number of unordered collections:
\begin{equation}
\label{eqn:unordered collections}
(i_1,j_1,\lambda_1),...,(i_u,j_u,\lambda_u) 
\end{equation}
where: \\
	
\begin{itemize}
		
\item $\lambda_s \in \BN$ with $\sum_{s=1}^u \lambda_s = k$ \\
		
\item $(i_s,j_s) \in \zzz$ with $\sum_{s=1}^u [i_s;j_s) = \bd$ \\
		
\item $j_s - i_s \leq \mu \lambda_s$ for all $s \in \{1,...,u\}$ \\
		
\end{itemize} 
	
\end{lemma}

\noindent In \eqref{eqn:geq leq}, we will show that the dimension of $\CA_{\leq \mu|\bd,k}$ is in fact equal to the number of unordered collections \eqref{eqn:unordered collections}. The argument below follows that of \cite{FHHSY, Shuf, Tor}. \\

\begin{proof} To any partition $\lambda = (\lambda_1 \leq ... \leq \lambda_u)$ of $k \in \BN$, we associate the linear map:
$$
\CA_{\leq \mu|\bd,k} \stackrel{\ph_\lambda}\longrightarrow \End(V^{\otimes u})(y_1,...,y_u)
$$
$$
X \leadsto X^{(\lambda_1,...,\lambda_u)} \text{ of \eqref{eqn:wheel}}
$$
Consider the dominance partial ordering $\lambda' > \lambda$ on partitions \footnote{Recall that $\lambda = (\lambda_1 \leq ... \leq \lambda_u) \leq \lambda' = (\lambda'_1 \leq ... \leq \lambda'_v)$ if $\lambda_u+...+\lambda_{u-s+1} \leq \lambda'_v+... +\lambda'_{v-s+1}$ for all $s$, but $\lambda_u+...+\lambda_1 = \lambda'_v+...+\lambda'_1$.}, and define:
$$
\CA_{\leq \mu|\bd,k}^{\lambda} = \bigcap_{\lambda' > \lambda} \text{Ker } \ph_{\lambda'}
$$
Since $\CA_{\leq \mu|\bd,k}^{(k)} = \CA_{\leq \mu|\bd,k}$, then the desired bound on $\dim \CA_{\leq \mu|\bd,k}$ would follow from:
\begin{equation}
\label{eqn:claim}
\dim \ph_\lambda \left( \CA_{\leq \mu|\bd,k}^{\lambda} \right) \leq \# \left \{ \text{quasi-ordered } (i_1,j_1),...,(i_u,j_u) \in \frac {\BZ^2}{(n,n)\BZ}, \right. 
\end{equation}
$$
\left. \text{such that } \sum_{s=1}^u [i_s;j_s) = \bd \text{ and } j_s - i_s \leq \mu \lambda_s \text{ for all }s \in \{1,...,u\}\right\}
$$
for any $\lambda = (\lambda_1 \leq ... \leq \lambda_u)$. Above, ``quasi-ordered" means that collections $\{(i_1,j_1),...,(i_u,j_u)\}$ are considered only up to arbitrarily swapping any $(i_a,j_a)$ and $(i_b,j_b)$ for which $\lambda_a=\lambda_b$. By Proposition \ref{prop:symmetric}, any $Y \in \text{Im } \ph_\lambda$ is of the form:
$$
Y(y_1,...,y_u) \in \frac {\End(V^{\otimes u}) [y_1^{\pm 1},...,y_u^{\pm 1}]}{\prod_{1\leq s < t \leq u} \prod_{d=0}^{\lambda_s - 1} (y_s \oq^{2d} - y_t \oq^{-2})(y_s \oq^{2d} - y_t \oq^{2\lambda_t})}
$$
However, if $Y = \ph_\lambda(X)$ for some $X \in \CA_{\leq \mu|\bd,k}^{\lambda}$, we claim that $Y$ is a Laurent polynomial. Indeed, let us show that $Y$ is regular at $y_s \oq^{2d} - y_t \oq^{-2}$. We have:
\begin{multline} 
\underset{y_s \oq^{2d} = y_t \oq^{-2}}{\Res} \left( \underset{\{z_{c_r} = y_r, ..., z_{c_{r+1}-1} = y_r \oq^{2(\lambda_r - 1)} \}_{\forall r \in \{1,...,u\}}}{\Res} X \right) = \\
= \underset{y_t \oq^{2d} = y_s \oq^{-2}}{\Res} \left( \underset{\{z_{c_r} = y_r, ..., z_{c_{r+1}-1} = y_r \oq^{2(\lambda'_r - 1)} \}_{\forall r \in \{1,...,u\}}}{\Res} X \right)
\end{multline} 
where $\lambda'$ is obtained from $\lambda$ by replacing $\lambda_s$ and $\lambda_t$ by $\lambda_s-d-1$ and $\lambda_t+d+1$. The right-hand side of the expression above vanishes because $X \in \CA_{\leq \mu|\bd,k}^{\lambda}$ and $\lambda'>\lambda$. Similarly, one can show that the residue of $Y$ vanishes at $y_s \oq^{2d} - y_t \oq^{2\lambda_t}$ for any $s<t$ and $0 \leq d < \lambda_s$. This precisely implies that $Y$ is a Laurent polynomial:
\begin{equation}
\label{eqn:clog}
Y(y_1,...,y_u) = \sum_{1 \leq \alpha_1,\beta_1,...,\alpha_u,\beta_u \leq n}^{h_1,...,h_u \in \BZ}  \text{coefficient } \cdot y_1^{h_1}... y_u^{h_u} E_{\alpha_1\beta_1} \otimes ... \otimes E_{\alpha_u \beta_u}
\end{equation}
Since the matrices $R$ and $\tR$ have total degree 0, the horizontal degree of $Y$ is equal to that of $X$, namely $\bd$, so we conclude that the only summands with non-zero coefficient in \eqref{eqn:clog} satisfy:
$$
(h_1+...+h_u)\bde + \deg E_{\alpha_1 \beta_1} + ... + \deg E_{\alpha_u \beta_u} = \bd
$$
Finally, the slope condition on $X$ implies an analogous slope condition on $Y$: in each variable $y_s$, we have:
$$
n h_s + \alpha_s - \beta_s \leq \mu \lambda_s 
$$
Therefore, the number of coefficients that one gets to choose in \eqref{eqn:clog} is at most the number of collections $(j_s = \alpha_s + n h_s , i_s = \beta_s)$ satisfying the conditions in the right-hand side of \eqref{eqn:claim} (we only take quasi-ordered collections because of \eqref{eqn:symm y}).
	
\end{proof}

\subsection{} Definition \ref{def:shuf aff} of the shuffle algebra allows us to associate to any $X = X_{1...k}(z_1,...,z_k) \in \CA^+$ the matrix-valued power series $X^{(k)}(y) \in \End(V)[y^{\pm 1}]$ (namely the object arising in \eqref{eqn:wheel} for the composition $(k)$ of $k$). \\

\begin{proposition}
\label{prop:quasi}

For any $A \in \CA_{\bd,k}$ and $B \in \CA_{\be,l}$, we have:	
\begin{equation}
\label{eqn:quasi 1}
(A*B)^{(k+l)}(y) = A^{(k)}(y) B^{(l)}(y) \oq^{2k e_n}
\end{equation}
where $e_n$ is the last component of the vector $\be = (e_1,...,e_n) \in \zz$. \\
	
\end{proposition} 

\begin{proof} The proof is precisely the $u=1$ case of the proof of Proposition \ref{prop:wheel preserved}, since the equality of braids therein indicates the fact that:
\begin{equation}
\label{eqn:quasi 2}
(A*B)^{(k+l)}(y) = A^{(k)}(y) B^{(l)}(y \oq^{2k})
\end{equation}
Since $B \in \CA_{\be,l}$, formulas \eqref{eqn:horizontal} and \eqref{eqn:degree} imply the homogeneity property: 
$$
B(z_1\xi,...,z_k \xi) = \xi^{e_n} B(z_1,...,z_k)
$$
As $R$--matrices are invariant under rescaling variables, the rational function $B^{(l)}(y)$ of \eqref{eqn:wheel} also satisfies $B^{(l)}(y \xi) = \xi^{e_n} B^{(l)}(y)$. Then \eqref{eqn:quasi 2} implies \eqref{eqn:quasi 1}. 
	
\end{proof}	

\noindent For all $(i , j) \in \zzz$, define the linear maps:
\begin{equation}
\label{eqn:alpha}
\bigoplus_{k = 0}^\infty \CA_{[i;j),k} \stackrel{\alpha_{[i;j)}}\longrightarrow \fff
\end{equation}
$$
\alpha_{[i;j)} \left( X_{1...k}(z_1,...,z_k) \right) = \text{coefficient of } E_{ji} \text{ in } X^{(k)}(y) (1- q^2)^k \oq^{\frac {k(i-j) + (j-i) + k  - 2k\bari}n} 
$$
(recall that $E_{ij} = E_{\bari \barj} y^{\left \lfloor \frac {i-1}n \right \rfloor - \left \lfloor \frac {j-1}n \right \rfloor}$). \\

\begin{corollary}
\label{cor:quasi}
	
For any $k,l \in \BN$ and $(i,j) \in \zzz$, we have:
\begin{equation}
\label{eqn:quasi}
\alpha_{[i;j)}(A * B) = \alpha_{[s;j)}(A) \alpha_{[i;s)}(B) \cdot \oq^{\frac {k(s-i) - l(j-s)}n}
\end{equation}
whenever $\deg A = ([s;j),k)$ and $\deg B = ([i;s),l)$ for some $s$ between $i$ and $j$. If such an $s$ does not exist, then the RHS of \eqref{eqn:quasi} is set equal to 0, by convention. \\
	
\end{corollary} 

\begin{proof} The corollary is an immediate consequence of \eqref{eqn:quasi 1} and \eqref{eqn:alpha}. The only thing we need to check is that the power of $\oq$ is the same in the left as in the right-hand sides of \eqref{eqn:quasi}, which happens due to the elementary identity:
$$
\oq^{\frac {(k+l)(i-j)-2 (k+l) \bari}n} = \oq^{\frac {k(s-j)-2k \oo{s}}n} \oq^{\frac {l(i-s) -2 l \bari}n} \oq^{- 2k \left( \left \lfloor \frac {s-1}n \right \rfloor - \left \lfloor \frac {i-1}n \right \rfloor\right)} \oq^{\frac {k(s-i) - l(j-s)}n}
$$
since if $\be = [i;s)$, then $e_n = \left \lfloor \frac {s-1}n \right \rfloor - \left \lfloor \frac {i-1}n \right \rfloor$.	
	
\end{proof}

\subsection{} Let $\CB_{\mu|\bd} = \CA_{\leq \mu|\bd,\frac {|\bd|}{\mu}}$. Particular importance will be given to the subalgebra:
\begin{equation}
\label{eqn:pieces b}
\CB^+_{\mu} = \bigoplus_{\bd\in \zz}^{|\bd| \in \mu \BN} \CB_{\mu|\bd}
\end{equation}
We will call $\CB_\mu^+$ a \textbf{slope subalgebra}. As a consequence of Proposition \ref{prop:multi}, the leading order term $\Delta_\mu$ of \eqref{eqn:leading order} restricts to a coproduct on the extended subalgebra:
\begin{equation}
\label{eqn:b enhanced}
\tCB^+_\mu = \Big \langle \CB^+_\mu, \psi_s^{\pm 1} \Big \rangle_{s \in \{1,...,n\}} \Big /\text{relations \eqref{eqn:psi commute}, \eqref{eqn:psi x}} 
\end{equation}
\text{ } \\

\begin{lemma}
\label{lem:min}

If $X \in \CB_\mu^+$ is primitive with respect to the coproduct $\Delta_\mu$, and:
\begin{equation}
\label{eqn:luna}
\alpha_{[i;j)}(X) = 0 
\end{equation}
for all $(i,j) \in \zzz$ such that $\deg X \in [i;j) \times \BN$, then $X = 0$. \\

\end{lemma}

\begin{proof} The assumption $X \in \CB^+_\mu$ implies that $\deg X = (\bd,k)$ with:
\begin{equation}
\label{eqn:zoe}
|\bd| = \mu k
\end{equation}
As we observed in the proof of Lemma \ref{lem:magic}, it suffices to show that $\ph_\lambda(X) = 0$ for all partitions $\lambda$, which we will do in reverse dominance order of the partition $\lambda$. The base case is when $\lambda = (k)$, which is satisfied because $\ph_{(k)}(X) = 0$ is precisely the content of the assumption \eqref{eqn:luna}. For a general partition $\lambda \neq (k)$, we may invoke the induction hypothesis to conclude that $\ph_{\lambda'}(X) = 0$ for all partitions $\lambda' > \lambda$, and in this case $\ph_\lambda(X)$ takes the form of \eqref{eqn:clog}. However, the fact that $X$ is a primitive element requires every summand appearing in the RHS of \eqref{eqn:clog} to satisfy:
$$
nh_s + \alpha_s - \beta_s < \mu \lambda_s
$$
for all $1 \leq s \leq u$ (we have $u > 1$, since $\lambda \neq (k))$. However:
$$
\sum_{s=1}^u (n h_s + \alpha_s - \beta_s) = |\bd| \stackrel{\eqref{eqn:zoe}} = \mu k = \mu \sum_{s=1}^u \lambda_s 
$$
This yields a contradiction, hence $\ph_\lambda(X) = 0$, thus completing the induction step. 

\end{proof} 

\subsection{} We will now construct particular elements of $\CB_\mu^+$, which together with Lemma \ref{lem:magic} will yield a PBW basis of the shuffle algebra, leading to the proof of Theorem \ref{thm:main}. Consider the following notion of symmetrization, analogous to \eqref{eqn:symmetrization}:
\begin{equation}
\label{eqn:symmetrization aff}
\sym \ X = \sum_{\sigma \in S(k)} R_\sigma \cdot X_{\sigma(1)...\sigma(k)}(z_{\sigma(1)},...,z_{\sigma(k)}) \cdot R_\sigma^{-1} 
\end{equation}
where $R_\sigma$ is the product of $R_{ij} \left(\frac {z_i}{z_j} \right)$ associated to any braid lift of $\sigma$. For instance:
\begin{equation}
\label{eqn:big omega}
R_{\omega_k}(z_1,...,z_k) = \prod_{i=1}^{k-1} \prod_{j=i+1}^k R_{ij} \left( \frac {z_i}{z_j} \right) 
\end{equation}
lifts the longest element $\omega_k \in S(k)$. Consider the matrix-valued rational functions:
\begin{align} 
&Q(x) = q^{-1}\sum_{1 \leq i, j \leq n} \frac {(x\oq^2)^{\delta_{i < j}}}{1-x\oq^2} E_{ij} \otimes E_{ji} \label{eqn:q} \\
&\bQ(x) = - q \sum_{1 \leq i, j \leq n} \frac {(x\oq^2)^{\delta_{i \leq j}}}{1-x\oq^2} E_{ij} \otimes E_{ji} \label{eqn:qq}
\end{align}
which have, up to scalar, the same simple pole and residue as $\tR(x)$:
\begin{equation}
\label{eqn:residues match}
q \cdot \underset{x = \oq^{-2}}{\Res} Q(x) = - q^{-1} \cdot \underset{x = \oq^{-2}}{\Res} \bQ(x) = (q^{-1}-q)^{-1} \cdot \underset{x = \oq^{-2}}{\Res} \tR(x) = (12)
\end{equation}
(see formula \eqref{eqn:blob}). Moreover, it is easy to check the following identity:
\begin{equation}
\label{eqn:identity}
Q(x) + \bQ(x) = \tR(x) - \sum_{1 \leq i \neq j \leq n} E_{ii} \otimes E_{jj}
\end{equation}
In the following, we will invoke \eqref{eqn:notation} and \eqref{eqn:tensor convention} by writing:
$$
E_{bc}^{(a)} = 1 \otimes ... \otimes 1 \otimes E_{\barb \barc} z_a^{\left \lfloor \frac {b-1}n \right \rfloor - \left \lfloor \frac {c-1}n \right \rfloor} \otimes 1 \otimes ... \otimes 1
$$
for all $a \in \{1,...,k\}$, $b,c \in \BZ$ (in the RHS, the non-trivial term is on the $a$-th spot). \\

\begin{proposition}
\label{prop:root}
	
For any $(i,j) \in \zzz$ and $\mu \in \BQ$ such that $k = \frac {j-i}{\mu} \in \BN$, set:
\begin{equation}
\label{eqn:def p}
F_{[i;j)}^\mu = F_{[i;j)}^{(k)} := \esym \ R_{\omega_k}(z_1,...,z_k)  
\end{equation}
$$
\prod_{a=1}^k \left[ \tR_{1a} \left( \frac {z_1}{z_a} \right) ... \tR_{a-2,a} \left( \frac {z_{a-2}}{z_a} \right) Q_{a-1,a} \left( \frac {z_{a-1}}{z_a} \right) E^{(a)}_{s_{a-1} s_a} \oq^{\frac {2\oo{s_a}}n} \right] 
$$
\begin{equation}
\label{eqn:def pp}
\bF_{[i;j)}^{\mu} = \bF_{[i;j)}^{(k)} := (-q^2\oq^{\frac 2n})^{-k} \cdot \esym \ R_{\omega_k}(z_1,...,z_k) 
\end{equation}
$$
\prod_{a=1}^k \left[ \tR_{1a} \left( \frac {z_1}{z_a} \right) ... \tR_{a-2,a} \left( \frac {z_{a-2}}{z_a} \right) \bQ_{a-1,a} \left( \frac {z_{a-1}}{z_a} \right) E^{(a)}_{s_{a-1}' s_a'} \oq^{\frac {2\oo{s_a'}}n} \right] 
$$
where $s_a = j - \lceil \mu a \rceil$, $s'_a = j - \lfloor \mu a \rfloor$. Then $F_{[i;j)}^\mu, \bF_{[i;j)}^\mu \in \CB_\mu^+$, and:
\begin{align}
&\Delta_\mu \left( F_{[i;j)}^\mu \right) = \sum_{s \in \{i,...,j\}} F_{[s;j)}^\mu \frac {\psi_i}{\psi_s} \otimes F_{[i; s)}^\mu \label{eqn:coproduct p} \\
&\Delta_\mu \left( \bF_{[i;j)}^\mu \right) = \sum_{s \in \{i,...,j\}} \frac {\psi_s}{\psi_j} \bF_{[i;s)}^\mu  \otimes \bF_{[s; j)}^\mu \label{eqn:coproduct pp}
\end{align}
where we set $F_{[i;j)}^\mu = \bF_{[i;j)}^\mu = 0$ if $\frac {j-i}{\mu} \notin \BN$. \\
	
\end{proposition}

\begin{proof} Note that if $Q$, $\bQ$ were replaced by $\tR$ in formulas \eqref{eqn:def p}, \eqref{eqn:def pp}, then the right-hand sides of the aforementioned formulas would precisely equal:
\begin{equation}
\label{eqn:they}
E_{s_0s_1} \oq^{\frac {2\oo{s_1}}n} * ... * E_{s_{k-1}s_k} \oq^{\frac {2\oo{s_k}}n} \quad \text{and} \quad E_{s'_0 s'_1} \oq^{\frac {2\oo{s'_1}}n} * ... * E_{s'_{k-1} s'_k} \oq^{\frac {2\oo{s'_k}}n}
\end{equation}
The fact that the shuffle elements \eqref{eqn:they} satisfy the wheel conditions is simply a consequence of iterating Proposition \ref{prop:wheel preserved} a number of $k-1$ times. As far as the elements $F_{[i;j)}^\mu$, $\bF_{[i;j)}^\mu$ are concerned, the fact that they satisfy the wheel conditions is proved similarly with the fact that \eqref{eqn:they} satisfies the wheel conditions: this is because Proposition \ref{prop:wheel preserved} uses the fact that $\Res_{x = \oq^{-2}} \tR(x)$ is a multiple of the permutation matrix, and we have already seen in \eqref{eqn:residues match} that the residues of $Q(x)$, $\bQ(x)$, $\tR(x)$ are the same up to scalar. We leave the details to the interested reader. \\

\noindent Let us prove that the shuffle elements \eqref{eqn:def p} and \eqref{eqn:def pp} lie in $\CB_\mu^+$. We will only prove the former case, since the latter case is analogous. We have:
\begin{equation}
\label{eqn:symm f}
F := F_{[i;j)}^{\mu} = \sym \ R_{\omega_k}(z_1,...,z_k) X_{1...k} (z_1,...,z_k)
\end{equation}
where $X$ is the expression on the second line of \eqref{eqn:def p}. Since $F$ has degree $([i;j),k)$ with $j-i = \mu k$, it remains to prove that $F$ has slope $\leq \mu$. To this end, for any $l \in \{0,...,k\}$ we need to look at the first $l$ tensor factors of $\sym \ R_{\omega_k} X$ and isolate the terms of minimal $|\text{hdeg}|$. If $Y$ is a $k$--tensor, we will henceforth use the phrase ``initial degree of $Y$" instead of ``total $|\text{hdeg}|$ of the first $l$ factors of $Y$". Because:
\begin{align} 
&\lim_{x \rightarrow 0} R_{ab}(x) = \sum_{i,j=1}^n q^{\delta_i^j} E^{(a)}_{ii} \otimes E^{(b)}_{jj} + \sum_{i > j} (q-q^{-1}) E^{(a)}_{ij} \otimes E^{(b)}_{ji} \label{eqn:limit 0} \\
&\lim_{x \rightarrow \infty} R_{ab}(x) = \sum_{i,j=1}^n q^{-\delta_i^j} E^{(a)}_{ii} \otimes E^{(b)}_{jj} - \sum_{i < j} (q-q^{-1}) E^{(a)}_{ij} \otimes E^{(b)}_{ji} \label{eqn:limit infty}
\end{align}
for all indices $a$ and $b$, we obtain the following easy (but very useful) fact: \\

\begin{claim}
\label{claim:useful}

For any $1 \leq a \neq b \leq k$, multiplying a $k$--tensor $Y$ by: 
$$
R_{ab} \left(\frac {z_a}{z_b} \right) \quad \text{or} \quad \tR_{ab} \left(\frac {z_a}{z_b} \right) \quad \text{or} \quad Q_{ab} \left(\frac {z_a}{z_b} \right) \quad \text{or} \quad \bQ_{ab} \left(\frac {z_a}{z_b} \right)
$$
(either on the left or on the right) cannot decrease the minimal initial degree of $Y$. \\

\end{claim} 

\noindent Therefore, it suffices to compute the minimal initial degree of:
\begin{equation}
\label{eqn:sigma x}
X_{\sigma(1)...\sigma(k)}(z_{\sigma(1)},...,z_{\sigma(k)}) = 
\end{equation}
$$
\prod_{a=1}^k \left[ \tR_{\sigma(1) \sigma(a)} \left( \frac {z_{\sigma(1)}}{z_{\sigma(a)}} \right) ... \tR_{\sigma(a-1),\sigma(a)} \left( \frac {z_{\sigma(a-1)}}{z_{\sigma(a)}} \right) Q_{\sigma(a-1),\sigma(a)} \left( \frac {z_{\sigma(a-1)}}{z_{\sigma(a)}} \right) E^{(\sigma(a))}_{s_{a-1} s_{a}} \oq^{\frac {2\oo{s_a}}n} \right]
$$
for any permutation $\sigma$ of $\{1,...,k\}$. Claim \ref{claim:useful} implies that the minimal initial degree (henceforth denoted ``m.i.d.") comes from the various $E_{s_{a-1}s_a}$ factors:
\begin{equation}
\label{eqn:usul}
\text{\mtd of \eqref{eqn:sigma x}} = \sum_{a \in A} \left(s_{a-1} - s_a \right) + \#
\end{equation}
where $A = \{\sigma^{-1}(1),...,\sigma^{-1}(l)\}$ and the number $\#$ counts those $a \in A$ such that $a - 1 \notin A$. This number appears in the right-hand side of \eqref{eqn:usul} because $Q(\infty)$ has 0 on the diagonal, by definition. It is elementary to show that:
\begin{equation}
\label{eqn:muad}
\text{RHS of \eqref{eqn:usul}} = \sum_{a \in A} \left( \lceil \mu a \rceil - \lceil \mu ( a - 1) \rceil \right) + \# \geq \sum_{s=1}^t \left( \lceil \mu \beta_s \rceil - \lfloor \mu \alpha_s \rfloor \right)
\end{equation}
if $A$ splits up into consecutive blocks of integers:
\begin{equation}
\label{eqn:set of variables}
A = \{\alpha_1+1,...,\beta_1\} \sqcup \{ \alpha_2+1,...,\beta_2 \} \sqcup ... \sqcup \{\alpha_t+1,...,\beta_t \}
\end{equation}
with $0 \leq \alpha_1$, while $\beta_s < \alpha_{s+1}$ for all $s$, and $\beta_t \leq k$. Since:
\begin{equation}
\label{eqn:dib}
\text{RHS of \eqref{eqn:muad}} \geq \mu\sum_{s=1}^t (\beta_s-\alpha_s) = \mu \cdot \# A = \mu l
\end{equation}
we conclude that $F \in \CA^+_{\leq \mu}$. Because $|\hdeg F| = j-i$ and $\vdeg F = k$, then $F \in \CB^+_\mu$. \\

\noindent Moreover, the terms of minimal initial degree in $F$ correspond to those situations where we have equality in all the inequalities above, and these require $\mu l \in \BZ$ and:
$$
A = \{1,...,l\}
$$
Let us now compute the summands which achieve the minimal initial degree in:
\begin{equation}
\label{eqn:x}
R_{\omega_k} X = R_{\omega_{l}}(z_1,...,z_{l}) \squiggly{\left[ R_{1,l+1} \left( \frac {z_1}{z_{l+1}} \right) ... R_{l, k} \left( \frac {z_{l}}{z_k} \right) \right]} 
\end{equation}
$$
R_{\omega_{k-l}} (z_{l+1},...,z_k) E^{(1...l)}_{s_0|...|s_{l}} \squiggly{\left[ \tR_{1,l+1} \left( \frac {z_1}{z_{l+1}} \right) ... Q_{l,l+1} \left( \frac {z_{l}}{z_{l+1}} \right)  ... \tR_{l, k} \left( \frac {z_{l}}{z_k} \right) \right]} E^{(l+1...k)}_{s_{l}| ... | s_{k}}
$$
where the notation $E^{(u...v)}_{c_{u-1}|c_u|...|c_{v-1}|c_v}$ is shorthand for:
\begin{equation}
\label{eqn:shorthand}
\prod_{a=u}^{v} \left[\tR_{ua} \left( \frac {z_{u}}{z_a} \right) ... \tR_{a-2,a} \left( \frac {z_{a-2}}{z_a} \right) Q_{a-1,a} \left( \frac {z_{a-1}}{z_a} \right) E^{(a)}_{c_{a-1} c_a} \oq^{\frac {2\oo{c_a}}n} \right] 
\end{equation}
If the terms with the squiggly red underline were not present in \eqref{eqn:x}, then we would conclude that the terms of minimal initial degree would be precisely:
\begin{equation}
\label{eqn:precisely}
\mtd R_{\omega_k} X = R_{\omega_{l}} E^{(1...l)}_{s_0|...|s_{l}} \otimes R_{\omega_{k-l}} E^{(l+1...k)}_{s_{l}| ... | s_{k}}
\end{equation}
and upon symmetrization, this would almost imply \eqref{eqn:coproduct p}. However, we must deal with the contribution of the terms with the squiggly red underline. As we have seen in the discussion above (specifically \eqref{eqn:limit 0} and \eqref{eqn:limit infty}), these factors only contribute a diagonal matrix to the terms of minimal initial degree. Specifically, if:
\begin{align*}
&E^{(1...l)}_{s_0|...|s_{l}} = \sum_{x_a,y_a =1}^n \text{coefficient} \cdot E_{x_1y_1} \otimes ... \otimes E_{x_{l} y_{l}} \otimes 1^{\otimes k-l} \\
&E^{(l+1...k)}_{s_{l}| ... | s_{k}} = \sum_{x_a,y_a =1}^n \text{coefficient} \cdot 1^{\otimes l} \otimes  E_{x_{l+1}y_{l+1}} \otimes ... \otimes E_{x_k y_k} 
\end{align*}
(above, $y_{l} = s_{l} = x_{l+1}$ in all summands with non-zero coefficient) then the terms of minimal initial degree in $R_{\omega_k} X$ yield the following value for the coproduct \eqref{eqn:leading order}:
$$
\Delta_\mu(R_{\omega_k} X) = \sum_{x_a,y_a =1}^n \psi_{x_{l+1}}^{-1} ... \psi_{x_k}^{-1} \underbrace{q^{\sum^{1 \leq a \leq l}_{l < b \leq k} \delta_{x_a}^{x_b}}}_{\text{first squiggle}} R_{\omega_{l}}(E_{x_1y_1} \otimes ... \otimes E_{x_{l} y_{l}} \otimes 1^{\otimes k- l}) 
$$
$$ 
\psi_{y_{l+1}} ... \psi_{y_k} \underbrace{q^{-\sum^{1 \leq a \leq l}_{l < b \leq k} \delta_{y_a}^{x_b}}}_{\text{second squiggle}} R_{\omega_{k-l}} (1^{\otimes l} \otimes  E_{x_{l+1}y_{l+1}} \otimes ... \otimes E_{x_k y_k}) \cdot \text{coefficient}
$$
where the various $\psi_a^{\pm 1}$ factors arise from the diagonal terms of the series $S(x)$, $T(x)$ (see \eqref{eqn:cop shuf aff 3}, \eqref{eqn:leading order}). Using \eqref{eqn:psi x}, we may move the product of $\psi$'s on the left to join the product of $\psi$'s in the middle, at the cost of cancelling the powers of $q$:
\begin{equation}
\label{eqn:bor}
\Delta_\mu(R_{\omega_k} X) = \sum_{x_a,y_a =1}^n \text{coefficient } \cdot
\end{equation}
$$ 
R_{\omega_{l}}(E_{x_1y_1} \otimes ... \otimes E_{x_{l} y_{l}} \otimes 1^{\otimes k- l})  \frac {\psi_{y_{k-l+1}}... \psi_{y_k}}{\psi_{x_{k-l+1}}... \psi_{x_k}} R_{\omega_{k-l}} (1^{\otimes l} \otimes  E_{x_{l+1}y_{l+1}} \otimes ... \otimes E_{x_k y_k})
$$
As a consequence of the following straightforward claim: \\

\begin{claim}
\label{claim:tensors}

The quantity \eqref{eqn:shorthand} is a sum of tensors $E_{x_uy_u} \otimes ... \otimes E_{x_v y_v}$ where:
$$
\# \{x_u \equiv r \text{ mod }n\} - \# \{y_u \equiv r \text{ mod }n\} = \delta_{c_{u-1}}^r - \delta_{c_v}^r
$$
for any $r \in \BZ/n\BZ$. \\

\end{claim}

\noindent we may rewrite \eqref{eqn:bor} as:
$$
\Delta_\mu(R_{\omega_k} X) = R_{\omega_l} E^{(1...l)}_{s_0|...|s_{l}} \frac {\psi_{s_k}}{\psi_{s_l}} \otimes R_{\omega_{k-l}} E^{(l+1...k)}_{s_{l}| ... | s_{k}}
$$
Upon symmetrization with respect to those permutations $\sigma \in S(k)$ which preserve the set $\{1,...,l\}$, this yields precisely \eqref{eqn:coproduct p}. 
	
\end{proof}

\subsection{} According to Lemma \ref{lem:min}, the elements $F_{[i;j)}^\mu, \bF_{[i;j)}^\mu \in \CB_\mu^+$ are completely determined by the coproduct relations \eqref{eqn:coproduct p} and \eqref{eqn:coproduct pp}, together with their value under the linear functionals \eqref{eqn:alpha}. Let us therefore compute the latter: \\

\begin{proposition}
\label{prop:alpha f}
	
For any $(i,j) \in \zzz$ and $\mu \in \BQ$ such that $\frac {j-i}{\mu} \in \BN$, we have:
\begin{align}
&\alpha_{[u;v)} \left( F_{[i;j)}^\mu \right) = \delta_{(u,v)}^{(i,j)} (1 - q^2) \oq^{\frac {\gcd(k,j-i)}n}\label{eqn:alpha p} \\
&\alpha_{[u;v)} \left( \bF_{[i;j)}^\mu \right) = \delta_{(u,v)}^{(i,j)} (1 - q^{-2}) \oq^{- \frac {\gcd(k,j-i)}n} \label{eqn:alpha pp}
\end{align}
for any $(u,v) \in \zzz$ such that $[u;v)  = [i;j)$. \\
\end{proposition}

\begin{proof} Recall that $F = F_{[i;j)}^\mu$ is given by the symmetrization \eqref{eqn:symm f}, namely: 
$$
F = \sum_{\sigma \in S(k)} R_\sigma \cdot \sigma \Big( R_{\omega_k} \cdot \text{second line of \eqref{eqn:def p}} \Big) \sigma^{-1} \cdot R_\sigma^{-1}
$$	
where $R_\sigma$ is an arbitrary braid which lifts the permutation $\sigma$. Of the $k!$ summands in the right-hand side, only the one corresponding to the identity permutation is involved in the iterated residue of $F$ at $z_k = z_{k-1} \oq^2$,..., $z_2 = z_1 \oq^2$, hence we obtain:
\begin{equation}
\label{eqn:underset}
\underset{\{z_{1} = y, z_2 = y \oq^2, ..., z_k = y \oq^{2k-2} \}}{\Res} F = R_{\omega_k}(y,y\oq^2,...,y\oq^{2k-2}) 
\end{equation}
$$
\prod_{a=1}^k \left[ \tR_{1a} \left( \oq^{2-2a} \right) ... \tR_{a-2,a} \left( \oq^{-4} \right) \cdot \left( \underset{x = \oq^{-2}}{\Res} Q_{a-1,a} \left( x \right) \right) \cdot E^{(a)}_{s_{a-1} s_a} \oq^{\frac {2\oo{s_a}}n} \Big|_{z_a \mapsto y \oq^{2a-2}}\right] 
$$
(in $E_{bc}^{(a)} = E_{\barb \barc}^{(a)} z_a^{\left \lfloor \frac {b-1}n \right \rfloor - \left \lfloor \frac {c-1}n \right \rfloor}$, we must specialize $z_a = y \oq^{2a-2}$). Note that:
$$
R_{\omega_k}(z_1,...,z_k) = \prod_{a=1}^{k-1} \prod_{b=a+1}^k R_{ab} \left(\frac {z_a}{z_b} \right) = R_{12} \left(\frac {z_1}{z_2} \right) ... R_{1k} \left(\frac {z_1}{z_k} \right) \prod_{b=k}^{3} \prod_{a=b-1}^2 R_{ab} \left(\frac {z_a}{z_b} \right)
$$
due to \eqref{eqn:ybe aff}. Since $\tR$ is given by \eqref{eqn:tr} and $\underset{x=\oq^{-2}}{\Res} Q(x) = q^{-1} \cdot (12)$, we have:
\begin{multline*}
\text{LHS of \eqref{eqn:underset}} = q^{1-k} R_{12} \left(\oq^{-2} \right) ... R_{1k} \left(\oq^{2-2k} \right) \prod_{b=k}^{3} \prod_{a=b-1}^2 R_{ab} \left( \oq^{2a-2b} \right) \\ \prod_{b=1}^k \left[ \prod_{a=1}^{b-2} R_{ba} \left(\oq^{2b-2a-2}\right) \cdot (b-1,b) \cdot E^{(b)}_{s_{b-1} s_b} \oq^{\frac {2\oo{s_b}}n} \Big|_{z_b \mapsto y \oq^{2b-2}}\right] 
\end{multline*}
If we move the permutations $(b-1,b)$ all the way to the right, then we obtain:
\begin{multline*}
\text{LHS of \eqref{eqn:underset}} = q^{1-k} R_{12} \left(\oq^{-2} \right) ... R_{1k} \left(\oq^{2-2k} \right) \prod_{b=k}^{3} \prod_{a=b-1}^2 R_{ab} \left( \oq^{2a-2b} \right) \\ \prod_{b=3}^k \left[ \prod_{a=1}^{b-2} R_{b,a+1} \left(\oq^{2b-2a-2}\right)\right] \cdot \prod_{b=1}^k  E^{(1)}_{s_{b-1} s_b} \oq^{\frac {2\oo{s_b}}n} \Big|_{z_b \mapsto y \oq^{2b-2}} \cdot \begin{pmatrix}
1 & ... & k \\ 2 & ... & 1 \end{pmatrix} \stackrel{\eqref{eqn:unitary r}}= 
\end{multline*}
\begin{multline*} 
= q^{1-k} \prod_{2 \leq a < b \leq k} f \left(\oq^{2a-2b} \right) \cdot R_{12} \left(\oq^{-2} \right) ... R_{1k} \left(\oq^{2-2k} \right) \cdot \\ \cdot (E_{ji} \otimes 1 \otimes ... \otimes 1) \oq^{\sum_{b=1}^k (2b-2) \left( \left \lfloor \frac {s_{b-1}-1}n \right \rfloor - \left\lfloor \frac {s_b-1}n \right \rfloor \right) + \frac {2\oo{s_b}}n}  \cdot \begin{pmatrix}
1 & ... & k \\ 2 & ... & 1 \end{pmatrix}
\end{multline*} 
With this in mind, \eqref{eqn:wheel} implies that:
$$
F^{(k)} = E_{ji} \cdot \frac {q^{1-k}\oq^{2\sum_{b=1}^k \left( \left \lfloor \frac {s_b-1}n \right \rfloor - \left \lfloor \frac {s_k-1}n \right \rfloor + \frac {\oo{s_b}}n \right)}}{(q^{-1} - q)^{k-1}} = E_{ji} \cdot \frac {\oq^{\sum_{b=1}^k \frac {2s_b}n - 2k \left \lfloor \frac {i-1}n \right \rfloor}}{(1 - q^2)^{k-1}}
$$
Given formula \eqref{eqn:alpha} and the elementary identity:
$$
\sum_{a=1}^k \left\lceil \frac {ad}k \right \rceil  =  \frac {dk + d + k - \gcd(d,k)}2 
$$
we conclude \eqref{eqn:alpha p}. Formula \eqref{eqn:alpha pp} is proved analogously. 
	
\end{proof}

\subsection{} Let $a\in \BZ$ and $b \in \BN$ be coprime, set $g = \gcd(n,a)$ and:
\begin{equation}
\label{eqn:mu}
\mu = \frac ab
\end{equation}
Recall the discussion in Subsection \ref{sub:slope stuff}, in which the algebra:
\begin{equation}
\label{eqn:cool}
\CE_\mu = U_q(\dot{\fgl}_{\frac ng})^{\otimes g}
\end{equation}
was graded by $\zz \times \BZ$, and its root generators were indexed as:
\begin{equation}
\label{eqn:rooty}
f_{[i;j)}^\mu, \quad \forall (i,j) \in \zzz \ \text{ such that } \ \frac {j-i}{\mu} \in \BN
\end{equation}
Comparing \eqref{eqn:cop 1} and \eqref{eqn:alpha affine} with \eqref{eqn:coproduct p} and \eqref{eqn:alpha p}, respectively, Lemma \ref{lem:any} implies that there exists a bialgebra homomorphism:
\begin{equation}
\label{eqn:pink}
\CE_\mu^+ \stackrel{\Upsilon_\mu}\longrightarrow \CB_\mu^+, \qquad f_{[i;j)}^{\mu} \leadsto F_{[i;j)}^{\mu}
\end{equation}
Lemma \ref{lem:magic} for $|\bd| = k \mu$ implies that:
$$
\dim \CB_{\mu|\bd} \leq \# \left\{ \begin{array}{cc} \text{partitions } \bd = [i_1;j_1)+...+[i_u;j_u) \\ \text{s.t. } j_s-i_s \in \mu \BN, \ \forall s \in \{1,...,u\} \end{array} \right\}
$$
The right-hand side above is precisely equal to the dimension of the algebra $\CE^+_\mu$ in degree $\bd$, so Corollary \ref{cor:any} implies: \\

\begin{proposition}
\label{prop:iso small}

For any $\mu \in \BQ$, the map $\Upsilon_\mu : \CE_\mu^+ \rightarrow \CB_\mu^+$ is an isomorphism. \\

\end{proposition}

\subsection{} Since $\CB_\mu^+$ is isomorphic to the algebra $\CE_\mu^+$, we may present it instead in terms of simple and imaginary generators (see Subsection \ref{sub:pbw stuff} for the notation):
\begin{align}
&P_{[i;j)}^{\mu} = \Upsilon_\mu \left(p_{[i;j)}^{\mu} \right) \label{eqn:simple} \\
&P_{l\bde, r}^\mu = \Upsilon_\mu \left(p_{l\bde, r}^\mu \right) \label{eqn:imaginary}
\end{align} 
Since $\Upsilon_\mu$ preserves the maps $\alpha_{[u;v)}$, we have:
\begin{align*} 
&\alpha_{[u;v)} \left( P_{[i;j)}^{\mu} \right) = \delta_{(u,v)}^{(i,j)} \\
&\alpha_{[u;u+nl)} \left( P_{l \bde, r}^{\mu} \right) = \delta_{u \text{ mod }g}^r
\end{align*} 
for all $(u,v) \in \zzz$. By \eqref{eqn:alpha p} and \eqref{eqn:alpha pp}, the simple generators are given by:
\begin{equation}
\label{eqn:explicit simple}
P_{[i;j)}^{\mu} = \frac {F_{[i;j)}^{\mu}}{\oq^{\frac 1n}(1 - q^2)} = \frac {\bF_{[i;j)}^{\mu}}{\oq^{-\frac 1n}(1 - q^{-2})}\end{equation}
if $\gcd(j-i,\mu(j-i)) = 1$, but we do not know a closed formula for the imaginary generators \eqref{eqn:imaginary}. We will sometimes use the notation:
$$
P_{[i;j)}^{(k)} = P_{[i;j)}^{\mu}, \qquad P_{l\bde, r}^{(k')} = P_{l\bde, r}^\mu
$$
if $j-i = \mu k$,  $nl = \mu k'$, in order to emphasize the fact that $\deg P_{\bd}^{(k)} = (\bd,k)$. Set:
\begin{equation}
\label{eqn:strange}
P_{[i;j)}^{(k)} = P_{l\bde, \bari}^{(k)}
\end{equation}
if $j = i + nl$ and $\gcd(k,j-i)=1$. The reason for this (rather unusual) choice is to unify notation in what follows. Indeed, under the aforementioned assumptions, the numbers $a = j-i$ and $b = k$ are coprime. Since $n|a$, we have $g = \gcd(n,a) = n$ and so the corresponding slope subalgebra $\CB^+_\mu$ is isomorphic to $\uuip^{\otimes n}$. This slope algebra does not have simple generators (only imaginary ones), so in \eqref{eqn:strange} we are effectively relabeling its first imaginary generator as if it were a simple generator. \\

\begin{theorem}
\label{thm:iso}

We have an algebra isomorphism:
\begin{equation}
\label{eqn:algebra iso}
\DD^+ \stackrel{\Upsilon^+}\cong \CA^+, \qquad p_{[i;j)}^{(k)} \mapsto P_{[i;j)}^{(k)}, \qquad p_{l\bde, r}^{(k')} \mapsto P_{l\bde, r}^{(k')}
\end{equation}
where $\DD^+$ is the explicit algebra of Definition \ref{def:explicit d}. \\

\end{theorem}

\subsection{} To prove Theorem \ref{thm:iso}, we need to show that the simple and imaginary generators of $\CB_\mu^+$ satisfy the analogues of relations \eqref{eqn:rel 2 pbw} and \eqref{eqn:rel 3 pbw}. We recall the discussion of hinges from Subsection \ref{sub:hinge}, and the fact (Definition \ref{def:leading order}) that a shuffle element $X$ having slope $\leq \mu$ means that all the hinges of the terms in $\Delta(X)$ are situated left of the ray of slope $\mu$ from the origin (``left" refers to the usual direction in the lattice plane). We will also say that a point is ``strictly left" of the ray of slope $\mu$ to indicate that the point in question is not allowed to lie on the ray. \\

\begin{proposition}
\label{prop:coproduct}
	
Assume $\gcd(j-i,k) = 1$, and consider the lattice triangle $T$:
	
\begin{picture}(110,140)(-50,-20)
	
\put(0,20){\circle*{2}}\put(20,20){\circle*{2}}\put(40,20){\circle*{2}}\put(60,20){\circle*{2}}\put(0,40){\circle*{2}}\put(20,40){\circle*{2}}\put(40,40){\circle*{2}}\put(60,40){\circle*{2}}\put(0,60){\circle*{2}}\put(20,60){\circle*{2}}\put(40,60){\circle*{2}}\put(60,60){\circle*{2}}\put(0,80){\circle*{2}}\put(20,80){\circle*{2}}\put(40,80){\circle*{2}}\put(60,80){\circle*{2}}\put(0,100){\circle*{2}}\put(20,100){\circle*{2}}\put(40,100){\circle*{2}}\put(60,100){\circle*{2}}
	
\put(40,20){\line(-1,1){20}}
\put(20,40){\line(0,1){60}}
\put(40,20){\line(-1,4){20}}
	
\put(30,10){\scriptsize{$(0,0)$}}
\put(5,105){\scriptsize{$(j-i,k)$}}
\put(13,70){\scriptsize{$\mu$}}
\put(23,50){\scriptsize{$T$}}
	
\put(40,-10){\mbox{\emph{F{\scriptsize IGURE} 22}. Two types of lattice triangles}}
	
\put(115,65){or}
	
\put(180,20){\circle*{2}}\put(200,20){\circle*{2}}\put(220,20){\circle*{2}}\put(240,20){\circle*{2}}\put(180,40){\circle*{2}}\put(200,40){\circle*{2}}\put(220,40){\circle*{2}}\put(240,40){\circle*{2}}\put(180,60){\circle*{2}}\put(200,60){\circle*{2}}\put(220,60){\circle*{2}}\put(240,60){\circle*{2}}\put(180,80){\circle*{2}}\put(200,80){\circle*{2}}\put(220,80){\circle*{2}}\put(240,80){\circle*{2}}\put(180,100){\circle*{2}}\put(200,100){\circle*{2}}\put(220,100){\circle*{2}}\put(240,100){\circle*{2}}
	
\put(200,20){\line(1,4){20}}
\put(200,20){\line(0,1){60}}
\put(200,80){\line(1,1){20}}
	
\put(190,10){\scriptsize{$(0,0)$}}
\put(205,105){\scriptsize{$(j-i,k)$}}	
\put(191,48){\scriptsize{$\mu$}}
\put(203,60){\scriptsize{$T$}}
	
\end{picture}
	
\noindent uniquely determined as the triangle of maximal area situated completely to the left of the vector $(j-i,k)$, which does not contain any lattice points inside. Let $\mu$ denote the slope of one of the edges of $T$, as indicated in the pictures above. Then:
$$
\Delta \left(P_{[i;j)}^{(k)} \right) = P_{[i;j)}^{(k)} \otimes 1 + \frac {\psi_i}{\psi_j} \otimes P_{[i;j)}^{(k)} + \Big( \text{tensors with hinge strictly left of } T \Big) +
$$
\begin{equation}
\label{eqn:big cop}
+ \sum_{i \leq u < v \leq j} \begin{cases} F^\mu_{[v;j)} \frac {\psi_u}{\psi_v} \bF^\mu_{[i;u)} \otimes P_{[u;v)}^{(\bullet)}  &\text{for the picture on the left} \\ \\ \frac {\psi_v}{\psi_j} P_{[u;v)}^{(\bullet)} \frac {\psi_i}{\psi_u} \otimes \bF^\mu_{[v;j)} F^\mu_{[i;u)} &\text{for the picture on the right} \end{cases} 
\end{equation}
where $\bullet = k- \frac {j-v+u-i}{\mu}$. \\

\end{proposition}

\begin{proof} Due to \eqref{eqn:explicit simple}, we may replace all $P$'s by $F$'s in formula \eqref{eqn:big cop}. We will refine the proof of Proposition \ref{prop:root}, so we will freely adapt the notations therein: let $F$ and $X$ be given by \eqref{eqn:symm f}. The summands of $\Delta(F)$ with leftmost possible hinge correspond to those tensors of minimal initial degree, which is:
\begin{equation}
\label{eqn:evg}
\text{minimal initial degree of } F \geq x := \sum_{s=1}^t \left(\left \lceil \frac {(j-i) \beta_s}k \right \rceil - \left \lfloor \frac {(j-i) \alpha_s}k \right \rfloor \right)
\end{equation}
(the notation $\alpha_s,\beta_s$ is as in relation  \eqref{eqn:muad}). If we let $y := \sum_{s=1}^t (\beta_s - \alpha_s) = \# A$, where $A = \{\sigma^{-1}(1),...,\sigma^{-1}(l)\}$, then it is easy to see that the lattice point $(x,y)$ lies on the boundary or to the left of the lattice triangle $T$. This implies that all hinges of summands of $\Delta(F)$ lie on the boundary or to the left of the triangle $T$. The boundary cases correspond to equality in \eqref{eqn:evg}, and they explicitly are: \\

\begin{itemize}[leftmargin=*]

\item for the picture on the left in Figure 22: $A = \{1,...,\beta\} \sqcup \{\alpha+1,...,k\}$ where $\beta < \alpha$ have the property that $\mu \alpha, \mu \beta \in \BZ$ \\	
	
\item for the picture on the right in Figure 22: $A = \{\alpha+1,...,\beta\}$ where $\alpha < \beta$ have the property that $\mu \alpha, \mu \beta \in \BZ$ \\

\end{itemize}

\noindent We will only treat the situation in the first bullet (i.e. in the picture on the left), because that of the second bullet is analogous and will not be used in the present paper. Recall that $\mtd Y$ stands for the terms of minimal initial degree of a $k$--tensor $Y$, i.e. the smallest value of the total $|\text{hdeg}|$ of the first $l$ tensor factors of $Y$ (for fixed $l \leq k$). We will now show that the terms of minimal initial degree give rise precisely to the tensors on the second line of \eqref{eqn:big cop}. Explicitly, we have:
$$
R_{\omega_k} X = R_{\omega_\beta}(z_1,...,z_{\beta}) \squiggly{\left[ R_{1,\beta+1} \left( \frac {z_1}{z_{\beta+1}} \right) ... R_{\beta, \alpha} \left( \frac {z_{\beta}}{z_{\alpha}} \right) \right]} R_{\omega_{\alpha-\beta}}(z_{\beta+1},...,z_{\alpha})
$$
$$
\left[ R_{1,\alpha+1} \left( \frac {z_{1}}{z_{\alpha+1}} \right) ... R_{\beta, k} \left( \frac {z_{\beta}}{z_{k}} \right) \right] \squiggly{\left[ R_{\beta+1,\alpha+1} \left( \frac {z_{\beta+1}}{z_{\alpha+1}} \right) ... R_{\alpha, k} \left( \frac {z_{\alpha}}{z_{k}} \right) \right]} R_{\omega_{k-\alpha}}(z_{\alpha+1},...,z_k)
$$
$$
E^{(1...\beta)}_{s_0 | ... | s_{\beta}} \squiggly{\left[ \tR_{1,\beta+1} \left( \frac {z_1}{z_{\beta+1}} \right) ... Q_{\beta,\beta+1} \left( \frac {z_\beta}{z_{\beta+1}} \right) ... \tR_{\beta, \alpha} \left( \frac {z_\beta}{z_\alpha} \right) \right]} E^{(\beta+1...\alpha)}_{s_\beta | ... | s_{\alpha}}  \left[ \tR_{1,\alpha+1} \left( \frac {z_1}{z_{\alpha+1}} \right)  \right.
$$
$$
\left. ... \tR_{\beta, k} \left( \frac {z_\beta}{z_k} \right) \right] \squiggly{\left[ \tR_{\beta+1,\alpha+1} \left( \frac {z_{\beta+1}}{z_{\alpha+1}} \right) ... Q_{\alpha,\alpha+1} \left( \frac {z_\alpha}{z_{\alpha+1}} \right) ... \tR_{\alpha, k} \left( \frac {z_\alpha}{z_k} \right) \right]} E^{(\alpha+1...k)}_{s_\alpha| ... | s_{k}} 
$$	
where the symbol $E^{(u...v)}_{c_{u-1} | ... | c_v}$ is defined in \eqref{eqn:shorthand}. As we have seen in the latter part of Proposition \ref{prop:root}, the terms with the squiggly red underline contribute certain powers of $q$ to the minimal initial degree of $R_{\omega_k} X$. Specifically, let:
\begin{align*}
&E^{(1...\beta)}_{s_0 | ... | s_{\beta}} = \sum_{x_a,y_a =1}^n \text{coefficient} \cdot E_{x_1y_1} \otimes ... \otimes E_{x_{\beta} y_{\beta}} \otimes 1^{\otimes \alpha-\beta} \otimes 1^{\otimes k-\alpha} \\
&E^{(\beta+1...\alpha)}_{s_\beta | ... | s_{\alpha}} = \sum_{x_a,y_a =1}^n \text{coefficient} \cdot 1^{\otimes \beta} \otimes  E_{x_{\beta+1}y_{\beta+1}} \otimes ... \otimes E_{x_{\alpha} y_{\alpha}} \otimes 1^{\otimes k-\alpha} \\
&E^{(\alpha+1...k)}_{s_\alpha| ... | s_{k}}  = \sum_{x_a,y_a =1}^n \text{coefficient} \cdot 1^{\otimes \beta} \otimes 1^{\alpha - \beta} \otimes  E_{x_{\alpha+1}y_{\alpha+1}} \otimes ... \otimes E_{x_{k} y_{k}}
\end{align*}
(we note that $x_1 = j$, $y_\beta = s_\beta = x_{\beta + 1}$, $y_\alpha = s_\alpha = x_{\alpha+1}$, $y_k = i$ in all summands above that have non-zero coefficient). Then $\mtd \Delta(R_{\omega_k}X)$, which differs by certain powers of $\psi_s^{\pm 1}$ from $\mtd R_{\omega_k}X$, is given by:
\begin{equation}
\label{eqn:phew}
\mtd \Delta(R_{\omega_k} X) =  \sum_{x_a,y_a =1}^n \text{coefficient} \cdot \psi^{-1}_{x_{\beta+1}}... \psi^{-1}_{x_{\alpha}} R_{\omega_\beta} \left[ \prod_{1 \leq a \leq \beta}^{\alpha < c \leq k} R_{ac} \left( \frac {z_{a}}{z_{c}} \right) \right]
\end{equation}
$$ 
R_{\omega_{k-\alpha}}  (E_{x_1y_1} \otimes ... \otimes E_{x_{\beta} y_{\beta}} \otimes 1^{\otimes k-\beta})  \left[ \prod_{1 \leq a \leq \beta}^{\alpha < c \leq k} \tR_{ac} \left( \frac {z_{a}}{z_{c}} \right) \right] (1^{\otimes \alpha} \otimes  E_{x_{\alpha+1}+1,y_{\alpha+1}} \otimes ... \otimes E_{x_{k} y_{k}})
$$
$$
 \psi_{y_{\beta+1}} ... \psi_{y_{\alpha-1}} \psi_{y_\alpha+1} R_{\omega_{\alpha - \beta}}  (1^{\otimes \beta} \otimes  E_{x_{\beta+1},y_{\beta+1}} \otimes ... \otimes E_{x_{\alpha}, y_{\alpha}+1} \otimes 1^{k-\alpha})
$$
$$
(-q^{-2})  \underbrace{q^{\sum^{1 \leq a \leq \beta}_{\beta < b \leq \alpha} \delta^{x_a}_{x_b}}}_{\text{first squiggle}} \underbrace{q^{-\sum^{\alpha < c \leq k}_{\beta < b \leq \alpha} \delta^{x_c}_{x_b}}}_{\text{second squiggle}} \underbrace{q^{-\sum^{1 \leq a \leq \beta}_{\beta < b \leq \alpha} \delta^{y_a}_{x_b}}}_{\text{third squiggle}} \underbrace{q^{\sum^{\alpha \leq c \leq k}_{\beta < b \leq \alpha} \delta_{y_b}^{x_c}}}_{\text{fourth squiggle}} \underbrace{q^{\sum_{\alpha < c \leq k}^{\beta < b \leq \alpha} \delta_{x_c}^{x_b} - \delta_{y_c}^{y_b}}}_{\text{conjugation}} 
$$
Before we move on, we must explain three issues concerning the expression above: the power of $q$ labeled ``conjugation", why $y_\alpha = x_{\alpha+1} = s_\alpha$ were increased by 1, and why the factor $-q^{-2}$ arose. The first issue, namely the power of $q$, appeared from the diagonal terms of arbitrary conjugation matrices $R_\sigma$ and $R_{\sigma}^{-1}$ as in \eqref{eqn:symmetrization aff}, where $\sigma$ is any permutation which switches the variables $\{\beta+1,...,\alpha\}$ and $\{\alpha+1,...,k\}$ (this is because in the definition of the coproduct, the variables to the left of the $\otimes$ sign must all have smaller indices than the variables to the right of the $\otimes$ sign). The latter two issues happened because of the presence of $Q_{\alpha, \alpha+1}(z_{\alpha}/z_{\alpha+1})$ in the fourth squiggle. Because $Q_{\alpha, \alpha+1}(\infty)$ has 0 on the diagonal, its contribution of minimal initial degree comes from the immediately off-diagonal terms, which are:
$$
-\frac {\oq^{-2\delta_{t}^n}}q \sum_t ... \otimes 1 \otimes E_{t,t+1} \otimes E_{t+1,t} \otimes 1 \otimes ...
$$
Therefore, for any indices $u$ and $v$, we have:
$$
E_{u s_\alpha}^{(\alpha)} \cdot \left( -\frac {\oq^{-2\delta_{t}^n}}q \sum_t  E_{t,t+1}^{(\alpha)} E_{t+1,t}^{(\alpha+1)} \right) \cdot E_{s_{\alpha}v}^{(\alpha+1)} = (-q^{2}\oq^{2\delta_{s_\alpha}^n})^{-1} \cdot q^{\delta_{s_\alpha}^{s_\alpha}} E_{u,s_{\alpha}+1}^{(\alpha)} E_{s_{\alpha}+1,v}^{(\alpha+1)}
$$
Using \eqref{eqn:psi x}, we may move certain $\psi$ factors around in \eqref{eqn:phew}, in order to cancel the powers of $q$ with underbraces beneath:
$$
\mtd \Delta(R_{\omega_k} X) =  (-q^{2}\oq^{2\delta_{s_\alpha}^n})^{-1} \sum_{x_a,y_a =1}^n \text{coefficient} \cdot R_{\omega_\beta} \left[ \prod_{1 \leq a \leq \beta}^{\alpha < c \leq k} R_{ac} \left( \frac {z_{a}}{z_{c}} \right) \right]
R_{\omega_{k-\alpha}}
$$
$$ 
(E_{x_1y_1} \otimes ... \otimes E_{x_{\beta} y_{\beta}} \otimes 1^{\otimes k-\beta}) \frac {\psi_{y_{\beta+1}}... \psi_{y_\alpha+1}}{\psi_{x_{\beta+1}}... \psi_{x_{\alpha}}}  \left[ \prod_{1 \leq a \leq \beta}^{\alpha < c \leq k} \tR_{ac} \left( \frac {z_{a}}{z_{c}} \right) \right] 
$$
$$
(1^{\otimes \alpha} \otimes  E_{x_{\alpha+1}+1,y_{\alpha+1}} \otimes ... \otimes E_{x_{k} y_{k}})  R_{\omega_{\alpha - \beta}}  (1^{\otimes \beta} \otimes  E_{x_{\beta+1},y_{\beta+1}} \otimes ... \otimes E_{x_{\alpha}, y_{\alpha}+1} \otimes 1^{k-\alpha})
$$
As a consequence of Claim \ref{claim:tensors}, we may write the expression above as:
$$
\mtd \Delta(R_{\omega_k} X) =  (-q^2\oq^{\frac 2n})^{-1} \sum_{x_a,y_a =1}^n \text{coefficient} \cdot R_{\omega_\beta} \left[ \prod_{1 \leq a \leq \beta}^{\alpha < c \leq k} R_{ac} \left( \frac {z_{a}}{z_{c}} \right) \right]
R_{\omega_{k-\alpha}}
$$
\begin{equation}
\label{eqn:city}
E^{(1...\beta)}_{s_0 | ... | s_{\beta}} \frac {\psi_{s_\alpha + 1}}{\psi_{s_\beta}}  \left[ \prod_{1 \leq a \leq \beta}^{\alpha < c \leq k} \tR_{ac} \left( \frac {z_{a}}{z_{c}} \right) \right] E^{(\alpha+1...k)}_{s_\alpha + 1| ... | s_{k}} \otimes R_{\omega_{\alpha-\beta}} E^{(\beta+1...\alpha)}_{s_\beta | ... | s_{\alpha} + 1}
\end{equation}
Symmetrizing the expression above with respect to all permutations $\sigma \in S(k)$ which fix the set $A$ gives rise to $\mtd \Delta(F)$. To obtain the expression on the second line of \eqref{eqn:big cop}, it remains to establish the formulas below (let $u = s_\alpha + 1$ and $v = s_\beta$):
\begin{align}
&F^\mu_{[v;j)} = \sym \ R_{\omega_\beta} E^{(1...\beta)}_{s_0 | s_1 | ... | s_{\beta-1} | s_{\beta}} \label{eqn:1} \\ 
&F^{(\bullet)}_{[u;v)} = \sym \ R_{\omega_{\alpha - \beta}} E^{(1... \alpha - \beta)}_{s_{\beta}| s_{\beta+1} | ... s_{\alpha-1}|s_{\alpha} + 1} \label{eqn:2} \\
&\bF^\mu_{[i;u)} = \sym \ R_{\omega_{k-\alpha}} E^{(1...k-\alpha)}_{s_{\alpha} + 1 | s_{\alpha+1} | ... | s_{k-1} | s_{k}} (-q^{2}\oq^{\frac 2n})^{-1} \label{eqn:3}
\end{align}
To prove the formulas above, we start with an easy computation: \\

\begin{claim}
\label{claim:identity}

We have the identity:
\begin{equation}
\label{eqn:identity 1}
E^{(1...k)}_{c_0|c_1|...|c_{k-1}|c_k} = \bE^{(1...k)}_{c_0|c_1+1|...|c_{k-1}+1|c_k} \cdot (-q^{2} \oq^{\frac 2n})^{1-k} 
\end{equation}
where $\bE^{(u...v)}_{c_{u-1}|...|c_v} = \prod_{a=u}^{v} \left[\tR_{ua} \left( \frac {z_{u}}{z_a} \right) ... \tR_{a-2,a} \left( \frac {z_{a-2}}{z_a} \right) \bQ_{a-1,a} \left( \frac {z_{a-1}}{z_a} \right) E^{(a)}_{c_{a-1} c_a} \oq^{\frac {2\oo{c_a}}n} \right]$. \\

\end{claim}

\noindent Let us first show how the Claim allows us to complete the proof of the Proposition. Because of \eqref{eqn:identity 1}, formula \eqref{eqn:3} is equivalent to:
\begin{equation}
\bF^\mu_{[i;u)} = (-q^2 \oq^{\frac 2n})^{\alpha - k} \cdot \sym \ R_{\omega_{k-\alpha}} \bE^{(1...k-\alpha)}_{s_{\alpha} + 1 | s_{\alpha+1} + 1 | ... | s_{k-1} + 1 | s_{k}}\label{eqn:4}
\end{equation}
Then \eqref{eqn:1}, \eqref{eqn:2}, \eqref{eqn:4} follow from \eqref{eqn:def p}, \eqref{eqn:def pp} and the formulas below:
\begin{align*}
&j - \left \lceil \frac {(j-i)t}k \right \rceil = j - \left \lceil \mu t \right \rceil \qquad \qquad \qquad \qquad \ \ \ \forall \ t\in \{1,...,\beta\} \\
&j - \left \lceil \frac {(j-i)t}k \right \rceil + \delta_t^\alpha = v - \left \lceil \frac {(v-u)(t-\beta)}{\alpha-\beta} \right \rceil \quad \forall \ t\in \{\beta+1,...,\alpha\} \\
&j - \left \lceil \frac {(j-i)t}k \right \rceil + 1 = u - \left \lfloor \mu (t-\alpha) \right \rfloor \qquad \qquad \ \ \forall \ t\in \{\alpha+1,...,k\} 
\end{align*}
which are all straightforward consequences of our assumption on the triangle $T$. This completes the proof of formula \eqref{eqn:big cop} for the picture on the left in Figure 22 (as we said, the case of the picture on the right is analogous, and left to the interested reader). As for Claim \ref{claim:identity}, we start with the following identity:
\begin{multline}
(E_{j,t} \otimes 1) Q\left( \frac {z_1}{z_2} \right) (1 \otimes E_{t,i}) \oq^{\frac {2\oo{t}}n} = \\ = (E_{j,t+1} \otimes 1) \bQ\left( \frac {z_1}{z_2} \right) (1 \otimes E_{t+1,i}) \oq^{\frac {2\oo{t+1}}n} \cdot (-q^{2} \oq^{\frac 2n})^{-1} \label{eqn:identity q}
\end{multline} 
for all $i,j,t \in \BZ$. Indeed, by plugging in \eqref{eqn:q}--\eqref{eqn:qq}, formula \eqref{eqn:identity q} reads:
\begin{align*} 
&q^{-1} \sum_{u = 1}^n z_1^{\left \lfloor \frac {j-1}n \right \rfloor - \left \lfloor \frac {t-1}n \right \rfloor} z_2^{\left \lfloor \frac {t-1}n \right \rfloor - \left \lfloor \frac {i-1}n \right \rfloor} \frac {\left( \frac {z_1 \oq^2}{z_2} \right)^{\delta_{\overline{t}<u}}}{1 - \frac {z_1\oq^2}{z_2}} (E_{\barj u} \otimes E_{u \bari}) \oq^{\frac {2\oo{t}}n}  = \\
&= (- q) \sum_{u = 1}^n z_1^{\left \lfloor \frac {j-1}n \right \rfloor - \left \lfloor \frac {t}n \right \rfloor} z_2^{\left \lfloor \frac {t}n \right \rfloor - \left \lfloor \frac {i-1}n \right \rfloor} \frac {\left( \frac {z_1 \oq^2}{z_2} \right)^{\delta_{\overline{t+1}\leq u}}}{1 - \frac {z_1\oq^2}{z_2}} (E_{\barj u} \otimes E_{u \bari}) \oq^{\frac {2\oo{t+1}}n} \cdot (-q^{2} \oq^{\frac 2n})^{-1}
\end{align*}
which is elementary. Identity \eqref{eqn:identity 1} follows by $k-1$ applications of \eqref{eqn:identity q}. 

\end{proof}

\begin{proof} \emph{of Theorem \ref{thm:iso}:} As a consequence of Proposition \ref{prop:iso small}, the assignment
$$
f_{[i;j)}^\mu \text{ of \eqref{eqn:root generators} } \leadsto F_{[i;j)}^\mu \text{ of \eqref{eqn:def p}}
$$
yields an algebra homomorphism $\CE^+_\mu \rightarrow \CA^+$, for any $\mu \in \BQ$. To extend this to a homomorphism:
\begin{equation}
\label{eqn:upsilon}
\Upsilon^+ : \DD^+ \longrightarrow \CA^+
\end{equation}
we need to prove that formulas \eqref{eqn:rel 2 pbw} and \eqref{eqn:rel 3 pbw} hold with $p$, $f$ replaced by $P$, $F$. In order to show that \eqref{eqn:rel 2 pbw} holds in this setup, let us first show that the linear maps $\alpha_{[u;v)}$ take the same value on both sides of the equation. By \eqref{eqn:quasi}, we have:
$$
\alpha_{[u;v)} (\lhs \text{ of \eqref{eqn:rel 2 pbw}}) = \alpha_{[u;v)} \left( P_{[i;j)}^{(k)} P_{l\bde, r}^{(k')} \right)
 - \alpha_{[u;v)} \left( P_{l\bde, r}^{(k')} P_{[i;j)}^{(k)} \right) = 
$$
$$
= \alpha_{[u+nl;v)} \left( P_{[i;j)}^{(k)} \right) \alpha_{[u;u+nl)} \left( P_{l\bde, r}^{(k')} \right) \oq^{\frac dn} - \alpha_{[v-nl;v)} \left(P_{l\bde, r}^{(k')} \right) \alpha_{[u;v-nl)} \left( P_{[i;j)}^{(k)} \right) \oq^{-\frac dn} =
$$
$$
= \delta_{(u+nl,v)}^{(i,j)} \delta_{u \text{ mod } g}^r \oq^{\frac dn} -  \delta_{(u,v-nl)}^{(i,j)} \delta_{v \text{ mod } g}^r \oq^{-\frac dn} = \alpha_{[u;v)} (\rhs \text{ of \eqref{eqn:rel 2 pbw}})
$$
$\forall u,v$. Therefore, Lemma \ref{lem:min} reduces the equality \eqref{eqn:rel 2 pbw} to showing that:
$$
\text{The LHS of \eqref{eqn:rel 2 pbw} is a primitive element of } \CB^+_{\frac {j+ln-i}{k+k'}}
$$
We may depict the degree vectors of the elements $P_{[i;j)}^{(k)}$, $P_{l\bde, r}^{(k')}$, $P_{[i;j+nl)}^{(k+k')}$ as:

\begin{picture}(100,110)(-120,15)

\put(0,40){\circle*{2}}\put(20,40){\circle*{2}}\put(40,40){\circle*{2}}\put(60,40){\circle*{2}}\put(80,40){\circle*{2}}\put(0,60){\circle*{2}}\put(20,60){\circle*{2}}\put(40,60){\circle*{2}}\put(60,60){\circle*{2}}\put(80,60){\circle*{2}}\put(0,80){\circle*{2}}\put(20,80){\circle*{2}}\put(40,80){\circle*{2}}\put(60,80){\circle*{2}}\put(80,80){\circle*{2}}\put(0,100){\circle*{2}}\put(20,100){\circle*{2}}\put(40,100){\circle*{2}}\put(60,100){\circle*{2}}\put(80,100){\circle*{2}}

\put(20,40){\line(0,1){20}}
\put(60,80){\line(0,1){20}}
\put(20,40){\line(1,1){40}}
\put(20,40){\line(2,3){40}}
\put(20,60){\line(1,1){40}}

\put(0,30){\scriptsize{$(0,0)$}}
\put(-10,65){\scriptsize{$(j-i,k)$}}
\put(40,105){\scriptsize{$(j-i+nl,k+k')$}}
\put(60,70){\scriptsize{$(nl,k')$}}

\end{picture}

\noindent We need to show that all the hinges of summands of $\Delta(\lhs \text{ of \eqref{eqn:rel 2 pbw}})$ are to the left the vector $(j-i+nl,k+k')$. Since coproduct is multiplicative, the hinges of $\Delta(XY)$ are all among the sums of hinges of $\Delta(X)$ and $\Delta(Y)$, as vectors in $\BZ^2$. By definition, the hinges of $\Delta(P_{[i;j)}^{(k)})$ and $\Delta(P_{l\bde, r}^{(k')})$ lie to the left of the vectors $(j-i,k)$ and $(nl,k')$, respectively. The sum of any two such hinges lies to the left of the parallelogram in the picture, except for the sum of the two hinges below:
\begin{align*}
&\Delta(P_{[i;j)}^{(k)}) = ... + P_{[i;j)}^{(k)} \otimes 1 + ... \qquad \text{has a hinge at } (0,0) \\ 
&\Delta(P_{l\bde, r}^{(k')}) = ... + 1 \otimes P_{l\bde, r}^{(k')} + ...  \qquad \text{has a hinge at } (nl,k')
\end{align*}
Therefore, $\Delta(P_{[i;j)}^{(k)}P_{l\bde, r}^{(k')})$ and $\Delta(P_{l\bde, r}^{(k')}P_{[i;j)}^{(k)})$ both have a hinge at the point $(nl,k')$, but the corresponding summand in both coproducts is:
$$
P_{[i;j)}^{(k)} \otimes P_{l\bde, r}^{(k')} 
$$
We conclude that this summand vanishes in $\Delta(\lhs \text{ of \eqref{eqn:rel 2 pbw}})$, which therefore has all the hinges to the left of $(j-i+nl,k+k')$. This completes the proof of \eqref{eqn:rel 2 pbw}. \\

\noindent Let us now prove \eqref{eqn:rel 3 pbw} by induction on $k+k'$ (the base case $k+k' = 1$ is trivial). Recall that $\mu = \frac {j+j'-i-i'}{k+k'}$, and let us represent the degrees of $P_{[i;j)}^{(k)}$ and $P_{[i';j')}^{(k')}$ as:

\begin{picture}(100,130)(-115,0)

\put(0,20){\circle*{2}}\put(20,20){\circle*{2}}\put(40,20){\circle*{2}}\put(60,20){\circle*{2}}\put(80,20){\circle*{2}}
\put(0,40){\circle*{2}}\put(20,40){\circle*{2}}\put(40,40){\circle*{2}}\put(60,40){\circle*{2}}\put(80,40){\circle*{2}}\put(0,60){\circle*{2}}\put(20,60){\circle*{2}}\put(40,60){\circle*{2}}\put(60,60){\circle*{2}}\put(80,60){\circle*{2}}\put(0,80){\circle*{2}}\put(20,80){\circle*{2}}\put(40,80){\circle*{2}}\put(60,80){\circle*{2}}\put(80,80){\circle*{2}}\put(0,100){\circle*{2}}\put(20,100){\circle*{2}}\put(40,100){\circle*{2}}\put(60,100){\circle*{2}}\put(80,100){\circle*{2}}

\put(40,20){\line(0,1){80}}
\put(40,20){\line(-1,2){20}}
\put(20,60){\line(1,2){20}}
\put(60,60){\line(-1,2){20}}
\put(40,20){\line(1,2){20}}

\put(30,10){\scriptsize{$(0,0)$}}
\put(-13,64){\scriptsize{$(j-i,k)$}}
\put(0,107){\scriptsize{$(j+j'-i-i',k+k')$}}
\put(60,65){\scriptsize{$(j'-i',k')$}}

\end{picture}

\noindent We have the following formulas, courtesy of Proposition \ref{prop:coproduct}:
\begin{align*}
&\Delta \left(P_{[i;j)}^{(k)}\right) = P_{[i;j)}^{(k)} \otimes 1 + \frac {\psi_i}{\psi_j} \otimes P_{[i;j)}^{(k)}  +  \sum_{i \leq u < v \leq j} F^\mu_{[v;j)} \frac {\psi_u}{\psi_v} \bF^\mu_{[i;u)} \otimes P_{[u;v)}^{(\bullet)} + ... \\
&\Delta \left(P_{[i';j')}^{(k')}\right) = P_{[i';j')}^{(k')} \otimes 1 + \frac {\psi_{i'}}{\psi_{j'}} \otimes P_{[i';j')}^{(k')} + ... 
\end{align*}
where the ellipsis denotes terms whose hinges lie to the left of the line of slope $\mu$ (this convention will remain in force for the remainder of this proof), and $\bullet$ denotes the natural number which makes the two sides of the expressions above have the same vertical degree. Letting $\lhs$ denote the left-hand side of \eqref{eqn:rel 3 pbw}, we have:
$$
\Delta(\lhs) = \lhs \otimes 1 + \frac {\psi_i \psi_{i'}}{\psi_j \psi_{j'}} \otimes \lhs + \squiggly{q^{\delta_{j'}^{i} - \delta_{i'}^{i}} P_{[i;j)}^{(k)} \frac {\psi_{i'}}{\psi_{j'}}\otimes P_{[i';j')}^{(k')} - q^{\delta_{j'}^{j} - \delta_{i'}^{j}} \frac {\psi_{i'}}{\psi_{j'}}  P_{[i;j)}^{(k)} \otimes P_{[i';j')}^{(k')}} +
$$ 
$$
+ \sum_{i \leq u < v \leq j}  q^{\delta_{j'}^{i} - \delta_{i'}^{i}} F^\mu_{[v;j)} \frac {\psi_u}{\psi_v} \bF^\mu_{[i;u)} \frac {\psi_{i'}}{\psi_{j'}} \otimes P_{[u;v)}^{(\bullet)} P_{[i';j')}^{(k')} - 
$$
$$
- \sum_{i \leq u < v \leq j} q^{\delta_{j'}^{j} - \delta_{i'}^{j}} \frac {\psi_{i'}}{\psi_{j'}} F^\mu_{[v;j)} \frac {\psi_u}{\psi_v} \bF^\mu_{[i;u)} \otimes P_{[i';j')}^{(k')} P_{[u;v)}^{(\bullet)} + ...
$$
Relation \eqref{eqn:psi x} allows us to move $\psi$'s around, and show that the expression with the squiggly underline vanishes, while the expression on the second line yields:
$$
\Delta(\lhs) = \lhs \otimes 1 + \frac {\psi_i \psi_{i'}}{\psi_j \psi_{j'}} \otimes \lhs + 
$$ 
$$
+ \sum_{i \leq u < v \leq j}  F^\mu_{[v;j)} \frac {\psi_{i'}\psi_u}{\psi_{j'}\psi_v} \bF^\mu_{[i;u)} \otimes \Big[ q^{\delta_{j'}^{u} - \delta_{i'}^{u}} P_{[u;v)}^{(\bullet)}  P_{[i';j')}^{(k')} - q^{\delta_{j'}^{v} - \delta_{i'}^{v}} P_{[i';j')}^{(k')} P_{[u;v)}^{(\bullet)} \Big] + ...
$$
If we are only interested in the leading term $\Delta_\mu$ of the coproduct, we may neglect the terms represented by the ellipsis. The induction hypothesis allows us to replace the term in square brackets by the RHS of \eqref{eqn:rel 3 pbw} for $(i,j,k) \mapsto (u,v,\bullet)$, hence:
$$
\Delta_\mu(\lhs) = \lhs \otimes 1 + \frac {\psi_i \psi_{i'}}{\psi_j \psi_{j'}} \otimes \lhs + \sum^{i \leq u < v \leq j}_{[t;s) = [i';j')} 
$$
$$
 F^\mu_{[v;j)} \frac {\psi_t\psi_u}{\psi_v\psi_s} \bF^\mu_{[i;u)} \otimes F^\mu_{[t;v)} \bF^\mu_{[u;s)} \cdot \left(\delta_{j'}^{s}  \frac {q^{-\delta_{j'}^{i'}}}{q^{-1} - q} - \delta_{j'}^{i'} \frac {(q\oq^{\frac 1n})^{n - 2(\overline{k'(s-i')})}}{q^{-n} \oq^{-1} - q^{n} \oq} \right)
$$
Let $\gamma_{i'j'k'}(s) \in \BQ(q,\oq^{\frac 1n})$ denote the constant in the round brackets on the second line above. By \eqref{eqn:coproduct p} and \eqref{eqn:coproduct pp}, notice that the formula above matches $\Delta_\mu(\rhs)$. By Lemma \ref{lem:min}, to prove that $\lhs = \rhs$, it suffices to show that the two sides of equation \eqref{eqn:rel 3 pbw} take the same values under the maps \eqref{eqn:alpha}. To this end:
\begin{equation}
\label{eqn:alpha lhs}
\alpha_{[u;v)} (\lhs) \stackrel{\eqref{eqn:quasi}}= \delta^{(i',j)}_{(u,v)} \delta_{i}^{j'} q^{\delta_{j'}^{i} - \delta_{i'}^{i}} \oq^{\frac dn} - \delta^{(i,j')}_{(u,v)} \delta_{j}^{i'} q^{\delta_{j'}^{j} - \delta_{i'}^{j}} \oq^{-\frac dn}
\end{equation}
where $d = \gcd(k+k',j+j'-i-i')$, while \eqref{eqn:quasi}, \eqref{eqn:alpha p} and \eqref{eqn:alpha pp} imply:
\begin{equation}
\label{eqn:alpha rhs}
\alpha_{[u;v)}(\rhs) = \sum_{[t;s) = [i';j')} 
\end{equation}
$$
\delta_i^u \delta_s^t \delta_j^v (1-q^{2}) \oq^{\frac {\gcd(\mu(j-t),j-t)}n} (1-q^{-2}) \oq^{-\frac {\gcd(\mu(s-i),s-i)}n} \gamma_{i'j'k'}(s)
$$
The equality between the right-hand sides of \eqref{eqn:alpha lhs} and \eqref{eqn:alpha rhs} was established in \cite[Claim 4.6]{PBW}. This concludes the proof of \eqref{eqn:rel 3 pbw}, so there exists an algebra homomorphism $\Upsilon^+$ as in \eqref{eqn:upsilon}. In the remainder of this proof, we need to show that $\Upsilon^+$ is an isomorphism. As explained in \cite{PBW} (following the similar argument of \cite{BS}), one may use relations \eqref{eqn:rel 2 pbw} and \eqref{eqn:rel 3 pbw} to express an arbitrary product of the generators \eqref{eqn:simple} and \eqref{eqn:imaginary} as a linear combinations of products of the form:
\begin{equation}
\label{eqn:product}
\prod_{\mu \in \BQ}^\rightarrow x_\mu^{(i)}, \quad \text{all but finitely many of the } x_\mu^{(i)} \text{ are } 1
\end{equation}
where $\{x^{(i)}_\mu\}$ go over any linear basis of $\CB_\mu^+$, and $\prod^\rightarrow_\mu$ is taken in increasing order of $\mu$. \\

\begin{claim}
\label{claim:claim}

The elements \eqref{eqn:product} are all linearly independent in $\CA^+$. \\

\end{claim}

\noindent Let us first show how the Claim above completes the proof of the Theorem. The linear independence of the elements \eqref{eqn:product} implies that:
\begin{equation}
\label{eqn:geq leq}
\dim \CA_{\leq \mu|\bd,k} \geq \# \Big\{ \text{unordered collections \eqref{eqn:unordered collections}} \Big\}
\end{equation}
for all $\mu$, $k$, $\bd$ (indeed, \eqref{eqn:bound affine} implies that the number of products $\prod_{\rho \leq \mu} x_\rho^{(i)}$ is precisely equal to the number in the RHS of \eqref{eqn:geq leq}). Combining \eqref{eqn:geq leq} with Lemma \ref{lem:magic}, we conclude that the products \eqref{eqn:product} actually form a linear basis of $\CA^+$. This implies that $\Upsilon^+$ is an isomorphism, since Definition \ref{def:explicit d} says that the same products also form a linear basis of $\DD^+$. \\

\noindent Let us now prove Claim \ref{claim:claim}. For any $\mu$, we assume that the basis vectors $x_\mu^{(i)}$ of $\CB_\mu^+$ are ordered in non-decreasing order of $|\text{hdeg}|$, i.e.:
\begin{equation}
\label{eqn:ordering}
i \geq i' \quad \Rightarrow \quad |\hdeg x_\mu^{(i)}| \geq |\hdeg x_\mu^{(i')}|
\end{equation}
Now suppose we have a non-trivial linear relation among the various products \eqref{eqn:product}. We may rewrite this hypothetical relation as:
\begin{equation}
\label{eqn:hypothetical}
\prod^j_{\nu < \mu} x_\nu^{(j)} \cdot x_\mu^{(i)} = \sum \text{coefficient}  \prod^{j'}_{\nu < \mu} x_\nu^{(j')} \cdot x_\mu^{(i')}
\end{equation}
where all terms in the RHS have $i' < i$. Since the coproduct is multiplicative, then all the hinges of $\Delta(XY)$ are sums of hinges of $\Delta(X)$ and hinges of $\Delta(Y)$, as vectors in $\BZ^2$. Therefore, $\Delta(\lhs \text{ of \eqref{eqn:hypothetical}})$ has a single summand with hinge at the lattice point:
\begin{equation}
\label{eqn:hinge}
\left(|\hdeg x_\mu^{(i)}|, \vdeg x_\mu^{(i)} \right)
\end{equation}
and the corresponding summand is precisely:
\begin{equation}
\label{eqn:left}
\Delta(\lhs \text{ of \eqref{eqn:hypothetical}}) = ... + \psi   \prod^j_{\nu < \mu} x_\nu^{(j)} \otimes x_\mu^{(i)} + ...
\end{equation}
where $\psi$ stands for a certain (unimportant) product of $\psi_a^{\pm 1}$'s. Meanwhile, the coproduct of the RHS of \eqref{eqn:hypothetical} can only have a hinge at a lattice point \eqref{eqn:hinge} if $|\hdeg x_\mu^{(i)}| = |\hdeg x_{\mu}^{(i')}|$. The corresponding summand in the coproduct is:
\begin{equation}
\label{eqn:right}
\Delta(\rhs \text{ of \eqref{eqn:hypothetical}}) = ... + \sum \text{coefficient } \cdot \psi \prod^{j'}_{\nu < \mu} x_\nu^{(j')} \otimes x_\mu^{(i')} + ... 
\end{equation}
Since $x_\mu^{(i)}$ cannot be expressed as a linear combination of $x_\mu^{(i')}$ with $i' < i$, the right-hand sides of expressions \eqref{eqn:left} and \eqref{eqn:right} cannot be equal. This contradiction implies that there can be no relation \eqref{eqn:hypothetical}, which proves Claim \ref{claim:claim}. 
	
\end{proof} 

\begin{corollary} 
\label{cor:gen}

The algebra $\CA^+$ is generated by the $\vdeg = 1$ elements:
\begin{equation}
\label{eqn:degree one piece bis}
\Big \{ E_{ij} \Big \}_{(i,j) \in \zzz}
\end{equation}

\end{corollary} 

\noindent The corollary is an immediate consequence of Proposition \ref{prop:gen} and Theorem \ref{thm:iso}. In \cite[Theorem 1.2 and Definition 2.3]{Rectangular}, we will construct explicit quadratic relations satisfied by the generators $E_{ij} \in \CA^+$, in the language of formal series. 

\section{The double shuffle algebra with spectral parameters}
\label{sec:double}

\noindent In the previous Section, we constructed the extended shuffle algebra corresponding to the $R$--matrix with spectral parameters \eqref{eqn:explicit r}. We will now take two such extended shuffle algebras and construct their double, as was done in Subsections \ref{sub:pairing} and \ref{sub:double} for $R$--matrices without spectral parameters. This will conclude the proof of Theorem \ref{thm:main}. \\

\subsection{} Let $\oq_+ = \oq$ and $\oq_- = q^{-n} \oq^{-1}$. If $\tR^+(x) = \tR(x)$ is given by \eqref{eqn:tr}, then:
\begin{equation}
\label{eqn:definition r minus}
\tR^-(x) = \left[ \tR^{\dagger_1}\left(\frac 1x \right)^{-1} \right]^{\dagger_1}_{21} \in \End(V \otimes V)(x)
\end{equation}
is given by:
$$
\tR^-(x) = \sum_{1\leq i,j \leq n} E_{ii} \otimes E_{jj} \left(\frac {q^{-1} - x q \oq_-^2}{1-x \oq_-^2} \right)^{\delta_i^j} - (q - q^{-1}) \sum_{1 \leq i \neq j \leq n} E_{ij} \otimes E_{ji} q^{2(j-i)} \frac {(x \oq_-^2)^{\delta_{i<j}}}{1-x \oq_-^2} 
$$
Note that we have the equality:
\begin{equation}
\label{eqn:r plus minus}
\tR^-(x) = D_2 \tR^+(x) D_2^{-1} \Big|_{\oq_+ \mapsto \oq_-} 
\end{equation}
where $D = \text{diag}(q^2,...,q^{2n}) \in \End(V)$. \\

\begin{definition}
\label{def:minus algebra}

The shuffle algebra $\CA^-$ is defined just like in Definition \ref{def:shuf aff}, using $\oq_-$ instead of $\oq$, and the multiplication \eqref{eqn:shuf prod aff} uses $\tR^-$ instead of $\tR$. \\

\end{definition}

\noindent Because of \eqref{eqn:r plus minus}, the map:
\begin{equation}
\label{eqn:a plus minus}
\CA^+ \stackrel{\Phi}\longrightarrow \CA^-, \qquad X_{1...k}(z_1,...,z_k) \mapsto D_1...D_k X_{1...k}(z_1,...,z_k)\Big|_{\oq_+ \mapsto \oq_-}
\end{equation}
is a $\BQ(q)$--linear algebra isomorphism. The following elements of $\CA^-$ are the images of the elements \eqref{eqn:def p}--\eqref{eqn:def pp} under $\Phi$, times a factor of $\oq_-^{\frac {2\barj-2\bari}n}$:

\begin{equation}
\label{eqn:def f minus}
F_{[i;j)}^{(-k)} = \sym \ R_{\omega_k}(z_1,...,z_k) 
\end{equation}
$$
\prod_{a=1}^k \left[ \tR^-_{1a} \left( \frac {z_1}{z_a} \right) ... \tR^-_{a-2,a} \left( \frac {z_{a-2}}{z_a} \right) Q^-_{a-1,a} \left( \frac {z_{a-1}}{z_a} \right) E^{(a)}_{s_{a-1} s_a} \oq^{-\frac {2\oo{s_{a-1}}}n} \right] 
$$
\begin{equation}
\label{eqn:def ff minus}
\bF_{[i;j)}^{(-k)} = (-\oq^{\frac 2n})^{k} \cdot \sym \ R_{\omega_k}(z_1,...,z_k)  
\end{equation}
$$
\prod_{a=1}^k \left[ \tR^-_{1a} \left( \frac {z_1}{z_a} \right) ... \tR^-_{a-2,a} \left( \frac {z_{a-2}}{z_a} \right) \bQ^-_{a-1,a} \left( \frac {z_{a-1}}{z_a} \right) E^{(a)}_{s'_{a-1} s'_a} \oq^{-\frac {2\oo{s'_{a-1}}}n} \right] 
$$
where $s_a = j - \lceil \mu a \rceil$, $s'_a = j - \lfloor \mu a \rfloor$, $Q^- = D_2 Q D_2^{-1} |_{\oq \mapsto \oq_-}$, $\bQ^- = D_2 \bQ D_2^{-1}|_{\oq \mapsto \oq_-}$. \\

\subsection{} In Definition \ref{def:extended aff}, we defined the extended shuffle algebra by introducing new generators. We will now add two more central elements $c$ and $\barc$, and define instead:
\begin{equation} 
\label{eqn:extended aff pm}
\tCA^\pm = \frac {\Big \langle \CA^\pm, s_{[i;j)}^\pm, c^{\pm 1}, \barc^{\pm 1} \Big \rangle^{i \leq j}_{1  \leq i \leq n}}{c, \barc \text{ central and relations \eqref{eqn:ext rel 1 aff new} and  \eqref{eqn:ext rel 4 aff new}}} 
\end{equation} 
where:
\begin{equation}
\label{eqn:ext rel 1 aff new}
R\left(\frac xy \right) S^\pm_1(x) S^\pm_2(y) = S^\pm_2(y) S^\pm_1(x) R \left(\frac xy \right)
\end{equation}
\begin{equation}
\label{eqn:ext rel 4 aff new}
X^\pm \cdot  S^\pm_0(y) = S^\pm_0(y) \cdot \frac {R_{k0} \left(\frac {z_k}y \right) ... R_{10} \left(\frac {z_1}y \right)}{f \left(\frac {z_k}y \right) ... f \left(\frac {z_1}y \right)} X^\pm \tR^\pm_{10} \left( \frac {z_1}y \right) ... \tR^\pm_{k0} \left( \frac {z_k}y \right) 
\end{equation}
for any $X^\pm = X^\pm_{1...k}(z_1,...,z_k) \in \CA^\pm \subset \tCA^\pm$, where:
$$
S^\pm(x) = \sum_{1 \leq i, j \leq n, \ d \geq 0}^{d = 0 \text{ only if } i \leq j} s^\pm_{[i;j+nd)} \otimes \begin{cases} E_{ij} x^{-d} &\text{if } \pm = + \\ \\ E_{ji} x^d &\text{if } \pm = - \end{cases} 
$$
If we define the series $T^\pm(x)$ by \eqref{eqn:new series s} and \eqref{eqn:new series t}, then \eqref{eqn:ext rel 4 aff new} is equivalent to:
\begin{equation}
\label{eqn:ext rel 5 aff new}
T^\pm_0(y) \cdot X^\pm = \tR^\pm_{0k} \left( \frac y{z_k} \right) ... \tR^\pm_{01} \left( \frac y{z_1} \right) X^\pm \frac {R_{01} \left( \frac y{z_1} \right) ... R_{0k} \left( \frac y{z_k} \right)}{f \left( \frac y{z_1} \right) ... f \left( \frac y{z_k} \right)} \cdot T^\pm_0(y)
\end{equation} 

\subsection{} 

The algebras $\tCA^\pm$ are graded by $\zz \times \BZ$, with:
\begin{align*}
&\deg X^\pm_{1...k}(z_1,...,z_k) = (\bd, \pm k), \quad \ \forall  X^\pm \in \CA^\pm \\
&\deg s_{[i;j)}^\pm = (\pm [i;j),0), \qquad \qquad \quad \forall i \leq j
\end{align*} 
where $\bd \in \zz$ is defined according to \eqref{eqn:horizontal}. We write $\deg X = (\hdeg X, \vdeg X)$ to specify the components of the degree vector in $\zz$ and $\BZ$, respectively. The reason why we introduced central elements $c$ and $\barc$ to the algebras \eqref{eqn:extended aff pm} is to twist the coproduct. Specifically, let $\Delta_{\text{old}}$ be the coproduct of \eqref{eqn:cop shuf aff 1}--\eqref{eqn:cop shuf aff 3}, and define: 
\begin{equation}
\label{eqn:twisted coproduct}
\Delta : \tCA^\pm \longrightarrow \tCA^\pm \woo \tCA^\pm
\end{equation}
by the formulas $\Delta(c) = c \otimes c$, $\Delta(\barc) = \barc \otimes \barc$, as well as:
\begin{equation}
\label{eqn:twisted coproduct explicit}
\text{if } \Delta_{\text{old}}(X) = X_1 \otimes X_2 \quad \text{then} \quad \Delta(X) = X_1 c^{-(\hdeg X_2)_n} \barc^{- \vdeg X_2} \otimes X_2
\end{equation}
Since $\deg X$ is multiplicative in $X$, the fact that $\Delta_{\text{old}}$ is coassociative and an algebra homomorphism implies the analogous statements for the coproduct $\Delta$. \\

\subsection{} We must now prove an analogue of Proposition \ref{prop:bialg pair}, but the main difficulty in doing so is \eqref{eqn:bialg pair 4}: given $X, Y \in \End(V^{\otimes k})(z_1,...,z_k)$, we can still define the trace of $XY$, but the answer will be a rational function in $z_1,...,z_k$. To obtain a number, one must integrate out the variables $z_1,...,z_k$, and the choice of contours will be crucial. \\

\noindent Let us consider the following expressions, for any $\sigma \in S(k)$:
$$
R_{\sigma} = \prod^{1 \leq i < j \leq k}_{\sigma^{-1}(i) > \sigma^{-1}(j)} R_{ij} \left( \frac {z_i}{z_j} \right), \qquad \oR_{\sigma} = \prod^{1 \leq i < j \leq k}_{\sigma^{-1}(i) > \sigma^{-1}(j)} R_{ji} \left( \frac {z_j}{z_i} \right)
$$
Explicitly, the product in $R_\sigma$ is taken by following the crossings in the positive braid lifting the permutation $\sigma$, while $\oR_\sigma$ is defined as $\sigma R_{\sigma^{-1}} \sigma^{-1}$. It is elementary to prove the following equation for all $\sigma \in S(k)$:
$$
\oR_\sigma R_{\sigma \omega_k} = \sigma R_{\omega_k} \sigma^{-1}
$$
where $\omega_k$ denotes the longest permutation. We may use $\oR_\sigma^{-1}$ instead of $R_\sigma$ in \eqref{eqn:symmetrization aff} because they both lift the permutation $\sigma$, and thus we obtain:
\begin{equation}
\label{eqn:shuf plus minus}
I_1 * ... * I_k = \sum_{\sigma \in S(k)} R_{\sigma \omega_k} \prod_{a=1}^k \left[ I_a^{(\sigma(a))}(z_{\sigma(a)}) \prod_{b=a+1}^k \tR^\pm_{\sigma(a)\sigma(b)} \left(\frac {z_{\sigma(a)}}{z_{\sigma(b)}} \right) \right] \oR_\sigma
\end{equation}
for all $I_1,..., I_k \in \End(V)[z^{\pm 1}] \subset \CA^\pm$. By Corollary \ref{cor:gen}, any element of $\CA^\pm$ is a linear combination of the shuffle elements \eqref{eqn:shuf plus minus}, for various $I_1,...,I_k$. \\

\begin{proposition}
\label{prop:intro}

There is a pairing (of vector spaces):
\begin{equation}
\label{eqn:pairing aff}
\CA^+ \otimes \CA^- \stackrel{\langle \cdot , \cdot \rangle}\longrightarrow \fff
\end{equation}
given by:
\begin{equation}
\label{eqn:symm pair 1} 
\Big \langle I^+_1 * ... * I^+_k, X^-_{1...k}(z_1,...,z_k) \Big \rangle = (q^2-1)^k \int_{|z_1| \ll ... \ll |z_k|} 
\end{equation}
$$
\emph{Tr} \left( R_{\omega_k} \prod_{a=1}^k \left[ I_a^{(a)}(z_a) \prod_{b=a+1}^k \tR^+_{ab} \left(\frac {z_a}{z_b} \right) \right] \frac {X_{1...k}(z_1,...,z_k)}{\prod_{1 \leq i < j \leq k} f\left(\frac {z_i}{z_j} \right)} \right)
$$
\begin{equation}
\label{eqn:symm pair 2} 
\Big \langle X^+_{1...k}(z_1,...,z_k), J^-_1 * ... * J^-_k \Big \rangle = (q^2-1)^k \int_{|z_1| \ll ... \ll |z_k|} 
\end{equation}  
$$
\emph{Tr} \left( R_{\omega_k} \prod_{a=1}^k \left[ J_a^{(a)}(z_a) \prod_{b=a+1}^k \tR^-_{ab} \left(\frac {z_a}{z_b} \right) \right] \frac {X_{1...k}(z_1,...,z_k)}{\prod_{1 \leq i < j \leq k} f\left(\frac {z_i}{z_j} \right)} \right)
$$
for all $I_1,...,I_k, J_1,...,J_k \in \eEnd(V)[z^{\pm 1}]$ and all $X^\pm \in \CA^\pm$. The notation: 
$$
\int_{|z_1| \ll ... \ll |z_k|} F(z_1,...,z_k)
$$ 
refers to the iterated residue of $F$ at $0$: first in $z_1$, then in $z_2$,..., finally in $z_k$. \\

\end{proposition}

\noindent Note that the pairing $\langle X^+, Y^- \rangle$ is only non-zero for pairs of elements of opposite degrees, i.e. $X^+ \in \CA_{\bd,k}$ and $Y^- \in \CA_{-\bd,-k}$ for various $(\bd,k) \in \zz \times \BN$. \\

\begin{proof} Formula \eqref{eqn:symm pair 1} is well-defined as a linear functional in the second argument, while \eqref{eqn:symm pair 2} is well-defined as a linear functional in the first argument. Therefore, to show that \eqref{eqn:pairing aff} is well-defined as a linear functional in both arguments, we only need to show that \eqref{eqn:symm pair 1} and \eqref{eqn:symm pair 2} produce the same result when $X^\pm$ is of the form \eqref{eqn:shuf plus minus} (this statement implicitly uses Corollary \ref{cor:gen}, which states that any element in $\CA^\pm$ is a linear combination of the elements \eqref{eqn:shuf plus minus}). To this end, we have:
$$
\frac 1{(q^2-1)^k} \Big \langle I^+_1 * ... * I^+_k, J^-_1 * ... * J^-_k \Big \rangle \text{ according to \eqref{eqn:symm pair 1}} = \int_{|z_1| \ll ... \ll |z_k|} \sum_{\sigma \in S(k)}
$$
\begin{multline}
\label{eqn:former} 
\Tr \left(\prod_{1 \leq i < j \leq k} \frac {1}{f\left(\frac {z_i}{z_j} \right)} \cdot R_{\omega_k} \prod_{a=1}^k \left[ I_a^{(a)}(z_a) \prod_{b=a+1}^k \tR^+_{ab} \left(\frac {z_a}{z_b} \right) \right] \right. \\ \left. R_{\sigma \omega_k} \prod_{a=1}^k  \left[ J_a^{(\sigma(a))}(z_{\sigma(a)}) \prod_{b=a+1}^k \tR^-_{\sigma(a)\sigma(b)} \left(\frac {z_{\sigma(a)}}{z_{\sigma(b)}} \right) \right] \oR_\sigma \right) 
\end{multline} 
and:
$$
\frac 1{(q^2-1)^k} \Big \langle I^+_1 * ... * I^+_k, J^-_1 * ... * J^-_k \Big \rangle \text{ according to \eqref{eqn:symm pair 2}} = \int_{|z_1| \ll ... \ll |z_k|} \sum_{\sigma \in S(k)}
$$
\begin{multline}
\label{eqn:latter}
\Tr \left(\prod_{1 \leq i < j \leq k} \frac {1}{f\left(\frac {z_i}{z_j} \right)} \cdot R_{\omega_k} \prod_{a=1}^k \left[ J_a^{(a)}(z_a) \prod_{b=a+1}^k \tR^-_{ab} \left(\frac {z_a}{z_b} \right) \right] \right. \\ \left. R_{\sigma \omega_k}  \prod_{a=1}^k \left[  I_a^{(\sigma(a))}(z_{\sigma(a)}) \prod_{b=a+1}^k \tR^+_{\sigma(a)\sigma(b)} \left(\frac {z_{\sigma(a)}}{z_{\sigma(b)}} \right) \right] \oR_\sigma \right) 
\end{multline}
By using the cyclic property of the trace and the straightforward identity:
\begin{equation}
\label{eqn:r identity r}
\oR_\sigma R_{\omega_k} = \sigma R_{\sigma^{-1} \omega_k} \sigma^{-1} \prod^{1 \leq i < j \leq k}_{\sigma^{-1}(i) > \sigma^{-1}(j)} f \left( \frac {z_i}{z_j} \right)
\end{equation}
(which uses \eqref{eqn:unitary r}) we may rewrite the formulas above as:
\begin{equation} 
\label{eqn:former 2}
\text{RHS of \eqref{eqn:former}} = \sum_{\sigma \in S(k)} \int_{|z_1| \ll ... \ll |z_k|} \prod^{1 \leq i < j \leq k}_{\sigma^{-1}(i) < \sigma^{-1}(j)} \frac {1}{f\left(\frac {z_i}{z_j} \right)} \cdot 
\end{equation}
$$
\Tr \left(  \prod_{a=1}^k \left[ I_a^{(a)}(z_a) \prod_{b=a+1}^k \tR^+_{ab} \left(\frac {z_a}{z_b} \right) \right] R_{\sigma \omega_k} \sigma \prod_{a=1}^k  \left[ J_a^{(a)}(z_a) \prod_{b=a+1}^k \tR^-_{ab} \left(\frac {z_a}{z_b} \right) \right] R_{\sigma^{-1} \omega_k} \sigma^{-1} \right)
$$
and:
$$
\text{RHS of \eqref{eqn:latter}} = \sum_{\sigma \in S(k)} \int_{|z_1| \ll ... \ll |z_k|} \prod^{1 \leq i < j \leq k}_{\sigma^{-1}(i) < \sigma^{-1}(j)} \frac {1}{f\left(\frac {z_i}{z_j} \right)} \cdot
$$
$$
\Tr \left(\sigma \prod_{a=1}^k \left[ I_a^{(a)}(z_a) \prod_{b=a+1}^k \tR^+_{ab} \left(\frac {z_a}{z_b} \right) \right] R_{\sigma^{-1} \omega_k} \sigma^{-1} \prod_{a=1}^k  \left[ J_a^{(a)}(z_{a}) \prod_{b=a+1}^k \tR^-_{ab} \left(\frac {z_{a}}{z_{b}} \right) \right] R_{\sigma \omega_k} \right)
$$
If we replace $\sigma \mapsto \sigma^{-1}$ and replace $z_a \mapsto z_{\sigma(a)}$ in the latter formula, then it precisely matches \eqref{eqn:former 2} (this uses the fact that $\Tr(Y) = \Tr(\sigma Y \sigma^{-1})$ for any tensor $Y$), up to the fact that the order of the contours changes:
\begin{equation} 
\label{eqn:latter 2}
\text{RHS of \eqref{eqn:latter}} = \sum_{\sigma \in S(k)} \int_{|z_{\sigma(1)}| \ll ... \ll |z_{\sigma(k)}|} \prod^{1 \leq i < j \leq k}_{\sigma^{-1}(i) < \sigma^{-1}(j)} \frac {1}{f\left(\frac {z_i}{z_j} \right)} \cdot
\end{equation}  
$$
\Tr \left(\prod_{a=1}^k \left[ I_a^{(a)}(z_a) \prod_{b=a+1}^k \tR^+_{ab} \left(\frac {z_a}{z_b} \right) \right] R_{\sigma \omega_k} \sigma \prod_{a=1}^k  \left[ J_a^{(a)}(z_a) \prod_{b=a+1}^k \tR^-_{ab} \left(\frac {z_a}{z_b} \right) \right] R_{\sigma^{-1} \omega_k}  \sigma^{-1} \right)
$$
Therefore, we may conclude that \eqref{eqn:former} equals \eqref{eqn:latter} (which is what we need to prove), once we show that we may change the contours of the integral \eqref{eqn:former 2}:
\begin{equation}
\label{eqn:contours may be moved}
\text{from } \int_{|z_1| \ll ... \ll |z_k|} \text{ to } \int_{|z_{\sigma(1)}| \ll ... \ll |z_{\sigma(k)}|}
\end{equation}
The integrand of \eqref{eqn:former 2} has three kinds of poles: \\

\begin{itemize}[leftmargin=*]

\item $z_i = z_j q^{\pm 2}$ if $i<j$ and $\sigma^{-1}(i) < \sigma^{-1}(j)$, which arise from the zeroes of $f(x)$ \\

\item $z_i \oq_+^2 = z_j$ if $i<j$, which arise from the poles of $\tR^+(x)$ \\

\item $z_i \oq_-^2 = z_j$ if $\sigma^{-1}(i) < \sigma^{-1}(j)$, which arise from the poles of $\tR^-(x)$ \\

\end{itemize}

\noindent (note that the terms $R_{\sigma \omega_k}$ and $R_{\sigma^{-1}\omega_k}$ do not produce poles at $z_i=z_j$ in the integrand \eqref{eqn:former 2}, because such poles are canceled by the denominator of $f$). As we move the contours as in \eqref{eqn:contours may be moved}, the only poles encountered involve $z_i$ and $z_j$ for $i<j$ and $\sigma^{-1}(i) > \sigma^{-1}(j)$, so already the poles in the first bullet do not come up. Meanwhile, the poles in the second bullet may come up, and in the remainder of this proof, we will show that the corresponding residue is 0 (the situation of the poles in the third bullet is analogous, so we skip it). To this end, recall that:
$$
\tR_{12}^{+} (z)^{\dagger_1} \tR_{21}^{-} \left(\frac 1z \right)^{\dagger_1} = \text{Id}_{V \otimes V}
$$
by the definition of $\tR^-$ in \eqref{eqn:definition r minus}. This implies the identity:
$$
\Tr_{V \otimes V} \left( \tR_{21}^{+} (z) A_2 \tR_{12}^{-} \left(\frac 1z \right)  B_1 \right) = \Tr(A) \Tr(B)
$$
for any $A, B \in \End(V)$. Taking the residue at $z = \oq^{-2}$, we obtain:
\begin{equation}
\label{eqn:vv}
\Tr_{V \otimes V} \left( (12) A_2 \tR_{12}^{-}(\oq^{2}) B_1 \right) = 0
\end{equation}
We may generalize the formula above to:
\begin{equation}
\label{eqn:trace identity}
\Tr_{V^{\otimes k}} \left( (ij) A_{1...\widehat{i}...k}  \tR_{ij}^-(\oq^{2})  B_{1... \widehat{j}....k} \right)= 0
\end{equation} 
$\forall i \neq j$ and $A, B \in \End(V^{\otimes k-1})$. Formula \eqref{eqn:vv} implies \eqref{eqn:trace identity} because none the indices, other than the $i$--th and $j$--th, play any role in the vanishing of the trace. 

\begin{figure}[h]    
	\centering
	\includegraphics[scale=0.31]{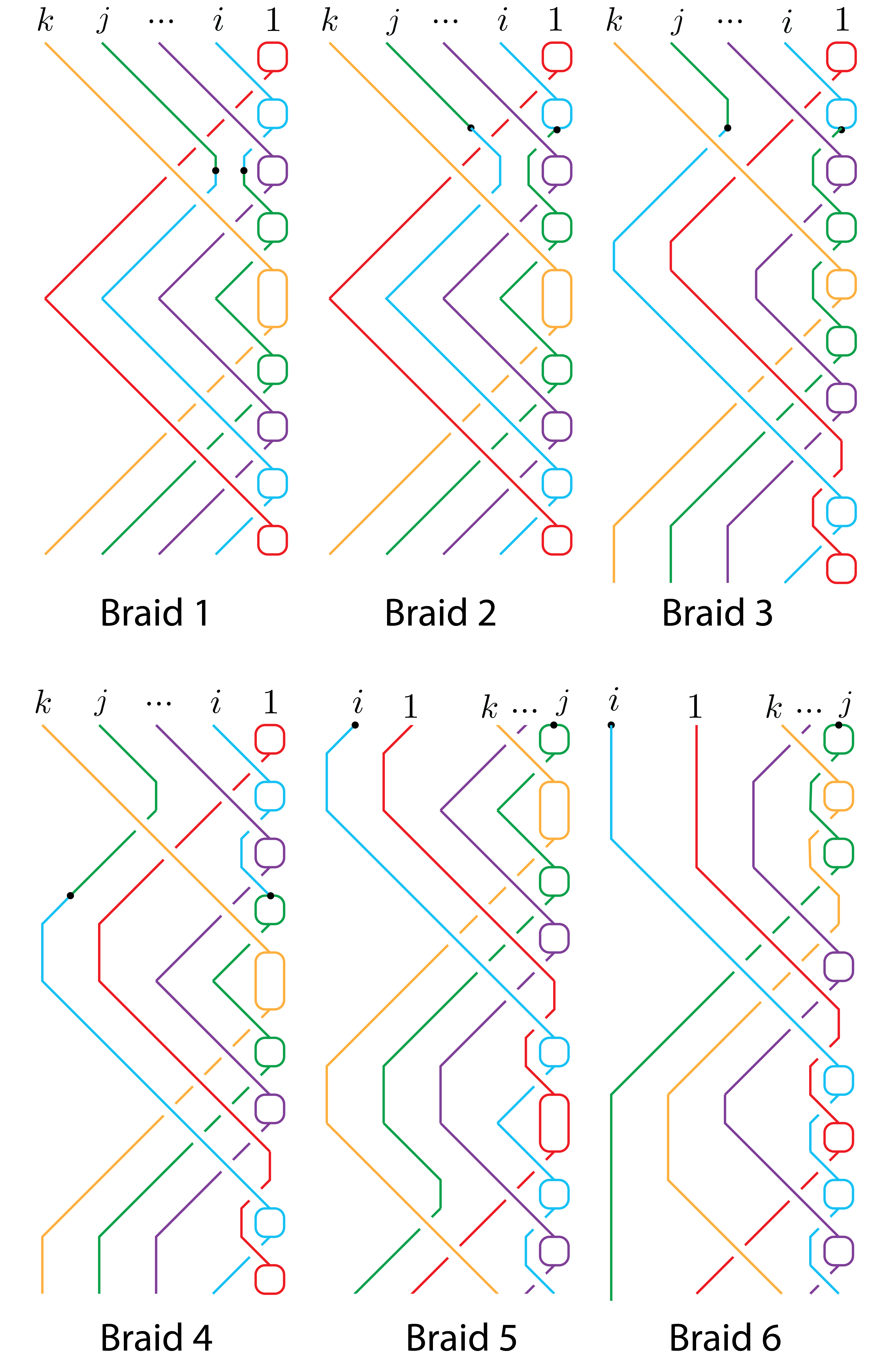}
\end{figure}

\noindent We will use \eqref{eqn:trace identity} to prove that the residue of the integrand of \eqref{eqn:former 2} at $z_i \oq_+^2 = z_j$ vanishes. Let us consider the expression on the second line of \eqref{eqn:former 2} for $\sigma = \omega_k$ and take its residue at $z_i \oq_+^2 = z_j$. The corresponding quantity is precisely represented by Braid 1 on the previous page (we draw braids top-to-bottom instead of left-to-right, for better legibility), if we make the convention that the variable on each strand is multiplied by $\oq_+^2$ and $\oq_-^2$ (respectively) as soon as it reaches the first and the second (respectively) box on the strand in question. Braids 1,2,3 and 4 are equivalent due to Reidemeister moves and the move in Figure 12. Braids 4 and 5 do not represent equal endomorphisms of $V^{\otimes k}$, but they are equal upon taking the trace (since $\Tr(AB) = \Tr(BA)$). Finally, Braids 5 and 6 are equal due to Reidemeister moves. The crossing between the blue and green strands in Braid 6 represents $\tR^-_{ij}(\oq^2)$, and we will call the portion of the braid above (respectively below) this crossing $A$ (respectively $B$). Then identity \eqref{eqn:trace identity} applies, thus showing that Braid 6 has zero trace, precisely what we needed to prove. \\

\noindent Strictly speaking, the argument just given covers the case $i=2$, $j=4$ and $k=5$, but it is obvious that we may replace the strands labeled 1,3,5 by any number of parallel strands, thus yielding the situation of arbitrary $i,j,k$. More crucial is the fact that we have only shown the vanishing of the residue at $z_i \oq_+^2 = z_j$ of the $\sigma = \omega_k$ summand of \eqref{eqn:former 2}. The case of general $\sigma$ would require one to insert the positive braid representing the permutation $\sigma \omega_k$ at the middle of the braids above and the positive braid representing the permutation $\sigma^{-1}\omega_k$ at the bottom of the braids above. The braid moves involved are analogous to the ones just performed, with the idea being to move the crossing between the blue and green strands to the very left of all other crossings. We leave the visual depiction of this fact to the interested reader, but we stress the fact that we only need to check the vanishing of the residue for those $i<j$ such that $\sigma^{-1}(i) > \sigma^{-1}(j)$. This implies that the green and blue strands do not cross except at the two points already depicted in the braids above, and this is what allows the argument to carry through.

\end{proof}

\begin{proposition}
\label{prop:pair aff}

There exists a bialgebra pairing:
\begin{equation}
\label{eqn:pairing aff ext}
\tCA^+ \otimes \tCA^- \stackrel{\langle \cdot, \cdot \rangle }\longrightarrow \fff
\end{equation}
generated by \eqref{eqn:pairing aff} and:
\begin{align} 
&\left\langle S_2^+ (y), S_1^- (x) \right\rangle = \tR^+ \left( \frac xy \right) & &\left \langle T_2^+ (y), T_1^- (x) \right \rangle = \tR^- \left( \frac xy \right) \label{eqn:bialg pair aff 1} \\
&\left \langle S_2^+ (y), T_1^- (x) \right \rangle = R \left(\frac xy \right) f^{-1} \left( \frac xy \right)  & &\left \langle T_2^+ (y), S_1^- (x) \right \rangle = R \left(\frac xy \right) \label{eqn:bialg pair aff 2}
\end{align}
(all rational functions above are expanded as power series in the region $|x| \ll |y|$). \\

\end{proposition}

\begin{proof} The proof follows that of Proposition \ref{prop:bialg pair} very closely, so we will only sketch the main ideas and leave the details as an exercise to the reader. Take any:
\begin{align*} 
&a,b \in \{X^+, S^+(x), T^+(x), \text{ for } X \in \CA^+\} \\ 
&c \in \{X^-, S^-(x), T^-(x), \text{ for } X \in \CA^-\}
\end{align*}
and define $\langle ab, c \rangle$ to be the RHS of \eqref{eqn:bialg 1}. Then if $\sum_{i} a_i b_i = 0$ holds in $\tCA^+$, we must show that the pairing:
$$
\left \langle \sum_{i} a_i b_i, c \right \rangle
$$
thus defined is 0. If at least one of $a,b,c$ is a coefficient of either $S^\pm(x)$ or $T^\pm(x)$, then the statement in question is proved just like in Proposition \ref{prop:bialg pair}, if one is careful to expand $x$ around $\infty^{\pm 1}$ (since \eqref{eqn:symm pair 1}--\eqref{eqn:symm pair 2} also involve integrals, we need to stipulate that $x$ should be closer to $\infty^{\pm 1}$ than any of the variables $z_1,..,z_k$). \\

\noindent The remaining case is when $a,b,c$ are all in $\CA^\pm$, and we must prove that:
\begin{equation}
\label{eqn:toto}
\Big \langle A^+ * B^+, Y^- \Big \rangle = \Big \langle B^+ \otimes A^+, \Delta(Y^-) \Big \rangle 
\end{equation}
for all $A_{1...k}(z_1,...,z_k), B_{1...l}(z_1,...,z_l) \in \CA^+$ and $Y_{1...k+l}(z_1,...,z_{k+l}) \in \CA^-$. By Corollary \ref{cor:gen}, it suffices to consider $A = I_1 * ... * I_k$ and $B = I_{k+1} * ... * I_{k+l}$ for various $I_1,...,I_{k+l} \in \End(V)[z^{\pm 1}]$. In this case, we may rewrite \eqref{eqn:symm pair 1} as:
\begin{equation}
\label{eqn:symm pair 3} 
\Big \langle I^+_1 * ... * I^+_k, X^-_{1...k}(z_1,...,z_k) \Big \rangle = (q^2-1)^k \int_{|z_1| \gg ... \gg |z_k|} 
\end{equation}
$$
\Tr \left( \prod_{b=k}^1 \left[ I^{(b)}_{k+1-b}(z_b) \prod_{a=b-1}^1 \tR^+_{ba} \left( \frac {z_b}{z_a} \right) \right] \oR_{\omega_k} \frac {X_{1...k}(z_1,...,z_k)}{\prod_{1 \leq i < j \leq k} f\left(\frac {z_i}{z_j} \right)} \right) 
$$
by reversing the order of the tensor factors of $V^{\otimes k}$ and relabeling the variables $z_a \mapsto z_{k+1-a}$, as well as using the symmetry property \eqref{eqn:symm aff} of $X$. Then we have:
\begin{multline}
\label{eqn:africa} 
\text{LHS of \eqref{eqn:toto}} = (q^2-1)^{k+l} \int_{|z_1| \gg ... \gg |z_{k+l}|} \\ \Tr \left(  \prod_{b=k+l}^{1} \left[ I_{k+l+1-b}^{(b)}(z_b) \prod_{a=b-1}^{1} \tR^+_{ba} \left(\frac {z_b}{z_a} \right) \right] \oR_{\omega_{k+l}} \frac {Y_{1...k+l}(z_1,...,z_{k+l})}{\prod_{1 \leq i < j \leq k+l} f\left(\frac {z_i}{z_j} \right)} \right) 
\end{multline}
Meanwhile, the right-hand side of \eqref{eqn:toto} is computed just like the right-hand side of \eqref{eqn:big want 1}, with the specification that $\Delta(Y)$ is expanded in the region when the first $k$ tensor factors are much smaller than the last $l$ tensor factors. Thus, we obtain: 
$$
\text{RHS of \eqref{eqn:toto}} = (q^2-1)^{k+l} \int_{|z_1| \gg ... \gg |z_{k+l}|} \Tr \left( \frac {Y_{1...k+l}(z_1,...,z_{k+l})}{\prod_{1 \leq i < j \leq k+l} f\left(\frac {z_i}{z_j} \right)} \right.
$$ 
$$
 \prod_{b=k+l}^{l+1} \left[ I_{k+l+1-b}^{(b)}(z_b) \prod_{a=b-1}^{l+1} \tR^+_{ba} \left(\frac {z_b}{z_a} \right) \right] \oR_{\omega_{k}}(z_{l+1},...,z_{l+k}) \left[ \tR^+_{l+1,l} \left( \frac {z_{l+1}}{z_l} \right) ... \tR^+_{l+k,1} \left( \frac {z_{l+k}}{z_1} \right) \right]
$$
$$
\left. \prod_{b=l}^{1} \left[ I_{k+l+1-b}^{(b)}(z_b) \prod_{a=b-1}^{1} \tR^+_{ba} \left(\frac {z_b}{z_a} \right) \right] \oR_{\omega_{l}}(z_1,...,z_l)  \left[ R_{l+k,1} \left( \frac {z_{l+k}}{z_1} \right) ... R_{l+1,l} \left( \frac {z_{l+1}}{z_l} \right) \right] \right) 
$$
We may move $\oR_{\omega_k}$ to the very right of the expression above, and obtain precisely \eqref{eqn:africa}. This completes the proof of the fact that the pairing \eqref{eqn:pairing aff ext} respects \eqref{eqn:bialg 1}. The situation of \eqref{eqn:bialg 2} is analogous, so we leave it as an exercise to the reader. 

\end{proof}

\subsection{} Having proved Proposition \ref{prop:pair aff}, we may construct the Drinfeld double:
\begin{equation}
\label{eqn:double aff}
\CA = \tCA^+ \otimes \tCA^{-,\op,\coop} \Big/ \left( s^+_{[i;i)} s^-_{[i;i)} - 1\right)
\end{equation}
We will often use the notation $\psi_i = (s^+_{[i;i)})^{-1}$.  \\

\begin{proposition}
\label{prop:commutators}

We have the following commutation relations in the algebra $\CA$:
\begin{equation}
\label{eqn:formula 1 final} 
S_0^\mp(w) \cdot_\pm X^\pm = 
\end{equation} 
$$
= \tR^\pm_{0k} \left(\frac {w c}{z_k} \right) ... \tR^\pm_{01} \left(\frac {w c}{z_1} \right) X^\pm R_{01} \left(\frac {w c}{z_1} \right) ... R_{0k} \left(\frac {w c}{z_k} \right) \cdot_\pm S_0^\mp(w) 
$$
\begin{equation}
\label{eqn:formula 2 final}
X^\pm \cdot_\pm T_0^\mp(w) =
\end{equation}
$$
=  T_0^\mp(w) \cdot R_{k0} \left(\frac {z_k}{w c} \right) ... R_{10} \left(\frac {z_1}{w c} \right) X^\pm \tR^\pm_{10} \left(\frac {z_1}{w c} \right) ... \tR^\pm_{k0} \left(\frac {z_k}{w c} \right) 
$$
if $X^\pm = X_{1...k}^\pm(z_1,...,z_k) \in \CA^\pm$, where we recall that $\cdot_+ = \cdot$ and $\cdot_- = \cdot^{\eop}$. Finally:
\begin{equation}
\label{eqn:e plus minus final} 
\left[ \left( \frac {E_{ij}}{z^d} \right) ^+, \left( \frac {E_{i'j'}}{z^{d'}} \right)^- \right] =
\end{equation}
$$
=  (q^2-1) \sum_{k \in \BZ} \left( s^+_{[j';i+nk)} t^+_{[j;i'+n(-d-d'-k))} c^{-d'} \barc^{-1} -  t^-_{[i;j'+nk)} s^-_{[i';j+n(d+d'-k))} c^{-d} \barc \right)
$$
(we set $s^\pm_{[i;j)} = t^\pm_{[i;j)} = 0$ if $i>j$). \\

\end{proposition}

\begin{proof} Formulas \eqref{eqn:formula 1 final} and \eqref{eqn:formula 2 final} are proved just like \eqref{eqn:drinfeld 1} and \eqref{eqn:drinfeld 2} (the presence of $c$ is due to the twist in the coproduct \eqref{eqn:twisted coproduct}), and so we leave them as exercises to the interested reader. As far as \eqref{eqn:e plus minus final} is concerned, we note the definition \eqref{eqn:twisted coproduct explicit} of the coproduct implies the following analogue of formula \eqref{eqn:delta up basic}:
\begin{equation}
\label{eqn:delta up final}
\Delta \left( \frac {E_{ij}}{z^d} \right)  = \frac {E_{ij}}{z^d}  \otimes 1 + \sum_{1 \leq x, y \leq n}^{a,b \geq 0} s^+_{[x;i+na)} t^+_{[j;y+nb)} c^{d+a+b} \barc^{-1} \otimes \frac {E_{xy}}{z^{d+a+b}} 
\end{equation}
in $\tCA^+$, as well as the following analogue of \eqref{eqn:delta down basic}:
\begin{equation}
\label{eqn:delta down final}
\Delta \left( \frac {E_{i'j'}}{z^{d'}} \right) = \sum_{1 \leq x, y \leq n}^{a,b \geq 0} \frac {E_{xy}}{z^{d'-a-b}} \otimes  t^-_{[y;j'+nb)} s^-_{[i';x+na)} c^{d'-a-b} \barc + 1 \otimes \frac {E_{i'j'}}{z^{d'}}
\end{equation}
in $\tCA^{-,\op,\coop}$. Then \eqref{eqn:e plus minus final} is simply an application of \eqref{eqn:dd}.

\end{proof}

\begin{proof} \emph{of Theorem \ref{thm:main} (in the formulation of Subsection \ref{sub:formulation}):} Let us write:
$$
\CA^0 \subset \CA
$$
for the subalgebra generated by the coefficients of the series $S^\pm(x)$ and $T^\pm(x)$. Then $\CA^0$ is isomorphic to the subalgebra $\DD^0$ of Subsection \ref{sub:explicit}, because they are both isomorphic to the algebra $\CE$ of Definition \ref{def:two realizations 2}. Moreover, Theorem \ref{thm:iso} (and its analogue when $\CA^+$ is replaced by $\CA^-$) give rise to algebra isomorphisms:
$$
\Upsilon^\pm : \DD^\pm \rightarrow \CA^\pm 
$$
Putting the preceding remarks together yields an isomorphism of vector spaces:
\begin{equation}
\label{eqn:iso big}
\DD \cong \DD^+ \otimes \DD^0 \otimes \DD^- \stackrel{\Upsilon}\rightarrow \CA^+ \otimes \CA^0 \otimes \CA^- \cong \CA
\end{equation}
where the first isomorphism holds by definition, and the last isomorphism follows from \eqref{eqn:double aff}. To show that $\Upsilon$ is an algebra isomorphism, one needs to show that: 
\begin{equation}
\label{eqn:upsilon is morphism}
\Upsilon(ab) = \Upsilon(a) \Upsilon(b)
\end{equation}
for all $a,b \in \DD$. By Proposition \ref{prop:gen}, we may assume that:
$$
a = x_1...x_k \alpha y_1... y_l \quad \text{and} \quad b = x_1'...x_{k'}' \beta y_1'...y_{l'}'
$$
for various $x_i,x_i' \in \DD^+$ of $\vdeg = 1$, $y_i,y_i' \in \DD^-$ of $\vdeg = -1$, and $\alpha,\beta \in \DD^0$. To compute the left-hand side of \eqref{eqn:upsilon is morphism}, one takes:
$$
ab = x_1...x_k \alpha y_1... y_l x_1'...x_{k'}' \beta y_1'...y_{l'}'
$$
and uses relations \squiggly{\eqref{eqn:formula 1}--\eqref{eqn:formula 4}, \eqref{eqn:formula 5}--\eqref{eqn:formula 8} and \eqref{eqn:plus minus}} to write it as:
\begin{equation}
\label{eqn:ab}
ab = \sum \text{coefficient} \cdot x_1''...x_u'' \gamma y_1''...y_v''
\end{equation}
for various $x_i'' \in \DD^+$ of $\vdeg = 1$, $y_i'' \in \DD^-$ of $\vdeg = -1$, and $\gamma \in \DD^0$. Similarly:
$$
\Upsilon(a)\Upsilon(b) = \Upsilon(x_1)...\Upsilon(x_k) \Upsilon(\alpha) \Upsilon(y_1)... \Upsilon(y_l) \Upsilon(x_1')...\Upsilon(x_{k'}') \Upsilon(\beta) \Upsilon(y_1')...\Upsilon(y_{l'}')
$$
can be expressed using relations \squiggly{\eqref{eqn:ext rel 4 aff new}--\eqref{eqn:ext rel 5 aff new}, \eqref{eqn:formula 1 final}--\eqref{eqn:formula 2 final} and \eqref{eqn:e plus minus final}} as:
\begin{equation}
\label{eqn:ups a ups b}
\Upsilon(a)\Upsilon(b) = \sum \text{coefficient} \cdot \Upsilon(x_1'')...\Upsilon(x_u'') \Upsilon(\gamma) \Upsilon(y_1'')...\Upsilon(y_v'')
\end{equation}
where the coefficients and the various $x_i'',y_i'',\gamma$ are the same ones as in \eqref{eqn:ab}. The reason for the latter fact is that the squiggly underlined relations above match each other pairwise. By its very definition in \eqref{eqn:iso big}, $\Upsilon$ applied to the right-hand side of \eqref{eqn:ab} is precisely the right hand side of \eqref{eqn:ups a ups b}, thus completing the proof. 

\end{proof}

\subsection{} We will now study the bialgebra pairing between $\tCA^+$ and $\tCA^-$ in more detail, with the goal of proving certain formulas that will be used in \cite{Rectangular}. Let us consider the restriction of the pairing \eqref{eqn:pairing aff ext} to the following subalgebras:
\begin{equation}
\label{eqn:pairing aff restrict}
\tCB_\mu^+ \otimes \tCB_\mu^- \xrightarrow{\langle \cdot , \cdot \rangle} \fff 
\end{equation}
for all $\mu \in \BQ \sqcup \infty$. \\

\begin{proposition}
\label{prop:restrict}

For all $\mu \in \BQ \sqcup \infty$, the pairing \eqref{eqn:pairing aff restrict} is a bialgebra pairing, i.e. it intertwines the product with the coproduct $\Delta_\mu$, in the sense of \eqref{eqn:bialg 1}--\eqref{eqn:bialg 2}. \\

\end{proposition}

\begin{proof} Let us prove \eqref{eqn:bialg 1}, and leave \eqref{eqn:bialg 2} as an exercise to the interested reader. Take formula \eqref{eqn:toto}, which we already proved in the course of Proposition \ref{prop:pair aff}:
\begin{equation}
\label{eqn:tina}
\Big \langle A^+ * B^+, Y^- \Big \rangle = \Big \langle B^+ \otimes A^+, \Delta(Y^-) \Big \rangle 
\end{equation}
If we let $A^+,B^+ \in \CB_\mu^+$ and $Y^- \in \CB_\mu^-$, then our task is equivalent to showing that the formula above holds with $\Delta$ replaced by $\Delta_\mu$. Since $Y^- \in \CB_\mu^-$, we have:
\begin{equation}
\label{eqn:david}
\Delta(Y^-) = \Delta_\mu(Y^-) + (\text{slope}< \mu) \otimes (\text{slope} > \mu)
\end{equation}
(the reason why the formula above differs from \eqref{eqn:leading order} is that slope lines are reflected across the horizontal line when going from $\CA^+$ to $\CA^-$, due to opposite vertical grading conventions). All the summands in the right-hand side of \eqref{eqn:david} other than $\Delta_\mu(Y^-)$ pair trivially with $B^+ \otimes A^+$, for degree reasons. This implies that \eqref{eqn:tina} holds with $\Delta$ replaced by $\Delta_\mu$.

\end{proof}

\noindent As a consequence of Proposition \ref{prop:restrict}, the Drinfeld double:
$$
\CB_\mu = \tCB_\mu^+ \otimes \tCB_\mu^{-,\op,\coop} \Big/ \Big( \text{identify } \psi_i, c \text{ from the two tensor factors}\Big)
$$
defined with respect to the coproduct $\Delta_\mu$, will be a subalgebra of the algebra $\CA$ of \eqref{eqn:double aff}. Moreover, the following diagram commutes, for all $\mu \in \BQ \sqcup \infty$:
\begin{equation}
\label{eqn:vertical diagram}
\xymatrix{
\CE_\mu \ar[d]_{\simeq} \ar@{^{(}->}[r] & \DD \ar[d]^{\simeq} \\
\CB_\mu \ar@{^{(}->}[r] & \CA}
\end{equation}
where the vertical map on the left is the natural double of the isomorphism $\Upsilon_\mu$ of Proposition \ref{prop:iso small}, and the vertical map on the right is the isomorphism $\Upsilon$ of \eqref{eqn:iso big}.  \\

\subsection{} Let us consider any $\mu \in \BQ \sqcup \infty$. As we have seen in Subsection \ref{sub:slope stuff}, the subalgebra $\CE_\mu \subset \DD$ is generated by the elements:
$$
\Big\{ f_{\pm [i;j)}^{(\pm k)}, \barf_{\pm [i;j)}^{(\pm k)} \Big\}_{(i,j) \in \zzz, k \in \BN}^{\mu = \frac {j-i}k} \quad \text{of \eqref{eqn:f mu}}
$$
which satisfy the coproduct relations of Proposition \ref{prop:endowed}, as well as formulas \eqref{eqn:f alpha}--\eqref{eqn:barf alpha} for their images under the maps $\alpha_{\pm [i;j)}$. When we pass the elements above through the vertical isomorphisms of diagram \eqref{eqn:vertical diagram}, we obtain:
$$
\Big\{ F_{\pm [i;j)}^{(\pm k)}, \bF_{\pm [i;j)}^{(\pm k)} \Big\}_{(i,j) \in \zzz, k \in \BN}^{\mu = \frac {j-i}k} \in \CB_\mu
$$
defined in \eqref{eqn:def p}--\eqref{eqn:def pp} when the sign is $+$, and in \eqref{eqn:def f minus}--\eqref{eqn:def ff minus} when the sign is $-$. As proved for the case $\pm = +$ in Proposition \ref{prop:root} (the case $\pm = -$ is analogous, and we leave it as an exercise to the interested reader), the elements $F,\bF$ satisfy the analogous coproduct relations as $f,\barf$. Similarly, we have the following formulas:
\begin{align}
&\alpha_{\pm [i;j)}(F^{(\pm k)}_{\pm [i';j')}) = \delta_{(i',j')}^{(i,j)} (1-q^2) \oq_{\pm}^{\frac {\gcd(j-i,k)}n} \label{eqn:corbu} \\
&\alpha_{\pm [i;j)}(\bF_{\pm [i';j')}^{(k)}) = \delta_{(i',j')}^{(i,j)} (1-q^{-2}) \oq_{\pm}^{-\frac {\gcd(j-i,k)}n} \label{eqn:vadu}
\end{align} 
where the linear maps:
\begin{equation}
\label{eqn:alpha maps}
\bigoplus_{k=0}^\infty \CA_{\pm [i;j),\pm k} \xrightarrow{\alpha_{\pm [i;j)}}\fff
\end{equation}
are given by:
\begin{align*}
&X_{1...k}(z_1,...,z_k) \stackrel{\alpha_{+[i;j)}}\longrightarrow \text{coefficient of } E_{ji} \text{ in } X^{(k)}(y) (1- q^2)^k \oq_+^{\frac {k(i-j) + (j-i) + k  - 2k\bari}n} \\
&Y_{1...k}(z_1,...,z_k) \stackrel{\alpha_{-[i;j)}}\longrightarrow \text{coefficient of } E_{ij} \text{ in } Y^{(k)}(y) (1- q^2)^k \oq_-^{\frac {k(j-i) + (i-j) + k  - 2(k-1)\barj-2\bari}n}
\end{align*}
Recall that $X^{(k)}(y) \in \text{End}(V)[y^{\pm 1}]$ was defined by Figure 14, for any $X \in \CA_k$. Meanwhile, for any $Y \in \CA_{-k}$, the symbol $Y^{(k)}(y)$ is defined analogously, but with $\oq$ replaced by $\oq_-$, and the operators $D_1...D_k$ placed in front of the braid in Figure 14. Formulas \eqref{eqn:corbu}--\eqref{eqn:vadu} were proved when $\pm = +$ in Proposition \ref{prop:root}, and the case when $\pm = -$ is an analogous exercise that we leave to the interested reader. \\

\subsection{}

As we have seen in Subsection \ref{sub:pbw stuff}, the subalgebra $\CE_\mu$ is also generated by the primitive elements:
$$
\Big\{ p_{\pm [i;i+a)}^{(\pm b)}, p_{\pm l\bde,r}^{\left(\pm \frac {ln}{\mu} \right)} \Big\}_{i \in \BZ/n\BZ}^{l \in \BN \frac ag, r \in \BZ/g\BZ} \quad \text{of \eqref{eqn:primitive elements}}
$$
where $\mu = \frac ab$ with $\gcd(a,b) = 1$, $b \geq 0$ and $g = \gcd(n,a)$. These primitive elements satisfy formulas \eqref{eqn:simple alpha}--\eqref{eqn:imaginary alpha}. When we pass the elements above through the vertical isomorphisms of diagram \eqref{eqn:vertical diagram}, we obtain:
\begin{equation}
\label{eqn:primitive in b}
\Big\{ P_{\pm [i;i+a)}^{(\pm b)}, P_{\pm l\bde,r}^{\left(\pm \frac {ln}{\mu} \right)} \Big\}_{i \in \BZ/n\BZ}^{l \in \BN \frac ag, r \in \BZ/g\BZ} \in \CB_\mu 
\end{equation}
The elements \eqref{eqn:primitive in b} are primitive for the coproduct $\Delta_\mu$, and moreover satisfy the following analogues of formulas \eqref{eqn:simple alpha} and \eqref{eqn:imaginary alpha}:
\begin{align}
&\alpha_{\pm [u;v)} \left(P_{\pm [i;i+a)}^\mu\right) = \pm \delta_{(u,v)}^{(i,i+a)} \label{eqn:simple alpha in b} \\
&\alpha_{\pm [s;s+ln)} \left(P_{\pm l\bde,r}^\mu\right) = \pm \delta_{s \text{ mod }g}^r \label{eqn:imaginary alpha in b}
\end{align}

\subsection{} Proposition \ref{prop:endowed} gives us formulas for the coproducts of the elements \eqref{eqn:f mu}. Meanwhile, the elements \eqref{eqn:primitive elements} are primitive, so there are no intermediate terms in their coproduct. Therefore, we may apply relation \eqref{eqn:dd} between the two halves of the bialgebra $\CB_\mu$, and obtain the following formulas: 
\begin{align*}
&\Big [ P_{\pm [i;j)}^\mu, F_{\mp [i';j')}^\mu \Big ] = \Big \langle P_{\pm [i;j)}^\mu, F_{\mp [i';j')}^\mu \Big \rangle \left[ \left( \frac {\psi_i}{\psi_j} \barc^{\frac {i-j}{\mu}} \right)^{\pm 1} - \left( \frac {\psi_{j'}}{\psi_{i'}} \barc^{\frac {j'-i'}{\mu}} \right)^{\pm 1} \right] \\
&\Big [ P_{\pm [i;j)}^\mu, \bF_{\mp [i';j')}^\mu \Big ] = \Big \langle P_{\pm [i;j)}^\mu, \bF_{\mp [i';j')}^\mu \Big \rangle \left[ \left( \frac {\psi_i}{\psi_j} \barc^{\frac {i-j}{\mu}} \right)^{\pm 1} - \left( \frac {\psi_{j'}}{\psi_{i'}} \barc^{\frac {j'-i'}{\mu}} \right)^{\pm 1} \right] 
\end{align*}
for all $(i,j),(i',j') \in \zzz$ such that $j-i=j'-i' \in \BN \mu$. Similarly, we have:
\begin{align*}
&\Big [ P_{\pm l\bde,r}^\mu, F_{\mp [i';j')}^\mu \Big ] = \Big \langle P_{\pm l\bde, r}^\mu, F_{\mp [i';j')}^\mu \Big \rangle \left[ c^{\mp l} \barc^{\mp \frac {nl}{\mu}}  - c^{\pm l} \barc^{\pm \frac {nl}{\mu}} \right] \\
&\Big [ P_{\pm l\bde,r}^\mu, \bF_{\mp [i';j')}^\mu \Big ] = \Big \langle P_{\pm l\bde, r}^\mu, \bF_{\mp [i';j')}^\mu \Big \rangle \left[ c^{\mp l} \barc^{\mp \frac {nl}{\mu}}  - c^{\pm l} \barc^{\pm \frac {nl}{\mu}} \right] 
\end{align*}
for all $l \in \BZ$ and $(i',j') \in \zzz$ such that $nl = j'-i' \in \BN \mu$, and $r \in \BZ/g\BZ$ where $g = \gcd(n,\text{numerator } \mu)$. Similar formulas were worked out in \cite{PBW} between the $p$ and $f,\barf$ generators in $\CE_\mu$, but with explicit numbers instead of the pairings in the right-hand side. Therefore, the fact that $\Upsilon_\mu$ is an isomorphism implies the following explicit formulas for the pairings above:
\begin{align}
&\Big \langle P_{\pm [i;j)}^\mu, F_{\mp [i';j')}^\mu \Big \rangle = \mp \delta^{(i,j)}_{(i',j')} \cdot \oq_\pm^{-\frac {\gcd(k,j-i)}n} \label{eqn:main pair 1} \\
&\Big \langle P_{\pm [i;j)}^\mu, \bF_{\mp [i';j')}^\mu \Big \rangle = \pm \delta^{(i,j)}_{(i',j')} \cdot \oq_\pm^{\frac {\gcd(k,j-i)}n} \label{eqn:main pair 2} \\
&\Big \langle P_{\pm l\bde, r}^\mu, F_{\mp [i';j')}^\mu \Big \rangle = \mp \delta_{i' \text{ mod }g}^r \cdot \oq_\pm^{-\frac {\gcd(k,nl)}n} \label{eqn:main pair 3} \\
&\Big \langle P_{\pm l\bde, r}^\mu, \bF_{\mp [i';j')}^\mu \Big \rangle = \pm \delta_{i' \text{ mod }g}^r \cdot \oq_\pm^{\frac {\gcd(k,nl)}n} \label{eqn:main pair 4} 
\end{align}
for all applicable indices. Comparing the formulas above with \eqref{eqn:simple alpha in b}--\eqref{eqn:imaginary alpha in b} yields: \\

\begin{proposition}
\label{prop:pair f}

For any $(i,j) \in \zzz$ and $k \in \BN$ such that $\mu = \frac {j-i}k$, we have:
\begin{align}
&\left \langle X, \bF_{-[i;j)}^{(-k)} \right \rangle = \alpha_{[i;j)}(X) \cdot \oq_+^{\frac {\gcd(j-i,k)}n} \label{eqn:pair f 1} \\
&\left \langle \bF_{[i;j)}^{(k)}, Y \right \rangle = \alpha_{-[i;j)}(Y) \cdot \oq_-^{\frac {\gcd(j-i,k)}n} \label{eqn:pair f 2}
\end{align}
for all $X \in \CB_\mu^+$ and $Y \in \CB_{\mu}^-$. \\

\end{proposition}

\begin{proof} We will only prove \eqref{eqn:pair f 2} as it will be used in \cite{Rectangular}, and leave the analogous formula \eqref{eqn:pair f 1} as an exercise to the interested reader. Comparing formulas \eqref{eqn:simple alpha in b}, \eqref{eqn:imaginary alpha in b} with \eqref{eqn:main pair 2}, \eqref{eqn:main pair 4} shows us that formula \eqref{eqn:pair f 2} holds when $Y$ is one of the primitive generators of $\CB_\mu^-$. Therefore, all that remains to show is that if \eqref{eqn:pair f 2} holds for $Y,Y' \in \CB_\mu^-$, then it also holds for $Y * Y'$. This follows by comparing:
$$
\left \langle \bF_{[i;j)}^\mu, Y * Y' \right \rangle \stackrel{\eqref{eqn:bialg 2},\eqref{eqn:coproduct pp}}= \begin{cases} \left \langle \bF_{[i;s)}^\mu, Y \right \rangle \left \langle \bF_{[s;j)}^\mu, Y' \right \rangle  &\text{if } \exists s \text{ s.t. } \hdeg Y = -[i;s), \hdeg Y' = -[s;j) \\ 0 &\text{otherwise} \end{cases} 
$$
with \eqref{eqn:pseudo}.

\end{proof}


\begin{thebibliography}{XXX}	

\bibitem{B} Beck J., {\em Braid group action and quantum affine algebras}, Commun. Math. Physics vol. 165 (1994), 555–568

\bibitem{BS}  Burban I., Schiffmann O., {\em On the Hall algebra of an elliptic curve, I}, Duke Math. J. 161 (2012), no. 7, 1171--1231

\bibitem{DF} Ding J., Frenkel I., {\em Isomorphism of two realizations of quantum affine algebra $U_q(\hgl_n)$}, Comm. Math. Phys. 156 (1993), no. 2, 277--300

\bibitem{DI} Ding J., Iohara K., {\em Generalization of Drinfeld quantum affine algebras}, Lett. Math. Phys., 41 (1997), no. 2, 181--193

\bibitem{D} Drinfeld V. G., {\em A new realization of Yangians and quantized affine algebras}, Soviet Math. Dokl. 36 (1988), 212--216

\bibitem{E} Enriquez B., {\em On correlation functions of Drinfeld currents and shuffle algebras}, Transform. Groups 5 (2000), no. 2, 111 - 120

\bibitem{FRT} Faddeev L., Reshetikhin N., Takhtajan L., {\em Quantization of Lie groups and Lie algebras}, Leningrad Math. J.1(1990) 193--226
	
\bibitem{FHHSY} Feigin B., Hashizume K., Hoshino A., Shiraishi J., Yanagida S., {\em A commutative algebra on degenerate $\BC \BP^1$ and MacDonald polynomials},  J. Math. Phys. 50 (2009), no. 9

\bibitem{FJMM} Feigin B., Jimbo M., Miwa T., Mukhin E., {\em Representations of quantum toroidal} $\fgl_n$, Journal of Algebra vol. 380 (2013), 78--108

\bibitem{FO} Feigin B., Odesskii A., {\em Vector bundles on elliptic curve and Sklyanin algebras}, Topics in Quantum Groups and Finite-Type Invariants, Amer. Math. Soc. Transl. Ser. 2, 185 (1998), Amer. Math. Soc., 65--84	
	
\bibitem{M} Miki K., {\em A $(q, \gamma)$ analog of the $W_{1+\infty}$ algebra}, J. Math. Phys., 48 (2007), no. 12
	
\bibitem{Mu} Mudrov A. I., {\em Reflection equation and twisted Yangians}, Journal of Mathematical Physics 48, 093501 (2007)

\bibitem{Shuf} Negu\cb t A., {\em Shuffle algebra revisited}, Int. Math. Res. Not., Volume 2014, Issue 22, 2014, 6242--6275
		
\bibitem{Tor} Negu\cb t A., {\em Quantum toroidal and shuffle algebras}, Adv. Math., Volume 372 (2020), 107288

\bibitem{PBW} Negu\cb t A., {\em PBW basis for $\UU$}, Transform. Groups (2022), doi.org/10.1007/s00031-022-09696-x

\bibitem{Rectangular} Negu\cb t A., {\em Deformed $W$--algebras in type $A$ for rectangular nilpotent}, Commun. Math. Phys. 389, 153--195 (2022) 

\bibitem{S} Schiffmann O., {\em Drinfeld realization of the elliptic Hall algebra}, Journal of Algebraic Combinatorics, vol 35 (2012), no 2, 237--262 
	
\bibitem{W} Wendlandt C., {\em The $R$-Matrix Presentation for the Yangian of a Simple Lie Algebra}, Comm. Math. Phys., October 2018, Volume 363, Issue 1, 289--332

\end{thebibliography}
\end{document}